\documentclass[letterpaper, 11pt]{article}

\usepackage[utf8]{inputenc}
\usepackage{geometry}
\usepackage[utf8]{inputenc}
\usepackage[colorlinks=true,linkcolor=blue,urlcolor=black,bookmarksopen=true, bookmarksopenlevel=1]{hyperref}
\usepackage{bookmark}
\usepackage{url}            % simple URL typesetting
\usepackage{booktabs}       % professional-quality tables
\usepackage{amsfonts,enumitem}       % blackboard math symbols
\usepackage{nicefrac}       % compact symbols for 1/2, etc.
\usepackage{microtype}      % microtypography

\usepackage{algorithmic}
\usepackage[linesnumbered,ruled,vlined]{algorithm2e}

\usepackage{graphicx}
\usepackage{subfigure}
\usepackage{booktabs} % for professional tables
\usepackage{amsmath}
\usepackage{amsthm}
\usepackage{float}
\usepackage{bm}
\usepackage{amssymb}
\usepackage{bbm}
\usepackage{color}
\usepackage{mathtools}
\usepackage[super]{nth}
\usepackage{etoolbox}
\usepackage{adjustbox}
\usepackage{enumitem}

\newtheorem{propo}{Proposition}[section]
\newtheorem{lemma}[propo]{Lemma}
\newtheorem{definition}[propo]{Definition}
\newtheorem{coro}[propo]{Corollary}
\newtheorem{thm}{Theorem}
\newtheorem{asmp}{Assumption}
\newtheorem{theorem}[propo]{Theorem}

\newcommand{\ip}[2]{\left\langle #1, #2 \right \rangle}

\def\reals{{\mathbb R}}
\def\prob{{\mathbb P}}

\def\cN{{\cal N}}
\def\eps{\varepsilon}

\def\E{\mathbb E}

\def\unsafe{{\sf UNSAFE}}
\def\safe{{\sf SAFE}}
\def\sens{\Delta}
\def\thresh{\tau}
\def\backoff{(k^*+3)\sens}
\def\popdist{D_{\phi}}
\def\robdist{D}
\def\HPTR{{\rm HPTR}}

\newcommand{\cB}{\mathcal B}

\newcommand{\cD}{\mathcal D}

\newcommand{\cM}{\mathcal M}

\newcommand{\cP}{\mathcal P}

\newcommand{\cT}{\mathcal T}

\newcommand{\cV}{\mathcal V}

\newcommand{\cX}{\mathcal X}
\newcommand{\cY}{\mathcal Y}

\DeclareMathOperator{\Tr}{Tr}

\title{Differential privacy and robust statistics in high dimensions}

\author{%
  Xiyang Liu\footnotemark[0] \thanks{Paul G. Allen School of Computer Science \& Engineering, 
  University of Washington, 
  \texttt{xiyangl@cs.washington.edu}} 
 \and Weihao Kong\footnotemark[2] \thanks{Google Research, \texttt{kweihao@gmail.com}} 
    \and
  Sewoong Oh\footnotemark[0] \thanks{Paul G. Allen School of Computer Science \& Engineering, 
  University of Washington, 
  \texttt{sewoong@cs.washington.edu}} 
}

\date{}

\usepackage{natbib}
\usepackage{graphicx}

\begin{document}

\maketitle

\begin{abstract}
We introduce a universal framework for characterizing the statistical efficiency of a 
statistical estimation problem with differential privacy guarantees. Our framework, which we call High-dimensional Propose-Test-Release (HPTR), builds upon three crucial components: the exponential mechanism from \cite{mcsherry2007mechanism}, robust statistics, and the  Propose-Test-Release mechanism from \cite{dwork2009differential}. 
Gluing all these together is the concept of resilience, which is central to robust statistical estimation. Resilience guides  the design of the algorithm, the sensitivity analysis, and the success probability analysis of the test step in Propose-Test-Release. 
The key insight is that if we design an exponential mechanism that accesses the data only via one-dimensional robust statistics, then the resulting local sensitivity can be dramatically reduced. Using resilience, we can  provide tight local sensitivity bounds. 
These tight bounds readily translate into near-optimal utility guarantees in several cases.
We give a general recipe for applying HPTR to a given instance of a statistical estimation problem and demonstrate it on canonical problems of mean estimation, linear regression, covariance estimation, and principal component analysis. 
We introduce a general utility analysis technique that proves that
 HPTR nearly achieves the optimal sample complexity under several scenarios studied in the literature.

\end{abstract} 
\newpage
\tableofcontents

\newpage
\section{Introduction} 
\label{sec:intro}

Estimating a parameter of a  distribution from i.i.d.~samples is a canonical problem in statistics. For such problems, characterizing the computational and statistical cost of ensuring differential privacy (DP) has gained significant interest with the rise of  DP as the de facto measure of privacy. This is spearheaded by exciting and foundational algorithmic advances, e.g.,  \cite{barber2014privacy,KV17,KLSU19,kamath2020private,cai2019cost}. However, the computational efficiency of these algorithms  often comes at the cost of  requiring superfluous assumptions that are not necessary for statistical efficiency, such as known bounds on the parameters or knowledge of higher-order moments. Without such assumptions,  the optimal sample  complexity remains unknown even for 
 canonical statistical estimation  problems under  differential privacy. 
Further,  each algorithm needs to be  customized to those assumptions or to the problem instances.

We take an alternative route of focusing only on the  statistical cost of differential privacy without concerning computational efficiency. Our goal is to introduce a general unifying framework that $(1)$ can be readily applied to each problem instance, $(2)$ provides  a tight characterization of the statistical complexity involved, and $(3)$ requires  minimal assumptions. 
Achieving this goal critically relies on 
three key ingredients: the exponential mechanism introduced in  \cite{mcsherry2007mechanism}, robust statistics, and  the Propose-Test-Release mechanism introduced in \cite{dwork2009differential}.  
We first explain 
 these three components of our approach,  and then demonstrate the utility of  our proposed framework, called High-dimensional Propose-Test-Release  (HPTR), in canonical example problems of mean estimation, linear regression, covariance estimation, and principal component analysis. 

\medskip\noindent
{\bf  Exponential mechanism and sensitivity.} 
Differential privacy (DP) is  an  agreed upon measure of privacy that provides plausible deniability to the individual entries.
Given a dataset $S$ of size $n$ and its empirical distribution  $\hat p_S=(1/n)\sum_{x_i\in S} \delta_{x_i}$, its {\em neighborhood} is defined as ${\cal N}_{S} = \{ S' :  |S'| =|S|, d_{\rm 
TV}(\hat p_S,\hat p_{S'} )\leq 1/n \}$, which is a set of datasets at Hamming distance\footnote{There are two notions of a neighborhood in DP, which are equally popular. We use the one based on exchanging an entry, but all the analyses can seamlessly be applied to the one that allows for insertion and deletion of an entry.} at most one from $S$, and $d_{\rm TV}(\cdot)$ is the total variation. 
Plausible deniability is achieved by introducing the right amount of randomness. A randomized estimator $\hat\theta(S)$ is said to be  $(\varepsilon,\delta)$-differentially private for some target $\varepsilon\geq 0$ and $\delta\in[0,1]$ if  
 $\prob(\hat \theta(S)\in A) \leq e^\varepsilon \prob(\hat\theta(S')\in A) +\delta $ for all neighboring datasets $S,S'$ and all measurable subset $A\subseteq\reals^p$ \cite{dwork2006calibrating}. 
 Consider a binary hypothesis testing on two hypotheses, $H_0$, where the estimate came from a dataset $S$,  and $H_1$, where the the estimate came from a dataset $S'$ that is a neighbor of $S$.  The DP condition with a sufficiently small $(\varepsilon,\delta)$  ensures that an adversary cannot  succeed in this test with high confidence \cite{composition}, which provides   plausible deniability.

%for all neighborhood ${\cal N}_{\hat p_n}$ while making estimation error,  $\ell(\hat\Theta(\hat p_n) ,\theta) $ for some choice of a loss $\ell(\cdot,\cdot)$, small  in expectation or with high probability.  ${\cal R}_{\varepsilon,\delta}$is a piecewise linear function defined by $\varepsilon$ and $\delta$. This binary hypothesis testing against all neighboring dataset is motivated by a powerful adversary who knows all the entries (that are shared between ) but one ($\hat p_n$ and $\hat q_n$). small ROC for the pair () as measured by the ${\cal R}_{\varepsilon,\delta}$ ensures that that powerful adversary cannot test with high confidence what the remaining data point is.

The {\em sensitivity} plays a  crucial role in designing DP estimators.
Consider an example of mean estimation, where the error is measured in the Mahalanobis distance defined as  $D_p(\hat\mu) = \|\Sigma_p^{-1/2}(\hat\mu - \mu_p)\|$, where $\mu_p$ and $\Sigma_p$ are the mean and covariance of the sample-generating distribution $p$. 
This is a preferred error metric since it has unit variance in all directions and is invariant to a linear transformation of the samples. 
A corresponding  empirical loss is  $D_{\hat p_S}(\hat\mu) = \|\Sigma_{\hat p_S}^{-1/2}(\hat\mu - \mu_{\hat p_S})\|$.  
The exponential mechanism from \cite{mcsherry2007mechanism} produces an $(\varepsilon,0)$-DP estimate $\hat \mu$ by sampling from a distribution that approximately and stochastically minimizes this empirical loss: 
\begin{eqnarray*}
    \hat\mu\;\sim\; \frac{1}{Z(S)}e^{-\frac{\varepsilon}{2\sens} D_{\hat p_S} (\hat\mu) }\;,
\end{eqnarray*}
where $Z(S)=\int \exp\{-(\varepsilon/2\sens)D_{\hat p _S}(\hat \mu) d\hat\mu$.
The  sensitivity   is defined as  $\sens :=\max_{\hat\mu,S,S'\in{\cal N}_{S}}  | D_{\hat p_S}(\hat \mu) - D_{\hat p_{S'}} (\hat \mu) |$, which is  the influence of one data point on the loss.
From this definition, the $(\varepsilon,0)$-DP guarantee follows immediately (e.g., Lemma~\ref{lem:exp_mech}).

 Using the exponential mechanism is crucial in HPTR for two reasons: adaptivity and flexibility. First, it naturally adapts to the geometry of the problem, which is encoded in the loss. 
For example,  a more traditional Gaussian mechanism \cite{dwork2014algorithmic} needs  to  estimate $\Sigma_{p}$ privately in order to add a Gaussian noise tailored to that estimated $\Sigma_p$, 
which increases sample complexity significantly. 
On the other hand, the exponential mechanism 
seamlessly adapts to $\Sigma_p$  without explicitly and privately estimating it. 
Further, the exponential mechanism allows us significant flexibility to design different  loss functions, some of which can dramatically reduce the sensitivity.  Discovering such a loss function is the main focus of this paper. 

%Typically, a Lipshitz continuous estimator $\tilde \theta$ is used with an explicit bound in the domain of the samples $x_i \in  B_r$ to bound the sensitivity. Once we have a bound on the sensitivity, one can immediately get a $(\varepsilon\delta)$-DP estimate using, for example, the exponential mechanism: $\hat\Theta \sim   (1/Z) e^{-(\varepsilon/(2\sens))\ell(\hat\theta,\hat\theta(\hat p_n))}$. Critical is that we know $B_r$ prior to collecting any data. However, in statistical estimation we do not have such a  bounded domain; 
%{\bf Sewoong. This might be too much detail: Typical remedy is to adaptively  find the domain that includes most of the data points and reveal the domain privately using  part of the privacy budget. Once revealed, this is a public information, and standard privatization mechanisms can be applied subsequently.  However, this approach require additional assumptions, is not agnostic and requires knowledge of other information, etc. } 

One major challenge  is that the sensitivity is unbounded when the support of the  distribution is unbounded. 
A common  solution is to privately  estimate a bounded domain that the samples lie in  and use it to bound the sensitivity (e.g., \cite{KV17,KLSU19,liu2021robust}). 
%Diverting from typical approaches that use part of the privacy budget to privately estimate the domain, 
We propose a fundamentally different approach  using robust statistics.

\medskip\noindent
{\bf Robust statistics and  resilience.} The {\em resilience} proposed in~\cite{steinhardt2018resilience} plays a critical role in robust statistics. 
For the mean, for example, a dataset $S$ is said to be   $(\alpha,\rho)$-resilient for some $\alpha\in[0,1]$ and $\rho>0$ if for all $v\in\reals^d$ with $\|v\|=1$ and all subset $T\subseteq S$ of size at least $|T|\geq \alpha n$, 
\begin{eqnarray}  
    \big|\, \langle v,\mu_{\hat p_T }\rangle  - \langle v,\mu_{\hat p_S} \rangle  \, \big|\; \leq \; \frac{ \rho}{\alpha } \;. 
\end{eqnarray}
A  more precise statement  is in Definition~\ref{def:resilience}.
This measures how resilient the empirical mean is to subsampling or deletion of a fraction of the samples. 
This resilience is a central concept in robust  statistical estimation when a fraction of the dataset is arbitrarily corrupted by an adversary \cite{steinhardt2018resilience,zhu2019generalized}.  
We show and exploit the fact that resilience is fundamentally related to the sensitivity of robust statistics. 

For each direction $v\in\reals^d$ with $\|v\|=1$, we construct a  robust mean of a one-dimensional projected dataset, also known as trimmed mean, $S_v=\{\langle v, x_i\rangle \in \reals \}_{x_i\in S}$, as follows. For some  $\alpha\in[0,1/2)$, remove $\alpha n$  data points corresponding to the largest entries in $S_v$ and also remove the $\alpha n$ smallest entries. The mean of the remaining $(1-2\alpha)n$ points is the robust one-dimensional mean, which we denote by $\langle v, \mu_{\hat p_v}^{(robust)} \rangle \in \reals$. 
From the resilience above, we know that the mean of the removed top part is upper bounded by $ \langle v, \mu_{\hat p_S}\rangle + \rho/\alpha $. The mean of the removed bottom part is lower bounded by $ \langle v, \mu_{\hat p_S}\rangle - \rho/\alpha$. Hence, the effective support of this robust one-dimensional mean estimator is upper and lower bounded by the same. This can be readily translated into a bound in sensitivity of the estimate, $\langle v, \mu_{\hat p_v}^{(robust)} \rangle $ (e.g., Lemma~\ref{lem:local_asmp}). 
A similar sensitivity bound holds  for the robust one-dimensional variance estimator, $v^\top \Sigma^{(robust)}_{\hat p_v}\,v,$ defined similarly.

We  propose an approach that  critically relies on this observation that {\em one-dimensional robust statistics have low sensitivity on resilient datasets,  i.e., datasets satisfying the resilience property with small $\rho$.} 

% 1-d robust - sensitivity 
This suggests that if we can design a score function  that only depends on one-dimensional robust statistics of the data, it might inherit the low sensitivity of those robust statistics. 
 To this end, we first transform the target error metric  into an equivalent expression that  only depends on one-dimensional (population)  mean, $ \langle v, \mu_p \rangle $, and variance, $v^\top \Sigma_p v$, i.e.,   $$\|\Sigma_p^{-1/2}(\hat\mu-\mu_p)\| \; =\; \max_{v\in\reals^d,\|v\|=1} \frac{\langle v, \hat\mu \rangle - \langle v, \mu_{p} \rangle }{\sqrt{v^\top \Sigma_{p}\, v} } \;,$$ 
which follows from   Lemma~\ref{lem:equidist}.  
Next, we  replace the population statistics with robust empirical ones to define a new empirical loss,  $D_{\hat p_S} (\hat\mu) = \max_{v\in\reals^d,\|v\|=1} (\langle v, \hat\mu \rangle - \langle v, \mu_{\hat p_v}^{(robust)} \rangle )/ \sqrt{v^\top \Sigma_{\hat p_v}^{\rm (robust)}\, v}$. 
Precise definitions of these robust statistics can be found in Eq.~\eqref{eq:defproj}.
For resilient datasets, 
such a score function has a  dramatically smaller sensitivity  compared to  those that rely on high-dimensional robust statistics. 
For mean estimation under a  sub-Gaussian distribution, the sensitivity of the proposed loss is $\tilde{O}(1/n)$, whereas a loss using a  high-dimensional robust statistics has 
$\Omega(\sqrt{d}/n)$ sensitivity.

Such an improved sensitivity  immediately leads to a better  utility guarantee of the exponential mechanism.  
We explicitly prescribe such loss functions for the canonical problems of mean estimation, linear regression, covariance estimation,  and principal component analysis. 
This leads to   near-optimal utility in most cases and improves upon the state-of-the-art in others, as we demonstrate in Section~\ref{sec:intro_main}. 
Further, this approach can potentially be more generally applied to a much broader class of problems. 
One remaining challenge is that the tight sensitivity bound we provide holds only for a resilient   dataset.  To reject bad datasets, we adopt the Propose-Test-Release (PTR) framework pioneered in the seminal work of  \cite{dwork2009differential}.  

\medskip\noindent 
{\bf Propose-Test-Release and local sensitivity.} 
The tight sensitivity bound we provide on the proposed exponential mechanism is {\em local} in the sense that it only holds for resilient datasets. 
 However, differential privacy must be guaranteed for any input, whether it is resilient (with desired level of $\alpha$ and $\rho$)  or not. 
We adopt Propose-Test-Release introduced in \cite{dwork2009differential} to  handle such locality of  sensitivity.  In the  first step, one proposes an upper bound on the sensitivity of the loss $D_S(\hat \theta)$, determined by the resilience of the dataset, which in turn is determined solely by the distribution family of interest and the target error rate. In the second step, one tests if the combination of the given dataset $S$, sensitivity bound $\sens$, and the exponential mechanism with loss $D_S(\hat\theta)$ satisfy the DP conditions. A part of the privacy budget is used to test this in a differential private manner, such that the subsequent exponential mechanism can depend on the result of this test, i.e., we only proceed to the third step if $S$ passes the test.  Otherwise, the process stops and outputs a predefined symbol, $\perp$. In the third step, one releases the DP estimate via the exponential mechanism. This ensures DP for any input $S$.
We are adopting the Propose-Test-Release mechanism pioneered in \cite{dwork2009differential}, which we explain in detail in Section~\ref{sec:dp}. 
The resulting framework, which we call High-dimensional Propose-Test-Release (HPTR) is provided in Section~\ref{sec:intro_algo}.

\medskip\noindent
{\bf Contributions.} We introduce a novel (computationally inefficient) algorithm for differentially private statistical estimation, with the goal of characterizing the achievable sample complexity for various problems with minimal assumptions. 
The proposed framework, which we call High-dimensional Propose-Test-Release (HPTR), makes a fundamental connection between differential privacy and robust statistics, thus achieving a sample complexity that significantly improves upon other state-of-the-art approaches. 
HPTR is a generic framework that can be  seamlessly  applied to various statistical estimation problems, 
as we demonstrate for mean estimation, linear regression, covariance estimation, and principal component analysis. 
Further, our analysis technique, which requires minimal assumptions, also seamlessly generalizes to all problem instances of interest. 

HPTR uses three crucial components:  
the exponential mechanism, 
robust statistics, 
and the Propose-Test-Release mechanism from \cite{dwork2009differential}. 
Building upon the inherent adaptivity and flexibility of the exponential mechanism, we propose using a novel loss function (also called a score function in a typical design of exponential mechanisms) that accesses the data only via one-dimensional robust statistics. 
The use of 1-D robust statistics is critical, because it dramatically reduces the sensitivity. We prove this sensitivity bound using the fundamental concept of resilience, which is central in robust statistics. 
This novel robust exponential mechanism is incorporated within the PTR framework to ensure that differential privacy is guaranteed on all input datasets, including those that are not necessarily compliant with the statistical assumptions. 
One byproduct of using robust statistics is that robustness comes for free. 
HPTR is inherently robust to adversarial corruption of the data and  achieves the optimal robust error rate under standard data corruption models. 

We present informal version of our main theoretical results in Section~\ref{sec:intro_main}.   
We present HPTR algorithm in detail in 
Section~\ref{sec:intro_algo}.  
We provide a sketch of the proof and the main technical contributions in 
Section~\ref{sec:intro_sketch}. 
Detailed explanations of the setting, main results, and the proofs for each instance of the problems are presented in Sections~\ref{sec:mean}--\ref{sec:pca} for mean estimation, linear regression, covariance estimation, and principal component analysis, respectively.

\medskip\noindent{\bf Notations.} 
Let  $[n]=\{1,2,\ldots,n\}$.
For $x\in{\mathbb R}^d$, we use $\|x\|=(\sum_{i\in[d]} (x_i)^2)^{1/2}$ to denote the Euclidean norm. 
For $X\in {\mathbb R}^{d_1\times d_2}$, we use $\|X\| = \max_{\|v\|_2=1} \| X v \|_2$ to denote the spectral norm. The $d\times d$ identity matrix is ${\mathbf I}_{d\times d}$. 
The Kronecker product is denoted by $x \otimes y$ for $x\in\reals^{d_1}$ and $y\in\reals^{d_2}$, such that $(x\otimes y)_{(i-1)d+j} =x_iy_j$ for $i\in[d_1]$ and $j\in[d_2]$. 
Whenever it is clear from context, we use $S$ to denote both a set of data points and also the set of indices of those data points. We use $S\sim S'$ to denote that two datasets $S,S'$ of  size $n$ are neighbors, such that $d_{\rm TV}(\hat p_S,\hat p_{S'})\leq 1/n$ where $d_{\rm TV}(\cdot)$ is the total variation and $\hat p_S$ is the empirical distribution of the data points in $S$.
We use $\mu(S)$ and $\Sigma(S)$ to denote mean and covariance of the data points in a dataset $S$, respectively. We use $\mu_p$ and $\Sigma_p$ to denote mean and covariance of the distribution $p$. 
%$\widetilde{O}$ and $\widetilde{\Omega}$ hide poly-logarithmic factors in $d,n,1/\alpha$,  $R$, and the failure probability.

% -----------------------------------

\subsection{Main results and related work} 
\label{sec:intro_main}
For each canonical problem of interest in statistical estimation,  HPTR can readily be applied to, in most cases,  
significantly improve upon known achievable sample complexity. Most of the lower bounds we reference are copied in Appendix~\ref{sec:lb} for completeness. 

% ---------------------------------- 
\subsubsection{DP mean estimation}

We apply our proposed HPTR framework to the standard DP mean estimation problem, where i.i.d.~samples $S=\{x_i\in\reals^d\}_{i=1}^n$ are drawn from a distribution $P_{\mu,\Sigma}$ with an unknown mean $\mu$ (which corresponds to $\theta$ in the general notation) and an unknown covariance $\Sigma \succ 0$, and we want to produce a DP estimate $\hat\mu$ of the mean. The resulting error is measured in Mahalanobis distance, $D_{P_{\mu,\Sigma}}(\hat\mu)=\|\Sigma^{-1/2}(\hat\mu-\mu)\|$, which is scale-invariant and naturally captured the uncertainty in all directions. 

This problem is especially challenging since we aim for a tight  guarantee that adapts to the unknown $\Sigma$ as measured in the Mahalanobis distance 
without sufficient samples to directly estimate $\Sigma$, as we explain below.
Despite being a canonical problem in DP statistics, the optimal sample complexity is not known even for standard distributions: sub-Gaussian and heavy-tailed distributions. We characterize the optimal sample complexity of the two problems by providing the guarantee of HPTR and the matching sample complexity lower bounds. A precise definition of sub-Gaussian distributions is provided in Eq.~\eqref{eq:def_subgauss}.

\begin{thm}[DP sub-Gaussian mean estimation algorithm, Corollary~\ref{coro:mean_subgaussian} informal]
    Consider a dataset $S =\{x_i\in\reals^d\}_{i=1}^n$ of $n$ i.i.d.~samples from a sub-Gaussian distribution with mean $\mu$ and covariance $\Sigma$. There exists an $(\eps,\delta)$-differentially private algorithm $\hat\mu(S)$ that given
    $$n \;=\; \tilde{O}_{\xi, \zeta}\Big(\, \frac{d}{\xi^2} + \frac{d}{ \varepsilon\xi} \,\Big)\;,$$
    achieves Mahalanobis error $\|\Sigma^{-1/2} (\hat\mu(S)-\mu) \| \le \xi$ with probability $1-\zeta$, where $\tilde{O}_{\xi, \zeta}$ hides the logarithmic dependency on $\xi, \zeta$ and we assume $\delta=e^{-O(d)}$.
\end{thm}
 
HPTR is the first algorithm for sub-Gaussian mean estimation with unknown covariance that matches the best known sample complexity lower bound of $n=\widetilde \Omega(d/\xi^2 + d/(\xi\varepsilon))$ from \cite{KV17,KLSU19} up to logarithmic factors in error $\xi$ and failure probability $\zeta$. Existing algorithms are suboptimal as they  require either a larger sample size or strictly Gaussian assumptions. 

Advances in DP mean estimation started with computationally efficient approaches of \cite{KV17,KLSU19,barber2014privacy}. We discuss the results as follows, and omit the polynomial factors in $\log(1/\delta)$.  When the covariance $\Sigma$ is known,  Mahalanobis  error $\xi$ can be
achieved with $n=\tilde O(d/\xi^2 + d/(\xi\varepsilon))$ samples. 
Under a relaxed assumption that ${\bf I}_{d\times d} \preceq \Sigma \preceq \kappa  {\bf I}_{d\times d}$ with a known $\kappa$,  $n=\tilde O(d/\xi^2 + d/(\xi\varepsilon) + d^{1.5}/\varepsilon)$ samples are required using  a specific preconditioning approach tailored for the assumption and the knowledge of $\kappa$.  
For general unknown $\Sigma$,   $O(d^2/\xi^2 + d^2/(\xi\varepsilon))$ samples are required using an  explicit DP estimation of the covariance. Empirically, further gains can be achieved with CoinPress \cite{biswas2020coinpress}. 

Computationally inefficient approaches followed with a goal of identifying the fundamental optimal sample complexity with minimal assumptions~\cite{bun2019private,aden2020sample}. For the unknown covariance setting, the best known result under Mahalanobis error is achieved by \cite{brown2021covariance}. 
Through analyzing the differentially private Tukey median estimator introduced in \cite{liu2021robust}, \cite{brown2021covariance} shows that $n=\tilde O(d/\xi^2 + d/(\xi\varepsilon))$ is sufficient even when the covariance is unknown. However, the approach heavily relies on the assumption that the distribution is strictly Gaussian. For sub-Gaussian distributions, \cite{brown2021covariance} proposes a different approach achieving sample complexity of 
$n=\tilde O ( d/\xi^2  + d/(\xi\varepsilon^2) )$ samples with a sub-optimal  $(1/\varepsilon^2)$ dependence.  

Beyond the sub-Gaussian setting, it is natural to consider the DP mean estimation for distributions with heavier tails. We apply HPTR framework to the more general mean estimation problems for hypercontractive distributions. 
    A distribution $P_{\mu,\Sigma}$ with mean $\mu$ and covariance $\Sigma$ is 
    $(\kappa,k)$-hypercontractive if for all $v\in\reals^d$, 
    $\E_{x\sim P_X}[ |\langle v , (x-\mu)\rangle |^k] \leq \kappa^k (v^\top\Sigma v)^{k/2}$. 
The assumption of hypercontractivity is similar to the bounded $k$-th moment assumptions, except requiring an additional lower bound on the covariance. This additional assumption is necessary for our setting to make sure the Mahalanobis error guarantee is achievable. We state our main result for hypercontractive mean estimation as follows. For simplicity of the statement, we assume $k, \kappa$ are constants.
\begin{thm}[DP hypercontractive mean estimation algorithm, Corollary~\ref{coro:mean_kmoment} informal]
Consider a dataset $S =\{x_i\in\reals^d\}_{i=1}^n$ of $n$ i.i.d.~samples from a $(\kappa, k)$-hypercontractive distribution with mean $\mu$ and covariance $\Sigma$. There exists an  $(\eps,\delta)$-differentially private algorithm $\hat\mu(S)$ that given
 $$n\;=\; \tilde{O}_d \Big(\frac{d}{\xi^{2}} + \frac{d}{\varepsilon\xi^{1+1/(k-1)} }   
    \Big) \;,$$ 
      achieves Mahalanobis error $\|\Sigma^{-1/2} (\hat\mu(S)-\mu) \| \le \xi$ with  probability at least $0.99$, where $\tilde{O}_{d}$ hides a logarithmic factor on $d$, and we assumes $\delta=e^{-O(d)}$.
\end{thm}

We prove an $n=\Omega( d/ \varepsilon\xi^{1+1/(k-1)} )$ sample complexity lower bound for hypercontractive DP mean estimation in Proposition~\ref{thm:lowerbound_mean_hypercontractive} to show the optimality of our upper bound result. 
Notice that the first term $\tilde{O}_d(d/\xi^2)$ in the upper bound cannot be improved up to logarithmic factors even if we do not require privacy, thus HPTR is the first algorithm that achieves optimal sample complexity for hypercontractive mean estimation under Mahalanobis distance up to logarithmic factors in $d$. 
When the covariance is known, 
an existing  DP mean estimator of \cite{kamath2020private} achieves a stronger $(\varepsilon,0)$-DP  with a similar sample size of $n=\widetilde O(  d / \xi^2 +  d / (\varepsilon\xi^{1+1/(k-1)} ) )$, and no prior result is known for the unknown covariance case.
%Such a bounded search space is critical in achieving a stronger {\em pure} privacy guarantee with $\delta=0$.

% ---------------------------------- 
\subsubsection{DP linear regression}
\label{sec:ex_lr}
% Let ${\cal D}_1$ be a distribution of $x_i\in\reals^d$ which is zero mean sub-Gaussian with covariance $\Sigma$ and sub-Gaussian proxy $0\prec \Gamma\preceq c \Sigma$ for some constant $c$. Let ${\cal D}_2$ be a distribution of  $\eta_i\in\reals$ which is a zero mean one-dimensional sub-Gaussian with variance $\gamma^2$ and sub-Gaussian proxy $\gamma_0^2\leq c\gamma^2$ for some constant $c$.  A multiset of i.i.d.~labeled samples $S=\{(x_i, y_i )\}_{i=1}^n$ is generated from a linear model with noise  $\eta_i$ independent of $x_i$:
% $y_i\;=\;x_i^\top\beta+\eta_i$\;, 
% where the input $x_i$ and the independent noise $\eta_i$ are i.i.d. samples from ${\cal D}_1$ and ${\cal D}_2$. 

We next apply {\HPTR} framework to DP linear regression. Given i.i.d. samples $S=\{(x_i,y_i)\}_{i\in [n]}$ from a distribution $P_{\beta, \Sigma, \gamma^2}$ of a linear model: $y_i = x_i^\top \beta+\eta_i$, where the input $x_i\in\reals^d$ has zero mean and covariance $\Sigma$ and  
the noise $\eta_i\in\reals$ has variance $\gamma^2$ satisfying $\E[x_i\eta_i]=0$, the goal of DP linear regression is to output a DP estimate $\hat\beta$ of the unknown model parameter $\beta$, without knowledge about the covariance $\Sigma$. 
%We further assume  $\E[x_i\eta_i]=0$, which is equivalent to assuming that the true parameter  $\beta=\Sigma^{-1}\E[y_ix_i]$, namely the least square solution given infinite samples equals $\beta$.  
%In DP linear regression, we want to output a DP estimate $\hat\beta$ of the unknown model parameter $\beta$, 
% assuming  that the covariance $\Sigma\succ 0$  is unknown. %When $\gamma^2$ is unknown, we can have a DP estimator to get $\gamma^2$, which is sample efficient as it is a one-dimensional estimation problem as we show in Appendix~\ref{sec:lr_gamma}. 
The resulting error is measured in $D_{P_{\beta,\Sigma,\gamma^2}}(\hat\beta)= (1/\gamma)\|\Sigma^{1/2}(\hat\beta -\beta )\|$ which is equivalent to the standard \textit{root excess risk} of the estimated predictor $\hat\beta$. 
Similar to Mahalanobis distance for mean estimation, this is challenging since we aim for a tight guarantee that adapts to the unknown $\Sigma$ without having enough samples to directly estimate $\Sigma$.

\begin{thm}[DP sub-Gaussian linear regression, Corollary~\ref{coro:linear_regression_subgaussian} informal]
    Consider a dataset $S =\{(x_i, y_i)\}_{i=1}^n$ generated from a linear model $y_i=x_i^\top \beta+\eta_i$ for some  $\beta\in \reals^d$, where  $\{x_i\}_{i\in [n]}$ are i.i.d. sampled from zero-mean $d$-dimensional sub-Gaussian distribution with unknown covariance $\Sigma$, and $\{\eta_i\}_{i\in [n]}$ are i.i.d. sampled from zero mean one-dimensional sub-Gaussian with  variance $\gamma^2$. We further assume the data $x_i$ and the noise $\eta_i$ are independent. There exists a $(\eps,\delta)$-differentially private algorithm $\hat\beta(S)$ that given
    $$n \;=\; \tilde{O}_{\xi, \zeta}\Big(\, \frac{d}{\xi^2} + \frac{d}{ \varepsilon\xi} \,\Big)\;,$$
    achieves error $(1/\gamma)\|\Sigma^{1/2} (\hat\beta(S)-\beta) \| \le \xi $ with probability $1-\zeta$, where $\tilde{O}_{\xi, \zeta}$ hides the logarithmic dependency on $\xi, \zeta$ and we assume $\delta=e^{-O(d)}$.
\end{thm}

  HPTR is the first algorithm for sub-Gaussian distributions with an unknown covariance $\Sigma$ that up to logarithmic factors matches the lower bound of  $n=\tilde\Omega(d/\xi^2+d/(\xi\varepsilon))$ assuming $\varepsilon < 1$ and  $\delta<n^{-1-\omega}$ for some $\omega>0$ from  \cite[Theorem 4.1]{cai2019cost}.
%For completeness, we provide the lower bound in Appendix~\ref{sec:lb}. 
%under the assumption that $\|x_i\|$ is bounded, $\eta_i$ is independent Gaussian noise and $\beta$ is bounded. 
An existing algorithm for DP linear regression from \cite{cai2019cost} is suboptimal as it require $\Sigma$ to be close to the identity matrix, which is equivalent to assuming that we know $\Sigma$.  
%which is equivalent to providing a weaker guarantee on the non-adaptive error on  $\|\hat\beta-\beta\|$. 
\cite{dwork2009differential} proposes to use PTR and B-robust regression algorithm from \cite{hampel2011robust} for differentially private linear regression under i.i.d. data assumptions (also exponential time), but only asymptotic consistency is proven as  $n\rightarrow \infty$.
Under an alternative setting where the data is deterministically given without any probabilistic assumptions, 
significant advances in DP linear regression has been made 
\cite{vu2009differential,kifer2012private,mir2013differential,dimitrakakis2014robust,bassily2014private,wang2015privacy,foulds2016theory, minami2016differential, wang2018revisiting, sheffet2019old}. 
The state-of-the-art guarantee is achieved in \cite{wang2018revisiting, sheffet2019old} which under our setting translates into a sample complexity of $n= O(d^{1.5}/(\xi\varepsilon))$. The extra $d^{1/2}$ factor is due to the fact that no statistical assumption is made, and cannot be improved under the deterministic setting (not necessarily i.i.d.) that those algorithms are designed for. 

%provides DP linear algorithms for fixed design matrix under independent gaussian noise.
%We assume $\kappa$, $\tilde{\kappa}$ and integer $k\geq 4$ are constants.

Similar to mean estimation, we also consider the DP linear regression for distributions with heavier tails, and apply HPTR framework to the linear regression problem under $(k,\kappa)$-hypercontractive distributions (see Definition~\ref{def:hyper}). 
HPTR can handle both independent and dependent noise, and we state the independent noise case here the dependent noise case in Section~\ref{sec:lr:dependent}.
For simplicity of the statement, we assume $k, \kappa$ are constants.

\begin{thm}[DP hypercontractive linear regression with independent noise, Corollary~\ref{coro:linear_regression_hypercontractive} informal]
    Consider a dataset $S =\{(x_i, y_i)\}_{i=1}^n$ generated from a linear model $y_i=x_i^\top \beta+\eta_i$ for some  $\beta\in \reals^d$, where  $\{x_i\}_{i\in[n]}$ are i.i.d. sampled from zero-mean $d$-dimensional $(\kappa, k)$-hypercontractive distribution with unknown covariance $\Sigma$  and $\eta_i$ are i.i.d. sampled from zero mean one-dimensional $( \kappa, k)$-hypercontractive distribution with  variance $\gamma^2$. We further assume the data $x_i$ and the noise $\{\eta_i\}_{i\in [n]}$ are independent. There exists an $(\eps,\delta)$-differentially private algorithm $\hat\beta(S)$ that given
    $$n \;=\; \tilde{O}_{d}\Big(\, \frac{d}{\xi^2} + \frac{d}{ \varepsilon \xi^{1+1/(k-1)}} \,\Big)\;,$$
    achieves error $(1/\gamma) \|\Sigma^{1/2} (\hat\beta(S)-\beta) \| \le \xi$ with probability $0.99$, where $\tilde{O}_d$ hides a logarithmic factor on $d$, and we assume $\delta=e^{-O(d)}$.
\end{thm}

The first term in the sample complexity cannot be improved as it matches the classical lower bound of linear regression even without privacy constraint. For the second term, the sub-Gaussian lower bound of $n=\tilde\Omega(d/(\varepsilon\xi))$ from \cite[Theorem 4.1]{cai2019cost} continues to hold in the hypercontractive setting. We do not have a matching lower bound for the second term. 
%under a general ambient dimension $d$, we prove a lower bound of $n=\Omega(1/(\varepsilon\xi^{1+1/(k-1)}))$  when $d=O(1)$ (Proposition~\ref{pro:lr_lb_indep}), and we believe our upper bound is tight. 
To the best of our knowledge, HPTR is the first algorithm for linear regression that guarantees $(\varepsilon,\delta)$-DP under hypercontractive distributions with independent noise. 

When applied to the setting where noise $\eta_i$ is dependent on the input vector $x_i$, HPTR is the first algorithm for linear regression that guarantees $(\varepsilon,\delta)$-DP. We refer the readers to Section~\ref{sec:lr:dependent} for more detailed description of our result.

\subsubsection{DP covariance estimation}
We present HPTR applied to  covariance estimation from i.i.d. samples under a Gaussian distribution ${\cal N}(0,\Sigma)$.  The main reason for this choice is that the Mahalanobis error $\|\Sigma^{-1/2} \hat{\Sigma}\Sigma^{-1/2} -\mathbf{I}_{d \times d}\|_F$ of the Kronecker product $x_i\otimes x_i$ is proportional to the natural error metric of total variation for Gaussian distributions. The strength of HPTR framework is that it can  be seamlessly applied to general distributions, for example sub-Gaussian or heavytailed, but the resulting Mahalanobis error becomes harder to interpret as it involves respective fourth moment tensors. 

 \begin{thm} [DP Gaussian covariance estimation, Corollary~\ref{coro:cov_gaussian} informal]
 Consider a dataset $S=\{x_i\}_{i=1}^n$ of $n$ i.i.d. samples from $\cN(0, \Sigma)$. There exists a $(\varepsilon, \delta)$-differentially private estimator $\hat{\Sigma}$ that given
 $$
 n \;=\; \tilde{O}_{\xi, \zeta}\Big(\,\frac{d^2}{\xi^2} + \frac{d^2 }{\xi\varepsilon} \,\Big)\;,
 $$
 achieves error $\|\Sigma^{-1/2} \hat{\Sigma}\Sigma^{-1/2} -\mathbf{I}_{d \times d}\|_F \le \xi$ with probability $1-\zeta$, where $\tilde{O}_{\xi, \zeta}$ hides the logarithmic dependency on $\xi, \zeta$ and we assume $\delta=e^{-O(d)}$.
\end{thm}

This Mahalanobis distance  guarantee (for the  Kronecker product, $\{x_i\otimes x_i\}$, of the samples) implies that the estimated Gaussian distribution is close to the underlying one in total variation distance
 (see for example \cite[Lemma 2.9]{KLSU19}): $d_{\rm TV}({\cal N}(0,\hat\Sigma),{\cal N}(0,\Sigma)) = O(\| \Sigma^{-1/2}\hat\Sigma\Sigma^{-1/2} - {\bf I}_{d\times d}\|_F) = O(\xi)$. 
The sample complexity is near-optimal, matching a lower  bound of 
$n=\Omega(d^2/\xi^2 + \min\{d^2,\log(1/\delta)\}/(\varepsilon\xi) )$ 
up to a logarithmic factor when $\delta=e^{-\Theta(d)}$. 
 The first term follows from  the classical estimation of the covariance without DP. The second term follows from extending the lower bound in \cite{KLSU19} constructed for pure differential privacy with $\delta=0$  and matches the second term in our upper bound  when $\delta=e^{-\Theta(d^2)}$. 
We note that a similar  upper bound is achieved  by the state-of-the-art (computationally inefficient) algorithm in \cite{aden2020sample}, which improves over HPTR in the lower order terms not explicitly shown in this informal version of our theorem.  Both HPTR and \cite{aden2020sample,amin2019differentially} improve upon computationally efficient approaches of \cite{KV17,KLSU19} which require additional assumption that ${\bf I}_{d\times d} \preceq \Sigma \preceq \kappa {\bf I}_{d\times d}$ with a known $\kappa$. Recently, \cite{kamath2021private}   introduced a novel preconditioning approach that is polynomial time  and removes the upper and lower bounds  on $\Sigma$ completely, but requires sample complexity of $n=\tilde O(d^2/\xi^2 + d^2{\rm polylog}(1/\delta)/(\xi\varepsilon) + d^{5/2}{\rm polylog}(1/\delta)/\varepsilon)$.

% ---------------------------------- 
\subsubsection{DP principal component analysis}

We next apply HPTR to the task of estimating the top PCA direction from i.i.d.~sampless

\begin{thm}[DP sub-Gaussian principle component analysis, Corollary~\ref{coro:pca}]
Consider a dataset $S =\{x_i\in\reals^d\}_{i=1}^n$ of $n$ i.i.d.~samples from a zero-mean sub-Gaussian distribution with unknown covariance $\Sigma$. There exists an  $(\eps,\delta)$-differentially private estimator $\hat{u}$ that given
    $$n \;=\; \tilde{O}_{\xi, \zeta}\Big(\, \frac{d}{\xi^2} + \frac{d}{ \varepsilon\xi} \,\Big)\;,$$
    achieves error $1-\frac{\hat{u}^\top\Sigma\hat{u}}{\|\Sigma\|} \le \xi$ with probability $1-\zeta$, where $\tilde{O}_{\xi, \zeta}$ hides the logarithmic dependency on $\xi, \zeta$ and we assume $\delta=e^{-O(d)}$.
\end{thm}
 
 HPTR is the first estimator for sub-Gaussian distributions to nearly match the information-theoretic lower bound of $n = \Omega(d/\xi^2 + \min\{d,\log(1/\delta)\}/(\varepsilon\xi))$ as we showed in Proposition~\ref{thm:lowerbound_pca_subgaussian}. 
 The first term $\Omega(d/\xi^2)$  is unavoidable even without DP (Proposition~\ref{prop:lowerbound_pca}). 
 The second term in the lower bound follows from Proposition~\ref{thm:lowerbound_pca_subgaussian}, which matches the second term  in the upper bound  when $\delta=e^{-\Theta(d)}$.  
 Existing DP PCA approaches from \cite{blum2005practical,PPCA, kapralov2013differentially,dwork2014analyze,hardt2012beating, hardt2013beyond, hardt2013robust} are designed for arbitrary samples not  necessarily drawn i.i.d.~and hence require a larger samples size of $n=\tilde O(d/\xi^2 + d^{1.5}/(\xi\varepsilon))$.
  This is also unavoidable for more general deterministic data, as it matches an information theoretic lower bound~\cite{dwork2014analyze} under weaker assumptions on the data than i.i.d.~Gaussian. 

\begin{thm}[DP hypercontractive principle component analysis, Corollary~\ref{coro:pca_hyper}]
Consider a dataset $S =\{x_i\in\reals^d\}_{i=1}^n$ of $n$ i.i.d.~samples from a zero-mean $(\kappa, k)$-hypercontractive distribution with unknown covariance $\Sigma$. There exists an $(\eps,\delta)$-differentially private estimator $\hat{u}$ that given
    $$n \;=\; \tilde{O}_{\xi, d}\Big(\, \frac{d}{\xi^{(2k-2)/(k-2)}} + \frac{d}{ \varepsilon\xi^{1+2/(k-2)}} \,\Big)\;,$$
    achieves error $1-\frac{\hat{u}^\top\Sigma\hat{u}}{\|\Sigma\|} \le \xi$ with probability $0.99$, where $\tilde{O}_{\xi, d}$ hides the logarithmic dependency on $\xi, d$ and we assume $\delta=e^{-O(d)}$.
\end{thm}

 HPTR is the first estimator for hypercontractive distributions to 
guarantee differential privacy for PCA with sample complexity scaling as $O(d)$.  
 As a complement of our algorithm, we proved a $n=\Omega(d/\xi^2  +  \min\{d,\log(1/\delta)\}/(\xi^{1+2/(k-2)} \varepsilon) )$ information-theoretic lower bound in Proposition~\ref{thm:lowerbound_pca_hyper}.  
 The first term $\Omega(d/\xi^2)$ in the lower bound is needed even without DP, and  
 there is a gap of factor $O(\xi^{-2/(k-2)})$ compared to the first term in our upper bound.  
The second term in the lower bound matches the last term  in the upper bound when $\delta= e^{-\Theta(d)}$. 
 
 %follows from  Proposition~\ref{thm:lowerbound_pca_hyper}, which 
  
\subsection{Algorithm} 
\label{sec:intro_algo} 

The proposed 
High-dimensional Propose-Test-Release (HPTR) is not computationally efficient, as the {\sc Test} step requires enumerating over a certain neighborhood of the input dataset and the {\sc Release} step requires enumerating over all directions in high dimension. 
The strengths of HPTR is that $(i)$ the same framework can be seamlessly applies to many problems as we demonstrate in Sections~\ref{sec:mean}--\ref{sec:pca}; $(ii)$  a unifying recipe can be applied for all those instances to give tight utility guarantees as we explicitly prescribe in Section~\ref{sec:recipe}; and $(iii)$ the algorithm  is simple and the analysis is clear such that it is transparent how the distribution family of interest translates into the utility guarantee (via resilience).

As a byproduct of the simplicity of the algorithm and clarity of the analysis, we achieve  the state-of-the-art  sample complexity for all those problem instances with minimal assumptions, oftentimes nearly matching the information  theoretic lower bounds. 
As a byproduct of the use of robust statistics,  we guarantee robustness against adversarial corruption for free (e.g., Theorems~\ref{thm:robust_mean}, \ref{thm:robust_linear_regression}, \ref{thm:robust_pca_utility}).

We describe the framework for general statistical estimation problem where data is drawn i.i.d.~from a distribution represented by two unknown parameters $\theta\in\reals^p$ and $\phi$ and is coming from a known family of distributions. 
An example of a problem instance of this type would be mean estimation from heavy-tailed distributions that are $(\kappa,k)$-hypercontractive with unknown mean $\mu$ (which in the general notation is $\theta$) and unknown covariance $\Sigma$ (which corresponds to $\phi$). 

The main component is an exponential mechanism in {\sc Release} step below that uses a loss function $D_S(\hat\theta)$ and a support $B_{\tau,S}$ defined as 
\begin{eqnarray*}
B_{\thresh,S}=\{\hat\theta\in\reals^p: \robdist_S(\hat\theta) \leq \thresh \}\;. 
\end{eqnarray*} 
Bounding the support of the exponential mechanism is important since the sensitivity  also depends on $\hat\theta$ in many problems of interest. 
We discuss this in detail in the example of mean estimation in Section~\ref{sec:mean_proof_strategy}. 
The specific choices of the threshold $\tau$ only depend on the tail of the distribution family of interest and not the parameters $\theta$ or $\phi$ or the data. In particular, we use the resilience property of the distribution family to prescribe the choice of $\thresh$ for each problem instance that gives us the tight utility guarantees. 
As explained in Section~\ref{sec:intro}, we use one-dimensional robust statistics to design the loss functions, which we elaborate for each problem instances in Sections~\ref{sec:mean}--\ref{sec:pca}, where we also explain how to choose the sensitivity for each case based on the resilience of the distribution family only.

After we {\sc Propose} the choice of the sensitivity $\sens$ and threshold $\thresh$ for the problem instance in hand, we {\sc Test} to make sure that the given dataset $S$ is consistent with the assumptions made when selecting $D_S(\hat\theta)$, $\sens$, and $\thresh$. This is done by testing the safety of the exponential mechanism, by privately checking the margin to safety, i.e., how many data points need to me changed from $S$ for the exponential mechanism to violate differential privacy conditions.
If the margin is large enough, HPTR proceeds to {\sc Release}. Otherwise, it halts and outputs $\perp$. 
A set of such unsafe datasets is defined as 
\begin{align}
    \label{def:unsafe}
    & \unsafe_{(\varepsilon,\delta,\thresh)} \; =\;
    \Big\{S'\subseteq \reals^{d\times n}\,|\, \text{$\exists S''\sim S'$ and $\exists E\subseteq \reals^{p}$ such that } \nonumber \\
    &\; \prob_{\hat\theta\sim r_{(\varepsilon,\sens,\thresh,S'')}}(\hat\theta\in E) > e^\varepsilon \prob_{\hat\theta\sim r_{(\varepsilon,\sens,\thresh,S')}} (\hat\theta\in E) + \delta
     \text{ or } \prob_{\hat\theta\sim r_{(\varepsilon,\sens,\thresh,S')}}(\hat\theta\in E) > e^\varepsilon \prob_{\hat\theta\sim r_{(\varepsilon,\sens,\thresh,S'')}} (\hat\theta\in E) + \delta \Big\}\;,
\end{align}
where $r_{(\varepsilon,\sens,\tau,S)} $ denotes the pdf of the exponential mechanism in Eq.~\eqref{def:exp}. 
Given a loss (or a distance) function, $\robdist_S(\hat\theta)$, which is a surrogate for the target measure of error, $\popdist(\hat\theta,\theta)$, 
High-dimensional Propose-Test-Release (HPTR) proceeds as follows: 
\begin{itemize} 
    \item[1.] {\sc \bf Propose}: 
 %   \begin{itemize}
        %\item[1.1.] 
        Propose a target bound $\sens $ on local sensitivity and a target threshold $\thresh$ for  safety.
 %       \item[1.2.] Compute     the sensitivity margin $m_\sens = \min_{S'} d_H(S,S')$ such that there exists a $\hat\theta\in B_{\thresh+\backoff,S}$ satisfying  $\| D(\hat\theta,S)-D(\hat\theta,S')\| > \sens $. 
%    \end{itemize} 
    \item[2.] {\sc \bf Test}: 
    \begin{itemize}
%        \item[2.1.] If $\hat{m}_\sens = m_\sens + {\rm Lap}(3/\varepsilon) < (3/\varepsilon)\log(3/\delta)$ then output $\perp$, and otherwise continue. 
    \item[2.1.]     Compute the safety margin $m_\thresh = \min_{S'} d_H(S,S')$ such that $S'\in \unsafe_{(\varepsilon/2,\delta/2,\thresh)}$.
    \item[2.2.]      If $\hat{m}_\thresh = m_\thresh  + {\rm Lap}(2/\varepsilon)<(2/\varepsilon)\log(2/\delta)$, then output $\perp$, and otherwise continue.
    \end{itemize} 
    \item[3.] {\sc \bf Release}: Output $\hat\theta$ sampled from a distribution with a pdf:
    \begin{eqnarray}
    r_{(\varepsilon,\sens,\thresh,S)}(\hat \theta)\;\;=\;\;\left\{\begin{array}{rl} \frac{1}{Z}\,\exp\Big\{-\frac{\varepsilon}{4 \sens}\robdist_S(\hat\theta) \Big\} 
    & \text{ if } \hat\theta \in B_{\thresh,S}\;, \\ 
    0 & \text{ otherwise\;,}
    \end{array}\right. 
    \label{def:exp}
    \end{eqnarray}
where $Z=\int_{B_{\thresh,S}} \exp\{-(\varepsilon \robdist_S(\hat\theta))/(4\sens)\} d\hat\theta$.
\end{itemize} 

 It is straightforward to show that $(\varepsilon,\delta)$-differential privacy is satisfied for all input $S$. 

\begin{thm} 
    \label{thm:privacy} 
    For any dataset $S\subset{\cal  X}^{n}$, distance function  $D_S:\reals^p  \to \reals $ on that dataset, and parameters $\varepsilon,\delta,\sens$ and $\thresh$, HPTR is $(\varepsilon,\delta)$-differentially private.
\end{thm}
\begin{proof}
    The differentially private margin $\hat m_\tau $  is $(\varepsilon/2,0)$-differentially private, because the sensitivity of the margin is one, and we are adding a Laplace noise with parameter $2/\varepsilon$. 
    The {\sc Test} step (together with the exponential mechanism) is $(0,\delta/2)$-differentially private since there is a probability $\delta/2$ event that a unsafe dataset $S$ with a small margin $m_\tau$ is classified as a safe dataset and passes the test. 
    On the complimentary event, namely, that the dataset that passed the {\sc Test} is indeed safe, the {\sc Release} step is $(\varepsilon/2,\delta/2)$-differentially private since we use \unsafe$_{(\varepsilon/2,\delta/2,\tau)}$ in the {\sc Test} step.   
\end{proof} 
% ----------------------------------------

% ----------------------------------------
\subsubsection{Utility analysis of HPTR for statistical estimation}
\label{sec:recipe}

We prescribe the following three-step recipe as a guideline for applying HPTR to each specific statistical estimation problem and obtaining a utility guarantee. Consider a  problem of estimating an unknown $\theta$ from samples from a  generative model $P_{\theta,\phi}$, where the error is measured in   $\popdist(\hat\theta,\theta)$.
\begin{itemize} 
    \item Step 1: Design 
a surrogate $\robdist_S(\hat\theta)$ for the target error metric  $\popdist(\hat\theta,\theta)$ using only one-dimensional robust statistics on $S$. 
    \item Step 2: Assuming  {\em resilience} of the dataset, propose an appropriate sensitivity bound $\sens$ and threshold $\thresh$ and analyze the utility of HPTR. 
    \item Step 3: For each specific family of generative models $P_{\theta,\phi}$ with a known tail bound, characterize the resulting resilience and substitute it in the utility analysis from the previous step, which gives the final guarantee. 
\end{itemize}
 %We characterize universal conditions that are not problem specific in Theorem~\ref{thm:utility}  that guarantees tight error bounds (Section~\ref{sec:utility}).  
%We check these universal conditions for each problem instance in the following, 
We demonstrate how to apply this recipe and carry out the utility analysis for mean estimation (Section \ref{sec:mean}), linear regression (Section \ref{sec:lin}),  covariance estimation (Section \ref{sec:cov}), and PCA (Section~\ref{sec:pca}). 
We explain and justify the use of one-dimensional robust statistics in Step 1 and the assumption on the  resilience of the dataset in Step 2  in the next section using the mean estimation problem as a canonical example. 
It is critical to construct $\robdist_S(\hat\theta)$ using only one-dimensional and robust statistics; this ensures that $\robdist_S(\hat\theta)$ has a small sensitivity as demonstrated in Section~\ref{sec:mean1}. 
We prove error bounds only assuming resilience of the dataset; this relies on a fundamental connection between sensitivity and resilience as explained in Section~\ref{sec:mean2}.

\subsection{Technical contributions and proof sketch}
\label{sec:intro_sketch}

%We make the following technical contributions. We introduce HPTR, which is a general framework for solving DP statistical estimation, and  prescribe how to select the loss function of the exponential mechanism based on one-dimensional robust statistics.  We use resilience of the probabilistic families  of interest to bound the sensitivity of the one-dimensional statistics, which leads to a streamlined analysis technique for the utility of HPTR as we sketch below. We use the robustness of the one dimensional robust statistics to show that $k^*$-hop neighbor of a resilient set $S$ (in the Hamming distance) still achieves low sensitivity for the prescribed loss function for some sufficiently large $k^*$ that ensures safety in the {\sc Test} step.  

We use the canonical example of mean estimation to explain our proof sketch.  
%A more detailed proof sketch can be found in Section~\ref{sec:mean_proof_strategy}.
For i.i.d.~samples from  a sub-Gaussian distribution $P_{\mu,\Sigma}$ with mean $\mu$ and covariance $\Sigma$, we show in Theorem~\ref{thm:mean} that HPTR achieves a near optimal sample complexity of $n=\tilde O(d/\alpha^2 + d/(\alpha\varepsilon))$ to get Mahalanobis error $\|\Sigma^{-1/2}(\hat \mu -\mu)\|=\tilde O(\alpha)$ for some target accuracy $\alpha\in[0,1]$.

Our proof strategy is to first show  that the robust one-dimensional statistics  have small sensitivity if the  dataset is resilient. Consequently,  the loss function $\robdist_S(\hat\mu)$ has a small {\em local} sensitivity, i.e.~the sensitivity is small if restricted to $\hat\mu$ close to $\mu$ and if the dataset is resilient. 
To ensure DP, we run {\sc Release} only when those two locality conditions are satisfied; we first {\sc Propose} the sensitivity $\sens$ and a threshold $\thresh$, and then we {\sc Test} that DP guarantees are met on the given dataset with those choices. 
We prove that resilient datasets   pass this safety test  with a high probability and   achieve the desired accuracy. 
 We give a sketch of the main steps below.  %Noting that $\rho_i(c\alpha)=\Theta(\rho_i(\alpha))$ for $i\in\{1,2\}$, we do not track the constant factor in the resilience. 
 
% 1-d robuststatistic has small sensitivity 
% D has small, but local sensitivity 
% local 1 - propose \Delta = XXX
% local 2 - propose \tau = XXX 
% check violation - Theorem 5
% utility analysis for sample complexity 

\medskip\noindent{\bf One-dimensional robust statistics have small  sensitivity on resilient datasets.} 
A set $S=\{x_i\in\reals^d\}_{i\in[n]}$ of i.i.d.~samples from a sub-Gaussian distribution has the following resilience property with probability $1-\zeta$ if $n=\tilde \Omega(d/\alpha^2)$,  where $\tilde\Omega$ hides polylogarithmic factors   $\alpha$ and the failure probability $\zeta$:
    \begin{eqnarray*}
    \Big|\frac{1}{|T|}\sum_{x_i\in T} \langle v, x_i-\mu\rangle\Big| \;\leq\; {2\sigma_v\sqrt{\log(1/\alpha)}} \;, \text{ and }  
    \Big|\frac{1}{|T|}\sum_{x_i \in T} \big(\, \langle v, x_i-\mu\rangle^2-\sigma_v^2 \,\big) \Big| \;\leq\; 2  \sigma_v^2 \log(1/\alpha) \;, 
    \end{eqnarray*} 
    for any subset $T\subset S$ of size at least $\alpha n$ and for any unit norm $v\in\reals^d$
    where $\sigma_v^2 = v^\top \Sigma v$ (Lemma~\ref{lem:mean_subgaussian}). This means that the $\alpha$-tail of the distribution (when projected down to one dimension) cannot be too far from the true one in mean and variance.
%There are two equivalent definitions of resilience. We are using the one in  Lemma~\ref{lem:deviate} for the purpose of explaining our proof idea, but a more common one is in Definition~\ref{def:resilience}. 
For mean estimation, we use the loss function of $D_S(\hat\mu) = \max_{v\in\reals^d,\|v\|=1} \langle v, \hat\mu-\mu({\cal M}_{v,\alpha}) \rangle/\sqrt{v^\top \Sigma({\cal M}_{v,\alpha}) v}$, where $\mu(T)$ and $\Sigma(T)$ are mean and covariance of a dataset $T$ and ${\cal M}_{v,\alpha}\subset S$ is defined as follows. 
For each direction $v$, we partition $S$ into three sets ${\cal T}_{v,\alpha},{\cal M}_{v,\alpha},$ and ${\cal B}_{v,\alpha}$. ${\cal T}_{v,\alpha}\subset S$ is the subset of datapoints corresponding to the largest $\alpha n$ datapoints in $\{\langle v, x_i \rangle\}_{x_i\in S} $, the   projected data points in  the direction $v$. ${\cal B}_{v,\alpha}\subset S$ corresponds to the smallest $\alpha n$ values, and ${\cal M}_{v,\alpha}\subset S$ is   the remaining $(1-2\alpha)n $ data points. 

We show that the robust projected mean, 
 $\langle v, \mu({\cal M}_{v,\alpha}) \rangle$ has sensitivity $O(\sigma_v \sqrt{\log(1/\alpha)}/n)$.
Under the resilience above, 
the top $\alpha$-tail, ${\cal T}_{v,\alpha}$, has the  empirical mean that is 
no more than $O(\sigma_v\sqrt{\log(1/\alpha)}\ )$ away from the true projected mean $\langle v, \mu \rangle$, and the same is true for ${\cal B}_{v,\alpha}$. 
It follows that there exists at least one data point in ${\cal T}_{v,\alpha}$ and one data point in ${\cal B}_{v,\alpha}$ that are no more than $O(\sigma_v\sqrt{\log(1/\alpha)})$ away from $\mu_v$. 
This implies that the range of the  middle subset ${\cal M}_{v,\alpha}$ is provably bounded by $O(\sigma_v \sqrt{\log(1/\alpha)})$, and the sensitivity of the robust mean $\langle v , \mu({\cal M}_{v,\alpha}) \rangle$ is guaranteed to be $O(\sigma_v \sqrt{\log(1/\alpha)}/ n )$. We can similarly show that $v^\top \Sigma({\cal M}_{v,\alpha}) v $ has sensitivity $O(\sigma_v^2 \log(1/\alpha)/  n)$.

Under the above sensitivity bounds for the one dimensional statistics,  
 it follows (for example, in  Eq.~\eqref{eq:sensbound_D}) that the sensitivity of the loss function $D_S(\hat\mu)$ is bounded by $ O(\sqrt{\log(1/\alpha)}/n)$
as long as $D_S(\hat\mu) \leq \thresh := C\alpha\sqrt{\log(1/\alpha)}$ and the dataset is resilient. It is worth noting here that since the sensitivity is only small when $D_S(\hat\mu)\le \tau$, our exponential mechanism only samples from the set $B_{\tau,S}$,  which contains only the hypotheses with small scores.
%\begin{eqnarray}
% |\,\robdist_S(\hat\mu) -\robdist_{S'}(\hat\mu)\, | \; \leq \; C' \frac{\rho_1}{\alpha n}\Big(1 + \frac{\rho_1 \|\Sigma^{-1/2}(\hat\mu-\mu)\|}{\alpha}\Big)\;,
%\label{eq:intuition_sensitivity}
%\end{eqnarray}
%for some constant $C'$ and all neighboring datasets $S'$, assuming $\rho_2$ is sufficiently small.
%Note that this sensitivity bound is {\em local} for two reasons; for this sensitivity to be small (i.e.~$O(\rho_1/(\alpha n))$), we require $S$  to be resilient  and $\hat\mu$ to be close to $\mu$. Thus the meaning of {\em local} here is two folded while traditionally local sensitivity in the privacy literature only concerns the sensitivity of a particular dataset $S$.
We handle this locality with {\sc Test} step that checks that the 
DP conditions are satisfied for the given dataset and the choice of $\sens:=C'\sqrt{\log(1/\alpha)}/n$ and $\thresh:=C\alpha\sqrt{\log(1/\alpha)}$. 
It is critical for ensuring DP that these choices only depend on the resilience (which is the property of the distribution family of interest, which in this case is sub-Gaussian) and the target accuracy, and not on the dataset $S$.   

%, which  bounds the support of the exponential mechanism to be within 
%The second locality is due to the multiplicative dependence of $\robdist_S(\hat\mu)$ on $(1/\sigma_v({\cal M}_{v,\alpha}))$. 
%By setting the support of the exponential mechanism to be 
%$B_{\thresh,S}=\{\hat\mu:\robdist_S(\hat\mu)\leq\tau\}$ with a choice of $\tau=O(\rho_1 )$. Consequently, we require $\rho_1^2/\alpha \ll 1$ for the second term in Eq.~\eqref{eq:intuition_sensitivity} to be dominated by the first. 
%Fortunately, this is indeed true for all scenarios we are interested in. For sub-Gaussian distributions, $\rho_1^2 = \alpha^2\log(1/\alpha) \ll \alpha$. For $k$-th moment bounded distributions with $k>3$, $\rho_1^2=\alpha^{2-2/k}\ll \alpha$. For covariance bounded distributions, we do not hope to get a  Mahalanobis distance guarantee. Instead, we aim for a Euclidean distance guarantee whose sensitivity does not depend  on $\hat\mu$ and  we do  not require $\rho_1^2/\alpha \ll 1$  (Section~\ref{sec:mean-2moment}).

\medskip \noindent
{\bf Sample complexity analysis.} 
Assuming the sensitivity is bounded by $\sens=C'\sqrt{\log(1/\alpha)}/ n$, which we ensure with the safety test, we analyze the  utility of the exponential mechanism. For a target accuracy of  $\|\Sigma^{-1/2}(\hat\mu-\mu)\|=  O(\alpha\sqrt{\log(1/\alpha)})$, we consider two sets,  $B_{\rm out}=\{\hat\mu : \|\Sigma^{-1/2}(\hat\mu-\mu)\|\leq c_0 \alpha\sqrt{\log(1/\alpha)} \}$ and 
$B_{\rm in}=\{\hat\mu : \|\Sigma^{-1/2}(\hat\mu-\mu)\|\leq c_1 \alpha\sqrt{\log(1/\alpha)} \}$, for some $c_0>c_1$. 
The exponential mechanism achieves accuracy $c_0\alpha\sqrt{\log(1/\alpha)} $ with  probability $1-\zeta$ if 
\begin{eqnarray*}
    {\mathbb P}(\hat\mu\notin B_{\rm out}) \;\leq\; 
    \frac{{\mathbb P}(\hat\mu\notin B_{\rm out})}{{\mathbb P}(\hat\mu\in B_{\rm in})} \;\lesssim\; \frac{{\rm Vol}(B_{\tau,S})}{{\rm Vol}(B_{\rm in})}    \frac{e^{-\frac{\varepsilon}{4\sens}c_0\alpha\sqrt{\log(1/\alpha)} }}{e^{-\frac{\varepsilon}{4\sens}c_1\alpha\sqrt{\log(1/\alpha)}}} \leq e^{O(d)}e^{-\frac{\varepsilon}{4\sens}(c_0-c_1)\alpha\sqrt{\log(1/\alpha)}} \leq \zeta\;,
\end{eqnarray*}
where the second inequality  requires 
$\robdist_S(\hat\mu)\simeq \|\Sigma^{-1/2}(\hat\mu-\mu)\|$,  which we show in Lemma~\ref{lem:outerset}. 
Since  $\sens=O(\sqrt{\log(1/\alpha)}/n)$, it is sufficient to have a large enough $c_0$ and $n=\tilde O((d+\log(1/\zeta))/(\alpha\varepsilon) )$ with a large enough constant. Together with the sample size required to ensure resilience, this gives the desired sample complexity of $n=\tilde O(d/\alpha^2 + (d+\log(1/\zeta))/(\alpha\varepsilon))$ where $\tilde O$ hides polylogarithmic factors in $1/\alpha$ and $1/\delta$.

\medskip\noindent{\bf Safety test.}  
We are left to show that for a resilient dataset, the failure probability of the safety test, 
${\mathbb P}(m_\tau + {\rm Lap}(2/\varepsilon) < (2/\varepsilon)\log(2/\delta) )$, is less than $\zeta$. 
This requires the safety margin to be large enough, i.e.~$m_\tau\geq k^*=(2/\varepsilon)\log(4/(\delta\zeta))$. 
Recall that the safety margin is defined as  the Hamming distance to the closest dataset to $S$ where the $(\varepsilon/2,\delta/2)$-DP condition of the exponential mechanism is violated. 
We therefore need to show that the DP condition is satisfied for not only $S$ but any dataset $S'$ at Hamming distance at most $k^*$ from $S$. 
We treat $S'$ as a {\em corrupted} version of a resilient $S$ by a fraction $k^*/n$-corruption. Since we are using robust statistics that are designed to be robust against data corruption, we can show that the corrupted resilient set still has a low sensitivity for $D_{S'}(\hat\mu)$. Building upon the proof techniques developed in \cite{brown2021covariance} for a safety test for a Tukey median based exponential mechanism, we  use the fact that $S'$ is a corrupted version of a resilient dataset $S$ to show that the safety test passes with high probability.

\section{Preliminary on differential privacy and Propose-Test-Release}
\label{sec:dp}

We give the backgrounds on differential privacy and the Propose-Test-Release mechanism. We say two datasets $S$ and $S'$ of the same size are neighboring if the Hamming distance between them is at most one. There is another equally popular  definition where injecting or deleting one data point to $S$ is considered as a neighboring dataset. All our analysis generalizes to that definition also, but notations get slightly heavier.    

\begin{definition}[\cite{dwork2006calibrating}]
    \label{def:dp}
    We say a randomized algorithm $M:\cX^n\rightarrow\cY$ is $(\varepsilon, \delta)$-differentially private  if for all neighboring databases  $S\sim S'\in \cX^n$, and all $Y\subseteq \cY$, we have $\prob(M(S)\in Y)\leq e^{\varepsilon}\prob(M(S')\in Y)+\delta$.
\end{definition}

HPTR relies on the exponential mechanism for its adaptivity and flexibility. 

\begin{definition}[Exponential mechanism \cite{mcsherry2007mechanism}]
    The exponential mechanism $M_{\rm exp}:\cX^n\rightarrow \Theta$ takes database $S\in \cX^n$, candidate space $\Theta$, score function $D_S(\hat{\theta})$ and sensitivity $\Delta$ as input, and select output with probability proportional to $\exp\{- \varepsilon D_S(\hat{\theta})/2\Delta\}$. 
\end{definition}

The exponential mechanism is $(\varepsilon,0)$-DP if the sensitivity of $D_S(\hat\theta)$ is bounded by $\sens$. 
\begin{lemma}[\cite{mcsherry2007mechanism}]
    \label{lem:exp_mech}
    If  $\max_{\hat{\theta}\in \Theta}\max_{S\sim S'} | D_S(\hat{\theta})-D_{S'}(\hat{\theta}) | \leq \sens$, then the exponential mechanism is $(\varepsilon,0)$-DP.
\end{lemma}

Starting from the seminal paper \cite{dwork2009differential}, there are increasing efforts to apply differential privacy to statistical problems, where the dataset consists of i.i.d.~samples from a distribution. There are two main challenges. First, the support is typically not bounded, and hence, the sensitivity is unbounded. \cite{dwork2009differential} proposed to resolve this by using robust statistics, such as median,  to estimate the mean. The second challenge is that while median is quite insensitive on i.i.d.~data, this low sensitivity is only local and holds only for i.i.d.~data from a certain class of distributions. This led to the original definition of local sensitivity in the following. 

\begin{definition}[Local Sensitivity]
    We define local sensitivity of dataset $S\in \cX^n$ and function $f:\cX^n\rightarrow \reals$ as $\Delta_f(S):=\max_{S'\sim S}|f(S)-f(S')|$.
\end{definition}

\cite{dwork2009differential} introduced the Propose-Test-Release mechanism to resolve both issues. First, a certain  robust statistic $f(S)$, such as median, mode, Inter-Quantile Range (IQR), or  B-robust regression model  \cite{hampel2011robust} is chosen as a query. It can be used to approximate a target statistic of interest, such as mean, range, or linear regression model, or the robust statistic itself could be the target. 
Then, the PTR mechanism proceeds in three steps. In the propose step, a local sensitivity $\sens$ is proposed such that $\sens_f(S)\leq \sens$ for all $S$ that belongs to a certain family. In the test step, a safety margin $m$, which is how many data points have to be changed to violate the local sensitivity, is computed and a private version of the safety margin, $\hat m$, is compared with a threshold. If the safety margin is large enough, then the algorithm outputs $f(S)$ via a Laplace mechanism with parameter $2\sens/\varepsilon$. Otherwise, the algorithm halts and outputs $\perp$.

\begin{definition}[Propose-Test-Release (PTR) \cite{dwork2009differential, vadhan2017complexity}]
For a query function $f: \cX^n\rightarrow \reals$, the ${\rm PTR}$ mechanism $M_{\rm PTR}:\cX^n\rightarrow \reals$ proceeds as follows:
\begin{itemize} 
    \item[1.] {\sc \bf Propose}: 
        Propose a target bound $\sens\geq 0 $ on local sensitivity.
    \item[2.] {\sc \bf Test}: 
    \begin{itemize}
    \item[2.1.]     Compute  $m = \min_{S'} d_H(S,S')$ such that local sensitivity of $S'$ satisfies $\Delta_f(S')\geq \sens$.
    \item[2.2.]      If $\hat{m} = m  + {\rm Lap}(2/\varepsilon)<(2/\varepsilon)\log(1/\delta)$ then output $\perp$, and otherwise continue.
    \end{itemize} 
    \item[3.] {\sc \bf Release}: Output $f(S)+{\rm Lap}(2\sens/\varepsilon)$.
\end{itemize} 

\end{definition}
It immediately follows that PTR is $(\varepsilon,\delta)$-differentially private for any input dataset. 

\begin{lemma}[\cite{dwork2009differential, vadhan2017complexity}]
$M_{\rm PTR}$ is $(\varepsilon, \delta)$-DP
\end{lemma}

Given a robust statistic of interest,  the art is in identifying the family of datasets with small local sensitivity and showing that the sensitivity is small enough to provide good  utility. 
For example, for privately releasing the mode, for the family of distributions whose occurrences of the mode  is at least $(4/\varepsilon) \log(1/\delta)$ larger than the occurrences of the second most frequent value, the local sensitivity is zero and PTR outputs the true mode with probability at least $1-\delta$ 
\cite{vadhan2017complexity}. 
Such a specialized PTR mechanism for zero local sensitivity is also called the stability based method.

In general, a naive method of computing $m$ in the TEST step requires enumerating  over all possible databases $S\in \cX^n$. For typical one-dimensional data/statistics, for example median estimation, this step can be computed efficiently.
This led to a fruitful line of research in DP statistics on one-dimensional data. \cite{dwork2009differential, brunel2020propose} propose PTR mechanisms for the range and the median of of a 1-D smooth distribution and  
\cite{avella2019differentially, avella2020role, brunel2020propose} propose PTR mechanisms that can  estimating median and mean of a 1-D sub-Gaussian distribution.  
The stability-based method introduced in \cite{vadhan2017complexity} can be used to release private  histograms, among other things, which can be subsequently used as a black box to solve several important problems including range estimation of a 1-D sub-Gaussian distribution \cite{KV17,KLSU19,liu2021robust} or a 1-D heavy-tailed distribution \cite{kamath2020private,liu2021robust}, and general counting queries.  
PTR and stability-based mechanisms are powerful tools  when estimating robust statistics of a distribution from i.i.d.~samples. 

Even if computational complexity is not concerned, however, directly  applying PTR to high dimensional  distributions can increase the  statistical cost significantly, which has limited the application of PTR. One exception is the recent work of   \cite{brown2021covariance}. For the mean estimation problem with Mahalanobis error metric of $\|\Sigma^{-1/2}(\hat\mu-\mu)\|$,  the private Tukey median mechanism introduced in \cite{liu2021robust} is studied. One major limitation of the utility analysis  is that private Tukey median requires the support to be bounded. In \cite{liu2021robust}, this is circumvented by assuming the covariance $\Sigma$ is known, in which case one can find a support with, for example, the private histogram of \cite{vadhan2017complexity}. Instead,  \cite{brown2021covariance} proposed using private Tukey median inside the PTR mechanism and designed an advanced safety test for high-dimensional problems. This naturally bounds the support that adapts to the geometry of the problem without explicitly and privately estimating $\Sigma$. 
One notable byproduct of this approach is that the resulting exponential mechanism is no longer pure DP, but rather $(\varepsilon,\delta)$-DP. This is because the resulting exponential mechanism has a support that depends on the dataset $S$, and hence two exponential mechanisms on two  neighboring datasets have different supports. 
The limitations of the private Tukey median are that $(i)$ it requires  symmetric distributions, like Gaussian distributions, and do not generalize to even sub-Gaussian distributions, and  $(ii)$ it only works for mean estimation. To handle the first limitation, \cite{brown2021covariance} propose another PTR mechanism using Gaussian noise, which works for more general sub-Gaussian distributions but achieves sub-optimal sample complexity. 

HPTR builds upon this advanced PTR with the high-dimensional safety test from \cite{brown2021covariance}. However, there are major challenges in applying this safety test to HPTR, which we overcome with the resilience property of the dataset and the robustness of the loss function.  
For private Tukey median, the sensitivity is always one for any $\hat\mu$ and any $S$, and the only purpose of the safety test is to ensure that the support is not too different between two neighboring datasets. For HPTR, the sensitivity is local in two ways: it requires  $S$ to be resilient and the estimate  $\hat\mu$ to be sufficiently close to $\mu$. To ensure a large enough margin when running the safety test, HPTR requires this local sensitivity to hold not just for the given $S$ but for all $S'$ within some Hamming distance from $S$. We use the fact that this larger neighborhood is included in an even larger set of databases that are  adversarial corruption of the $\alpha$-fraction of the original resilient dataset $S$ with a certain choice of $\alpha$.  The robustness of our loss function implies that the bounded sensitivity is preserved under such corruption of a resilient dataset. This is critical in proving that a resilient dataset passes the safety test  with high probability.

We take a first-principles approach to design a universal framework for DP statistical estimation that blends  exponential mechanism, robust statistics, and PTR. The exponential mechanism in HPTR adapts to the geometry of the problem without explicitly estimating any other parameters and also gives us the flexibility to apply to a wide range of problems.  
The choice of the loss functions that only depend on one-dimensional statistics is critical in achieving the low sensitivity, which directly translates into near optimal utility guarantees for several canonical problems. Ensuring differential privacy is achieved by building upon the advanced PTR framework of \cite{brown2021covariance}, with a few  critical differences. Notably, the safety analysis uses the resilience of robust statistics in a fundamental way.

On the other hand, 
there is a different way of handling local sensitivity, which is known as smooth sensitivity. 
Introduced in \cite{nissim2007smooth},  smooth sensitivity is a smoothed version of local sensitivity on the neighborhood of the dataset, defined as \begin{eqnarray*}
\Delta^{\rm smooth}_{f}(S) \; = \; \max_{S'\in \cX^n}\{\Delta_f(S')e^{-\varepsilon d_{\rm H}(S, S')}\}
\end{eqnarray*} 
Note that, in general, computing smooth sensitivity is also computationally inefficient with an exception of \cite{avella2021privacy}. 
Using smooth sensitivity, \cite{lei2011differentially,smith2011privacy,chaudhuri2012convergence,avella2021privacy} leverage robust M-estimators for differentially private estimation and inference. The intuition is based on the fact that the influence function of the M-estimators can be used to bound the smooth sensitivity. The applications include: linear regression, location estimation, generalized linear models, private testing. However, these approaches require restrictive assumptions on the dataset that need to be checked (for example via PTR) and 
fine-grained analyses on the statistical complexity is challenging;
there is no sample complexity analysis comparable to ours. 
One exception is \cite{bun2019average}, which proposes a smooth  sensitivity based approach and gives an upper bound on the sub-Gaussian mean estimation error for a finite $n$, but only for one-dimensional data. 

\section{Mean estimation} 
\label{sec:mean}

In a standard mean estimation, we are given i.i.d.~samples $S=\{x_i\in\reals^d\}_{i=1}^n$ drawn from a distribution $P_{\mu,\Sigma}$ with an unknown mean $\mu$ (which corresponds to $\theta$ in the general notation) and an unknown covariance $\Sigma \succ 0$ (which corresponds to $\phi$ in the general notation), and we want to produce a 
DP estimate $\hat\mu$ of the mean. The resulting error is best measured in Mahalanobis distance,  $D_\Sigma(\hat\mu,\mu)=\|\Sigma^{-1/2}(\hat\mu-\mu)\|$, because this is a scale-invariant distance; every direction has unit variance after whitening by $\Sigma$. 

This problem is especially challenging as  we aim for a tight  guarantee that adapts to the unknown $\Sigma$ as measured in the Mahalanobis distance 
without enough samples to directly estimate $\Sigma$ (see Section~\ref{sec:intro_main} for a survey).
Despite being a canonical problem in DP statistics, the optimal sample complexity is not known even for standard distributions: sub-Gaussian and heavy-tailed distributions. 
We characterize the optimal sample complexity by showing that 
HPTR matches the known lower bounds 
in Section~\ref{sec:mean3}. 
This follows directly from the general three-step strategy outlined in Section~\ref{sec:recipe}. 
%We first  design  $\robdist_S(\hat\mu)$. Next, we    propose   $\sens$ and $\thresh$  based on certain resilience properties of the dataset  and analyze the utility. Finally, we apply this analysis to each family of distributions of interest.  

% ----------------------------------------
\subsection{Step 1: Designing the surrogate \texorpdfstring{$\robdist_S(\hat\mu)$}{}  for the Mahalanobis distance}
\label{sec:mean1}

We want to privately release $\hat\mu$ with small Mahalanobis distance $\|\Sigma^{-1/2}(\hat\mu-\mu)\|$. 
In the exponential mechanism in {\sc Release} step, we propose using the surrogate distance,  
\begin{eqnarray}
    \label{def:mean_distance}
    \robdist_S(\hat\mu) \;\; = \;\; \max_{v:\|v\|\leq 1}  \frac{\langle  v ,\hat\mu \rangle - \mu_v({\cal M}_{v,\alpha})  }{\sigma_v({\cal M}_{v,\alpha}) }  \;,
\end{eqnarray}
where 
the robust one-dimensional mean $\mu_v({\cal M}_{v,\alpha})$ and variance $\sigma_v^2({\cal M}_{v,\alpha})$ are defined as follows. 
We partition $S=\{x_i\}_{i=1}^n$  into three sets ${\cal B}_{v,\alpha}$, ${\cal M}_{v,\alpha}$, and ${\cal T}_{v,\alpha}$, by considering 
 a set of projected data points $S_v=\{\langle v , x_i\rangle \}_{x_i\in S}$ and  
letting ${\cal B}_{v,\alpha}$ be the data points corresponding to the subset of  bottom $(2/5.5)\alpha n$ data points with smallest values in $S_v$, ${\cal T}_{v,\alpha}$ be the subset of top  $(2/5.5)\alpha n$ data points with largest  values, and ${\cal M}_{v,\alpha}$ be the subset of remaining  $(1-(4/5.5)\alpha) n$ data points. 
For a fixed direction $v$, define 
\begin{eqnarray}
\mu_v({\cal M}_{v,\alpha})\;=\;\frac{1}{|{\cal M}_{v,\alpha}|}
\sum_{x_i\in{\cal M}_{v,\alpha}}\langle v, x_i \rangle \; , \text { and }\;
\sigma^2_v({\cal M}_{v,\alpha})=\frac{1}{|{\cal M}_{v,\alpha}|}\sum_{x_i\in {\cal M}_{v,\alpha}}( \langle v , x_i \rangle - \mu_v({\cal M}_{v,\alpha}))^2 \;, \label{eq:defproj}
\end{eqnarray} 
which are robust estimates of the population projected mean $\mu_v = \langle v , \mu \rangle$ and the population projected variance $\sigma_v^2=v^\top\Sigma v$.  

\medskip
\noindent 
{\bf General guiding principles for designing $\robdist_S(\hat\mu)$.} We propose %follow Our choice of $\robdist_S(\hat\mu)$ is guided by 
the following three design principles that apply more generally to all problem instances of interest. 
The first guideline is that   it should recover the target error metric $D_\Sigma(\hat\mu,\mu)= \|\Sigma^{-1/2}(\hat\mu-\mu)\|$ when we  substitute the population statistics, e.g.~$\mu_v$ and $\sigma_v$ for mean estimation, for their robust counterparts:  
$\mu_v({\cal M}_{v,\alpha})$ and $\sigma_v({\cal M}_{v,\alpha})$.  
This  ensures that minimizing $D_S(\hat\mu)$ is approximately equivalent to minimizing the target metric $D_\Sigma(\hat\mu,\mu)=\|\Sigma^{-1/2}(\hat\mu-\mu)\|$ (Lemma~\ref{lem:outerset}). For mean estimation, this equivalence is shown  in the following lemma. 
\begin{lemma}
    \label{lem:equidist}
    For any $\mu\in\reals^d$ and $0\prec \Sigma \in\reals^{d\times d}$, let $\mu_v=\langle v , \mu \rangle$ and  $\sigma_v^2 = v^\top\Sigma v$. Then,  we have 
    \begin{eqnarray*}
        \|\Sigma^{-1/2}(\hat\mu-\mu)    \| \;\;=\;\; \max_{v:\|v\|\leq 1} \frac{\langle v, \hat\mu \rangle - \mu_v}{\sigma_v} \;.
    \end{eqnarray*} 
\end{lemma}
\begin{proof}
    Let $\hat\mu-\mu=\sum_{\ell=1}^d a_\ell u_\ell $ with $a_\ell=\langle u_\ell,\hat\mu-\mu\rangle $, $\|a\|=\|\hat\mu-\mu\|$ and $u_\ell$'s are the singular vectors of $\Sigma$.
    Similarly, let $v=\sum_{\ell=1}^d b_\ell u_\ell$ with $\|b\|=1$. Then we have 
    $$\|\Sigma^{-1/2}(\hat\mu-\mu)    \|^2=\sum (a_\ell^2/\sigma_\ell) \text{ and }   \frac{\langle v, (\hat\mu-\mu)\rangle}{ \sigma_v} = \frac{\langle a,b\rangle }{ \sqrt{\sum b_\ell^2\sigma_\ell }}\;.$$ 
    From Cauchy-Schwarz, we have $ \langle a,b\rangle^2 \leq (\sum b_\ell^2\sigma_\ell)(\sum a_\ell^2\sigma_\ell^{-1})$, which proves that $$\|\Sigma^{-1/2}(\hat\mu-\mu)    \| \geq  \max_{v:\|v\|=1} (1/\sigma_v) \langle v, (\hat\mu-\mu)\rangle\;.$$  
    
    To show equality,  we find $v$ that makes Cauchy-Schwarz inequality tight. Let $v=\sum_{\ell=1}^d b_\ell u_\ell$ with a choice of $b_\ell = (1/Z)a_\ell \sigma_\ell^{-1}$ and $Z= \sqrt{\sum_\ell a_\ell^2\sigma_\ell^{-2}}$. This implies $\|b\|=1$ and 
    \begin{eqnarray*}
        \langle a,b\rangle \;=\; \frac1Z \sum_{\ell=1}^d (1/\sigma_\ell) a_\ell^2, \text{ and } \;\;
        \sqrt{\sum b_\ell^2 \sigma_\ell} \; =\; \frac{1}{Z}\sqrt{\sum_{\ell=1}^d  
        (1/\sigma_{u_\ell}) a_\ell^2} \;,
    \end{eqnarray*}
    which implies that there exists a $v$ such that 
    $\|\Sigma^{-1/2}(\hat\mu-\mu)\|=(1/\sigma_v)\langle v,\hat\mu-\mu\rangle $ and $\|\Sigma^{-1/2}(\hat\mu-\mu)\| \leq \max_{v:\|v\|=1}(1/\sigma_v)\langle v,\hat\mu-\mu\rangle $. 
\end{proof}

 The second guideline is that $\robdist_S(\hat\mu)$ should  depend only on the  {\em one-dimensional} statistics of the data.  This is critical as 
the sensitivity of high-dimensional statistics  increases with the ambient dimension $d$.
For example, consider using the robust mean estimate $\hat\mu_{\rm robust}(S)\in\reals^d$   from \cite{dong2019quantum} and using  the Euclidean distance $\robdist_S(\hat\mu) = \|\hat\mu-\hat\mu_{\rm robust}(S) \| $ in the exponential mechanism, where we are assuming $\Sigma={\bf I}$ for simplicity.
It can be shown that, even for Gaussian distributions, this requires $n=\widetilde\Omega(d^{3/2}/(\varepsilon\alpha) + d/\alpha^2)$ samples to achieve an accuracy $\|\hat{\mu}-\mu\|=\widetilde{O}(\alpha)$. This is significantly sub-optimal compared to what HPTR achieves in  Corollary~\ref{coro:mean_subgaussian}, which leverages the fact that sensitivity of  one-dimensional statistic is  dimension-independent. 
 
The last guideline  is to  use robust statistics.  Robust statistics have small sensitivity on {\em resilient} datasets, which is critical in achieving the near-optimal guarantees. 
We elaborate on it in Section~\ref{sec:mean_proof_strategy}.

% ------------------------------------
\subsection{Step 2: Utility analysis under resilience}
\label{sec:mean2} 

For utility, we prefer smaller $\sens$ and $\thresh$ to ensure that the exponential mechanism samples $\hat\mu$ closer to the minimum of $D_S(\hat\mu)\approx \|\Sigma^{-1/2}(\hat\mu-\mu)\|$. 
%$\|\Sigma^{-1/2}(\hat\mu-\mu)\| \approx \frac{d \sens}{\varepsilon}$
However, aggressive choices can violate  the DP condition and hence fail the safety test. Near-optimal utility can be achieved by selecting $\sens$ and $\thresh$ based on the {\em resilience}  of the dataset   defined as follows.  

\begin{definition}[Resilience for mean estimation \cite{steinhardt2018resilience,zhu2019generalized}]
    \label{def:resilience}
    For some $\alpha\in(0,1)$, $\rho_1\in\reals_+$, and $\rho_2\in\reals_+$, we say a set of $n$ data points $S_{\rm good}$ is $(\alpha,\rho_1,\rho_2)$-resilient with respect to $(\mu,\Sigma)$ if for any $T\subset S_{\rm good}$ of size $|T| \geq (1-\alpha)n$, the following holds %for some universal constant $c>0$ and 
    for all $v\in\reals^d$ with $\|v\|=1$:     
    \begin{eqnarray}
         \Big| \frac{1}{|T|}\sum_{x_i\in T} \langle v, x_i \rangle -\mu_v   \Big|  & \leq &  \rho_1 \, \sigma_v    \;,\text{ and } \label{def:rho1}\\
      \Big| \frac{1}{|T|}\sum_{x_i\in T}  \big(\langle v, x_i\rangle-\mu_v \big)^2  - \sigma_v^2  \Big| & \leq &  \rho_2 \, \sigma_v^2  \;, \label{def:rho2}
    \end{eqnarray}
    where $\mu_v=\langle v , \mu\rangle$ and $\sigma_v^2 = v^\top \Sigma v$.
    %For some $\rho_1:[0,1]\to\reals_+$ and $\rho_2:[0,1]\to\reals_+$, we say a set of $n$ data points $S_{\rm good}$ is $(\rho_1,\rho_2)$-resilient with respect to $(\mu,\Sigma)$ if for any $T\subset S_{\rm good}$ of size $|T|=(1-\alpha)n$, the following holds for all $\alpha\in(0,c)$ for some $c>0$ and for all $v\in\reals^d$ with $\|v\|=1$:     
    %\begin{eqnarray}
    %     \Big| \frac{1}{|T|}\sum_{i\in T} \langle v, x_i \rangle -\mu_v   \Big|  & \leq &  \rho_1(\alpha)\, \sigma_v    \;,\text{ and } \label{def:rho1}\\
    %  \Big| \frac{1}{|T|}\sum_{i\in T}  \big(\langle v, x_i\rangle-\mu_v \big)^2  - \sigma_v^2  \Big| & \leq &  \rho_2(\alpha)\, \sigma_v^2  \;, \label{def:rho2}
    %\end{eqnarray}
    %where $\mu_v=\langle v , \mu\rangle$ and $\sigma_v^2 = v^\top \Sigma v$.
\end{definition}

Originally, resilience is introduced in the context of robust statistics.
% meaning / connection to tail / small vs large
Resilience measures how sensitive the sample statistics are to removing an $\alpha$-fraction of the data points. 
 A dataset from a distribution with a lighter tail has smaller resilience  $(\rho_1,\rho_2)$. For example, sub-Gaussian distributions have $\rho_1=O(\alpha\sqrt{\log(1/\alpha)})$ and $\rho_2 =O(\alpha\log(1/\alpha))$ (Lemma~\ref{lem:mean_subgaussian}),  which is smaller than the resilience of heavy-tailed distributions with bounded $k$-th moment, i.e.~$\rho_1 = O(\alpha^{1-1/k})$ and $\rho_2 = O(\alpha^{1-2/k})$ (Lemma~\ref{lem:mean_kmoment}).
 Resilience plays a crucial role in robust statistics, where the resilience of a dataset determines the minimax sample complexity of estimating population statistics from adversarially corrupted samples \cite{steinhardt2018resilience,zhu2019generalized}. 

In the context of differential privacy, 
our design of HPTR is guided by our analysis showing that the sensitivity of one-dimensional robust statistics is fundamentally governed by resilience. 
Leveraging this three-way connection  between the use of robust statistics in the algorithm, the resilience of the data, and the sensitivity of the distance $\robdist_S(\hat\mu)$ is crucial in achieving the near-optimal utility.

Concretely, we consider $\alpha$ as a free parameter that we can choose depending on the target accuracy. 
 For example, let $\|\Sigma^{-1/2}(\hat\mu-\mu)\|=32\rho_1$ be our target accuracy. 
 Note that we did not optimize the constants in our analysis and they can be further tightened. 
In the case of sub-Gaussian distributions, we have $\rho_1 = C'\alpha\sqrt{\log(1/\alpha)}$ w.h.p.~when the sample size is large enough. This determines the value of  $\alpha$ that  achieves a target accuracy and also the choice of $\sens$ and $\tau$ as follows. 

The robust statistics of a resilient dataset (i.e., one with small resilience) cannot  change too much when  a small  fraction of the dataset is changed. 
This is made precise in Lemma~\ref{lem:local_asmp} which shows, for example, that the robust mean $\mu_v({\cal M}_{v,\alpha})$ can only change by 
$O(\rho_1 /(\alpha n))$ when one data point is arbitrarily changed. 
This implies the sensitivity of $\robdist_S(\hat\mu)$ is also small: $\sens=O(\rho_1 /(\alpha n))$.
Choosing  $\thresh=42\rho_1$ to be larger by a constant factor from the target accuracy, we show that a sample size of  $n = O(d/(\varepsilon\alpha))$ is sufficient  to achieve the desired utility. 

\begin{thm}[Utility guarantee for mean estimation]
    \label{thm:mean}  
There exist positive constants $c$ and $C$ such that for any  $(\alpha,\rho_1,\rho_2)$-resilient set $S$ with respect to some  $(\mu\in\reals^d ,\Sigma\succ 0)$  satisfying $\alpha\in(0,c)$, $\rho_1<c$, $\rho_2<c$, and $\rho_1^2\leq c \alpha$, $\HPTR$ with the choices of the distance function in Eq.~\eqref{def:mean_distance},  $\sens=110 \rho_1/(\alpha n)$, and $\thresh=42\rho_1$  achieves 
    $\|\Sigma^{-1/2}(\hat\mu-\mu) \|\leq 32 \rho_1$ with probability $1-\zeta$, if \begin{eqnarray*}
         n \;\geq \; C \frac{d+\log(1/(\delta\zeta))}{\varepsilon \alpha}  %+ \frac{\rho_1(5.5\alpha)^2}{\alpha} 
         \; . 
    \end{eqnarray*} 
 %   There exist positive constants $c$ and $C$ such that for any  $(\rho_1,\rho_2)$-resilient set $S$ and for any $\alpha\in(0,c)$ satisfying $\rho_1(5.5\alpha)<c$, $\rho_2(5.5\alpha)<c$, and $\rho_1(5.5\alpha)^2\leq c \alpha$, $\HPTR$ with the choices of the distance function in Eq.~\eqref{def:mean_distance},  $\sens=20 \rho_1(\alpha/2)/(\alpha n)$, and $\thresh=42\rho_1(5.5\alpha)$  achieves 
%    $\|\Sigma^{-1/2}(\hat\mu-\mu) \|\leq 32 \rho_1(5.5\alpha)$ with probability $1-\zeta$, if \begin{eqnarray*}
 %        n \;\geq \; C \frac{d+\log(1/(\delta\zeta))}{\varepsilon \alpha}  %+ \frac{\rho_1(5.5\alpha)^2}{\alpha} 
%         \; . 
    %\end{eqnarray*} 
\end{thm}

This theorem shows how a resilient dataset (which is a deterministic condition) implies small error for HPTR. 
We make formal connections to standard assumptions on the sample generating distributions and their respective resiliences in  Section~\ref{sec:mean3}, where we also discuss the  optimality of this utility guarantee.    
For example, sub-Gaussian distributions have $\rho_1 = O(\alpha\sqrt{\log(1/\alpha)})$ when $n\geq C' d/(\alpha \log(1/\alpha))^2$ for any $\alpha$ smaller than a universal constant.  This implies that 
HPTR achieves a target accuracy of $\|\Sigma^{-1/2}(\hat\mu-\mu)\| \leq \tilde\alpha$ with sample size $\tilde{O}(\frac{d}{\tilde\alpha^2}+\frac{d}{\tilde\alpha \varepsilon})$ where $\tilde{O}$ hides logarithmic factors in $1/\alpha$, $\delta$, and $\zeta$.   
We explain the intuition behind our analysis   and provide a complete proof in Sections~\ref{sec:mean_proof_strategy}--\ref{sec:mean_proof2}. One by-product of using robust statistics is that we get robustness for free, which we show next.

\subsubsection{Robustness of HPTR}

One by-product of using robust statistics  is that HPTR is also robust to adversarial corruption. We therefore provide a more general guarantee that simultaneously achieves DP and robustness. Suppose we are given a dataset $S$ that is a corrupted version of a resilient dataset $S_{\rm good}$. 

\begin{asmp}[$\alpha_{\rm corrupt}$-corruption]  \label{asmp:mean} 
Given a set $S_{\rm good}=\{\tilde{x}_i\in\reals^d\}_{i=1}^n$ of $n$ data points, an adversary inspects all data points, selects $\alpha_{\rm corrupt} n$ of the data points, and replaces them with arbitrary dataset $S_{\rm bad}$ of size $\alpha_{\rm corrupt} n$. 
The resulting corrupted dataset is called $S=\{x_i\in\reals^d\}_{i=1}^n$. 
\end{asmp}  

 This adaptive adversary is strong, as the corruption can adapt to the entire dataset (for example it covers  the Huber contamination model  \cite{huber1964robust} and the non-adaptive adversarial model \cite{lecue2020robust}). This threat model is now standard in robust statistics literature \cite{steinhardt2018resilience}.  
If the original $S_{\rm good}$ is resilient, we show that the same guarantee as Theorem~\ref{thm:mean} holds under corruption up to an $\alpha_{\rm corrupt}$ fraction of $S_{\rm good}$ for sufficiently small $\alpha_{\rm corrupt}\leq (1/5.5)\alpha$. The factor $1/5.5$ is due to the fact that the algorithm 
treats some of the good data points as outliers (which is at most $4\alpha_{\rm corrupt}$ due to the top and bottom tails cut in the definition of ${\cal M}_{v,(2/5.5)\alpha}$) and we need to handle neighboring datasets up to $(0.5/5.5)\alpha n$ Hamming distance. Hence, we need to ensure resilience for $\alpha$ at least $5.5$ times larger than the corruption $\alpha_{\rm corrupt}$. 

\begin{definition}[Corrupt good set]
    \label{def:corruptgoodset}
    We say a dataset $S$ is $(\alpha_{\rm corrupt},\alpha,\rho_1,\rho_2 )$-corrupt good with respect to $(\mu,\Sigma)$ if it is an $\alpha_{\rm corrupt}$-corruption of an $(\alpha,\rho_1,\rho_2)$-resilient dataset $S_{\rm good}$.
\end{definition}

We get the following theorem showing that HPTR can tolerate up to $(1/5.5)\alpha$ fraction of the data being arbitrarily corrupted. 
\begin{thm}[Robustness]
    \label{thm:robust_mean}  There exist positive constants $c$ and $C$ such that for any  $((2/11)\alpha,\alpha,\rho_1,\rho_2)$-corrupt good set $S$   with respect to $(\mu\in\reals^d,\Sigma\succ 0)$ satisfying $\alpha<c$,   $\rho_1<c$, $\rho_2<c$, and $\rho_1^2\leq c \alpha$, $\HPTR$ with the distance function in Eq.~\eqref{def:mean_distance},  $\sens=110 \rho_1/(\alpha n)$, and $\thresh=42\rho_1$  achieves 
    $\|\Sigma^{-1/2}(\hat\mu-\mu) \|\leq 32 \rho_1$ with probability $1-\zeta$, if \begin{eqnarray*}
         n \;\geq \; C  \frac{d+\log(1/(\delta\zeta))}{\varepsilon \alpha}  %+ \frac{\rho_1(5.5\alpha)^2}{\alpha}
          \; .
    \end{eqnarray*}
 %   There exist positive constants $c$ and $C$ such that for any  $(\alpha,\rho_1,\rho_2)$-corrupt good set $S$ and for any $\alpha\in(0,c)$ satisfying  $\rho_1(5.5\alpha)<c$, $\rho_2(5.5\alpha)<c$, and $\rho_1(5.5\alpha)^2\leq c \alpha$, $\HPTR$ with the distance function in Eq.~\eqref{def:mean_distance},  $\sens=20 \rho_1(\alpha/2)/(\alpha n)$, and $\thresh=42\rho_1(5.5\alpha)$  achieves 
  %  $\|\Sigma^{-1/2}(\hat\mu-\mu) \|\leq 32 \rho_1(5.5\alpha)$ with probability $1-\zeta$, if \begin{eqnarray*}
   %      n \;\geq \; C  \frac{d+\log(1/(\delta\zeta))}{\varepsilon \alpha}  %+ \frac{\rho_1(5.5\alpha)^2}{\alpha}
%          \; .
    %\end{eqnarray*}
\end{thm}
In  Sections~\ref{sec:mean_proof_strategy}--\ref{sec:mean_proof2}, we prove this more general result. When there is no adversarial corruption, 
 Theorem~\ref{thm:mean} immediately follows as a special case by selecting $\alpha$ as a free parameter depending on the target accuracy. The constants in all the theorems can be improve if we track them more carefully, and we did not attempt to optimize them in this paper.

\subsubsection{Proof strategy for Theorem~\ref{thm:robust_mean}} 
\label{sec:mean_proof_strategy}

%\subsubsection{Proof strategy} 
%\label{sec:resilience_sensitivity}

We  show in Section~\ref{sec:mean3_sens} that the robust one-dimensional statistics, $\mu_v({\cal M}_{v,\alpha})$ and $\sigma^2_v({\cal M}_{v,\alpha})$, have small sensitivity if the  dataset is resilient. Consequently,  $\robdist_S(\hat\mu)$ has a small {\em local} sensitivity, i.e.~the sensitivity is small  if restricted to $\hat\mu$ close to $\mu$ and if the dataset is resilient. To ensure DP, we run {\sc Release} only when those two locality conditions are satisfied; we first {\sc Propose} the sensitivity $\sens$ and a threshold $\thresh$, and then we {\sc Test} that DP guarantees are met on the given dataset with those choices. 
Resilient datasets $(i)$ pass this safety test  with a high probability and $(ii)$ achieve the desired accuracy, both of which rely on our general analysis of HPTR with a general distance function (Theorem~\ref{thm:utility}). 
 We give sketches of the main steps below.  %Noting that $\rho_i(c\alpha)=\Theta(\rho_i(\alpha))$ for $i\in\{1,2\}$, we do not track the constant factor in the resilience. 
 
% 1-d robuststatistic has small sensitivity 
% D has small, but local sensitivity 
% local 1 - propose \Delta = XXX
% local 2 - propose \tau = XXX 
% check violation - Theorem 5
% utility analysis for sample complexity 

\medskip\noindent{\bf One-dimensional robust statistics have small  sensitivity on resilient datasets.} 
Consider the robust projected mean $\mu_v({\cal M}_{v,\alpha})$ for some small enough $\alpha>0$. 
If $S$ is $(\alpha,\rho_1, \rho_2)$-resilient, then the following technical lemma shows that the top and bottom $(2/5.5)\alpha$-tails cannot deviate too much from the mean. 
\begin{lemma}[Lemma 10 from \cite{steinhardt2018resilience}] 
    \label{lem:deviate}
    For a $(\alpha,\rho_1,\rho_2)$-resilient dataset $S$ with respect to $(\mu,\Sigma)$ and any $0\leq\tilde\alpha\leq\alpha$, the following holds for any subset $T\subset S$ of size at least $\tilde\alpha n$ and for any unit norm $v\in\reals^d$: 
    \begin{eqnarray}
    \Big|\frac{1}{|T|}\sum_{x_i\in T} \langle v, x_i-\mu\rangle\Big| &\leq& \frac{2 - \tilde\alpha}{\tilde\alpha}\,\rho_1 \,\sigma_v\;, \text{ and } \label{eq:resilienceontail1} \\ 
    \Big|\frac{1}{|T|}\sum_{x_i \in T} \big(\, \langle v, x_i-\mu\rangle^2-\sigma_v^2 \,\big) \Big| &\leq& \frac{2-\tilde\alpha}{\tilde\alpha}\,\rho_2 \, \sigma_v^2\;. \label{eq:resilienceontail2}
    \end{eqnarray} 
\end{lemma} 
Under the definitions in Eq.~\eqref{def:mean_distance}, 
the top $(2/5.5)\alpha$-tail denoted by ${\cal T}_{v,\alpha}$ and bottom $(2/5.5)\alpha$-tail denoted by ${\cal B}_{v,\alpha}$ have the  empirical means that are  
no more than $O(\sigma_v\rho_1 /\alpha)$ away from the true projected mean $\mu_v$, respectively.
It follows that there exists at least one data point in ${\cal T}_{v,\alpha}$ and one data point in ${\cal B}_{v,\alpha}$ that are no more than $O(\sigma_v\rho_1 /\alpha)$ away from $\mu_v$. 
This implies that the range of the  middle subset ${\cal M}_{v,\alpha}$ is provably bounded by $O(\sigma_v \rho_1 /\alpha)$, and the sensitivity of the robust mean $\mu_v({\cal M}_{v,\alpha})$ is guaranteed to be $O(\sigma_v \rho_1 /(\alpha n ))$. We can similarly show that $\sigma_v^2({\cal M}_{v,\alpha})$ has sensitivity $O(\sigma_v^2 \rho_1^2/ (\alpha^2 n))$ as shown in Eq.~\eqref{eq:varsensitivitybound}.  
Note that these sensitivity bounds are {\em local} in the sense that it requires the data to be $(\alpha,\rho_1,\rho_2)$-resilient.

\medskip 
\noindent 
{\bf Small  local sensitivity of $\robdist_S(\hat\mu)$.}  
Under the above sensitivity bounds for 
$\mu_v({\cal M}_{v,\alpha})$ and $\sigma_v^2({\cal M}_{v,\alpha})$, 
 it follows after some calculations as shown in Eq.~\eqref{eq:sensbound_D} that the sensitivity for a resilient dataset $S$ is bounded by 
\begin{eqnarray}
 |\,\robdist_S(\hat\mu) -\robdist_{S'}(\hat\mu)\, | \; \leq \; C' \frac{\rho_1}{\alpha n}\Big(1 + \frac{\rho_1 \|\Sigma^{-1/2}(\hat\mu-\mu)\|}{\alpha}\Big)\;, 
\label{eq:intuition_sensitivity}
\end{eqnarray}
for some constant $C'$ and all neighboring datasets $S'$, assuming $\rho_2$ is sufficiently small.
Note that this sensitivity bound is {\em local} for two reasons; for this sensitivity to be small (i.e.~$O(\rho_1/(\alpha n))$), we require $S$  to be resilient  and $\hat\mu$ to be close to $\mu$. Thus the meaning of {\em local} here is two folded while traditionally local sensitivity in the privacy literature only concerns the sensitivity of a particular dataset $S$.
We handle these two locality with {\sc Test} step that, among other things, checks that the 
DP conditions are satisfied for the given dataset and the choice of $\sens$ and $\thresh$, which  
bounds the support of the exponential mechanism to be within 
%The second locality is due to the multiplicative dependence of $\robdist_S(\hat\mu)$ on $(1/\sigma_v({\cal M}_{v,\alpha}))$. 
%By setting the support of the exponential mechanism to be 
$B_{\thresh,S}=\{\hat\mu:\robdist_S(\hat\mu)\leq\tau\}$ with a choice of $\tau=O(\rho_1 )$. Consequently, we require $\rho_1^2/\alpha \ll 1$ for the second term in Eq.~\eqref{eq:intuition_sensitivity} to be dominated by the first. 
Fortunately, this is indeed true for all scenarios we are interested in. For sub-Gaussian distributions, $\rho_1^2 = \alpha^2\log(1/\alpha) \ll \alpha$. For $k$-th moment bounded distributions with $k>3$, $\rho_1^2=\alpha^{2-2/k}\ll \alpha$. For covariance bounded distributions, we do not hope to get a  Mahalanobis distance guarantee. Instead, we aim for a Euclidean distance guarantee whose sensitivity does not depend  on $\hat\mu$ and  we do  not require $\rho_1^2/\alpha \ll 1$  (Section~\ref{sec:mean-2moment}).

\medskip \noindent
{\bf Sample complexity analysis.} 
Assuming the sensitivity of $\robdist_S(\hat\mu)$ is bounded by $\sens=O(\rho_1/(\alpha n))$, which we ensure with the safety test, we analyze the  utility of the exponential mechanism. For a target accuracy of  $\|\Sigma^{-1/2}(\hat\mu-\mu)\|=O(\rho_1)$, we consider two sets $B_{\rm out}=\{\hat\mu : \|\Sigma^{-1/2}(\hat\mu-\mu)\|\leq c_0 \rho_1\}$ and 
$B_{\rm in}=\{\hat\mu : \|\Sigma^{-1/2}(\hat\mu-\mu)\|\leq c_1\rho_1 \}$ for some $c_0>c_1$. 
The exponential mechanism achieves accuracy $c_0\rho_1 $ with  probability $1-\zeta$ if 
\begin{eqnarray*}
    {\mathbb P}(\hat\mu\notin B_{\rm out}) \;\leq\; 
    \frac{{\mathbb P}(\hat\mu\notin B_{\rm out})}{{\mathbb P}(\hat\mu\in B_{\rm in})} \;\lesssim\; \frac{{\rm Vol}(B_{\tau,S})}{{\rm Vol}(B_{\rm in})}    \frac{e^{-\frac{\varepsilon}{4\sens}c_0\rho_1  }}{e^{-\frac{\varepsilon}{4\sens}c_1\rho_1}} \leq e^{O(d)}e^{-\frac{\varepsilon}{4\sens}(c_0-c_1)\rho_1} \leq \zeta\;,
\end{eqnarray*}
where the second inequality  requires 
$\robdist_S(\hat\mu)\simeq \|\Sigma^{-1/2}(\hat\mu-\mu)\|$,  which we show in Lemma~\ref{lem:outerset}. 
Since the volume ratio is $ {\rm Vol}(B_{\tau,S})/{\rm Vol}(B_{\rm out}) = e^{O(d)}$, $\tau=O(\rho_1)$, and $\sens=O(\rho_1 /(\alpha n))$, it is sufficient to have a large enough $c_0$ and $n=O((d+\log(1/\zeta))/(\alpha\varepsilon) )$ with a large enough constant. 

\medskip\noindent{\bf Safety test.}  
We are left to show that for a resilient dataset, the failure probability of the safety test, 
${\mathbb P}(m_\tau + {\rm Lap}(2/\varepsilon) < (2/\varepsilon)\log(2/\delta) )$, is less than $\zeta$. 
This requires the safety margin to be large enough, i.e.~$m_\tau\geq k^*=(2/\varepsilon)\log(4/(\delta\zeta))$. 
Recall that the safety margin is defined as  the Hamming distance to the closest dataset to $S$ where the $(\varepsilon/2,\delta/2)$-DP condition of the exponential mechanism is violated. 
We therefore need to show that the DP condition is satisfied for not only $S$ but any dataset $S'$ at Hamming distance at most $k^*$ from $S$. 

Consider two exponential mechanisms $r_{(\varepsilon,\sens,\thresh,S')}$ and 
$r_{(\varepsilon,\sens,\thresh,S'')}$ on  neighboring datasets $S'$ and $S''$. Since  $B_{\tau,S'}\neq B_{\tau,S''}$, we separately analyze  the intersection $B_{\tau,S'}\cap B_{\tau,S''}$ and the differences  $B_{\tau,S'}\setminus B_{\tau,S''}$ and $B_{\tau,S''}\setminus B_{\tau,S'}$. 
In the intersection, we show that the two probability distributions are within a multiplicative factor $e^{\varepsilon/2}$ of each other: 
\begin{eqnarray*}
    {\mathbb P}_{r_{(\varepsilon,\sens,\thresh,S')}}(\hat\mu\in A) \;\leq \; e^{\varepsilon/2} {\mathbb P}_{r_{(\varepsilon,\sens,\thresh,S'')}}(\hat\mu\in A)\;, 
\end{eqnarray*}
for all $A \subseteq B_{\tau,S'}\cap B_{\tau,S''}$, $S'$ within Hamming distance $k^*$ from a resilient dataset $S$, and $S''\sim S'$. 
The main challenge is that $S'$ is no longer a resilient dataset but a $k^*$-neighbor of a resilient dataset. Since such $S'$ is $(k^*/n,\alpha,\rho_1,\rho_2)$-corrupt good (Definition~\ref{def:corruptgoodset}), we  show that corrupt good sets also inherit the bounded local sensitivity  of a resilient dataset seamlessly as shown in Lemma~\ref{lem:local_asmp}. 

In the set difference, we show that the total probability mass ${\mathbb P}_{r_{(\varepsilon,\sens,\thresh,S)}} (\hat\mu\in B_{\tau,S}\setminus B_{\tau,S'})$ and 
${\mathbb P}_{r_{(\varepsilon,\sens,\thresh,S')}} (\hat\mu\in B_{\tau,S'}\setminus B_{\tau,S})$  are bounded by  $\delta$, respectively, as long as the overlap of the two supports are large enough. This requires $\tau\gg \sens k^*$, as we show in  Appendix~\ref{sec:proof_safetymargin}, 
which is satisfied for $n \geq (\log(1/(\delta\zeta))/(\alpha \varepsilon))$.

%This can There are two ways this can happen: violating the $\delta$ part of the DP condition and violating the $e^\varepsilon$ part. For the exponential mechanism of HPTR, $\delta$-violation happens if the support $B_{\thresh,S}$ is too small. This can be translated into a necessary condition: $\thresh \gtrsim  \rho$. The $e^\varepsilon$-violation happens is $\sens$ is smaller than the actual local  sensitivity. Eq.~\eqref{eq:intuition_sensitivity}  gives a necessary condition:$\sens \gtrsim\frac{(\thresh+1)\rho}{\alpha n}$. 

\medskip\noindent
{\bf Outline.} 
The analyses for the accuracy and the safety test build upon a universal analysis of HPTR in Theorem~\ref{thm:utility}, which holds more generally for any distance function $\robdist_\phi(\hat\theta)$ in the estimation problems of interest. 
For mean estimation, we show in 
Sections~\ref{sec:mean_proof1}-\ref{sec:mean3_sens} that the sufficient conditions of Theorem~\ref{thm:utility} are met for the choices of constants and parameters: 
$\rho=\rho_1$, $c_0=31.8$, $c_1=10.2$, $k^*=(2/\varepsilon)\log(4/(\delta\zeta))$, $\tau=42\rho_1$, and $\sens=110\rho_1 /(\alpha n)$. 
We can set $c_2$ to be a large constant and will only change the constant factor in the sample complexity which we do not track.
A proof of Theorem~\ref{thm:robust_mean} is provided in Section~\ref{sec:mean_proof2}, from which Theorem~\ref{thm:mean} follows immediately. 
All the lemmas assume $((1/5.5)\alpha,\alpha,\rho_1 ,\rho_2 )$-corrupt good set $S$, $\alpha\leq  0.015$, 
%$\rho_1(5\alpha) \leq 0.05$ and $\rho_2( 5\alpha) \leq 0.035$. 
$\rho_1\leq 0.013$, and $\rho_2 \leq 0.0005$. 
We omit this assumption in stating the lemmas  for brevity. 
%We also frequently use the following monotonicity of resilience. 
%\begin{lemma}
%\label{lem:monotonicity}
%    If $0<\alpha_1\leq\alpha_2\leq 1/2$ then $\rho_1(\alpha_1)\leq\rho_1(\alpha_2)$ and 
%    $\rho_2(\alpha_1)\leq\rho_2(\alpha_2)$.
%\end{lemma}

\subsubsection{Resilience implies robustness} 
\label{sec:mean_proof1}

For the assumption~\ref{asmp_resilience} in Theorem~\ref{thm:utility}, we show that $\robdist_S(\hat\mu)$ is a good approximation of the true distance $\|\Sigma^{-1/2}(\hat\mu-\mu)\|$ in Lemma~\ref{lem:outerset}.
We first show that the one-dimensional mean and the variance of the filtered out ${\cal M}_{v,\alpha}$ are robust. 

\begin{lemma} 
For any unit norm $v\in\reals^d$,  
    $ |\langle v, \mu  - \mu({\cal M}_{v,\alpha})\rangle| \leq 6 \rho_1 \,\sigma_v$ and $0.9 \sigma_v \leq \sigma_v({\cal M}_{v,\alpha}) \leq 1.1 \sigma_v $.  
    \label{lem:deviation}
\end{lemma}
\begin{proof}
    For the mean bound, 
    \begin{align} 
      &|\langle v, \mu  - \mu({\cal M}_{v,\alpha})\rangle| \nonumber \\
      &\;\;\leq\;    \frac{|{\cal M}_{v,\alpha}\cap S_{\rm bad}|}{|{\cal M}_{v,\alpha}|} |\langle v,\mu(S_{\rm bad} \cap {\cal M}_{v,\alpha})-\mu \rangle| + \frac{|{\cal M}_{v,\alpha}\cap S_{\rm good}|}{|{\cal M}_{v,\alpha}|}  |\langle v, \mu(S_{\rm good} \cap {\cal M}_{v,\alpha})-\mu \rangle | \nonumber \\
      &\;\;\leq \; \frac{(1/5.5)\alpha}{1-(4/5.5)\alpha}
      \frac{2\rho_1 \sigma_v}{(1/5.5)\alpha}  + \frac{1-(1/5.5)\alpha}{1-(4/5.5)\alpha}\rho_1 \sigma_v \nonumber \\
      &\;\; \leq \;  (2\rho_1 + \rho_1 )\sigma_v/(1-(4/5.5)\alpha) \;, \label{eq:meanubound}
    \end{align}
    The second inequality follows from 
    the following.  
    First, 
    $|\langle v, \mu(S_{\rm good} \cap {\cal M}_{v,\alpha})- \mu\rangle |\leq \sigma_v \rho_1 $ by the definition of resilience and that fact that $| S_{\rm good} \cap {\cal M}_{v,\alpha} | \geq (1-(5/5.5)\alpha)n$.
    Next, since $|\langle v , \mu(S_{\rm bad} \cap {\cal M}_{v,\alpha}) - \mu \rangle |$ is less than $|\langle v, \mu(S_{\rm good} \cap {\cal T}_{v,\alpha}) - \mu \rangle |$ or $|\langle v, \mu(S_{\rm good} \cap {\cal B}_{v,\alpha}) - \mu \rangle |$, 
    both of which are at most $2\rho_1  \sigma_v/(1/5.5)\alpha $, from applying Lemma~\ref{lem:deviate} with a set size at least $(1/5.5)\alpha n$,
    we have 
    \begin{eqnarray*}
    | \langle v , \mu(S_{\rm bad} \cap {\cal M}_{v,\alpha}) - \mu \rangle | \;\leq\; 
    \frac{2}{(1/5.5)\alpha}\rho_1  \sigma_v\;.
    \end{eqnarray*} 
    The mean bound follows from \eqref{eq:meanubound} and $\alpha\leq 0.1$. 
    %, and  $\rho_1(5\alpha)\geq\rho_1(\alpha)$. 
    For the variance upper bound, 
    \begin{eqnarray*}
        \sigma_v({\cal M}_{v,\alpha})^2 \;=\; \frac{1}{(1-(4/5.5)\alpha)n}\sum_{x_i\in {\cal M}_{v,\alpha}}  \langle v , x_i-\mu({\cal M}_{v,\alpha}) \rangle^2      
        \;\leq\;  \frac{1}{(1-(4/5.5)\alpha)n}\sum_{x_i\in{\cal M}_{v,\alpha}}\langle v , x_i-\mu \rangle^2 \;,
    \end{eqnarray*}
    where the first inequality follows from the fact that subtracting the empirical mean $\mu({\cal M}_{v,\alpha})$ minimizes the second moment. We can decompose the empirical deviation and show an upper bound first: 
    \begin{eqnarray} 
         &&\frac{ \sum_{x_i \in{\cal M}_{v,\alpha}}(\langle v , x_i-\mu \rangle^2 -\sigma_v^2)  }{(1-(4/5.5)\alpha)n} \nonumber\\
        & = & 
        \frac{ \sum_{x_i \in {\cal M}_{v,\alpha}\cap S_{\rm bad}} (\langle v, x_i-\mu \rangle^2-\sigma_v^2 ) }{(1-(4/5.5)\alpha)n}
        + \frac{  \sum_{x_i\in {\cal M}_{v,\alpha}\cap S_{\rm good}} (\langle v, x_i-\mu \rangle^2-\sigma_v^2)}{(1-(4/5.5)\alpha)n }  \nonumber\\ 
        &\leq & \frac{ (1/5.5)\alpha(2\rho_2/(1/5.5)\alpha)\sigma_v^2 + (1-(4/5.5)\alpha)\rho_2 \sigma_v^2}{1-(4/5.5)\alpha} \;\leq\; 6 \rho_2 \sigma_v^2\;, \label{eq:varianceubound}
    \end{eqnarray}
    where in the second inequality we used 
    resilience on ${\cal M}_{v,\alpha}\cap S_{\rm good}$ of size at least $1-(5/5.5)\alpha$. For $x_i\in S_{\rm bad}\cap {\cal M}_{v,\alpha}$, we  use
    the fact that 
    \begin{eqnarray*}
         \big|\,\langle v,x_i-\mu\rangle^2 -\sigma_v^2\,\big|  & \leq & \max \Big\{ \frac{ \sum_{j\in S_{\rm good}\cap {\cal T}_{v,\alpha}} (\langle v, x_j-\mu\rangle^2-\sigma_v^2) }{|S_{\rm good}\cap {\cal T}_{v,\alpha}|}, \frac{ \sum_{j\in S_{\rm good}\cap {\cal B}_{v,\alpha}} (\langle v,x_j-\mu\rangle^2-\sigma_v^2)  }{|S_{\rm good}\cap {\cal B}_{v,\alpha}|}   \Big\}\\
        & \leq & 
        \frac{2 \rho_2\sigma_v^2}{(1/5.5)\alpha}
        \;,
    \end{eqnarray*}
    where we used Eq.~\eqref{eq:resilienceontail2} in 
    Lemma~\ref{lem:deviate} for sets with size at least $(1/5.5)\alpha n$. 
    For the variance deviation lower bound, 
    \begin{align}
         &\frac{\sum_{x_i\in{\cal M}_{v,\alpha}}
         (\langle v,  x_i  -\mu({\cal M}_{v,\alpha}) \rangle^2-\sigma_v^2)}{(1-(4/5.5)\alpha)n} 
         =  \frac{ \sum_{x_i\in {\cal M}_{v,\alpha} } \big(\,\langle v, x_i-\mu\rangle^2 - \sigma_v^2 - \langle v , \mu-\mu({\cal M}_{v,\alpha})\rangle^2 \, \big)}{(1-(4/5.5)\alpha)n} \nonumber\\
        & \;\;\;\;\geq\; \frac{ \sum_{x_i\in {\cal M}_{v,\alpha}\cap S_{\rm bad}} (\langle v, x_i-\mu \rangle^2-\sigma_v^2 ) }{(1-(4/5.5)\alpha)n}
        + \frac{  \sum_{x_i\in {\cal M}_{v,\alpha}\cap S_{\rm good}} (\langle v, x_i-\mu \rangle^2-\sigma_v^2)}{(1-(4/5.5)\alpha)n } -36\rho_1^2\sigma_v^2\;, \nonumber \\
        & \;\;\;\;\geq\; 
        -\frac{2\rho_2 \sigma_v^2}{1-(4/5.5)\alpha}-\frac{1-(4/5.5)\alpha}{1-(4/5.5) \alpha}\rho_2 \sigma_v^2 -36\rho_1^2\sigma_v^2 
        \;\geq\; -(3.2\rho_2+36\rho_1^2)\sigma_v^2\;,
        \label{eq:variancelbound} 
    \end{align} 
    where we used $\alpha\leq 0.1$, the first term only uses the fact that $|S_{\rm bad}|\leq (1/5.5)\alpha n$, the second term uses resilience, and the last term uses the mean bound we proved earlier. 
    In \eqref{eq:varianceubound} and \eqref{eq:variancelbound}, assuming $\rho_1  \leq 0.04$, and $\rho_2  \leq 0.035$, we have 
    $\sqrt{1+6\rho_2}\leq 1.1 $ and 
    $\sqrt{1-3.2\rho_2-36\rho_1^2}\geq 0.9$.
\end{proof}

We show  that resilience implies our estimate of the distance is robust. 
\begin{lemma} 
     If $\hat\mu\in B_{\thresh,S}$ and $\thresh=42\rho_1$ then $\big|\, \|\Sigma^{-1/2}(\hat\mu-\mu) \|-\robdist_S(\hat\mu) \,\big| \leq  6 \rho_1 + 0.1\thresh\leq 10.2 \rho_1 $.  
    \label{lem:outerset}
\end{lemma} 
\begin{proof}
    From Lemma~\ref{lem:deviation}, we know that for all $\hat\mu\in B_{t,S}$,  
    \begin{eqnarray}
        \robdist_S(\hat\mu) \;=\; \max_{\|v\|=1} \frac{\langle v , \hat\mu-\mu({\cal M}_{v,\alpha}) \rangle }{\sigma_v({\cal M}_{v,\alpha})} \;\geq\; \max_{\|v\|=1} \frac{\langle v , \hat\mu - \mu \rangle - 6\rho_1  \sigma_v }{1.1\sigma_v} \;. 
        \label{eq:distance_lb}
    \end{eqnarray}
    and 
       \begin{eqnarray}
        \robdist_S(\hat\mu) \;=\; \max_{\|v\|=1} \frac{\langle v , \hat\mu-\mu({\cal M}_{v,\alpha}) \rangle }{\sigma_v({\cal M}_{v,\alpha})} \;\leq\; \max_{\|v\|=1} \frac{\langle v , \hat\mu - \mu \rangle + 6\rho_1  \sigma_v }{0.9 \sigma_v} \;. 
        \label{eq:distance_ub}
    \end{eqnarray}
    Applying Lemma~\ref{lem:equidist}, we get 
     $0.9 \robdist_S(\hat\mu) - 6\rho_1 \leq \|\Sigma^{-1/2}(\hat\mu-\mu) \|\leq 1.1 \robdist_S(\hat\mu) + 6 \rho_1 $. 
     Since $\robdist_S(\hat\mu)\leq\thresh$, we 
     get the desired bound.

%For $t=10\rho_1(5\alpha)$, this gives 
%$ \langle v, \hat\mu - \mu \rangle \leq 17 \rho_1(5\alpha) \sigma_v  $ for all $\|v\|=1$.
%This implies that $\|\Sigma^{-1/2}(\hat\mu-\mu)\| =\sigma_v^{-1} \|\hat\mu -\mu\| \leq 17\rho_1(5\alpha)$, for a specific choice of $v=(1/\|\hat\mu-\mu\|)(\hat\mu-\mu)$. 
\end{proof}

%We show that our proposed distance function is a close approximation of the true distance. 

%\begin{lemma} 
%    \label{lem:distanceappx}
%    If $\hat\mu \in B_{t=10\rho_1(5\alpha),S} $ then $\big|D(\hat\mu,S)-\|\Sigma^{-1/2}(\hat\mu-\mu)\|  \big| \leq 7\rho_1(5\alpha)$.
%\end{lemma}
%\begin{proof}
%    From \eqref{eq:distance_lb}, it follows that 
%    $1.1 D(\hat\mu,S) \geq -6\rho_1(5\alpha) + \max_{\|v\|=1} \langle v, \hat\mu-\mu \rangle /\sigma_v $. Together with the fact that $ \max_{\|v\|=1} \langle v, \hat\mu-\mu \rangle /\sigma_v \geq \|\Sigma^{-1/2}(\hat\mu-\mu)\|$, this  gives 
%    \begin{eqnarray*}
%    D(\hat\mu,S)- \|\Sigma^{1/2}(\hat\mu-\mu)\| \;\geq\; -0.1 D(\hat\mu,S)-6\rho_1(5\alpha)     \;\geq\; -7\rho_1(5\alpha)\;,
%    \end{eqnarray*}
%    where we used the fact that $\hat\mu\in B_{10\rho_1(5\alpha),S}$ and hence $D(\hat\mu,S)\leq 10\rho_1(5\alpha)$.
%    For an upper bound, we use 
%    \begin{eqnarray*}
%    \max_{\|v\|=1} \frac{\langle v , \hat\mu-\mu({\cal M}_{v,2\alpha}) \rangle }{\sigma_v({\cal M}_{v,2\alpha})} \;\leq\; \max_{\|v\|=1} \frac{\langle v , \hat\mu - \mu \rangle + 6\rho_1(5\alpha) \sigma_v }{0.9\sigma_v} \;, 
%    \end{eqnarray*}
%    which implies that $D(\hat\mu,S)- \|\Sigma^{1/2}(\hat\mu-\mu)\|\leq 7\rho_1(5\alpha)$ using Lemma~\ref{lem:equidist}.
%\end{proof}

% --------------------------------------------  
\subsubsection{Bounded volume} 
\label{sec:proof_vol}

%We show that resilience of data implies that the proposed $D(\hat\mu,S)$ has a small sensitivity. Hence, the small and dimension independent sensitivity bound  $\sens=18\rho_1(5.5\alpha)$ achieves a large margin $m=(1/2)\alpha n+1$ in Lemma~\ref{lem:sensitivitymargin}; the sensitivity margin $m$ is defined in step 1 of HPTR as the minimum number of samples that need to be changed to change the distance by more than $\sens$. 
We show that the assumption~\ref{asmp_vol} in Theorem~\ref{thm:utility} is satisfied for robust estimate $\robdist_S(\hat\mu)$.
\begin{lemma}
For $\rho=\rho_1$, 
$c_1=10.2$, $\thresh=42\rho_1$, 
$\sens =110\rho_1/(\alpha n)$, and $c_2\geq\log(67/12)+\log((c_0+2c_1)/c_1)$, 
we have 
 $(7/8)\thresh-(k^*+1)\sens>0$,  
        \begin{eqnarray*}
        \frac{{\rm Vol}(B_{ \thresh+(k^*+1)\sens+c_1\rho,S}) }
        {{\rm Vol}(B_{ (7/8)\thresh-(k^*+1)\sens-c_1\rho,S})}
        &\leq& e^{c_2 d}\;, \text{ and }\\
        \frac{{\rm Vol}(\{\hat\mu:\|\Sigma^{-1/2}(\hat\mu-\mu )\| \leq (c_0+2c_1)\rho\})}
        {{\rm Vol}(\{\hat\mu: \|\Sigma^{-1/2} (\hat \mu-\mu)\| \leq c_1\rho\})}
        &\leq &e^{c_2 d}\;.
        \end{eqnarray*}
    \label{lem:vol} 
\end{lemma}

\begin{proof}
The second part of assumption~\ref{asmp_vol} follows from the fact that 
\begin{eqnarray*}
    {\rm Vol}(\{\hat\mu:\|\Sigma^{-1/2}(\hat\mu-\mu)\|\leq r\}) = c_d |\Sigma| r^d\;,
\end{eqnarray*}
where $|\Sigma|=\prod_{j=1}^d \sigma_j(\Sigma)$ is the determinant of $\Sigma$ and $\sigma_j(\Sigma)$ is the $j$-th singular value, 
for some constant $c_d$ that only depends on the dimension
and selecting $c_2\geq \log((c_0+2c_1)/c_1)$. 

The first part is tricky as we do not yet have handle on the set $B_{t,S}$ for $t>\thresh$. In particular, we do not know how $\robdist_S(\hat\mu)$ relates to $\|\Sigma^{-1/2}(\hat\mu-\mu)\|$ for such a $\hat\mu$ outside of $B_{\thresh,S}$. To this end, we use the following corollary. 

\begin{coro}[{Corollary of Lemma~\ref{lem:outerset}}]
     If $\hat\mu\in B_{2\thresh,S}$ and $\thresh=42\rho_1 $ then $\big|\, \|\Sigma^{-1/2}(\hat\mu-\mu) \|-\robdist_S(\hat\mu) \,\big| \leq   14.2 \rho_1 $.  
    \label{coro:outerset}
\end{coro} 

We will show that
$(7/8)\thresh-(k^*+1)\sens>0$. 
As this implies that $\thresh + (k^*+1)\sens  \leq 2 \thresh$, we can use the above corollary to show that 
\begin{eqnarray*}
    \frac{{\rm Vol}(B_{ \thresh+(k^*+1)\sens+c_1\rho,S}) }
        {{\rm Vol}(B_{ (7/8)\thresh-(k^*+1)\sens-c_1\rho,S})}
        &\leq& \frac{{\rm Vol}\big(\,\{\hat\mu:\|\Sigma^{-1/2}(\hat\mu-\mu)\| \leq \thresh+(k^*+1)\sens+c_1\rho + 14.2\rho_1 \}\,\big)}{{\rm Vol}\big(\,\{\hat\mu:\|\Sigma^{-1/2}(\hat\mu-\mu)\| \leq (7/8)\thresh-(k^*+1)\sens-c_1\rho - 14.2\rho_1\}\,\big)} \\
        &=& \Big( \frac{\thresh+(k^*+1)\sens+c_1\rho + 14.2\rho_1 }
        {(7/8)\thresh-(k^*+1)\sens-c_1\rho - 14.2\rho_1 } \Big)^d\\
        &\leq& (67/12)^d \; \leq \; e^{c_2 d}\;,
\end{eqnarray*}
for the choices of $\rho=\rho_1 $, 
$c_1=10.2$, $\thresh=42\rho_1 $, 
$\sens =110\rho_1 /(\alpha n)$, and 
$c_2\geq \log(67/12)$ where we used the fact that for $n\geq C\log(1/(\delta\zeta))/(\alpha \varepsilon)$ with a large enough constant $C$, we have $(k^*+1)\sens\leq 0.3 \rho_1$. 
It follows that the condition $(7/8)\thresh-(k^*+1)\sens>0$ is also satisfied.  
\end{proof}

% -----------------------------------
\subsubsection{Resilience implies  bounded local sensitivity}
\label{sec:mean3_sens}

We show that resilience implies the assumption \ref{asmp_local} in Theorem~\ref{thm:utility} (Lemma~\ref{lem:local_asmp}). 
However, since local sensitivity needs to be established first for not just the given set $S$ but also Hamming distance $k^*+1$ neighborhood of $S$, we need robustness results for this broader regime. 
Assuming $(k^*+1)/n\leq \alpha/11$, we can extend robustness results analogously as follows. We consider a set $S'$ with $k$ data points arbitrarily changed from $S$. This implies that $S'$ is a $((1/5.5)\alpha+(k/n),\alpha,\rho_1,\rho_2)$-corrupt good set with respect to $(\mu,\Sigma)$. 
We first prove the analogous bounds to Lemma~\ref{lem:deviation} for this $S'$.
\begin{lemma}
    For an $((1/5.5)\alpha+\tilde\alpha,\alpha,\rho_1,\rho_2)$-corrupt good set $S'$ with respect to $(\mu,\Sigma)$,  $\tilde\alpha\leq(1/11)\alpha$, and any unit norm $v\in\reals^d$,  
    $ |\langle v, \mu  - \mu({\cal M}_{v,\alpha})\rangle| \leq 14 \rho_1 \,\sigma_v$ and $0.9 \sigma_v \leq \sigma_v({\cal M}_{v,\alpha}) \leq 1.1 \sigma_v $. 
    %if $\alpha\leq0.015$, $\rho_1(5.5\alpha)\leq 0.013$, and $\rho_2(5.5 \alpha)\leq 0.0005$.  
    \label{lem:deviation2}
\end{lemma}
\begin{proof}
    Analogous to \eqref{eq:meanubound}, we have 
    \begin{eqnarray*} 
      |\langle v, \mu  - \mu({\cal M}_{v,\alpha})\rangle| 
      &\leq & \frac{(1/5.5)\alpha+\tilde\alpha}{1-(4/5.5)\alpha}
      \frac{2\rho_1 \sigma_v}{(1/5.5)\alpha-\tilde\alpha}  + \frac{1-(1/5.5)\alpha-\tilde\alpha}{1-(4/5.5)\alpha}\rho_1\sigma_v \nonumber \\
      &\leq & 14\rho_1 \sigma_v\;,
    \end{eqnarray*}
    where we used the fact that $(5/5.5)\alpha+\tilde\alpha \leq \alpha$.
    Analogous to \eqref{eq:varianceubound}, we have 
        \begin{eqnarray*}
        \frac{ \sum_{x_i\in{\cal M}_{v,\alpha}}(\langle v , x_i-\mu({\cal M}_{v,\alpha}) \rangle^2 -\sigma_v^2)  }{(1-(4/5.5)\alpha)n}  
        &\leq & \frac{ ((1/5.5)\alpha+\tilde\alpha)(\frac{2\rho_2}{(1/5.5)\alpha-\tilde\alpha})\sigma_v^2 + (1-(1/5.5)\alpha-\tilde\alpha)\rho_2 \sigma_v^2}{1-(4/5.5)\alpha}\\ 
        &\leq& 14 \rho_2  \sigma_v^2\;.
    \end{eqnarray*}
    Analogous to \eqref{eq:variancelbound}, we have  
    \begin{align*}
         \frac{\sum_{x_i\in{\cal M}_{v,\alpha}}
         (\langle v, x_i  -\mu({\cal M}_{v,\alpha}) \rangle^2-\sigma_v^2)}{(1-(4/5.5)\alpha)n} 
         &\;\geq\; 
        -\frac{((1/5.5)\alpha+\tilde\alpha)2\rho_2\sigma_v^2}{(1-(4/5.5)\alpha)((1/5.5)\alpha - \tilde\alpha)}- \rho_2 \sigma_v^2 -14^2\rho_1^2\sigma_v^2 
        \\ 
        &\;\geq\; - (7.3\rho_2 + 196\rho_1^2 ) \sigma_v^2\;.  
    \end{align*}
    For $\alpha\leq0.045$, $\rho_1 \leq 0.013$, and $\rho_2 \leq 0.0005$, we have the desired bounds.     
\end{proof}

\begin{lemma}
    \label{lem:distanceperturbed} 
    For an $((1/5.5)\alpha+\tilde\alpha,\alpha, \rho_1,\rho_2)$-corrupt good set $S'$ with respect to $(\mu,\Sigma)$ and  $\tilde\alpha\leq(1/11)\alpha$, if $\hat\mu \in B_{t,S'}$ for some $t > 0$  %B_{10\rho_1(5\alpha),S} $ 
    then we have  $  \|\Sigma^{-1/2}(\hat\mu-\mu)\| \leq 14\rho_1  + 1.1t$ and $\big|D(\hat\mu,S') - \|\Sigma^{-1/2}(\hat\mu-\mu) \| \big|\leq 14\rho_1  + 0.1t$. 
\end{lemma}
\begin{proof}
    Analogously to the proof of Lemma~\ref{lem:outerset}, we have 
    \begin{eqnarray*}
        1.1D(\hat\mu,S') &\geq& -14\rho_1  +  \|\Sigma^{-1/2}(\hat\mu-\mu)\|\;, \text{ and } \\
        0.9D(\hat\mu,S') &\leq& 14\rho_1 + \|\Sigma^{-1/2}(\hat\mu-\mu)\| \;.
    \end{eqnarray*}
    %This gives %$\big|\,D(\hat\mu,S')-\|\Sigma^{-1/2}(\hat\mu-\mu)\|\,\big| \leq(1/9)\|\Sigma^{-1/2}(\hat\mu-\mu)\|  + (14/0.9)\rho_1(5.5\alpha) $.     From Lemma~\ref{lem:outerset} we have $ \|\Sigma^{-1/2}(\hat\mu-\mu) \|\leq 17 \rho_1(5\alpha)$, which 
    This gives the desired bound. 
\end{proof}

%\begin{lemma}
%    \label{lem:sensitivitymargin}
%    For any $\hat\mu\in B_{10\rho_1(5\alpha),S}$, if  $m \leq(1/2)\alpha n$ then  $\max_{S':d_H(S',S)\leq m} |D(\hat\mu,S)-D(\hat\mu,S')|\leq \sens$ where $\sens=25\rho_1(5.5\alpha)$. 
%\end{lemma}

%\begin{proof} 
%    From Lemma~\ref{lem:distanceperturbed} with $\tilde\alpha=m /n$, we know that for all $S'$ such that $d_H(S',S)\leq (1/2)\alpha n$, $|D(\hat\mu,S')- \|\Sigma^{1/2}(\hat\mu-\mu)\||\leq 18\rho_1(5.5\alpha)$. 
%    From Lemma~\ref{lem:distanceappx}, we know that  $|D(\hat\mu,S)- \|\Sigma^{1/2}(\hat\mu-\mu)\| |\leq7\rho_1(5\alpha)$. 
%    By triangular inequality, this implies the desired bound.
%\end{proof}

The sensitivity of $\robdist_S(\hat\mu)$ is {\em local} in two ways. 
First, we get the desired sensitivity bound for a dataset $S$ that behaves nicely, which is captured by the notion of $((1/5.5)\alpha,\alpha,\rho_1,\rho_2)$-corrupt good set $S$. 
Secondly, the sensitivity bound requires the estimate  parameter $\hat\mu$ to be close to $\mu$ in $\|\Sigma^{-1/2}(\hat\mu-\mu)\|$. 
Both {\em locality in  dataset} and {\em locality in estimate} are ensured by the safety test (Test step in HPTR).
To show that corrupt good datasets pass the safety test, the following lemma establishes that those datasets have small local sensitivity. 

\begin{lemma}
    \label{lem:local_asmp}
    For $\sens = 110\rho_1 /(\alpha n)$, $\thresh = 42\rho_1 $, 
    and an $((1/5.5\alpha),\alpha,\rho_1,\rho_2)$-corrupt good  $S$, if 
    \begin{eqnarray}
        n\;=\;\Omega\Big(
        \frac{\log(1/(\delta\zeta))}{\alpha\varepsilon} 
        % + \frac{\rho_1(5.5\alpha)^2}{\alpha} 
        \Big)\;,
    \end{eqnarray} 
    then the local sensitivity in assumption~\ref{asmp_local} is satisfied. 
\end{lemma}

%\begin{remark}

\medskip\noindent{\bf Remark.}
Note that to keep $\sens=O(\rho_1/(\alpha n))$ that we want (and is critical in getting the final utility guarantee), we need the extra corruption to be $k^*/n = O(\alpha)$. This implies $n=\Omega(k^*/\alpha)=\Omega(\log(1/(\delta\zeta))/(\varepsilon\alpha))$. Further, $k^*=\Omega(\log(1/(\delta\zeta))/\varepsilon)$ cannot be improved, as it is critical in achieving small failure probability in the testing step. 
    Hence, the sample complexity of $\Omega(\log(1/(\delta\zeta))/(\varepsilon\alpha))$ cannot be improved under current proof strategy. 
%    \label{rem:tightness}
%\end{remark}

\begin{proof} 
    Since $S$ is $((1/5.5)\alpha,\alpha,\rho_1,\rho_2)$-corrupt good and $d_H(S,S') \leq k^*$, it follows that $S'$ is  $((1/5.5)\alpha+\tilde\alpha,\alpha ,\rho_1,\rho_2)$-corrupt good with $\tilde\alpha = (k^*/n)$. 
    We further assume that $\tilde\alpha\leq (1/11)\alpha$, which follows from $k^*=(2/\varepsilon)\log(4/(\delta\zeta))$ and $n=\Omega(\log(1/\delta\zeta)/(\varepsilon\alpha))$ with a large enough constant.
    We show that 
    this resilience implies that $S'$ is dense around the boundary of ${\cal M}_{v,\alpha}$, which in turn implies low sensitivity. 
    
    Recall that ${\cal T}_{v,\alpha}\subset S$
    is the set of data points corresponding to the largest $(2/5.5)\alpha n$  data points in the projected set $S'_{(v)}=\{\langle v,x_i\rangle\}_{x_i\in S'} $ and ${\cal B}_{v,\alpha}\subset S$ is the bottom set. Let $S_{\rm good}$ denote the original uncorrupted resilient dataset.
    Applying Lemma~\ref{lem:deviate} to $S_{\rm good}\cap {\cal T}_{v,\alpha}$ (and $S_{\rm good}\cap {\cal B}_{v,\alpha}$) of size at least $(1/11)\alpha$ (since corruption fraction is at most $(1/5.5)\alpha+\tilde\alpha\leq(1.5/5.5)\alpha$), 
    \begin{eqnarray*}
        \big|\,\langle v, \mu(S_{\rm good}\cap {\cal T}_{v,\alpha}) - \mu \rangle \,\big| \;\leq\; \frac{2\rho_1  \sigma_v}{(1/11)\alpha}\;, \text{ and } 
        \big|\,\langle v, \mu(S_{\rm good}\cap {\cal B}_{v,\alpha}) - \mu \rangle \,\big| \;\leq\; \frac{2\rho_1 \sigma_v}{(1/11)\alpha}\;.
    \end{eqnarray*}
    This implies that there is at least one good data point that is closer to the center than the means of the upper tail and the bottom tail: 
    \begin{eqnarray*}
        \min_{x_i\in S_{\rm good}\cap {\cal T}_{v,\alpha} }\big|\,\langle v, x_i  - \mu \rangle \,\big| \;\leq\; \frac{2\rho_1 \sigma_v}{(1/11)\alpha}\;, \text{ and } 
        \min_{x_i\in S_{\rm good}\cap {\cal B}_{v,\alpha}} \big|\,\langle v, x_i - \mu \rangle \,\big| \;\leq\; \frac{2\rho_1 \sigma_v}{(1/11)\alpha}\;.
    \end{eqnarray*}
    It follows that  the distance between two closest points in ${\cal T}_{v,\alpha}$ and ${\cal B}_{v,\alpha}$ is bounded by 
    \begin{eqnarray}
        \min_{x_i\in S_{\rm good}\cap {\cal T}_{v,\alpha}} \langle v,x_i \rangle - \max_{x_i\in S_{\rm good}\cap {\cal B}_{v,\alpha}} \langle v, x_i \rangle\;\leq\; (44/\alpha)\rho_1 \sigma_v\;, 
        \label{eq:dense}
    \end{eqnarray}
    when $\mu\in{\cal M}_{v,\alpha}$. When $\mu\in{\cal T}_{v,\alpha}$ or $\mu\in{\cal B}_{v,\alpha}$, it is straightforward that the above inequality holds.
    %The scalar dataset   $S'_{(v)}=\{\langle v, x_i\rangle \}_{x_i \in S'}$, therefore,  is dense at the boundary of the top and bottom $(2/5.5)\alpha$ quantiles. 
    This implies low sensitivity as follows.  
    
    Recall that ${\cal M}_{v,\alpha}(S')$ denote the middle part after filtering out the top and bottom $(2/5.5)\alpha$ quantiles from $\{\langle v , x_i\rangle\}_{x_i\in S'} $. 
    For a neighboring dataset  $S''$ and the corresponding $S''_{(v)}$, consider a scenario where one point $x_i$ in ${\cal M}_{v,\alpha}(S')$ is replaced by another point $\tilde{x}_i$.
If 
$\langle v, \tilde{x}_i\rangle \in [\, \max_{x_i\in S_{\rm good} \cap {\cal B}_{v,\alpha}} \langle v, x_i\rangle \,,\, \min_{x_i\in S_{\rm good} \cap {\cal T}_{v,\alpha}} \langle v, x_i\rangle \,]$, then 
Eq.~\eqref{eq:dense} implies that $|\langle v,x_i -\tilde x_i\rangle |\leq (44/\alpha)\rho_1 \sigma_v$. 
Otherwise, ${\cal M}_{v,\alpha}(S'')$ will have $x_i$ replaced by either $\arg\min_{j\in S_{\rm good} \cap {\cal T}_{v,\alpha}} \langle v,  x_j\rangle $ or $\arg\max_{j\in S_{\rm good} \cap {\cal B}_{v,\alpha}} \langle v, x_j \rangle $. In either case, Eq.~\eqref{eq:dense} implies that $|\langle v,x_i -\tilde x_i\rangle |\leq (44/\alpha)\rho_1 \sigma_v$. 
The other case of when the replaced sample $x_i\in S$ is not in ${\cal M}_{v,\alpha}$ follows similarly. 

From this, we get the following bounds on the sensitivity of the robust mean and robust variance. Note that using robust statistics is critical in getting such small sensitivity bounds. 
Let $\mu'=\mu({\cal M}_{v,\alpha}(S'))$ and  
$\mu''=\mu({\cal M}_{v,\alpha}(S''))$ where we write the dataset $S'$  in ${\cal M}_{v,\alpha}(S')$ explicitly,
\begin{eqnarray}
    \big|\, \langle v, \mu'-\mu'' \rangle   \,\big| &\leq& \frac{44\rho_1 \sigma_v}{\alpha(1-(4/5.5)\alpha)n}\;. \label{eq:meansensitivitybound}
\end{eqnarray}

For the variance bound,  let 
$\sigma_v'^2=\sigma_v^2({\cal M}_{v,\alpha}(S')) = (1/|{\cal M}_{v,\alpha}(S')|)\sum_{x_i'\in{\cal M}_{v,\alpha}(S')}\langle v, x_i'-\mu' \rangle^2$   and $\sigma_v''^2=\sigma_v^2({\cal M}_{v,\alpha}(S''))$.
%Since $(1-4\alpha)n\sigma_v'^2 = \sum_{ {\cal M}_{v,2\alpha}(S')}\langle v, x_i-\mu' \rangle^2 = \sum_{ {\cal M}_{v,2\alpha}(S')}(\langle v, x_i-\mu \rangle^2 - \langle v , \mu-\mu' \rangle^2)$ for any $\mu$, we will bound each term separately. Choosing $\mu=\mu''$,
Since 
$(1-(4/5.5)\alpha)n\sigma_v'^2 = 
\sum_{ x_i'\in {\cal M}_{v,\alpha}(S')}\langle v, x_i'-\mu' \rangle^2 = \sum_{x_i'\in {\cal M}_{v,\alpha}(S')}(\langle v, x_i'-\mu'' \rangle^2 - \langle v , \mu''-\mu' \rangle^2)$, we have 
$(1-(4/5.5)\alpha)n(\sigma_v'^2-\sigma_v''^2)= \sum_{x_i'\in {\cal M}_{v,\alpha}(S')}\langle v,x_i'-\mu''\rangle^2-\sum_{x_i''\in {\cal M}_{v,\alpha}(S'')}\langle v,x_i''-\mu''\rangle^2 - (1-(4/5.5)\alpha) n \langle v, \mu''-\mu' \rangle^2$. We bound each term separately. 
Note that ${\cal M}_{v, \alpha}(S')$ and ${\cal M}_{v, \alpha}(S'')$ only differ in at most one data point. We denote those by 
 $x'$ and $x''$ respectively.  Then, 
    \begin{align*}
         &\Big|\, \sum_{ x_i' \in {\cal M}_{v,\alpha}(S')}\langle v, x_i'-\mu''  \rangle^2-\sum_{x_i'' \in {\cal M}_{v,\alpha}(S'')}\langle v, x_i''-\mu''  \rangle^2 \,\Big| = \big|\, \langle v , x'-\mu''\rangle^2- \langle v,x''-\mu''\rangle^2 \,\big|  \\
         &\;\;\;\;=\;  \big|\,\langle v, x' + x'' - 2\mu'' \rangle \langle v, x' -  x'' \rangle \,\big| \\
         &\;\;\; = \;  
         \big|\,\langle v, x' - \mu' \rangle +
         \langle v, \mu' - \mu'' \rangle+
         \langle v, x'' - \mu'' \rangle\big| \, 
         \big| \langle v, x'- x'' \rangle \,\big| \\ 
         &\;\;\;\leq \; 3\Big( \frac{44\rho_1\sigma_v}{\alpha} \Big)^2\;,
    \end{align*}
    and 
    \begin{eqnarray*} 
      (1-(4/5.5)\alpha) n \langle v,\mu'-\mu''\rangle^2 &\leq &
     (1-(4/5.5)\alpha)n \frac{(44\rho_1 \sigma_v)^2}{(\alpha(1-(4/5.5)\alpha)n)^2}
    \;.    
    \end{eqnarray*}
    This implies that 
    \begin{eqnarray}
     |\sigma_v'^2-\sigma_v''^2| \; \leq \;
    \frac{(44\rho_1(\alpha/2)\sigma_v)^2}{ (1-(4/5.5)\alpha)n \alpha^2}\Big(3+\frac{1}{(1-(4/5.5)\alpha)n}\Big)\;\leq\; 
    \frac{4(44\rho_1 \sigma_v)^2}{(1-(4/5.5)\alpha)n \alpha^2}\;.
    \label{eq:varsensitivitybound}
    \end{eqnarray}
    Together, we get the following bound on the sensitivity of $D(\hat\mu,S')$. Since $\max_v a_v - \max_v b_v\leq \max_v |a_v-b_v|$, we have   
    \begin{eqnarray*}
        \big|\,\robdist _{S'}(\hat\mu)- \robdist_{S''}(\hat\mu)\,\big| &\leq& \max_{v:\|v\|=1} \Big| \frac{\langle v , \hat\mu-\mu'\rangle }{\sigma_v'}-\frac{\langle v , \hat\mu-\mu''\rangle }{\sigma_v''}\Big| \\
        &\leq& \max_{v:\|v\|=1} \frac{|\langle v, \mu'-\mu''\rangle |}{\sigma_v'}+ \frac{|\langle v , \hat\mu-\mu''\rangle |}{\sigma_v}\,\Big|\frac{\sigma_v}{\sigma_v'}-\frac{\sigma_v}{\sigma_v''}\Big|\\ 
        &\leq& \frac{44\rho_1 }{0.9\alpha(1-(4/5.5)\alpha)n} + \|\Sigma^{-1/2}(\hat\mu-\mu'')\| \max_v \frac{\sigma_v}{\sigma_v'\sigma_v''(\sigma_v'+\sigma_v'')} |\sigma_v'^2-\sigma_v''^2 |\\
        &\leq&  \frac{44 \rho_1 }{0.9\alpha(1 - (4/5.5) \alpha)n} + \frac{5312\rho_1^2}{\alpha^2(1-(4/5.5)\alpha)n}\|\Sigma^{-1/2}(\hat\mu-\mu'')\|
        \;, 
    \end{eqnarray*}
    where we used triangular inequality in the second inequality and the third inequality follows from 
    $\sigma_v'\geq 0.9\sigma_v$  (Lemma~\ref{lem:deviation2}),   Eqs.~\eqref{eq:meansensitivitybound}, and 
    Lemma~\ref{lem:equidist}, and 
    the last inequality follows from 
    and $\sigma_v''\geq 0.9\sigma_v$ and 
    \eqref{eq:varsensitivitybound}. 
    
    From Lemma~\ref{lem:distanceperturbed}, $\hat\mu\in B_{\thresh+\backoff,S}$ implies  
    $\|\Sigma^{-1/2}(\hat\mu-\mu) \|\leq  14\rho_1 + 1.1(\thresh+\backoff)$.
    From Lemma~\ref{lem:deviation2}, 
    $ \|\Sigma^{-1/2}(\mu-\mu'')\|\leq 14\rho_1 $. We apply triangular inequality and show  
    that $\|\Sigma^{-1/2}(\hat\mu-\mu'') \|\leq c \alpha  /\rho_1 $ for the choices of $\sens$, $k^*$, $\thresh$ and $n$, with an arbitrarily small constant $c$:
    \begin{eqnarray*}
        \|\Sigma^{-1/2}(\hat\mu-\mu'')\| &\leq & 28\rho_1 + 1.1(\thresh+\backoff)  \\
        &\leq &C \rho_1 +C\frac{\rho_1 \log(1/(\delta\zeta))}{\varepsilon\alpha n} \\
        &\leq& 2C\rho_1  % \;\leq\; \frac{c\alpha n}{\rho_1 }
        \;,
    \end{eqnarray*}
    for some constant $C>0$
    where $\sens=110\rho_1 /(\alpha n)$, $\thresh= 42\rho_1  $, 
    $k^*=(2/\varepsilon)\log(4/(\delta\zeta))$, and $n\geq C' \log(1/(\delta \zeta))/(\varepsilon\,\alpha )$.
%    In the last inequality we used the assumption that $\rho_1(5.5\alpha)^2 \leq c n\alpha$.
    Under the assumption that 
    $ \rho_1^2\leq c \alpha  $ and $\alpha\leq c$ for some small enough $c$, 
    this implies
    \begin{eqnarray} 
    |\robdist_{S'}(\hat\mu)-\robdist_{S''}(\hat\mu)|&\leq& \frac{44\rho_1 }{0.9(1-(4/5.5)\alpha)\alpha n} \Big(1+ \frac{121\rho_1 }{\alpha}
     2C\rho_1  \Big)\nonumber\\
    %\frac{c\alpha n}{\rho_1(\alpha/2)}\Big)\\
    &\leq& \frac{( 44/0.9) \rho_1}{\alpha n}
    \frac{1+44c}{1-(4/5.5)c}\;\leq\; \sens \;=\; \frac{110\rho_1}{\alpha n}\;. \label{eq:sensbound_D}
    \end{eqnarray}

\end{proof}

% -----------------------------------

\subsubsection{Proof of Theorem~\ref{thm:robust_mean}}
\label{sec:mean_proof2}
    
We show that the sufficient conditions of Theorem~\ref{thm:utility} are met for the choices of constants and parameters: 
$p=d$, $\rho=\rho_1 $, $c_0=31.8$, $c_1=10.2$, $\tau=42\rho_1 $, and $\sens=110\rho_1 /(\alpha n)$. 
We can set $c_2$ to be a large constant and will only change the constant factor in the sample complexity.

The assumptions~\ref{asmp_vol}, \ref{asmp_local}, and \ref{asmp_resilience} follow from Lemmas~\ref{lem:vol}, \ref{lem:local_asmp}, and \ref{lem:outerset}, respectively. 
The assumption~\ref{asmp_sens} follows from \begin{eqnarray*}
\sens = \frac{110 \rho_1 }{\alpha n} \leq  \frac{1.2\rho_1 \varepsilon}{32(c_2 d + (\varepsilon/2) + \log(16/(\delta\zeta)))}
= 
\frac{(c_0-3c_1)\rho \varepsilon }{32(c_2 d + (\varepsilon/2) + \log(16/(\delta\zeta)))}\;,
\end{eqnarray*}
for large enough $n \geq C' (d + \log(1/(\delta\zeta)))/(\alpha\varepsilon)$. 
This finishes the proof of Theorem~\ref{thm:robust_mean} from which  Theorem~\ref{thm:mean} immediately follows.

% ----------------------------------------
\subsection{Step 3: Near-optimal guarantees}
\label{sec:mean3}
% ----------------------------------------

We provide utility guarantees for popular families of distributions in private or robust mean estimation literature: %\cite{barber2014privacy,lai2016agnostic,steinhardt2018resilience,zhu2019generalized}: 
sub-Gaussian \cite{barber2014privacy,lai2016agnostic,steinhardt2018resilience,zhu2019generalized,KV17,KLSU19,cai2019cost,bun2019private,biswas2020coinpress,aden2020sample,brown2021covariance,diakonikolas2019robust,diakonikolas2017beingrobust,dong2019quantum,hopkins2020mean,diakonikolas2018robustly,huang2021instance}, $k$-th moment bounded \cite{barber2014privacy,lai2016agnostic,steinhardt2018resilience,zhu2019generalized,kamath2020private}, and covariance bounded \cite{barber2014privacy,lai2016agnostic,steinhardt2018resilience,zhu2019generalized,kamath2020private,dong2019quantum,hopkins2020robust,depersin2019robust,depersin2021robustness}. 
We apply known resilience bounds of each family of distributions and substitute them in Theorems~\ref{thm:mean} and \ref{thm:robust_mean}. In all cases, the resulting sample complexity is near-optimal, which follows from matching information-theoretic lower bounds. 

Since we aim for Mahalanobis distance error bounds, corresponding mean resilience we need in  Definition \ref{def:resilience} scales linearly in the projected standard deviation. For sub-Gaussian distributions, this requires the projected variance  $v^\top \Sigma v$ to be lower bounded by how fast the tail is decreasing, capture by the sub-Gaussian proxy $\Omega(v^\top \Gamma v)$ in  Eq.~\eqref{eq:def_subgauss} (Section~\ref{sec:mean-subgaussian}). 
For $k$-th moment bounded distributions with $k>3$, this requires the projected variance to be lower bounded by $\Omega({\mathbb E}[|\langle v,x-\mu \rangle|^k]^{2/k})$, a condition known as hypercontractivity  (Section~\ref{sec:mean-kmoment}). 
When we do not have such lower bounds on the covariance, HPTR can only hope to achieve Euclidean distance error bounds. Under our design principle, this translates into the choice of $\robdist_S(\hat\mu)=\max_{\|v\|\leq 1} \langle v , \hat\mu \rangle -\mu_v({\cal M}_{v,\alpha})$. We give an example of this scenario with  covariance bounded distributions (Section~\ref{sec:mean-2moment}).

\subsubsection{Sub-Gaussian distributions}
\label{sec:mean-subgaussian}

We say a distribution $P$ is sub-Gaussian with proxy $\Gamma$ if  for all $\|v\|=1$ and $t\in\reals$, 
\begin{eqnarray}
    \E_{x\sim P} \big[ \,\exp(t\,\langle v , x\rangle)\,\big] \;\leq\; \exp\Big( \frac{t^2\,v^\top\Gamma v}{2}\Big)\;. 
    \label{eq:def_subgauss}
\end{eqnarray}
Under this standard sub-Gaussianity, we are only guaranteed mean resilience of Eq.~\eqref{def:rho1}, for example, with R.H.S scaling as $\rho_1 \sqrt{ v^\top\Gamma v}$ instead of $\rho_1 \sqrt{  v^\top\Sigma v}$. This implies that  
the Mahalanobis distance of any robust estimate can be made arbitrarily large by shrinking the covariance in one direction such that $v^\top\Sigma v \ll v^\top\Gamma v$. To avoid such degeneracy, we add an additional assumption that $\Sigma \succeq c \Gamma$, which is also common in robust statistics literature, e.g., \cite{jambulapati2020robust}. 
With this definition,  it is known that  sub-Gaussian samples are $(\alpha,O(\alpha\sqrt{\log(1/\alpha)}),O(\alpha\log(1/\alpha)))$-resilient. 

\begin{lemma} [{Resilience of sub-Gaussian samples \cite{zhu2019generalized} and \cite[Corollary 4]{jambulapati2020robust}}] 
    \label{lem:mean_subgaussian} 
    For any fixed $\alpha\in(0,1/2)$, consider a dataset $S =\{x_i\in\reals^d\}_{i=1}^n$ of $n$ i.i.d.~samples from a  sub-Gaussian distribution with  mean $\mu$, covariance $\Sigma$, and a sub-Gaussian proxy  $0 \prec \Gamma\preceq c_1 \Sigma$ for a constant $c_1$.
     There exist constants $c_2$ and $c_3>0$ such that if $n\geq c_2((d+\log(1/\zeta)) / (\alpha\log(1/\alpha))^2)$ then 
     $S$ 
    is $(\alpha,c_3\alpha\sqrt{\log(1/\alpha)},c_3\alpha\log(1/\alpha))$-resilient with respect to $(\mu,\Sigma)$ with  probability $1-\zeta$. 
\end{lemma} 

This lemma and  Theorem~\ref{thm:mean} imply the following utility guarantee.  Further, from Theorem~\ref{thm:robust_mean} the guarantee also holds under $\alpha$-corruption of the i.i.d. samples from a sub-Gaussian distribution. 

\begin{coro} 
    \label{coro:mean_subgaussian} 
    Under the hypothesis of Lemma~\ref{lem:mean_subgaussian} there exists a constant $c>0$ such that  for any $\alpha\in(0,c)$, a dataset of size 
     $$n \;=\; O\Big(\, \frac{d+\log(1/\zeta)}{(\alpha\log(1/\alpha))^2} + \frac{d+\log(1/(\delta\zeta))}{\alpha\varepsilon} \,\Big)\;,$$
      sensitivity of $\sens=O((1/n)\sqrt{\log(1/\alpha)})$, and threshold of $\tau=O(\alpha\sqrt{\log(1/\alpha)})$, 
     with large enough constants are sufficient for  \HPTR$(S)$ with  the distance function in Eq.~\eqref{def:mean_distance} to achieve \begin{eqnarray}
         \|\Sigma^{-1/2} (\hat\mu-\mu) \|\; =\; O(\alpha\sqrt{\log(1/\alpha)})\;,\label{eq:mean_subgauss}
     \end{eqnarray} with  probability $1-\zeta$. 
     Further, the same guarantee holds even if $ \alpha$-fraction of the samples are arbitrarily corrupted as in Assumption~\ref{asmp:mean}. 
\end{coro}

This sample complexity is near-optimal up to logarithmic factors  in $1/\alpha$ and $1/\zeta$ for $\delta=e^{-O(d)}$. 
Even for DP mean estimation without corrupted samples,  
HPTR is the first algorithm for  sub-Gaussian distributions with unknown covariance that nearly matches the lower bound of $n=\widetilde \Omega(d/\alpha^2 + d/(\alpha\varepsilon) + \log(1/\delta)/\varepsilon)$ from \cite{KV17,KLSU19}, where $\widetilde\Omega$ hides  polylogarithmic terms in $1/\zeta,1/\alpha,d,1/\varepsilon$ and $\log(1/\delta)$. The third term has a gap of $1/\alpha$ factor to our upper bound, but this term is dominated by other terms under the  assumption that $\delta=e^{-O(d)}$.
For completeness, we state the lower bound in Appendix~\ref{sec:lb}.
Existing algorithms are suboptimal as they  require either $n=\widetilde O ((d/\alpha^2) + (d(\log(1/\delta)^3)/(\alpha\varepsilon^2)))$ samples with $(1/\varepsilon^2)$ dependence to achieve the error rate of Eq.~\eqref{eq:mean_subgauss}  \cite{brown2021covariance}  or  extra conditions such as  strictly Gaussian distributions \cite{brown2021covariance,bun2019private} or known covariance matrices  \cite{KLSU19,aden2020sample,barber2014privacy}. 
%For strictly Gaussian distributions, \cite{brown2021covariance} achieves a slightly improved guarantee of  $n=\widetilde\Omega(\frac{d}{\alpha^2}+\frac{d}{\alpha\varepsilon} + \frac{\log(1/\delta)}{\varepsilon})$, where the improvement is in  the third term  scaling as $\log(1/\delta)/\varepsilon$ instead of $\log(1/\delta)/(\varepsilon\alpha)$.

The error bound is near-optimal in its dependence in $\alpha$ under $\alpha$-corruption. 
HPTR is the first estimator that is both $(\varepsilon,\delta)$-DP and also achieves the  robust error rate of 
$\|\Sigma^{-1/2}(\hat\mu-\mu)\|=O(\alpha\sqrt{\log(1/\alpha)})$, nearly matching the known information-theoretic lower bound of 
$\|\Sigma^{-1/2}(\hat\mu-\mu)\|=\Omega(\alpha)$ \cite{chen2018robust}. This lower bound holds for any estimator that is not necessarily private and regardless of how many samples are available. 
In comparison, the  existing robust and DP estimator from  \cite{liu2021robust}, which  runs in polynomial time, 
requires 
%only guarantees Euclidean error bound with known or identity 
the knowledge of the covariance matrix $\Sigma$ and a larger sample complexity of $n=\widetilde\Omega((d/\alpha^2) + (d^{3/2}\log(1/\delta))/(\alpha\varepsilon))$. If privacy is not required (i.e.,~$\varepsilon=\infty$), a robust mean estimator from \cite{zhu2019generalized} achieves the same error bound and sample complexity as ours. 
%which is a fundamental limit for robust mean estimation  even if  privacy is not required. 

% ----------------------------------

\subsubsection{Hypercontractive  distributions} 
\label{sec:mean-kmoment}
For an integer $k\geq 3$, a distribution $P_{\mu,\Sigma}$ is $k$-th moment bounded with a mean $\mu$ and covariance  $\Sigma$ if for all $\|v\|=1$, we have
$\E_{x\sim P_X}[ |\langle v , (x-\mu)\rangle |^k] \leq \kappa^k$ for some $\kappa>0$.
However, similar to sub-Gaussian case, Mahalanobis distance guarantees require an additional lower bound on the covariance. To this end, we assume  hypercontractivity, which is common in robust statistics literature, e.g., \cite{klivans2018efficient}.

\begin{definition}
    \label{def:hyper}
    A distribution $P_{\mu,\Sigma}$ is 
    $(\kappa,k)$-hypercontractive if for all $v\in\reals^d$, 
    $\E_{x\sim P_X}[ |\langle v , (x-\mu)\rangle |^k] \leq \kappa^k (v^\top\Sigma v)^{k/2}$. 
\end{definition}

Although samples from such heavy-tailed distributions are known to be not resilient, it is known  that it is $O(\alpha)$-close in total variation distance to an $(\alpha,O(\alpha^{1-1/k}),O(\alpha^{1-2/k}))$-resilient dataset. This means that the resulting dataset is $((1/11)\alpha,\alpha,O(\alpha^{1-1/k}),O(\alpha^{1-2/k}))$-corrupt good, for example.  
Note that 
hypercontractivity is invariant under affine transformations and $\kappa$ does not depend on the condition number of the covariance.  

\begin{lemma} [{Resilience of $k$-th moment bounded samples \cite[Lemma G.10]{zhu2019generalized}}] 
    \label{lem:mean_kmoment} 
    For any fixed $\alpha\in(0,1/2)$, consider a dataset $S=\{x_i\in\reals^d\}_{i=1}^n$ of $n$ i.i.d.~samples from a $(\kappa,k)$-hypercontractive  distribution with mean $\mu$ and  covariance $\Sigma\succ 0$ for some $k\geq3$. 
    For any $c_3>0$, there exist constants $c_1$ and $c_2 >0$ that only depend on $c_3$ such that if 
    $$n\;\geq \;c_1\Big(\frac{d}{\zeta^{2(1-1/k)}\alpha^{2(1-1/k)}} + \frac{k^2\alpha^{2-2/k} d\log d}{\zeta^{2-4/k}\kappa^2} + \frac{\kappa^2 d\log d}{\alpha^{2/k}}  \Big)\;,$$ 
    then  $S$ 
    is $( c_3\alpha,\alpha,c_2k\kappa \alpha^{1-1/k}\zeta^{-1/k},c_2k^2\kappa^2\alpha^{1-2/k} \zeta^{-2/k} )$-corrupt good  with respect to $(\mu,\Sigma)$  with probability $1-\zeta$. 
\end{lemma} 

This lemma and  Theorem~\ref{thm:mean} imply the following utility guarantee.  Further, from Theorem~\ref{thm:robust_mean} the guarantee also holds under $(1/5.5-c_3)\alpha$-corruption of the i.i.d. samples from a 
$(\kappa,k)$-hypercontractive distribution. Choosing appropriate constants,  we get the following result.

\begin{coro} 
    \label{coro:mean_kmoment} 
    Under the hypothesis of Lemma~\ref{lem:mean_kmoment} there exists a constant $c_{\kappa,k,\zeta}$ that only depends on $k$, $\kappa$, and $\zeta$ such that for any  $\alpha\in(0,c_{\kappa,k,\zeta})$, a dataset of size 
    $$n\;=\; O \Big(\frac{d+\log(1/(\delta\zeta))}{\varepsilon\alpha } +
    \frac{d}{\zeta^{2(1-1/k)}\alpha^{2(1-1/k)}} + \frac{k^2\alpha^{2-2/k} d\log d}{\zeta^{2-4/k}\kappa^2} + \frac{\kappa^2 d\log d}{\alpha^{2/k}}
    %\frac{d}{\zeta^{2(1-1/k)}\alpha^{1-2/k}} + \frac{\zeta^{2/k} d\log d}{k\alpha} + \frac{\gamma  \zeta^{1/k} d\log d}{k^{1/2} \alpha^{(1+2k)/2k}}
    \Big) \;,$$ 
      sensitivity of $\sens=O(1/(n\alpha^{1/k}))$, and threshold of $\tau=O(\alpha^{1-1/k})$, 
     with large enough constants are sufficient for  \HPTR$(S)$ with  the distance function in Eq.~\eqref{def:mean_distance} to achieve  $\|\Sigma^{-1/2}(\hat\mu-\mu)\|=O(k \kappa \zeta^{-1/k}\alpha^{1-1/k})$ with  probability $1-\zeta$. Further, the same guarantee holds even if $ \alpha$-fraction of the samples are arbitrarily corrupted as in Assumption~\ref{asmp:mean}.
\end{coro}
This sample complexity is near-optimal in its dependence in $d$, $1/\varepsilon$, and $1/\alpha$  when $\delta=e^{-\Theta(d)}$. 
Suppose $\zeta$, $k$, and $\kappa$ are $\Theta(1)$. 
Even for DP mean estimation without robustness,  
HPTR is the first algorithm that  achieves 
$\|\Sigma^{-1/2}(\hat\mu-\mu)\|=O(\alpha^{1-1/k})$ with $n=\widetilde O(\frac{d}{\alpha^{2(1-1/k)}} + \frac{d+\log(1/\delta)}{\varepsilon\alpha})$ samples, which nearly matches the known lower bounds. 
The first term $O(d/\alpha^{2(1-1/k)})$ cannot be improved  even if we do not require privacy.
The second term $O((d+\log(1/\delta))/\varepsilon\alpha)$ nearly matches the lower bound of 
$n=\Omega(  {\min\{d,\log((1-e^{-\varepsilon})/\delta)\}}/
{(\varepsilon\alpha)})$ 
for DP mean estimation  
that we show in  Proposition~\ref{thm:lowerbound_mean_hypercontractive}. 
In typical DP scenarios, we have $0<\varepsilon\leq 1$ and $\delta = e^{-\Theta(d)}$  \cite{barber2014privacy}, in which case the upper and lower bounds match. 
%\cite[Proposition 4]{barber2014privacy}.
An existing  DP mean estimator (without robustness) of \cite{kamath2020private} achieves a stronger $(\varepsilon,0)$-DP and a similar  accuracy but in Euclidean distance with a similar sample size of 
$n=\widetilde O(  \frac{d}{\alpha^{2(1-1/k)}} + \frac{d}{\varepsilon\alpha} )$. However, it  requires  a known or identity covariance matrix and a known bound on the unknown mean of the form $\mu\in[-R,R]^d$. Such a bounded search  space is critical in achieving 
 a stronger {\em pure} privacy guarantee with $\delta=0$.

%The best known  lower bound is for a more tighter notion of moment bound, $\E[\|\Sigma^{-1/2}(x-\mu)\|^k]\leq\gamma^k$, that implies the moment bound assumed in the above corollary.
%It is shown in \cite[Proposition 4]{barber2014privacy}  that $n=\Omega(  \frac{1}{\alpha^{2(1-1/k)}} + \frac{\min\{d,\log(1/\delta)\}}{\varepsilon\alpha})$ is necessary for recovery up to $\|\Sigma^{-1/2}(\hat\mu-\mu)\|= O(\alpha^{1-1/k})$. 

The error bound is optimal in its dependence in $\alpha$ under $\alpha$-corruption. 
The error bound $\|\Sigma^{-1/2}(\hat\mu-\mu)\|=O(\alpha^{1-1/k}) $ matches the following information-theoretic lower bound in Proposition~\ref{pro:mean_kmoment_lb}; no algorithm can distinguish two distributions whose means are at least $O(\alpha^{1-1/k})$ apart from $\alpha$-fraction of samples corrupted, even with infinite samples. %   which  is a fundamental limit for robust mean estimation even if  privacy is not required.  
HPTR is the first algorithm that guarantees  both differential privacy and robustness (i.e., the error only depends on $\alpha$ and not in $d$) for $k$-th moment bounded distributions. 
If privacy is not required (i.e.,~$\varepsilon=\infty$), a robust mean estimator from \cite{zhu2019generalized} achieves a similar  error bound and sample complexity as ours.
%and also  achieves  near-optimality in both robust accuracy and sample complexity. 

%For robust and DP mean estimation, HPTR is the first algorithm with provable guarantees, identifying the trade-off between $\alpha$, $n$, and the error. Further, 

\begin{propo}[Lower bound for robust mean estimation]
    \label{pro:mean_kmoment_lb}

For any $\alpha\in (0,1/2)$, there exist two distributions ${\cal D}_1$ and ${\cal D}_2$ satisfying the hypotheses of Lemma~\ref{lem:mean_kmoment} such that $d_{\rm TV}({\cal D}_1, {\cal D}_2)=\alpha$, and 
    \begin{eqnarray*}
        \|\Sigma^{-1/2}(\mu_1-\mu_2)\| =\Omega( \alpha^{1-1/k})\;.
    \end{eqnarray*}
\end{propo}

\begin{proof}
	We construct two scalar distributions  ${\cal D}_1$ and ${\cal D}_2$ with $d_{\rm TV}({\cal D}_1, {\cal D}_2)=\alpha$ as follows:
	\begin{eqnarray*}
		{\cal D}_1(x)=\begin{cases}
			(1-\alpha)/2, \;\;\; &\text{if }x\in \{-1,1\}\\
			\alpha\;\;\; &\text{if }x=-\alpha^{1/k}
			%\\
			%0 \;\;\;&\text{otherwise}
		\end{cases}\;, \;\;\text{ and } \;\;\;
		{\cal D}_2(x)=\begin{cases}
			(1-\alpha)/2, \;\;\; &\text{if }x\in \{-1,1\}\\
			\alpha\;\;\; &\text{if }x=\alpha^{1/k}
			%\\
			%0 \;\;\;&\text{otherwise}
		\end{cases}
	\end{eqnarray*}
	The variance is $\Omega(1)$ for both distributions and $|\E_{x\sim {\cal D}_1}[x] - \E_{x\sim {\cal D}_2}[x]| = 2\alpha^{1-1/k} $. 
	Then it suffices to  show that ${\cal D}_1$ and ${\cal D}_2$ are both $(O(1),k)$-hypercontractive.  In fact, we know $\E_{x\sim {\cal D}_1}[x]=-\alpha^{1-1/k}$, $\E_{x\sim {\cal D}_1}[x^2]=\E_{x\sim {\cal D}_2}[x^2]=1-\alpha+\alpha^{1-2/k}$ and $\E_{{\cal D}_1}[|x|^k]=2-\alpha$. Since $\alpha\in (0,1/2)$, there exists a constant $c$ such that $\E_{x\sim {\cal D}_1}[|x-\mu_1|^k]\leq c$, which concludes the proof. 

\end{proof}
% \begin{theorem}[Lower bound]
% 	Let $P_{\mu, \Sigma}$ be the set of $k$-th moment bounded distribution with mean $\mu\in \reals^d$ and covariance $\Sigma\in \reals^{d\times d}$, where $\E_{x\sim P_x\in P_{\mu, \Sigma}}[|\ip{v}{\Sigma^{-1/2}(x-\mu)}|^k]\leq c^k$ for some $c>0$ and any $\|v\|=1$. Let $\Theta_{\varepsilon, \delta}$ be a set of $(\varepsilon, \delta)$-DP estimator. Then
% 	\begin{eqnarray}
% 		\inf_{\hat\mu\in \Theta_{\varepsilon, \delta}}\sup_{P_x\in  P_{\mu, \Sigma}}\E_{x\in P_x}\|\hat{\mu}-\mu\|_\Sigma^2\gtrsim \frac{c^2}{n}+c^{2} \min \left\{\left(\frac{d^{2} \wedge \log ^{2} \frac{1}{\delta}}{n^{2} \varepsilon^{2}}\right)^{\frac{k-1}{k}}, 1\right\}
% 	\end{eqnarray}
% \end{theorem}
% \begin{theorem}[Lower bound]
% \label{thm:lowerbound_mean_hypercontractive}
% 	Let $P_{\mu, \Sigma}$ be the set of $k$-th moment bounded distribution with mean $\mu\in \reals^d$ and covariance $\Sigma\in \reals^{d\times d}$, where $\E_{x\sim P_x\in P_{\mu, \Sigma}}[|\ip{v}{\Sigma^{-1/2}(x-\mu)}|^k]\leq c^k$ for some $c>0$ and any $\|v\|=1$. Suppose $\hat{\mu}$ is a $(\varepsilon, \delta)$-DP estimator such that
% 	\begin{eqnarray}
% 		\E_{x\sim  P_x}\|\hat{\mu}-\mu\|_\Sigma^2\leq \alpha^{1-1/k} \;.
% 	\end{eqnarray}
% 	Then 
% 	\begin{eqnarray}
% 	n=\Omega\left(\frac{d\wedge\log(1/\delta)}{\alpha \varepsilon}\right)\;.
% 	\end{eqnarray}
% \end{theorem}
\begin{propo}[Lower bound for DP mean estimation]
\label{thm:lowerbound_mean_hypercontractive}
	Let $\cP_{\mu, \Sigma, k}$ be the set of $(1,k)$-hypercontractive distributions with mean $\mu\in \reals^d$ and covariance $\Sigma\in \reals^{d\times d}$.  Let  $\cM_{\varepsilon,\delta}$ be a class of $(\varepsilon, \delta)$-DP estimators using $n$ i.i.d.~samples from $P\in{\cal P}_{\mu,\Sigma,k}$.
	Then, for $\varepsilon\in(0,10)$, there exists a constant $c$ such that 
	\begin{eqnarray*}
	\inf_{\hat{\mu}\in \cM_{\varepsilon,\delta}}
	\sup_{\mu\in\reals^d,\Sigma\succ 0 ,P\in \cP_{\mu, \Sigma, k}}\E_{S\sim  P^n}[\|\Sigma^{-1/2}(\hat\mu(S)-\mu)\|^2]
	 \ge c 
	\min\left\{\left(\frac{d\wedge \log((1-e^{-\varepsilon})/\delta)}{n\varepsilon}\right)^{2-2/k}, 1\right\}.
	\end{eqnarray*}
	
\end{propo}
\begin{proof}
We extend the proof of {\cite[Proposition 4]{barber2014privacy}} to hypercontractive distributions. Before we prove the lower bound, we first establish the private version of standard statistical estimation problem. Specifically, let $\cP$ denote a family of distributions of interest and $\theta:\cP\rightarrow \Theta$ denote the population parameter. The goal is to estimate $\theta$ from i.i.d. samples $x_1, x_2, \ldots, x_n\sim \cP$. Let $\hat{\theta}$ be an $(\varepsilon, \delta)$-differentially private estimator. Furthermore, let $\rho:\Theta\times\Theta\rightarrow \reals^+$ be a (semi)metric on parameter space $\Theta$ and $\ell:\reals^+\rightarrow \reals^+$ be a non-decreasing loss function with $\ell(0)=0$.

To measure the performance of our $(\varepsilon, \delta)$-DP estimator $\hat{\theta}$,  we define the \emph{minimax risk} as follows:
\begin{eqnarray}
\inf _{\hat{\theta}} \sup _{P \in \mathcal{P}} \mathbb{E}_{x_1, x_2, \ldots, x_n\sim P}\left[\ell\left(\rho\left(\hat{\theta}\left(x_{1}, \ldots, x_{n}\right), \theta(P)\right)\right)\right]\;.
\end{eqnarray}

To prove the lower bound of the minimax risk, we construct a well-separated family of distributions and convert the estimation problem into a testing problem. Specifically, let $\cV$ be an index set of finite cardinality. Define $\cP_{\cV}=\{P_v, v\in \cV\}\subset \cP$ be an indexed family of distributions. If for all $v\neq v' \in \cV$  we have $\rho(P_v, P_{v'})\geq 2t$, we say $\cP_\cV$ is \emph{$2t$-packing} of $\Theta$.

The proof of {\cite[Proposition 4]{barber2014privacy}} is based on following lemma.
\begin{lemma}[{\cite[Theorem 3]{barber2014privacy}}]
\label{thm:packing}
	  Fix $p\in [0,1]$, and let $\cP_\cV$ be a $2t$-packing of $\Theta$ such that $d_{\rm TV}(P_v, P_{v'})= p$. Let $\hat{\theta}$ be $(\varepsilon, \delta)$ differentially private estimator. Then 
	\begin{eqnarray}
		\frac{1}{|\mathcal{V}|} \sum_{\nu \in \mathcal{V}} P_{v}\left(\rho\left(\widehat{\theta}, \theta(P_v)\right) \geq t\right) \geq \frac{(|\mathcal{V}|-1) \cdot\left(\frac{1}{2} e^{-\varepsilon\lceil n p\rceil}-\delta \frac{1-e^{-\varepsilon[n p\rceil}}{1-e^{-\varepsilon}}\right)}{1+(|\mathcal{V}|-1) \cdot e^{-\varepsilon\lceil n p\rceil}}\;.
	\end{eqnarray}
\end{lemma}

In our problem, we set $\cP$ to be $\cP=\cP_{\mu, \Sigma, k}$. It suffices to construct such index set $\cV$ and indexed family of distributions $\cP_{\cV}$. We construct a similar packing set defined in the proof of {\cite[Proposition 4]{barber2014privacy}}. By {\cite[Lemma 6]{acharya2021differentially}}, there exists a finite set $\cV\subset \reals^d$ with cardinality $|\cV|=2^{\Omega(d)}$, $\|v\|=1$ for all $v\in \cV$, and $\|v-v'\|\geq 1/2$ for all $v\neq v'\in \cV$. Define $Q_0$ as $Q_0=\cN(0,\mathbf{I}_{d \times d})$, and $Q_{v}$ as a point mass on ${x=\alpha^{-1/k}cv}$, where $v\in \cV$. We construct $P_v$ as $P_v = \alpha Q_v+(1-\alpha)Q_0$. 

We first verify that $\cP_\cV\subset \cP$. It is easy to see $\mu(P_v)=\E_{x\sim P_v}[x]=\alpha^{1-1/k}v$ and $\Sigma(P_v)=\E_{x\sim P_v}[(x-\mu(P_v))(x-\mu(P_v))^\top]=(1-\alpha)\mathbf{I}_{d \times d}+\alpha(1-\alpha)\alpha^{-2/k}vv^\top$. This implies $\frac{1}{2}\mathbf{I}_{d \times d}\preceq \Sigma(P_v)\preceq \mathbf{I}_{d \times d}$. Since $\E\left[(X-\E[X])^{k}\right] \leq \E\left[X^{k}\right]$ for any $X\geq 0$, it suffices to show $\E_{x\sim P_v}[|\ip{u}{x} |^k]\leq C^k$ for some constant $C>0$ and any $\|u\|=1$. In fact, let $c_k$ denote $k$-th moment of standard Gaussian, we have  
\begin{eqnarray*}
	\E_{x\sim P_v}[|\ip{u}{x} |^k] = (1-\alpha)c_k+\alpha \left|\ip{u}{\alpha^{-1/k}v} \right|^k=O(1)\;.
\end{eqnarray*}

It is also easy to see that $d_{\rm TV}(P_v, P_{v'})=\alpha$. Let $\rho(\theta_1, \theta_2)=\|\theta_1-\theta_2\|$. We also have 
\begin{eqnarray*}
	t =\min_{v\neq v'\in \cV} \alpha^{1-1/k}\|v-v'\|
	\geq  \frac{1}{2}\alpha^{1-1/k}\;.
\end{eqnarray*}

Next, we apply the reduction of estimation to testing with this packing $\cV$. For $(\varepsilon, \delta)$-DP estimator $\hat\mu$, using Lemma~\ref{thm:packing}, we have
\begin{eqnarray*}
	\sup_{P\in \cP}\E_{S \sim P^n}[\|\Sigma(P)^{-1/2}(\hat\mu(S)-\mu(P))\|^2]
		&\geq & \frac{1}{|\cV|}\sum_{v\in \cV}\E_{S \sim P_v^n}[\|\Sigma(P_v)^{-1/2}(\hat\mu(S)-\mu(P_v))\|^2]\\
		&= &t^2\frac{1}{|\cV|}\sum_{v\in \cV}P_{v}\left(\|\Sigma(P_v)^{-1/2}(\hat\mu(S)-\theta(P_v))\|\geq t\right)\\
			&\asymp &t^2\frac{1}{|\cV|}\sum_{v\in \cV}P_{v}\left(\|\hat\mu(S)-\theta(P_v)\|\geq t\right)\\
		&\gtrsim &t^2 \frac{e^{d/2} \cdot\left(\frac{1}{2} e^{-\varepsilon\lceil n \alpha\rceil}-\frac{\delta}{1-e^{-\varepsilon}}\right)}{1+e^{d/2} e^{-\varepsilon\lceil n \alpha\rceil}}\;,
\end{eqnarray*}
where the last inequality follows from the fact that $d\geq 2$.

The rest of the proof follows from  {\cite[Proposition 4]{barber2014privacy}}. We choose 
\begin{eqnarray*}
	\alpha =\frac{1}{n \varepsilon} \min \left\{\frac{d}{2}-\varepsilon, \log \left(\frac{1-e^{-\varepsilon}}{4 \delta e^{\varepsilon}}\right)\right\}
\end{eqnarray*}
 so that 
 \begin{eqnarray*}
 	\sup_{P\in \cP}\E_{S\sim P^n}[\|\Sigma(P)^{-1/2}(\hat\mu(S)-\mu(P))\|^2] \gtrsim  \alpha^{2-2/k} \;.
 \end{eqnarray*}
 This means, for $\varepsilon\in(0,1)$, 
 \begin{eqnarray*}
 	\inf_{\hat{\mu}\in \cM_{\varepsilon, \delta}}\sup_{P\in \cP}\E_{S \sim P^n}[\|\Sigma(P)^{-1/2}(\hat\mu(S)-\mu(P))\|^2]\gtrsim \min\left\{\left(\frac{d\wedge \log((1-e^{-\varepsilon})/\delta)}{n\varepsilon}\right)^{2-2/k}, 1\right\}\;,
 \end{eqnarray*}
which completes the proof.
\end{proof}

% ----------------------------------
\subsubsection{Covariance bounded distributions}
\label{sec:mean-2moment} 

A distribution $P_{\mu,\Sigma}$ is covariance bounded with mean $\mu$ and covariance $\Sigma$ if $\|\Sigma\|\leq 1$. Contrary to the previous cases, the sample variance is not resilient as  $\{\langle v, x_i-\mu \rangle^2\}$ do not concentrate. To get around this issue,  we use the Euclidean distance:  $\popdist(\hat\mu,\mu)=\|\hat\mu-\mu\|$. %instead of $\|\Sigma^{-1/2}(\hat\mu-\mu)\|$. 
This leads to the surrogate Euclidean distance  of
\begin{eqnarray}
    \robdist_S(\hat\mu)=\max_{\|v\|\leq 1} \langle v , \hat\mu\rangle - \mu_v({\cal M}_{v,\alpha}) \;.\label{eq:dist_mean_2moment} 
\end{eqnarray}
As this does not depend on the robust variance, $\sigma^2_v({\cal M}_{v,\alpha}) $, we only require the following first order resilience.

\begin{lemma} [{Resilience of covariance bounded samples \cite[Lemma G.3]{zhu2019generalized}}] 
    \label{lem:mean_2moment} 
    For any fixed $\alpha\in(0,1/2)$, consider a dataset  $S=\{x_i\in\reals^d\}_{i=1}^n$ of $n$ i.i.d.~samples from a covariance bounded distribution with mean $\mu$ and covariance $\Sigma\succ 0$. 
    If $ n=\Omega ({d\log(d/\zeta)}/( 
    \alpha ) )$  
    then with probability $1-3\zeta$, for any subset $T\subset S$ of size $|T| \geq (1-\alpha)n$, there exists  
    a constant $C>0$ such that the following holds for all $\alpha\in(0,1/2)$ and for all $v\in\reals^d$ with $\|v\|=1$:     
    \begin{eqnarray*}
         \Big| \frac{1}{|T|}\sum_{x_i\in T} \langle v, x_i \rangle -\mu_v   \Big|  & \leq &  C\alpha^{1/2} \;,
        %  \frac{C\alpha^{1/2}}{\zeta^{1/2}} 
    \end{eqnarray*}
    where $\mu_v=\langle v , \mu\rangle$. 
\end{lemma} 

This lemma  and  Theorem~\ref{thm:robust_mean}, adapted for the new $\robdist_S(\hat\mu)=\max_{\|v\|\leq 1} \langle v , \hat\mu\rangle - \mu_v({\cal M}_{v,\alpha})$, imply the following utility guarantee. 

\begin{coro} 
    \label{coro:mean_2moment} 
    Under the hypothesis of Lemma~\ref{lem:mean_2moment} there exists a constant $c_{\zeta}$ that only depends on   $\zeta$ such that for $\alpha\in(0,c_{\zeta})$, a dataset of size 
    $$n\;=\; O \Big(\frac{d+\log(1/(\delta\zeta))}{\varepsilon\alpha } +
    \frac{d\log(d/\zeta)}{\alpha} 
    \Big) \;,$$ 
      sensitivity of $\sens=O(1/(n\sqrt{\alpha}))$, and threshold of $\tau=O(\sqrt{\alpha})$, 
     with large enough constants are sufficient for  \HPTR$(S)$ with  the distance function in Eq.~\eqref{eq:dist_mean_2moment} to achieve  $\| \hat\mu-\mu \|=O( \alpha^{1/2})$ with  probability $1-3\zeta$. Further, the same guarantee holds even if $ \alpha$-fraction of the samples are arbitrarily corrupted as in Assumption~\ref{asmp:mean}.
\end{coro}

This sample complexity is near-optimal in its dependence in $d$, $1/\varepsilon$, and $1/\alpha$ for $\delta=e^{-O(d)}$. 
It matches the information-theoretic lower bound of $n=\Omega(d/\varepsilon\alpha)$ from \cite{kamath2020private}.  
For completeness, we write the lower bound in Appendix~\ref{sec:lb}.
This problem is easier then the sub-Gaussian or $k$-th moment bounded settings,  
since the error is measured in Euclidean distance and hence one does not need to adapt to the unknown covariance. 
Therefore there exist other algorithms achieving  near-optimality  and even runs in polynomial time \cite{kamath2020private}.

The error rate is near-optimal under $\alpha$-corruption, matching the information-theoretic lower bound of $\|\hat\mu-\mu\|=\Omega(\alpha^{1/2})$ \cite{dong2019quantum}. Note that there exists an  DP and robust algorithm from \cite{liu2021robust} that achieves near-optimality in both error rate and sample complexity but requires an additional assumption that the spectral norm of the covariance is known and  the unknown mean is in a bounded set, $[-R,R]^d$, with a known $R$. %\alphayang{We removed R in the camera ready version}

\medskip
\noindent 
{\bf Remark.}
    \label{rem:mean_2mooment} 
    Corollary~\ref{coro:mean_2moment} is suboptimal as $(i)$  the error metric is Euclidean $\|\hat\mu-\mu\|$ instead of Mahalanobis $\|\Sigma^{-1/2}(\hat\mu-\mu)\|$, and  $(ii)$ sample complexity scales as $1/\zeta$ instead of $\log(1/\zeta)$. It remains an open problem if these gaps can be closed. 
    For the former, one could use the Stahel-Donoho outlyingness \cite{stahel1981robuste,donoho1982breakdown}, $$\robdist_S(\hat\mu)\;=\;\sup_{v\in\reals^d,\|v\|=1} \frac{|\langle v, \hat\mu\rangle-{\rm Med}(\langle v, S\rangle )|}{{\rm Med}(|\langle v,S\rangle -{\rm Med}(\langle v,S\rangle)|)}\;, $$ 
    in the exponential mechanism, which  replaces second moment based normalization by a first moment based one that is resilient. Here, ${\rm Med}(\langle v, S \rangle)$ is the median of $\{\langle v,x_i \rangle \}_{x_i\in S}$. 
    Further, replacing the median by the  median of means can improve the dependence on $\zeta$.
    Such directions have  been fruitful for robust but non-private mean estimation 
    \cite{depersin2021robustness}.  

% -----------------------------------------
\section{Linear regression}
\label{sec:lin}

In a standard linear regression, we have i.i.d.~samples $S=\{(x_i\in\reals^d,y_i\in \reals)\}_{i=1}^n$ from a distribution $P_{\beta,\Sigma,\gamma^2}$ of a linear model: 
\begin{eqnarray*}
	y_i\;=\;x_i^\top\beta+\eta_i\;,
\end{eqnarray*}
where the input $x_i\in\reals^d$ has zero mean and covariance $\Sigma$ and  
the noise $\eta_i\in\reals$ has  variance $\gamma^2$. 
We further assume  $\E[x_i\eta_i]=0$, which is equivalent to assuming that the true parameter  $\beta=\Sigma^{-1}\E[y_ix_i]$.  
In DP linear regression, we want to output a DP estimate $\hat\beta$ of the unknown model parameter $\beta$ (which corresponds to $\theta=\mu$ in the general notation),  
 assuming  that both covariance $\Sigma\succ 0$  and  the noise variance $\gamma^2$ (corresponding to $\phi=(\Sigma,\gamma)$ in the general notation) are unknown. 
The resulting error is measured in $D_{\Sigma,\gamma}(\hat\beta,\beta)=(1/\gamma)\|\Sigma^{1/2}(\hat\beta -\beta )\|$ which is equivalent to the (re-scaled) root excess prediction risk of the estimated predictor $\hat\beta$.
%because this  is invariant under linear transformations of the form $Qx_i$ and $Q^{-1}\beta$ for any invertible $Q$. 
Similar to Mahalanobis distance for mean estimation, this is challenging as we aim for a tight guarantee that adapts to the unknown $\Sigma$  without having enough samples to directly estimate $\Sigma$. We follow the three-step strategy of Section~\ref{sec:recipe} and provide utility guarantees.

%Let ${\cal D}_1$ be a zero mean one dimensional distribution with unknown variance $\sigma^2$ and ${\cal D}_2$ be a zero-mean distribution on $\reals^d$ with some unknown covariance $\Sigma$. A multiset of labeled samples $S_{\rm good}=\{(x_i\in \reals^d, y_i\in \reals)\}_{i=1}^n$ is generated from a linear model:
%\begin{eqnarray}
%	y_i\;=\;x_i^\top\beta+\eta_i\;,
%\end{eqnarray}
%where $\eta_i\sim {\cal D}_1$, and $x_i\sim {\cal D}_2$, for unknown $\beta$ and $\gamma\in [\sigma_{\min}, \sigma_{\max}]$.

%For some $\alpha$, we are given a corrupted dataset $S$ where an adversary  adaptively inspects all the samples in $S_{\rm good}$, removes $\alpha n$ of them, and replaces them with $S_{\rm bad}$ which are $\alpha n$  arbitrary points in ${\mathbb R}^d\times {\reals}$. 

\subsection{Step 1:  Designing the surrogate  \texorpdfstring{$\robdist_S(\hat\beta)$}{} for the error metric  \texorpdfstring{$(1/\gamma)\|\Sigma^{1/2}(\hat\beta-\beta)\|$}{}}
\label{sec:lr1}

In the {\sc Release} step of  HPTR, we propose  the following surrogate error metric for the exponential mechanism:
\begin{eqnarray}
     \robdist_S(\hat{\beta}) \;=\;  \max_{v:\|v\|\leq 1}   \frac{ \frac{1}{|{\cal N}_{v,\hat{\beta},\alpha}|} \sum_{x_i\in {\cal N}_{v,\hat{\beta},\alpha}}\langle v , x_i(y_i-x_i^\top\hat{\beta}) \rangle}{\sigma_v({\cal M}_{v,\alpha})\hat{\gamma}} \label{def:lr_dist}\;, 
\end{eqnarray}
where $\hat{\gamma}^2$ is defined as
\begin{eqnarray}
	\hat{\gamma}^2 = \min_{\bar{\beta}}\frac{1}{|\cB_{\bar{\beta}, \alpha}|}\sum_{i\in \cB_{\bar{\beta}, \alpha}}(y_i-x_i^\top\bar{\beta})^2\;.\label{def:robust_gamma}
\end{eqnarray}

We define ${\cal N}_{v,\hat{\beta},\alpha}$, ${\cal M}_{v,\alpha}$ and $\cB_{\bar{\beta},\alpha}$ as follows. For a fixed $v$,  ${\cal M}_{v,\alpha}$ is defined in Section~\ref{sec:mean1} as a subset of $S$ with size $(1-(4/5.5)\alpha)n$ that remains after removing $(4/5.5)\alpha n$ data points corresponding to the 
 top $(2/5.5)\alpha n$ and the bottom $(2/5.5)\alpha n$ samples when projected down to  $S_v = \{\ip{v}{x_i}\}_{i\in [n]}$. 
We denote a robust estimate of the variance 
 in direction $v$ as $\sigma_v({\cal M}_{v,\alpha})^2 = (1/|{\cal M}_{v,\alpha}|)\sum_{x_i\in {\cal M}_{v,\alpha}}\langle v , x_i \rangle^2 $, since  $x_i$'s are zero mean. 
Similarly, for fixed $\hat{\beta}$ and $v$, we consider a set of projected data points $S_{v, \hat{\beta}}=\{\langle v , x_i(y_i-x_i^\top\hat{\beta}) \rangle\}_{i\in [n]}$ and partition $S$ into three disjoint sets ${\cal B}_{v,\hat{\beta},\alpha}$, ${\cal N}_{v,\hat{\beta},\alpha}$, and ${\cal T}_{v,\hat{\beta},\alpha}$, where ${\cal B}_{v,\hat{\beta},\alpha}$ is the subset of $S$ corresponding to the bottom $(2/5.5)\alpha n$ data points with smallest values in $S_{v, \hat{\beta}}$, ${\cal T}_{v,\hat{\beta},\alpha}$ corresponds to the top $(2/5.5)\alpha n$ data points, and ${\cal N}_{v,\hat{\beta},\alpha}$ corresponds to  the remaining $(1-(4/5.5)\alpha) n$ middle data points. We use ${\cal T}_{v,\hat{\beta},\alpha},{\cal N}_{v,\hat{\beta},\alpha}$, and ${\cal B}_{v,\hat{\beta},\alpha}$ to denote both the set of paired examples $\{(x_i,y_i)\}$ and the set of indices of those examples, and it should be clear form the context which one we mean. 

For a fixed $\bar{\beta}$,  $\cB_{\bar{\beta}, \alpha}$ is defined as a subset of $S$ with size $(1-(3.5/5.5)\alpha )n$ that remains after removing the largest $(2/5.5)\alpha n$ data points in set $S_{\bar{\beta}}=\{(y_i-x_i^\top\bar{\beta})^2\}_{i\in [n]}$.

% Let $\hat{\gamma}^2=\min_{\beta}\frac{1}{|\cB_{\beta, \alpha}|}(y_i-x_i^\top \beta)^2$.

%Let $\tilde{\beta}:=\left[\begin{array}{c}
%\beta-\hat\beta  \\
% \gamma
%\end{array}\right] \in \reals^{d+1}$, $\tilde{x}_i :=\left[\begin{array}{c}
%x_i  \\
%\eta_i/\gamma
%\end{array}\right] \in \reals^{d+1}$ and $\tilde{v}:=\left[\begin{array}{c}
%v  \\
%0
%\end{array}\right]$. By definition, we know $\tilde{x}_i$ can be seen as samples from a zero mean distribution with covariance 
%$\tilde{\Sigma}:=\left[\begin{array}{cc}
%\Sigma & \\
% & 1
%\end{array}\right] $. We can also write our score function as follows:
%\begin{eqnarray}
%	\robdist_S(\hat{\beta}) \;=\;  \max_{v:\|v\|\leq 1}   \frac{ \frac{1}{|{\cal N}_{v,\hat{\beta},2\alpha}|} \sum_{i\in {\cal N}_{v,\hat{\beta},2\alpha}} \tilde{v}^\top , \tilde{x}_i\tilde{x}_i^\top \tilde{\beta}}{\sigma_v({\cal M}_{v,2\alpha})} \;, 
%\end{eqnarray}

This choice is justified by  Lemma~\ref{lem:true_dist_lr}, which  shows that if we replace the robust one-dimensional statistics by the true ones, we recover the target error metric. 
Hence, the exponential mechanism with distance $\robdist_S(\hat\beta)$ is approximately and stochastically minimizing  $\|\Sigma^{1/2}(\hat\beta-\beta)\|$.
For a more elaborate justification of using $\robdist_S(\hat\beta)$, we refer to a similar choice for mean estimation in Section~\ref{sec:mean1}.

\begin{lemma}
\label{lem:true_dist_lr}
For any $\beta \in\reals^d$, $0\prec \Sigma\in\reals^{d\times d}$, $\gamma>0$, let  $\sigma_v^2=v^\top \Sigma v$. If $\E[\eta_ix_i]=0 $, $y_i=x_i^\top \beta + \eta_i$, and $(x_i,y_i)\sim P_{\beta,\Sigma,\gamma^2}$ then we have 
	\begin{eqnarray*}
 \|\Sigma^{1/2}(\hat{\beta}-\beta) \|
 &=& 		\max_{v:\|v\|\leq 1}   \frac{ \E_{P_{\beta,\Sigma, \gamma^2}}[ \langle v , x_i(y_i-x_i^\top\hat{\beta}) \rangle ] }{\sigma_v} \;,\text{ and }\\
 \gamma^2 &=& \min_{\bar\beta\in\reals^d} \E[(y_i-x_i^\top \bar\beta)^2]\;.
		%\;\;=\;\; \max_{v:\|v\|\leq 1}\frac{\tilde{v}^\top\tilde{\Sigma}\tilde{\beta}}{\sigma_v}
		%\label{eq:true_dist_lr}
	\end{eqnarray*}
\end{lemma}

\begin{proof}
		We have
	\begin{eqnarray*}
		\max_{v:\|v\|\leq 1}   \frac{ \E_{P_{\beta,\Sigma, \gamma^2}}[ \langle v , x_i(y_i-x_i^\top\hat{\beta}) \rangle ] }{\sigma_v}&=& \max_{v:\|v\| \leq 1}   \frac{ \E_{P_{\beta,\Sigma, \gamma^2}}[ \langle v , x_i(x_i^\top(\beta-\hat{\beta})+\eta_i) \rangle ] }{\sigma_v}\\
	&=& \max_{v:\|v\| \leq 1}   \frac{  \langle v , \Sigma(\beta-\hat{\beta}) \rangle  }{\sigma_v} \;=\; \|\Sigma^{1/2}(\beta-\hat\beta)\|\;,
	\end{eqnarray*}
	where the second equality uses the fact that $\eta_i$ has zero mean and $x_i$ has covariance $\Sigma$. The last equality follows from Lemma~\ref{lem:vector_norm}. 
	%The second equality of Eq.~\eqref{eq:true_dist_lr} follows from the fact that $		\langle v , \Sigma(\beta-\hat{\beta}) \rangle = \langle \tilde{v} ,\tilde{ \Sigma}\tilde{\beta} \rangle\;$.
	For the noise, we have 
	$ \E[(y_i-x_i^\top \bar\beta)^2] = \E[(x_i^\top\beta+\eta_i-x_i^\top \bar \beta)^2]=\E[\eta_i^2] + \E[(\beta-\bar\beta)x_ix_i^\top(\beta-\bar\beta)] $, which follows from $\E[\eta_ix_i]=0$. This is minimized when $\bar\beta=\beta$, and the minimum is $\gamma^2$.
	\end{proof}

%\subsubsection{Constants}
%\begin{eqnarray}
%	\rho= \rho_1(5.5\alpha)\\
%	\sens = 20\rho_1(\alpha/2)\gamma/\alpha n\\
%	\thresh = 42\rho_1(5.5\alpha)\gamma
%\end{eqnarray}
%
\subsection{Step 2: Utility analysis under resilience}
\label{sec:lr2}

The following resilience is a fundamental property of the dataset that determines the sensitivity of $\robdist_S(\hat\beta)$.  
We refer to Section~\ref{sec:mean2} for a detailed explanation of how resilience relates to sensitivity. 

\begin{definition}[Resilience for linear regression]
    \label{def:resilience_lr}
    For some $\alpha\in (0,1)$, $\rho_1\in \reals_+$,  $\rho_2\in  \reals_+$, and $\rho_3\in  \reals_+$, we say a set of $n$ labelled data points $S_{\rm good}=\{(x_i\in\reals^d,y_i\in\reals)\}_{i=1}^n$ is $(\alpha,\rho_1,\rho_2, \rho_3, \rho_4)$-resilient with respect to $(\beta,\Sigma,\gamma) $ for some $\beta\in\reals^d$, positive definite $\Sigma\in\reals^{d\times d}$, and $\gamma>0$ if for any $T\subset S_{\rm good}$ of size $|T|\geq (1-\alpha)n$, 
    the following holds %for some universal constant $c>0$ and 
    for all $v\in \reals^d$ with $\|v\|=1$:     
    \begin{eqnarray}
	%\left\|\frac{1}{|T|}\sum_{i\in T}\Sigma^{-1/2}x_i(y_i-x_i^\top \beta)\right\| &\leq & \rho_1(\alpha)\gamma\; \text{ , } \label{def:res1}\\
 \Big|\frac{1}{|T|}\sum_{(x_i,y_i)\in T} \langle v,x_i\rangle (y_i-x_i^\top \beta)\Big| &\leq & \rho_1 \,\sigma_v \, \gamma\; \text{ , } \label{def:res1}\\
	\Big| \frac{1}{|T|}\sum_{x_i\in T} \langle v, x_i\rangle^2  - \sigma_v^2 \Big| & \leq &  \rho_2 \, \sigma_v^2 \label{def:res2} \;\text{ ,  }\\ 
	\Big| \frac{1}{|T|}\sum_{x_i\in T} \langle v, x_i \rangle   \Big|  & \leq &  \rho_3 \, \sigma_v    \;\text{ , and }\label{def:res3}\\
		\Big| \frac{1}{|T|}\sum_{(x_i, y_i)\in T} (y_i-x_i^\top\beta)^2-\gamma^2  \Big|  & \leq &  \rho_4 \, \gamma^2    \;\text{ , }\label{def:res4}
	\end{eqnarray}
    where  $\sigma_v^2 = v^\top \Sigma v$.
\end{definition}

For example, $n$ i.i.d.~samples from   sub-Gaussian $x_i$'s and sub-Gaussian $\eta_i$'s (independent of $x_i$'s) is $\big(\alpha,O(\alpha\log(1/\alpha)),O(\alpha\log(1/\alpha)),O(\alpha\sqrt{\log(1/\alpha)}), O(\alpha\log(1/\alpha))\big)$-resilient. 
 Resilient dataset implies a sensitivity of $\sens=O( \rho_1/(\alpha n)) = O( \log(1/\alpha)/n)$, where $\alpha$ is a free parameter determined by the target accuracy $(1/\gamma)\|\Sigma^{1/2}(\hat\beta-\beta)\| = O( \alpha \log(1/\alpha) )$. We show that a sample size of $O((d+\log(1/\delta))/(\varepsilon \alpha))$ is sufficient to achieve the target accuracy for any resilient dataset. 
In Section~\ref{sec:lr3}, we apply this theorem to resilient datasets from several sampling distributions of interest and characterize the trade-offs. 

\begin{thm}[Utility guarantee for linear regression]
    \label{thm:linear_regression}  There exist positive constants $c$ and $C$ such that for any  $(\alpha, \rho_1,\rho_2, \rho_3, \rho_4)$-resilient set  $S$ with respect to $(\beta,\Sigma\succ 0,\gamma>0)$ satisfying 
     $\alpha\in(0,c)$,$\rho_1<c$, $\rho_2 <c$, $\rho_3^2\leq c\alpha$ and $\rho_4<c$, $\HPTR$ with the distance function in Eq.~\eqref{def:lr_dist},  $\sens=110 \rho_1 /(\alpha n)$, and $\thresh=42\rho_1 $  achieves 
    $(1/\gamma)\|\Sigma^{1/2}(\hat\beta-\beta) \|\leq 32 \rho_1 $ with probability $1-\zeta$, if \begin{eqnarray}
         n \;\geq \; C\, \frac{d+\log(1/(\delta\zeta))}{\varepsilon \alpha}  \; .
    \end{eqnarray}
\end{thm}

%\medskip\noindent
%{\bf Remark.} Notice that the knowledge of $\gamma$ is critical when we {\sc Propose}  $\sens$ and $\thresh$.  If the noise variance is unknown, then we need to estimate $\gamma$ privately before running HPTR up to a sufficient multiplicative accuracy. 
%We give such an algorithm in Appendix~\ref{sec:lr_gamma} and discuss the challenges involved.

%\Xiyang{$\gamma$ can be unknown now}
%Currently, there is no such an algorithm, and we believe designing such an algorithm is outside the scope of this paper.    

\subsubsection{Robustness of HPTR}

One by-product of using  robust statistics in $\robdist_S(\hat\beta)$ is that robustness for HPTR comes for free under a standard data corruption model.

\begin{asmp}[$\alpha_{\rm corrupt}$-corruption]  \label{asmp:lr} 
Given a set $S_{\rm good}=\{(\tilde{x}_i\in\reals^d, \tilde{y}_i\in \reals)\}_{i=1}^n$ of $n$ data points, an adversary inspects all data points, selects $\alpha_{\rm corrupt} n$ of the data points, and replaces them with arbitrary dataset $S_{\rm bad}$ of size $\alpha_{\rm corrupt} n$. 
The resulting corrupted dataset is called $S=\{(x_i\in\reals^d, y_i\in \reals)\}_{i=1}^n$. 
\end{asmp}  

The same guarantee as Theorem~\ref{thm:linear_regression} holds under corruption up to a corruption of  $\alpha_{\rm corrupt}<(1/5.5)\alpha$ fraction of a $(\alpha,\rho_1,\rho_2,\rho_3, \rho_4)$-resilient dataset $S_{\rm good}$. The factor $(1/5.5)$ is due to the fact that the algorithm can remove $(4/5.5)\alpha$ fraction of the good points and a slack of $(0.5/5.5)\alpha$ fraction is needed to resilience of neighboring datasets. 

\begin{definition}[Corrupt good set]
    \label{def:corruptgoodset_lr}
    We say a dataset $S$ is $(\alpha_{\rm corrupt},\alpha,\rho_1,\rho_2, \rho_3, \rho_4 )$-corrupt good with respect to $(\beta,\Sigma, \gamma)$ if it is an $\alpha_{\rm corrupt}$-corruption of an $(\alpha,\rho_1,\rho_2, \rho_3, \rho_4)$-resilient dataset $S_{\rm good}$.
\end{definition}

\begin{thm}[Robustness]
    \label{thm:robust_linear_regression} 
    There exist positive constants $c$ and $C$ such that for any  $((2/11)\alpha,\alpha,\rho_1,\rho_2, \rho_3, \rho_4)$-corrupt good set $S$ with respect to $(\beta,\Sigma\succ 0,\gamma>0)$ satisfying $\alpha< c$, $\rho_1<c$, $\rho_2<c$, $\rho_3^2\leq c\alpha$ and $\rho_4<c$, $\HPTR$  with the distance function in Eq.~\eqref{def:lr_dist},  $\sens=110 \rho_1 /(\alpha n)$, and $\thresh=42\rho_1 $  achieves 
    $(1/\gamma) \|\Sigma^{1/2}(\hat\beta-\beta) \|\leq 32 \rho_1  $ with probability $1-\zeta$, if \begin{eqnarray}
         n \;\geq \; C\, \frac{d+\log(1/(\delta\zeta))}{\varepsilon \alpha}  \; .
    \end{eqnarray}
\end{thm}
We provide a proof in Sections~\ref{sec:lr_proof_strategy}-\ref{sec:lr_proof2}. When there is no adversarial corruption, Theorem~\ref{thm:linear_regression} immediately follows by selecting $\alpha$ as a free parameter. 

\subsubsection{Proof strategy for Theorem~\ref{thm:robust_linear_regression}}
\label{sec:lr_proof_strategy}

The overall proof strategy follows that of Section~\ref{sec:mean_proof_strategy} for mean estimation. We highlight the differences here. 

\begin{lemma}[Lemma~10 from \cite{steinhardt2018resilience}] 
    \label{lem:deviate_lr}
    For a  $(\alpha,\rho_1,\rho_2, \rho_3,\rho_4)$-resilient set $S$ with respect to $(\beta,\Sigma,\gamma)$ and any $0\leq \tilde\alpha \leq  \alpha$,  the following holds  for any subset $T\subset S$ of size at least $\tilde\alpha n$ and for any unit vector  $v\in \reals^d$: 
    \begin{eqnarray}
    % \left\|\frac{1}{|T|}\sum_{i\in T}\Sigma^{-1/2}x_i\eta_i\right\| &\leq & \frac{2-\alpha}{\alpha}\,\rho_1(\alpha)\gamma, \text{ and }  
     \Big|\frac{1}{|T|}\sum_{(x_i,y_i)\in T} \langle v,x_i\rangle (y_i-x_i^\top \beta)\Big| &\leq & \frac{2-\tilde\alpha}{\tilde\alpha}\rho_1 \,\sigma_v\,\gamma \;, 
    \label{eq:res_tail1}\\ 
    \left|\frac{1}{|T|}\sum_{x_i\in T}\, \langle v, x_i\rangle^2-\sigma_v^2 \,\right| &\leq& \frac{2-\tilde\alpha}{\tilde\alpha}\,\rho_2 \sigma_v^2\;, \label{eq:res_tail2}\\
    \left|\frac{1}{|T|}\sum_{x_i\in T}\, \langle v, x_i\rangle \,\right| &\leq& \frac{2-\tilde\alpha}{\tilde\alpha}\,\rho_3  \sigma_v\;,\text{ and }\label{eq:res_tail3} \\
    \Big| \frac{1}{|T|}\sum_{(x_i, y_i)\in T} (y_i-x_i^\top\beta)^2-\gamma^2  \Big|  & \leq &  \frac{2-\tilde{\alpha}}{\tilde{\alpha}}\rho_4 \, \gamma^2    \;\label{eq:res_tail4}\;.
    \end{eqnarray}
\end{lemma} 
This technical lemma is critical in showing that the sensitivity of one-dimensional statistics is bounded by the resilience of the dataset, such that the sensitivity of $\robdist_S(\hat\beta)$ for a resilient $S$ is bounded by 
\begin{eqnarray*}
    |\robdist_S(\hat\beta)-\robdist_{S'}(\hat\beta) | \; \leq \; C'\Big(1+\frac{\rho_3^2}{\alpha}\Big) 
    \frac{\rho_1\, + \, (1/\gamma) \|\Sigma^{1/2}(\hat\beta-\beta)\|}{\alpha n}   \;,
\end{eqnarray*}
for some constant $C'$ and for any neighboring dataset $S'$ as shown in Eq~\eqref{eq:lr_sens_bound}. The desired sensitivity bound is local in two ways: it requires $S$ to be resilient and $(1/\gamma) \|\Sigma^{1/2}(\hat\beta-\beta)\| = O( \rho_1 )$. 
Under the assumption that $\rho_3^2/\alpha=O(1)$ with a small enough constant, this achieves the desired bound $\sens=O(\rho_1 /(\alpha n))$ with $\hat\beta \in B_{\thresh,S}$ and $\thresh=O(\rho_1 )$.
The standard utility analysis of exponential mechanisms shows that the error of $(1/\gamma) \|\Sigma^{1/2}(\hat\beta-\beta)\| =O( \rho_1 )$ can be achieved when  $e^{O(d)-c\frac{\varepsilon}{\sens}\rho_1 } \leq \zeta$, which happens if $n=\Omega((d+\log(1/\zeta))/(\varepsilon \alpha))$ with a large enough constant. The {\sc Test} step checks the two localities by ensuring that DP conditions are met for the given dataset. 

\medskip\noindent
{\bf Outline.} Analogous to the mean estimation proof, the analyses of utility and safety test build upon the universal analysis of HPTR in Theorem~\ref{thm:utility}. For linear regression, 
 we show in Sections~\ref{sec:lr_proof_1}-\ref{sec:lr_proof_3} that the assumptions of Theorem~\ref{thm:utility} are met for a resilient dataset and the choices of constants and parameters: 
$\rho=\rho_1  $, $c_0=31.8$,
$c_1=10.2$, $\thresh=42\rho_1 $, 
$\sens = 110\rho_1 /(\alpha n)$, $\thresh=42\rho_1$, $k^*=(2/\varepsilon)\log(4/(\delta\zeta))$, and a large enough constant $c_2$, and assume that $\alpha<c$ and $\rho_1 <c$ for small enough constant $c$.  
A proof of Theorem~\ref{thm:robust_linear_regression} is provided in Section~\ref{sec:lr_proof2}, and Theorem~\ref{thm:linear_regression} immediately follows by selecting $\alpha$ as a free parameter. 

The above resilience properties also imply the following useful resilience on the $S_{\bar{\beta}}=\{(y_i-\bar{\beta}^\top x_i)^2\}_{i=[n]}$ for any vector $\bar{\beta}$.
\begin{lemma}[Resilience of residual square]
	Let $S_{\rm good}=\{(x_i\in \reals^d, y_i\in \reals)\}_{i=[n]}$ be $(\alpha, \rho_1, \rho_2, \rho_3, \rho_4)$-resilient with respect to $(\beta, \Sigma, \gamma)$. Let $\rho^* = \max\{\rho_1, \rho_2, \rho_4\}$. Then we have 
	\begin{enumerate}
		\item for any $T\in S_{\rm good}$ of size $|T|\geq (1-\alpha)n$ and any vector $\bar{\beta}\in \reals^d$,
	\begin{eqnarray}
		\left|\frac{1}{|T|}\sum_{(x_i,y_i)\in T}(y_i-\bar{\beta}^\top x_i)^2-(\gamma+\|\Sigma^{1/2}(\beta-\bar{\beta})\|)^2\right|\leq \rho^*(\gamma+\|\Sigma^{1/2}(\beta-\bar{\beta})\|)^2\;,
	\end{eqnarray}
	\item and for any $0\leq \tilde{\alpha}\leq \alpha$ and $T\in S_{\rm good}$ of size $|T|\geq \tilde{\alpha} n$, we have
	\begin{eqnarray}
		\left|\frac{1}{|T|}\sum_{(x_i,y_i)\in T}(y_i-\bar{\beta}^\top x_i)^2-(\gamma+\|\Sigma^{1/2}(\beta-\bar{\beta})\|)^2\right|\leq \frac{2-\tilde{\alpha}}{\tilde{\alpha}}\rho^*(\gamma+\|\Sigma^{1/2}(\beta-\bar{\beta})\|)^2\;.
	\end{eqnarray}
	\end{enumerate}
	\label{lem:res_residual}
\end{lemma}
\begin{proof}
	The proof follows directly from resilience properties of Eq.~\eqref{def:res1},  \eqref{def:res2} and \eqref{def:res4}.
\end{proof}
\subsubsection{Resilience implies robustness}
\label{sec:lr_proof_1}

To show that the assumption~\ref{asmp_resilience} in Theorem~\ref{thm:utility} is satisfied, we use the robustness of one-dimensional variance $\sigma_v({\cal M}_{v,\alpha})$ (Lemma~\ref{lem:sigmav_res1}) and show that $\robdist_S(\hat\beta)$ is a good approximation of $(1/\gamma) \|\Sigma^{1/2}(\hat\beta-\beta)\|$ (Lemma~\ref{lem:approximation_lr}).

\begin{lemma}
\label{lem:sigmav_res1}
    For an $((2/11)\alpha,\alpha,\rho_1,\rho_2, \rho_3, \rho_4)$-corrupt good set $S$ with respect to $(\beta,\Sigma,\gamma)$, and any unit norm vector $v\in\reals^d$,  we have  $0.9 \sigma_v \leq \sigma_v({\cal M}_{v,\alpha}) \leq 1.1 \sigma_v $. 
\end{lemma}

\begin{proof}
This follows from Lemma~\ref{lem:deviation}.
%Analogous to \eqref{eq:varianceubound}, we have
%        \begin{eqnarray*}
 %       \frac{ \sum_{i\in{\cal M}_{v,2\alpha}}(\langle v , x_i \rangle^2 -\sigma_v^2)  }{(1-4\alpha)n}  
  %      &=& \frac{ \sum_{i\in {\cal M}_{v,2\alpha}\cap S_{\rm bad}} (\langle v, x_i\rangle^2-\sigma_v^2 ) }{(1-4\alpha)n}
   %     + \frac{  \sum_{i\in {\cal M}_{v,2\alpha}\cap S_{\rm good}} (\langle v, x_i \rangle^2-\sigma_v^2)}{(1-4\alpha)n } \\ 
%        &\leq & \frac{ (\alpha)(\frac{2\rho_2(\alpha)}{\alpha})\sigma_v^2 + (1-\alpha)\rho_2(5\alpha)\sigma_v^2}{1-4\alpha}\\ 
 %       &\leq& 6 \rho_2(5\alpha) \sigma_v^2\;.
%    \end{eqnarray*}
%    Analogous to \eqref{eq:variancelbound}, we have  
%    \begin{align*}
 %        \frac{\sum_{i\in{\cal M}_{v,2\alpha}}
 %        (\langle v, x_i  \rangle^2-\sigma_v^2)}{(1-4\alpha)n} 
%         &\;\geq\; 
%        -\frac{\alpha\sigma_v^2}{1-4\alpha}-\frac{(1-\alpha)}{1-4\alpha}\rho_2(5\alpha)\sigma_v^2  
%        \\ 
%        &\;\geq\; - (2\alpha+2\rho_2(5\alpha)) \sigma_v^2\;.  
 %   \end{align*}
%    For $\alpha\leq0.015$, and $\rho_2(5\alpha)\leq 0.0005$, we have the desired bounds.     
\end{proof}

\begin{lemma}
\label{lem:gamma_approx}
    For an $((2/11)\alpha,\alpha,\rho_1,\rho_2, \rho_3, \rho_4)$-corrupt good set $S$ with respect to $(\beta,\Sigma,\gamma)$, and any unit norm vector $v\in\reals^d$,  we have  $0.99 \gamma \leq \hat{\gamma} \leq 1.01 \gamma $. 
\end{lemma}

\begin{proof}
	Analogous to the proof of Lemma~\ref{lem:deviate_lr}, for any fixed $\bar{\beta}$, we have
	\begin{eqnarray}
		&&\left|\frac{1}{|\cB_{\bar{\beta}, \alpha}|}\sum_{i\in \cB_{\bar{\beta}, \alpha}}(y_i-x_i^\top\bar{\beta})^2-(\gamma+\|\Sigma^{1/2}(\beta-\bar{\beta})\|)^2\right| \nonumber \\ 
		&\leq & \frac{|\sum_{\cB_{\bar{\beta}, \alpha}\cap S_{\rm good}}(y_i-x_i^\top\bar{\beta})^2-(\gamma+\|\Sigma^{1/2}(\beta-\bar{\beta})\|)^2|}{(1-(2/5.5)\alpha)n}\nonumber \\
		&&+\frac{|\sum_{\cB_{\bar{\beta}, \alpha}\cap S_{\rm bad}}(y_i-x_i^\top\bar{\beta})^2-(\gamma+\|\Sigma^{1/2}(\beta-\bar{\beta})\|)^2|}{(1-(2/5.5)\alpha)n}\nonumber\\
		&\stackrel{(a)}{\leq} & \frac{(1-(2/5.5)\alpha)n\rho^*(\gamma+\|\Sigma^{1/2}(\beta-\bar{\beta})\|)^2}{(1-(2/5.5)\alpha)n}+\frac{(2/11)\alpha n \cdot 2\rho^*(\gamma+\|\Sigma^{1/2}(\beta-\bar{\beta})\|)^2/((2/11)\alpha)}{(1-(2/5.5)\alpha)n}\nonumber\\
		&\stackrel{(b)}{\leq} & 4\rho^*(\gamma+\|\Sigma^{1/2}(\beta-\bar{\beta})\|)^2\;, \label{eq:approx_gamma_square}
	\end{eqnarray}
	where $(a)$ follows from Lemma~\ref{lem:res_residual}, and $(b)$ follows from our assumption that $\alpha\leq c$ for some small enough constant $c$.
	
	Let $F(\bar{\beta}) = \frac{1}{|\cB_{\bar{\beta}, \alpha}|}\sum_{i\in \cB_{\bar{\beta}, \alpha}}(y_i-x_i^\top\bar{\beta})^2$. We know $\hat{\gamma}^2 = \min_{\bar\beta}F(\bar{\beta})\leq F(\beta)$, which, together with Eq.~\eqref{eq:approx_gamma_square} implies
	\begin{eqnarray*}
		\hat{\gamma}^2\leq (1+4\rho^*)\gamma^2\leq 1.0201\gamma^2\;,
	\end{eqnarray*}
	when $\rho^*\leq c$ for some $c$ small enough.
	
	Also we have 
	\begin{eqnarray*}
		\hat{\gamma}^2\geq (1-4\rho^*)(\gamma+\|\Sigma^{1/2}(\beta-\bar{\beta})\|)^2\geq (1-4\rho^*)\gamma^2\geq 0.9801 \gamma^2.
	\end{eqnarray*}
	when $\rho^*\leq c$ for some $c$ small enough.
\end{proof}

\begin{lemma} 
     For an $((2/11)\alpha,\alpha,\rho_1,\rho_2, \rho_3,\rho_4 )$-corrupt good set $S$ with respect to $(\beta, \Sigma, \gamma)$, if $\hat\beta\in B_{\thresh,S}$ and $\thresh=42\rho_1 $ then $\big|\, \|\Sigma^{1/2}(\hat\beta-\beta) \|/\gamma-\robdist_S(\hat\beta) \,\big| \leq 0.15 \tau+1.1\rho_1 \leq 10.2\rho_1  $.  
    \label{lem:approximation_lr}
\end{lemma} 
\begin{proof}

By Lemma~\ref{lem:true_dist_lr}, Lemma~\ref{lem:bounded_norm} and resilience Eq.~\eqref{def:res1} and Eq.~\eqref{def:res2}, we have    

\begin{eqnarray*}
&&\left|\max_{v:\|v\|\leq 1}   \frac{ \frac{1}{|{\cal N}_{v,\hat{\beta},\alpha}|} \sum_{i\in {\cal N}_{v,\hat{\beta},\alpha}}\langle v , x_i(y_i-x_i^\top\hat{\beta}) \rangle}{\sigma_v}-\left\|\Sigma^{1/2}(\beta-\hat{\beta})\right\|\right|\\
& = &\left|\max_{v:\|v\|\leq 1}   \frac{ \frac{1}{|{\cal N}_{v,\hat{\beta},\alpha}|} \sum_{i\in {\cal N}_{v,\hat{\beta},\alpha}}\left(v^\top x_ix_i^\top(\beta-\hat{\beta})+v^\top x_i\eta_i\right)}{\sigma_v}-\max_{v:\|v\|\leq 1}\frac{v^\top\Sigma(\beta-\hat{\beta})}{\sigma_v}\right|\\
&\leq & \max_{v:\|v\|\leq 1}\left|\frac{v^\top\left(\frac{1}{|{\cal N}_{v,\hat{\beta},\alpha}|}\sum_{i\in {\cal N}_{v,\hat{\beta},\alpha}}x_ix_i^\top-\Sigma\right)(\beta-\hat\beta)}{\sigma_v}+\frac{v^\top\frac{1}{|{\cal N}_{v,\hat{\beta},\alpha}|}\sum_{i\in {\cal N}_{v,\hat{\beta},\alpha}}x_i\eta_i}{\sigma_v}\right|\\
&\leq  &\left\|\Sigma^{-1/2}\left(\frac{1}{|{\cal N}_{v,\hat{\beta},\alpha}|}\sum_{i\in {\cal N}_{v,\hat{\beta},\alpha}}x_ix_i^\top-\Sigma\right)(\beta-\hat{\beta})\right\|+\left\|\Sigma^{-1/2}\frac{1}{|{\cal N}_{v,\hat{\beta},\alpha}|}\sum_{i\in {\cal N}_{v,\hat{\beta},\alpha}}x_i\eta_i\right\|\\
&\leq & \rho_2  \|\Sigma^{1/2}(\beta-\hat\beta)\|+\rho_1 \gamma\;.\\
\end{eqnarray*}     

Together with Lemma~\ref{lem:sigmav_res1}, this implies 
%\begin{eqnarray}
%	\frac{1-\rho_1(5\alpha)}{1.1}\left\|\Sigma^{1/2}(\beta-\hat{\beta})\right\|-\frac{\rho_1(5\alpha)\gamma}{1.1}\;\;\leq\;\;\robdist_S(\hat\beta)\leq \frac{1+\rho_1(5\alpha)}{0.9}\left\|\Sigma^{1/2}(\beta-\hat{\beta})\right\|+\frac{\rho_1(5\alpha)\gamma}{0.9}\;.
%\end{eqnarray}
%and 
\begin{eqnarray*}
	\frac{0.9\robdist_S(\hat\beta)\hat{\gamma}-\rho_1 \gamma}{1+\rho_2 }\;\;\leq\;\;\left\|\Sigma^{1/2}(\beta-\hat{\beta})\right\|\;\;\leq\;\; \frac{1.1\robdist_S(\hat\beta)\hat{\gamma}+\rho_1 \gamma}{1-\rho_2 }\;.
\end{eqnarray*}

Assuming $\rho_2 \leq 0.013$, we have $0.86\robdist_S(\hat\beta)-1.1\rho_1 \leq\left\|\Sigma^{1/2}(\beta-\hat{\beta})\right\|/\gamma\leq 1.15\robdist_S(\hat\beta)+1.1\rho_1 $. Since $\robdist_S(\hat\beta)\leq \tau$, we get the desired bound.
\end{proof}
\subsubsection{Bounded Volume}
We show that the assumption~\ref{asmp_vol} in Theorem~\ref{thm:utility} is satisfied for robust estimate $\robdist_S(\hat\beta)$.

\begin{lemma}
For $\rho=\rho_1  $, $c_0=31.8$,
$c_1=10.2$, $\thresh=42\rho_1 $, 
$\sens =110\rho_1 /(\alpha n)$, and $c_2\geq\log(67/12)+\log((c_0+2c_1)/c_1)$, 
we have 
 $(7/8)\thresh-(k^*+1)\sens>0$,  
        \begin{eqnarray*}
        \frac{{\rm Vol}(B_{ \thresh+(k^*+1)\sens+c_1\rho,S}) }
        {{\rm Vol}(B_{ (7/8)\thresh-(k^*+1)\sens-c_1\rho,S})}
        &\leq& e^{c_2 d}\;, \text{ and }\\
        \frac{{\rm Vol}(\{\hat\theta:\|\Sigma^{1/2}(\hat \beta -\beta)\|/\gamma \leq (c_0+2c_1)\rho\})}
        {{\rm Vol}(\{\hat\theta: \|\Sigma^{1/2} (\hat \beta -\beta)\|/\gamma \leq c_1\rho\})}
        &\leq &e^{c_2 d}\;.
        \end{eqnarray*}
    \label{lem:vol_lr} 
\end{lemma}

\begin{proof}

The proof is similar to the proof of Lemma~\ref{lem:vol}. The second part of assumption~\ref{asmp_vol} follows from the fact that 
\begin{eqnarray*}
    {\rm Vol}(\{\hat\mu:\|\Sigma^{1/2}(\hat\beta-\beta)\|\leq r\}) = c_d |\Sigma| r^d\;,
\end{eqnarray*}
for some constant $c_d$ that only depends on the dimension
and selecting $c_2\geq \log((c_0+2c_1)/c_1)$. The first part follows from our choices of $c_0$, $c_1$, $\thresh$, $\sens$ and the following corollary.

\begin{coro}[{Corollary of Lemma~\ref{lem:approximation_lr}}]
     If $\hat\beta\in B_{2\thresh,S}$ and $\thresh=42\rho_1 $ then $\big|\, \|\Sigma^{1/2}(\hat\beta-\beta) \|/\gamma-\robdist_S(\hat\beta) \,\big| \leq   14.2 \rho_1 $.  
    \label{coro:approximation_lr}
\end{coro} 
\end{proof}

\subsubsection{Resilience implies bounded local sensitivity}
\label{sec:lr_proof_3}

We show that resilience implies the assumption \ref{asmp_local} in Theorem \ref{thm:utility}  (Lemma~\ref{lem:local_asmp_lr}). 
Assuming $(k^*+1)/n\leq \alpha/2$, we show a set $S'$ with at most $k^*$ data points arbitrarily changed from $S$ has bounded local sensitivity. This implies that $S'$ is a $((1/5.5)\alpha+(k^*/n),\alpha,\rho_1,\rho_2, \rho_3, \rho_4 )$-corrupt good set with respect to  $(\beta,\Sigma, \gamma)$.
\begin{lemma}
    For an $((1/5.5)\alpha+\tilde\alpha,\alpha,\rho_1,\rho_2, \rho_3,\rho_4)$-corrupt good set $S'$ with respect to $(\beta, \Sigma, \gamma)$,  $\tilde\alpha\leq(1/11)\alpha$, and any unit norm $v\in\reals^d$,  we have  $0.9 \sigma_v \leq \sigma_v({\cal M}_{v,\alpha}) \leq 1.1 \sigma_v $. 
\end{lemma}

\begin{proof} This follows from Lemma~\ref{lem:deviation2}. 
%Analogous to \eqref{eq:varianceubound}, we have
%        \begin{eqnarray*}
%        \frac{ \sum_{i\in{\cal M}_{v,2\alpha}}(\langle v , x_i \rangle^2 -\sigma_v^2)  }{(1-4\alpha)n}  
%        &=& \frac{ \sum_{i\in {\cal M}_{v,2\alpha}\cap S_{\rm bad}} (\langle v, x_i\rangle^2-\sigma_v^2 ) }{(1-4\alpha)n}
%        + \frac{  \sum_{i\in {\cal M}_{v,2\alpha}\cap S_{\rm good}} (\langle v, x_i \rangle^2-\sigma_v^2)}{(1-4\alpha)n } \\ 
%        &\leq & \frac{ (\alpha+\tilde\alpha)(\frac{2\rho_2(\alpha-\tilde\alpha)}{\alpha-\tilde\alpha})\sigma_v^2 + (1-\alpha-\tilde\alpha)\rho_2(5\alpha+\tilde\alpha)\sigma_v^2}{1-4\alpha}\\ 
%        &\leq& 14 \rho_2(5.5\alpha) \sigma_v^2\;.
%    \end{eqnarray*}
%    Analogous to \eqref{eq:variancelbound}, we have  
%    \begin{align*}
     %    \frac{\sum_{i\in{\cal M}_{v,2\alpha}}
      %   (\langle v, x_i  \rangle^2-\sigma_v^2)}{(1-4\alpha)n} 
    %     &\;\geq\; 
    %    -\frac{(\alpha+\tilde\alpha)\sigma_v^2}{1-4\alpha}-\frac{(1-\alpha-\tilde\alpha)}{1-4\alpha}\rho_2(5\alpha+\tilde\alpha)\sigma_v^2  
  %      \\ 
   %     &\;\geq\; - (3\alpha+2\rho_2(5.5\alpha)) \sigma_v^2\;.  
 %   \end{align*}
%    For $\alpha\leq0.045$, and $\rho_2(5.5\alpha)\leq 0.0005$, we have the desired bounds.     
\end{proof}
\begin{lemma}
\label{lem:gamma_approx2}
    For an $((1/5.5)\alpha+\tilde{\alpha},\alpha,\rho_1,\rho_2, \rho_3, \rho_4)$-corrupt good set $S'$ with respect to $(\beta,\Sigma,\gamma)$, and any unit norm vector $v\in\reals^d$,  we have  $0.99 \gamma \leq \hat{\gamma} \leq 1.01 \gamma $. 
\end{lemma}
\begin{proof}
	This proof follows from the proof of Lemma~\ref{lem:gamma_approx}.
\end{proof}

\begin{lemma}
    \label{lem:distance_perturbed_lr} 
    For an $((1/5.5)\alpha+\tilde\alpha,\alpha,\rho_1,\rho_2, \rho_3,\rho_4)$-corrupt good set $S'$ with respect to $(\beta, \Sigma, \gamma)$ and  $\tilde\alpha\leq(1/11)\alpha$, if $\hat\beta \in B_{t, S'}$
    then we have  $ \|\Sigma^{1/2}(\hat\beta-\beta)\|/\gamma \leq 1.1\rho_1  + 1.15t$ and $\big|\robdist_{S'} (\hat\beta) - \|\Sigma^{1/2}(\hat\beta-\beta) \| /\gamma\big|\leq 1.1\rho_1  + 0.15t$. 
\end{lemma}
\begin{proof}
    It follows from the proof of Lemma~\ref{lem:approximation_lr}. 
\end{proof}

\begin{lemma}
%[We show {Assumption~\ref{asmp_local}} is satisfied]
    \label{lem:local_asmp_lr}
    For $\sens=110\rho_1 /(\alpha n)$, $\thresh = 42\rho_1 $, 
    and an $((1/5.5)\alpha,\alpha,\rho_1,\rho_2, \rho_3,\rho_4)$-corrupt good  $S$, if 
    \begin{eqnarray*}
        n\;=\;\Omega\Big(
        \frac{\log(1/(\delta\zeta))}{\alpha\varepsilon}
        \Big)\;,
    \end{eqnarray*} 
    with a large enough constant  
    then the local sensitivity in assumption~\ref{asmp_local} is satisfied.
\end{lemma}

\begin{proof}
We follows the proof strategy of Lemma~\ref{lem:local_asmp} in Section~\ref{sec:mean3_sens}. 
Consider a dataset $S'$ which is 
at Hamming distance at most $(1/11)\alpha n$ from $S$ and  
corresponding partition $({\cal T'}_{v,\hat\beta,\alpha},{\cal N'}_{v,\hat\beta,\alpha},{\cal B'}_{v,\hat\beta,\alpha})$ of $S'$ for a specific direction $v$. 
By resilience property of the tails in  Eq.~\eqref{eq:res_tail1} and Eq.~\eqref{eq:res_tail2}, Lemma~\ref{lem:vector_norm}, and Lemma~\ref{lem:bounded_norm}, we have for any $v\in \reals^{d}$ with unit norm $\|v\|=1$ and any $\hat{\beta}\in \reals^{d}$, 

\begin{eqnarray}
	&&\frac{v^\top\frac{1}{|{\cal T}_{v, \hat{\beta}, \alpha}'\cap S_{\rm good}|}\sum_{i\in {\cal T}_{v, \hat{\beta}, \alpha}'\cap S_{\rm good}}\left(\left(x_ix_i^\top - {\Sigma}\right)(\beta-\hat\beta)+x_i\eta_i\right)}{\sigma_v} \nonumber \\
	&\leq & \left\|\Sigma^{-1/2}\left(\frac{1}{|{\cal T}_{v, \hat{\beta}, \alpha}'\cap S_{\rm good}|}\sum_{ i\in {\cal T}_{v, \hat{\beta}, \alpha}'\cap S_{\rm good}}\left(x_ix_i^\top -  {\Sigma}\right)(\beta-\hat\beta)\right)\right\|+
	\\&&\left\|\Sigma^{-1/2}\left(\frac{1}{|{\cal T}_{v, \hat{\beta}, \alpha}'\cap S_{\rm good}|}\sum_{i\in {\cal T}_{v, \hat{\beta}, \alpha}'\cap S_{\rm good}}x_i\eta_i\right)\right\| \nonumber \\
	&\leq & 
	\frac{2\rho_2}{(1/11)\alpha}\|\Sigma^{1/2}(\beta-\hat{\beta})\|+\frac{2\rho_1 }{(1/11)\alpha}\gamma\;, \label{eq:top_set_lr}
\end{eqnarray}
where $S_{\rm good}$ is the original uncorrupted resilient dataset. 
Similarly, we have
\begin{eqnarray*}
	\frac{v^\top\frac{1}{|{\cal B}_{v, \hat{\beta}, \alpha}'\cap S_{\rm good}|}\sum_{i\in {\cal B}_{v, \hat{\beta}, \alpha}'\cap S_{\rm good}}\left(\left(x_ix_i^\top - {\Sigma}\right)(\beta-\hat\beta)+x_i\eta_i\right)}{\sigma_v} 
	&\leq & \frac{2\rho_2 }{(1/11)\alpha}\|\Sigma^{1/2}(\beta-\hat{\beta})\|+\frac{2\rho_1 }{(1/11)\alpha}\gamma\;.
\end{eqnarray*}

This implies
\begin{eqnarray}
	\min_{i \in {\cal T}_{v, \hat{\beta}, \alpha}'\cap S_{\rm good}}\frac{v^\top \left(x_i x_i^\top(\beta-\hat{\beta})+x_i\eta_i\right)}{\sigma_v}-\max_{i \in {\cal B}_{v, \hat{\beta}, \alpha}'\cap S_{\rm good}}\frac{\tilde{v}^\top\left(x_i x_i^\top(\beta-\hat{\beta})+x_i\eta_i\right)}{\sigma_v}\nonumber\\
	\leq \;\;\;\;\frac{44\rho_1}{\alpha}\gamma+\frac{44\rho_2}{\alpha}\|\Sigma^{1/2}(\beta-\hat{\beta})\|\;. \label{eq:lr_bounded_support}
\end{eqnarray}
% This means the scalar dataset $\tilde{S}'_{v, \hat{\beta}}=\{v^\top (x_i x_i^\top(\beta-\hat{\beta})+x_i\eta_i)\}$ is dense at the top and bottom $2\alpha$ quantiles. 
Analogous to Lemma~\ref{lem:local_asmp}
, for a neighboring databases $S'$ and $S''$, the corresponding middle sets ${\cal N'}_{v,\hat\beta, \alpha}$ and ${\cal N''}_{v,\hat\beta, \alpha}$ differ at most by one entry. Denote those entry by $x'_i$ and $\eta'_i=y'_i-\langle \beta, x'_i\rangle $ in ${\cal N'}_{v,\hat\beta, \alpha}$ and $x_j^{\prime \prime}$ and $\eta_j^{\prime \prime}$ in ${\cal N}''_{v,\hat\beta, \alpha}$. Then, from Eq.~\eqref{eq:lr_bounded_support},  we have
\begin{eqnarray*}
	\left|v^\top \left(\left(x_i'x_i'^\top-x_j^{\prime \prime}x_j^{\prime \prime\top}\right)(\beta-\hat\beta)+x'_i\eta'_i-x_j^{\prime \prime}\eta_j^{\prime \prime}\right)\right|\leq \left(\frac{44\rho_1 }{\alpha}\gamma+\frac{44\rho_2 }{\alpha}\|\Sigma^{1/2}(\beta-\hat{\beta})\|\right) \sigma_v\;,
\end{eqnarray*}
which implies that 
\begin{eqnarray}
	&&\left|v^\top\frac{1}{(1-(4/5.5)\alpha) n}\sum_{i\in {\cal N}_{v,\hat{\beta}, \alpha}'}\left(x_i x_i^\top(\beta-\hat{\beta})+x_i\eta_i\right)-v^\top\frac{1}{(1-(4/5.5) \alpha) n}\sum_{i\in {\cal N}_{v,\hat{\beta}, \alpha}''}\left(x_i x_i^\top(\beta-\hat{\beta})+x_i\eta_i\right)\right| \nonumber\\
&&\hspace*{6cm}	\leq \;\;\;\frac{\sigma_v}{(1-(4/5.5) \alpha)n} \left(\frac{44\rho_1 }{\alpha}\gamma+\frac{44\rho_2 }{\alpha}\|\Sigma^{1/2}(\beta-\hat{\beta})\|\right) \;. 
	\label{eq:lr_sens_condition1}
\end{eqnarray}
By resilience properties in  Eq.~\eqref{def:res1} and Eq.~\eqref{def:res2}, and  Lemma~\ref{lem:bounded_norm}, Lemma~\ref{lem:true_dist_lr}, and the fact that ${\cal N}_{v, \hat{\beta}, \alpha}''\cap S_{\rm good}$ is at least of size $(1-\alpha)n$,   we have 
for the data points in ${\cal N}_{v, \hat{\beta}, \alpha}''\cap S_{\rm good}$,
\begin{eqnarray*}
	\frac{v^\top\frac{1}{|{\cal N}_{v, \hat{\beta}, \alpha}''\cap S_{\rm good}|}\sum_{i\in {\cal N}_{v, \hat{\beta}, \alpha}''\cap S_{\rm good}}\left(x_i x_i^\top(\beta-\hat{\beta})+x_i\eta_i\right)}{\sigma_v}
	\leq (1+\rho_2 )\|\Sigma^{1/2}(\hat{\beta}-\beta)\|+\rho_1 \gamma\;.
\end{eqnarray*} 
%For the data points in ${\cal N}_{v, \hat{\beta}, \alpha}''\cap S_{\rm bad}$,
By Eq.~\eqref{eq:top_set_lr}, for any $x''_i\in {\cal N}_{v, \hat{\beta}, \alpha}''\cap S_{\rm bad}$ (where $S_{\rm bad}=S''\setminus S_{\rm good}$) we have
\begin{eqnarray*}
	\frac{v^\top\left(x''_i x_i''^\top(\beta-\hat{\beta})+x''_i\eta''_i\right)}{\sigma_v} 
	&\leq & \frac{v^\top\frac{1}{|{\cal T}_{v, \hat{\beta}, \alpha}''\cap S_{\rm good}|}\sum_{i\in {\cal T}_{v, \hat{\beta}, \alpha}''\cap S_{\rm good}}\left(x_i x_i^\top(\beta-\hat{\beta})+x_i\eta_i\right)}{\sigma_v}\\
	&\leq &\Big(\frac{22\rho_2 }{\alpha}+1\Big)\|\Sigma^{1/2}(\hat{\beta}-\beta)\|+\frac{22\rho_1 }{\alpha}\gamma\;.
\end{eqnarray*}
Since $|S_{\rm bad}|\leq (1.5/5.5)\alpha n$ and $\alpha<c$ for some small enough constant $c$, we have
\begin{eqnarray}
	&&\frac{v^\top\frac{1}{(1-(4/5.5)\alpha)n}\sum_{i\in {\cal N}_{v, \hat{\beta}, \alpha}''}\left(x_i x_i^\top(\beta-\hat{\beta})+x_i\eta_i\right)}{\sigma_v} \nonumber \\
&=&\frac{v^\top\frac{1}{(1-(4/5.5)\alpha)n}\sum_{i\in {\cal N}_{v, \hat{\beta}, \alpha}''\cap S_{\rm bad}}\left(x_i x_i^\top(\beta-\hat{\beta})+x_i\eta_i\right)}{\sigma_v}+\nonumber\\
&&\frac{v^\top\frac{1}{(1-(4/5.5)\alpha)n}\sum_{i\in {\cal N}_{v, \hat{\beta}, \alpha}''\cap S_{\rm good}}\left(x_i x_i^\top(\beta-\hat{\beta})+x_i\eta_i\right)}{\sigma_v} \nonumber \\
&\leq &\frac{(6\rho_2+(1.5/5.5)\alpha)\|\Sigma^{1/2}(\hat{\beta}-\beta)\|+6\rho_1 \gamma}{1-(4/5.5)\alpha} + %\frac{1}{1-(4/5.5)\alpha}
\Big((1+\rho_2 )\|\Sigma^{1/2}(\hat{\beta}-\beta)\|+\rho_1 \gamma\Big) \nonumber \\
&\leq  & 7\rho_1 \gamma+(1+\alpha+7\rho_2 )\|\Sigma^{1/2}(\hat{\beta}-\beta)\| \;.
\label{eq:sensitvity_bound_lr1}
\end{eqnarray}

Analogous to Eq.~\eqref{eq:varsensitivitybound}, by using resilience properties in Eqs.~\eqref{def:res2} and \eqref{def:res3}, we have
\begin{eqnarray}
	|\sigma_v'^2-\sigma_v''^2| &=&  \frac{1}{(1-(4/5.5)\alpha) n}\left|\sum_{x_i\in {\cal N}_{v,\hat{\beta},\alpha}'}\ip{v}{x_i}^2-\sum_{x_i\in {\cal N}_{v,\hat{\beta},\alpha}''}\ip{v}{x_i}^2\right|\nonumber\\
	&\leq &\frac{64 \cdot 11^2 \cdot \rho_3^2\sigma_v^2}{\alpha^2(1-(4/5.5)\alpha)n}\;.\label{eq:varsensitivitybound_lr}
\end{eqnarray}

By Eqs.~\eqref{eq:sensitvity_bound_lr1}, \eqref{eq:lr_sens_condition1}, and  \eqref{eq:varsensitivitybound_lr}, we have
\begin{eqnarray}
\hspace*{-4cm}&&		\left|\robdist_{S'}(\hat{\beta})-\robdist_{S''}(\hat{\beta})\right| \nonumber \\
\leq && \max_{v:\|v\|=1}\left|\frac{v^\top\frac{1}{|{\cal N}_{v,\hat{\beta},\alpha}'|}\sum_{i\in {\cal N}_{v,\hat{\beta},\alpha}'}\left(x_i x_i^\top(\beta-\hat{\beta})+x_i\eta_i\right)}{\sigma_v'\hat{\gamma}'}-\frac{v^\top\frac{1}{|{\cal N}_{v,\hat{\beta},\alpha}''|}\sum_{i\in {\cal N}_{v,\hat{\beta},\alpha}''}\left(x_i x_i^\top(\beta-\hat{\beta})+x_i\eta_i\right)}{\sigma_v''\hat{\gamma}''}\right|\nonumber\\
\leq && \max_{v:\|v\|=1}\left|\frac{v^\top\left(\frac{1}{(1-(4/5.5)\alpha)n}\sum_{i\in {\cal N}_{v,\hat{\beta},\alpha}'}\left(x_i x_i^\top(\beta-\hat{\beta})+x_i\eta_i\right)-\frac{1}{(1-(4/5.5)\alpha)n}\sum_{i\in {\cal N}_{v,\hat{\beta}, \alpha}''}\left(x_i x_i^\top(\beta-\hat{\beta})+x_i\eta_i\right)\right) }{\sigma_v'\hat{\gamma}'}\right| \nonumber\\
&&+\max_{v:\|v\|=1}\frac{v^\top\frac{1}{|{\cal N}_{v,\hat{\beta},\alpha}''|}\sum_{x_i \in {\cal N}_{v,\hat{\beta},\alpha}''}\left(x_i x_i^\top(\beta-\hat{\beta})+x_i\eta_i\right)}{\sigma_v}\left|\frac{\sigma_v}{\sigma_v'\hat{\gamma}'}-\frac{\sigma_v}{\sigma_v''\hat{\gamma}''}\right| \nonumber\\
\leq && \frac{44 \rho_1 }{0.9\cdot 0.99(1-(4/5.5)\alpha)n\alpha} + \frac {44\rho_2 }{0.9\cdot 0.99 (1-(4/5.5) \alpha)n \alpha}\frac{\|\Sigma^{1/2} (\beta-\hat{\beta})\|}{\gamma} \nonumber\\
&&+\frac{64\cdot 11^2\cdot  \rho_3^2 \cdot 0.02\gamma} {0.9^3\alpha^2(1-(4/5.5)\alpha)n \cdot 0.99^2\gamma^2}\left(7\rho_1 \gamma+(1+\alpha+7\rho_2 )\|\Sigma^{1/2}(\hat{\beta}-\beta)\|\right)  \label{eq:lr_sens_bound}\\
\leq &&\left( \frac{0.12}{\alpha n}+\frac{0.016 }{\alpha n}\right)\frac{\|\Sigma^{1/2}(\hat{\beta}-\beta)\|}{\gamma}+\left(\frac{9\rho_1 }{\alpha n}+\frac{0.07 \rho_1 }{\alpha n}\right) \nonumber\\
\leq && \frac{0.2 }{\alpha n}\frac{\|\Sigma^{1/2}(\hat{\beta}-\beta)\|}{\gamma}+\frac{50\rho_1 }{\alpha n} \nonumber
\end{eqnarray}
where the last three inequalities follow from our assumptions that $\alpha\leq c$ and $\rho_2 \leq c$, $\rho_3 ^2\leq c\alpha $, $\rho_4\leq c$  with a small enough constant $c$ and Lemma~\ref{lem:gamma_approx2}. 
From Lemma~\ref{lem:distance_perturbed_lr}, we know if $\hat{\beta}\in B_{\thresh+(k^*+3)\sens, S}$, we have $ \|\Sigma^{1/2}(\hat\beta-\beta)\|/\gamma \leq 1.1\rho_1   + 1.15(\thresh+(k^*+3)\sens)$. We show that $\|\Sigma^{1/2}(\hat\beta-\beta) \|\leq 50\rho_1 \gamma$ for the choices of $\sens$, $k^*$, $\thresh$ and $n$:
    \begin{eqnarray*}
        1.1\rho_1 + 1.15(\thresh+\backoff) &\leq & 49 \rho_1 +\frac{50\rho_1 \log(1/(\delta\zeta))}{\varepsilon\alpha n} \\
        &\leq & 50\rho_1  
        \;,
    \end{eqnarray*}
    where $\sens=110\rho_1 /(\alpha n)$, $\thresh= 42\rho_1  $, 
    $k^*=(2/\varepsilon)\log(4/(\delta\zeta))$, $\varepsilon\leq \log(4/\delta\zeta)$ and $n\geq C' \log(1/(\delta \zeta))/(\varepsilon\,\alpha )$ for some large enough universal constant $C'>0$. 
    This implies
    \begin{eqnarray*} 
    |\robdist_{S'}(\hat\beta)-\robdist_{S''}(\hat\beta)|&\leq & \frac{110\rho_1 }{\alpha n} = \sens\;.
    \end{eqnarray*}
\end{proof}

\subsubsection{Proof of Theorem~\ref{thm:robust_linear_regression}}
\label{sec:lr_proof2} 

We show that the sufficient conditions of Theorem~\ref{thm:utility} are met for the choices of constants and parameters: 
$p=d$, $\rho=\rho_1 $, $c_0=31.8$, $c_1=10.2$, $\thresh=42\rho_1 $, and $\sens=110\rho_1 /(\alpha n)$. 
We can set $c_2$ to be a large constant and will only change the constant factor in the sample complexity. 
The assumptions~\ref{asmp_vol}, \ref{asmp_local}, and \ref{asmp_resilience} follow from Lemmas~\ref{lem:vol_lr}, \ref{lem:local_asmp_lr}, and \ref{lem:approximation_lr}, respectively. 
The assumption~\ref{asmp_sens} follows from \begin{eqnarray*}
\sens = \frac{110 \rho_1 }{\alpha n} \leq  \frac{1.2\rho_1 \varepsilon}{32(c_2 d + (\varepsilon/2) + \log(16/(\delta\zeta)))}
= 
\frac{(c_0-3c_1)\rho \varepsilon }{32(c_2 p + (\varepsilon/2) + \log(16/(\delta\zeta)))}\;,
\end{eqnarray*}
for large enough $n \geq C' (d + \log(1/(\delta\zeta)))/(\alpha\varepsilon)$. 
This finishes the proof of Theorem~\ref{thm:robust_linear_regression} from which Theorem~\ref{thm:linear_regression} follows immediately.  

\subsection{Step 3: Achievability guarantees}
\label{sec:lr3} 

We provide utility guarantees for popular families of distributions studied in the private or robust linear regression literature: sub-Gaussian \cite{diakonikolas2019efficient,gao2020robust, zhu2019generalized,cai2019cost,wang2018revisiting} and hypercontractive \cite{zhu2019generalized,klivans2018efficient,cherapanamjeri2020optimal,jambulapati2021robust, bakshi2021robust, prasad2018robust}. Similar to mean estimation, the resilience we need scales with the variance. For sub-Gaussian distributions, this requires a lower bound on the variance of the form $\sigma\preceq c\Gamma$ for the sub-Gaussian proxy $\Gamma$. For the $k$-th moment bounded distributions, we require hypercontractivity.

%\Xiyang{Private linear regression:\cite{sheffet2019old, wang2018revisiting} can you separate them as per Gaussian or hypercontractive?} 

\subsubsection{Sub-Gaussian distributions}

The most common scenario in linear regression is when both the input $x_i$ and the noise $\eta_i$ are sub-Gaussian as we defined in  
Eq.~\eqref{eq:def_subgauss} and independent of each other.  
The next lemma shows that the resulting dataset is $(O(\alpha\log(1/\alpha)),O(\alpha\log(1/\alpha)),O(\alpha\sqrt{\log(1/\alpha)}),O(\alpha\log(1/\alpha)))$-resilient, which follows from the covariance resilience of sub-Gaussian distributions.

%\begin{lemma}[{\cite[Corollary 4]{jambulapati2020robust}}]\label{lem:subgaussian_cov_res}
%Let $\delta>0$ and $\alpha\in (0,0.5)$. A dataset $S=\{x_1, x_2, \cdots, x_n\}$ consists of $n$ i.i.d. samples from ${\cal D}$, where ${\cal D}$ is a zero-mean sub-Gaussian on $\reals^d$ with covariance $\Sigma$ and  sub-Gaussian proxy $\Gamma\preceq c \Sigma$ for some constant $c$. If $n=\Omega\left(\frac{d+\log(1/\delta)}{\left(\alpha\log(1/\alpha)\right)^2}\right)$, then there exists a universal constant $C_1>0$ such that with probability $1-\delta$, for any subset $T\subset S $ and $|T|=(1-\alpha)n$, we have for any unit norm vector $v\in \reals^d$ with $\|v\|=1$,
%\begin{eqnarray}
%    \left|v^\top\left(\frac{1}{|T|}\sum_{i\in T}x_ix_i^\top - \Sigma\right)v\right|\leq C_1\alpha\log(1/\alpha)v^\top\Sigma v\;.
%\end{eqnarray}
%\end{lemma}

\begin{lemma}[Resilience for sub-Gaussian samples]
\label{lem:res_subgaussian_lr}

Let ${\cal D}_1$ be a distribution of $x_i\in\reals^d$ which is zero mean sub-Gaussian with covariance $\Sigma$ and sub-Gaussian proxy $0\prec \Gamma\preceq c \Sigma$ for some constant $c$. Let ${\cal D}_2$ be a distribution of  $\eta_i\in\reals$ which is a zero mean one-dimensional sub-Gaussian with variance $\gamma^2$ and sub-Gaussian proxy $\gamma_0^2\leq c\gamma^2$ for some constant $c$.  A multiset of i.i.d.~labeled samples $S=\{(x_i, y_i )\}_{i=1}^n$ is generated from a linear model with noise  $\eta_i$ independent of $x_i$:
$y_i\;=\;x_i^\top\beta+\eta_i$\;, 
where the input $x_i$ and the independent noise $\eta_i$ are i.i.d. samples from ${\cal D}_1$ and ${\cal D}_2$. 
 There exist constants $c_1$ and $c_2>0$ such that, for any  $\alpha\in(0,1/2)$, if $n\geq c_1((d+\log(1/\zeta)) / (\alpha\log(1/\alpha))^2)$ then,  with  probability $1-\zeta$,  
     $S$ 
    is $(\alpha, c_2\alpha\log(1/\alpha),c_2\alpha\log(1/\alpha), c_2\alpha\sqrt{\log(1/\alpha)}, c_2\alpha\log(1/\alpha))$-resilient with respect to $(\beta,\Sigma,\gamma)$. 
\end{lemma}
\begin{proof}
This follows from 
 \cite[Corollary 4]{jambulapati2020robust}.  
Let $\tilde{x}_i :=\left[\begin{array}{c}
\Sigma^{-1/2}x_i  \\
\eta_i/\gamma
\end{array}\right] \in \reals^{d+1}$. By definition, we know $\tilde{x}_i$ can be seen as samples from a zero mean sub-Gaussian distribution with covariance $\mathbf{I}_{(d+1)\times (d+1)}$.
By  {\cite[Corollary~4]{jambulapati2020robust}} and union bound, we know if $n=\Omega(d+\log(1/\zeta))/(\alpha\log(1/\alpha))^2$ then there exists a constant $C_1$ such that with probability $1-\zeta$, for any $T\subset S$ and $|T|\geq (1-\alpha)n$ and any unit vector $u\in \reals^{d+1}$, $v\in \reals^d$, we have
\begin{eqnarray}
	&&\left|u^\top \left(\frac{1}{|T|}\sum_{x_i\in T}\tilde{x}_i\tilde{x}_i^\top-\mathbf{I}_{(d+1)\times (d+1)} \right)u\right|\leq C_1\alpha\log(1/\alpha) \;,\label{eq:res_subgaussian}\\
	&&\left|v^\top \left(\frac{1}{|T|}\sum_{x_i\in T}\Sigma^{-1/2} x_ix_i^\top\Sigma^{-1/2}-\mathbf{I}_{d\times d} \right)v\right|\leq C_1\alpha\log(1/\alpha) \;, \text{ and }\label{eq:subres1}\\
	 &&\left|\frac{1}{|T|}\sum_{\eta_i\in T}\frac{\eta_i^2}{\gamma^2}-1\right|\leq C_1\alpha\log(1/\alpha)\;.\label{eq:subres2}
\end{eqnarray}

Let $u:=\left[\begin{array}{c}
u_1  \\
u_2
\end{array}\right]$ where $u_1\in \reals^d$ and $u_2\in \reals$ and $\|u_1\|^2+u_2^2=1$. Then Eq.~\eqref{eq:res_subgaussian} is equivalent to 
\begin{eqnarray}
	&&\left|u_1^\top \left(\frac{1}{|T|}\sum_{i\in T}\Sigma^{-1/2}x_ix_i^\top\Sigma^{-1/2}-\mathbf{I}_{d \times d} \right)u_1+\frac{2u_2}{\gamma}u_1^\top \frac{1}{|T|}\sum_{i\in T}\Sigma^{-1/2}x_i\eta_i+\frac{u_2^2}{\gamma^2}\frac{1}{|T|}\sum_{i\in T}(\eta_i^2-\gamma^2)\right|\nonumber\\
	\leq &&C_1\alpha\log(1/\alpha)\;.\label{eq:res_subgaussian2}
\end{eqnarray}

By Eq.~\eqref{eq:subres1} and \eqref{eq:subres2}, we know
\begin{eqnarray*}
&&\left|u_1^\top(\frac{1}{|T|}\sum_{i\in T}\Sigma^{-1/2}x_ix_i^\top\Sigma^{-1/2}-\mathbf{I}_{d \times d})u_1\right|\leq C_1\alpha\log(1/\alpha)\|u_1\|^2\\
&&\left|\frac{u_2^2}{\gamma^2}\frac{1}{|T|}\sum_{i\in T}(\eta_i^2-\gamma^2)\right|\leq C_1\alpha\log(1/\alpha)u_2^2\;.
\end{eqnarray*}

This means
\begin{eqnarray}
-C_1\alpha\log(1/\alpha)(1+\|u_1\|^2+u_2^2)\leq	\frac{2u_2}{\gamma}u_1^\top \frac{1}{|T|}\sum_{i\in T}\Sigma^{-1/2}x_i\eta_i \leq C_1\alpha\log(1/\alpha)(1+\|u_1\|^2+u_2^2)\;.\label{eq:res_subgaussian3}
\end{eqnarray}

For any unit vector $w\in \reals^d$, let $u_1 = 0.5w$. Thus, we have $u_2^2=0.75$. Eq.~\eqref{eq:res_subgaussian3} implies 
\begin{eqnarray}
\left|\frac{1}{\gamma}	w^\top\frac{1}{|T|}\sum_{i\in T}\Sigma^{-1/2}x_i\eta_i\right| \leq C_2\alpha\log(1/\alpha)\;,
\end{eqnarray}
for some constant $C_2$. This proves the first resilience in Eq.~\eqref{def:res1}.
The second, third and fourth resilience properties in Eqs.~\eqref{def:res2}, \eqref{def:res3} and \eqref{def:res4} follow from {\cite[Lemma~4.1]{dong2019quantum}}, {\cite[Corollary~4]{jambulapati2020robust}} and a union bound.

\end{proof}
The above resilience lemma and  Theorem~\ref{thm:robust_linear_regression} imply the following optimal utility guarantee. 

\begin{coro} 
    \label{coro:linear_regression_subgaussian} 
    Under the hypothesis of Lemma~\ref{lem:res_subgaussian_lr},  
 there exists a constant $c>0$ such that  for any $\alpha\in(0,c)$, a sample size of 
     $$n = O\Big(\frac{d+\log(1/\zeta)}{(\alpha\log(1/\alpha))^2} + \frac{ d+\log(1/(\delta\zeta))}{\alpha\varepsilon} \Big)\;,$$ 
      a sensitivity of $\sens=O(\log(1/\alpha) / n)$, and a threshold of $\thresh=O( \alpha\log(1/\alpha))$ 
     with large enough constants are sufficient for  \HPTR$(S)$ 
     with the distance function in Eq.~\eqref{def:lr_dist}  
     to achieve 
     \begin{eqnarray} 
     \frac{1}{\gamma} \|\Sigma^{1/2} (\hat\beta-\beta) \| \; = \; O( \alpha\log(1/\alpha) )\;,
     \label{eq:lr_sub_bound}
     \end{eqnarray} 
     with  probability $1-\zeta$. 
     Further, the same guarantee holds even if $ \alpha$-fraction of the samples are arbitrarily corrupted as in Assumption~\ref{asmp:lr}. 
\end{coro}

The sample complexity is nearly optimal. Even for DP linear regression without robustness, HPTR is the first algorithm for sub-Gaussian distributions with an unknown covariance $\Sigma$ that up to log factors matches the lower bound of  $n=\tilde\Omega(d/\alpha^2+d/(\alpha\varepsilon))$ assuming $\varepsilon < 1$ and  $\delta<n^{-1-\omega}$ for some $\omega>0$ from  \cite[Theorem 4.1]{cai2019cost}.
For completeness, we provide the lower bound in Appendix~\ref{sec:lb}. 
%under the assumption that $\|x_i\|$ is bounded, $\eta_i$ is independent Gaussian noise and $\beta$ is bounded. 
An existing algorithm for DP linear regression from \cite{cai2019cost} is suboptimal as it require $\Sigma$ to be close to the identity matrix, which is equivalent to assuming that we know $\Sigma$.  
%which is equivalent to providing a weaker guarantee on the non-adaptive error on  $\|\hat\beta-\beta\|$.

The error bound is nearly optimal under $\alpha$-corruption, namely HPTR is the first robust estimator that is both differentially private and also achieves the near-optimal error rate of $ (1/\gamma) \|\Sigma^{1/2}(\hat\beta-\beta)\|=O( \alpha\log(1/\alpha))$, matching the known information-theoretic lower bound of 
$(1/\gamma) \|\Sigma^{1/2}(\hat\beta-\beta)\|=\Omega( \alpha)$ \cite{gao2020robust} up to a log factor. This lower bound  holds for any robust estimator that is not necessarily  private and regardless of how many samples are available.  If privacy is not required (i.e., $\varepsilon=\infty$), a similar guarantee can be achieved by, for example, \cite{diakonikolas2019efficient}.

%When we have dependent noise $\eta_i$ with sub-Gaussian $x_i$ and sub-Gaussian $\eta_i$, resulting dataset is   $(O(\alpha), O(\alpha))$-resilient. Up to a polylogarithmic factor in $1/\alpha$, a similar guarantee as Corollary~\ref{coro:linear_regression_subgaussian} holds in this dependent noise case also, and we omit a  detailed analysis. 

%For robust linear regression, our error bound matches \cite{gao2020robust} up to a logarithmic factor in $1/\alpha$, where they only consider Gaussian distributions and get the optimal error bounds $\|\Sigma^{1/2} (\hat\beta-\beta) \|=O(\alpha\gamma)$ using $O(d/\alpha^2)$ samples (with exponential time). \cite{diakonikolas2019efficient, cherapanamjeri2020optimal} achieve the same error bound and sample complexity as ours but run in polynomial time. 

\subsubsection{Hypercontractive distributions with independent noise} 

We assume $x_i$ and $\eta_i$ are independent and 
$(\kappa,k)$-hypercontractive and $(\tilde\kappa,k)$-hypercontractive, respectively, as in Definition~\ref{def:hyper}. 
For the necessity of hypercontractive conditions for 
robust linear regression, we refer to \cite[Section F.5]{zhu2019generalized}. 
The next lemma shows that  the the resulting dataset has a subset of size at least $(1-\alpha)n$ that is $(O(\alpha),O(\alpha^{1-1/k}),O(\alpha^{1-2/k}),O(\alpha^{1-1/k}), O(\alpha^{1-2/k}))$-resilient.

\begin{lemma}[Resilience for hypercontractive samples]\label{lem:hyper_contractive_res} For some integer $ k \geq 4 $ and positive scalar parameters $\kappa$ and $\tilde\kappa$, let ${\cal D}_1$ be a $(\kappa, k)$-hypercontractive distribution on $x_i\in\reals^d$ with zero mean and covariance $\Sigma \succ 0$. Let ${\cal D}_2$ be a $(\tilde\kappa,k)$-hypercontractive distribution on $\eta_i\in\reals$ with zero mean and variance $\gamma^2$. 
A multiset of labeled samples $S=\{(x_i, y_i)\}_{i=1}^n$ is generated from a linear model:  $y_i\;=\;x_i^\top\beta+\eta_i$, 
where the input $x_i$ and the independent noise $\eta_i$ are i.i.d. samples from ${\cal D}_1$ and ${\cal D}_2$. 
For any $\alpha\in(0,1/2)$ and any constant $c_3>0$,  
there exist constants $c_1$ and $c_2>0$ that only depend on $c_3$ such that if %$n=\widetilde\Omega\left(\frac{d\log(1/\alpha)}{\alpha^{2(1-1/k)}}\right)$ 
\begin{eqnarray}
	n\;\geq\; c_1\Big(\frac{d}{\zeta^{2(1-1/k)}\alpha^{2(1-1/k)}} + \frac{k^2\alpha^{2-2/k} (1+1/\tilde{\kappa}^2)d\log d}{\zeta^{2-4/k}\kappa^2} + \frac{\kappa^2(1+\tilde{\kappa}^2) d\log d}{\alpha^{2/k}}  \Big)\;,
	\label{eq:lr_kmoment_sample}
\end{eqnarray} 
then  $S$ 
    is $(c_3\alpha,\alpha,c_2k\kappa\tilde{\kappa} \alpha^{1-1/k}\zeta^{-1/k},c_2k^2\kappa^2\alpha^{1-2/k}\zeta^{-2/k}, c_2k\kappa\alpha^{1-1/k}\zeta^{-1/k}, c_2k^2\tilde{\kappa}^2\alpha^{1-2/k}\zeta^{-2/k} )$-corrupt good  with respect to $(\beta,\Sigma,\gamma)$ with probability $1-\zeta$.  %The notation $\tilde\Omega$ hides poly-logarithmic factors in $d$.
  \end{lemma}
% {\bf Sewoong: write what $\tilde\Omega$ is hiding. }
\begin{proof}

Since of $x_i$ and $\eta_i$ are independent, we know 
\begin{eqnarray*}
	\E\left[\left|\ip{v}{\gamma^{-1}\Sigma^{-1/2}x\eta}\right|^{k}\right]= \E\left[\left|\ip{v}{\Sigma^{-1/2}x}\right|^{k}\right]\E\left[|\gamma^{-1}\eta|^k\right]\leq \kappa^k\tilde{\kappa}^k\;.
\end{eqnarray*}
This implies $\gamma^{-1}\Sigma^{-1/2}x\eta$ is a $k$-th moment bounded distribution with covariance $\mathbf{I}_{d \times d} $. By Lemma~\ref{lem:mean_kmoment},  under the sample complexity of \eqref{eq:lr_kmoment_sample}, 
 with probability $1-8\zeta$, there exists a subset $S_{\rm good}\subset S$ such that $|S_{\rm good}|\geq (1-\alpha )n$ and there exists a constant $C$ such that for any subset $T\subset S_{\rm good}$ and $|T|\geq (1-10\alpha)|S_{\rm good}|$, we have
\begin{eqnarray}
	\left\|\frac{1}{|T|}\sum_{i\in T}\frac{1}{\gamma} \Sigma^{-1/2}x_i\eta_i\right\| &\leq & Ck\kappa \tilde{\kappa}\gamma\alpha^{1-1/k}\zeta^{-1/k}\;. 
\end{eqnarray}
This proves the first resilience in Eq.~\eqref{def:res1}. 
The second resilience in Eq.~\eqref{def:res2}, third resilience in Eq.~\eqref{def:res3} and fourth resilience in Eq.~\eqref{def:res4} follow directly from Lemma~\ref{lem:mean_kmoment}.
\end{proof}

The above resilience lemma and  Theorem~\ref{thm:robust_linear_regression} imply the following  utility guarantee. HPTR is naturally robust against $(1/5.5-c_3)\alpha$-corruption of the data. Choosing appropriate constants, we get the following result. 

\begin{coro} 
    \label{coro:linear_regression_hypercontractive} 
    Under the hypothesis of Lemma~\ref{lem:hyper_contractive_res}, 
 there exists a constant $c>0$ such that for any $\alpha\leq c$ and $k^2 \kappa^2 \alpha^{1-2/k} \leq c$, it is sufficient to have a dataset of size 
 \begin{eqnarray}
	n=O \Big(\frac{d}{\zeta^{2(1-1/k)}\alpha^{2(1-1/k)}} + \frac{k^2\alpha^{2-2/k} (1+1/\tilde{\kappa}^2)d\log d}{\zeta^{2-4/k}\kappa^2} + \frac{\kappa^2(1+\tilde{\kappa}^2) d\log d}{\alpha^{2/k}}  + 
	\frac{d+\log(1/\delta)}{\alpha\varepsilon} \Big)\;,
\end{eqnarray} 
     %$n = \widetilde{O}(d\log(1/\alpha)/\alpha^{2-2/k} + (d+\log(1/\delta))/(\alpha\varepsilon))$ 
    a sensitivity of $\sens=O( 1/ ( n\alpha^{1/k}))$, and a threshold of $\thresh=O( \alpha^{1-1/k})$ 
     with large enough constants  for  \HPTR$(S)$ 
     with the distance function in Eq.~\eqref{def:lr_dist} 
      to achieve $(1/\gamma) \|\Sigma^{1/2} (\hat\beta-\beta) \|=O(k\kappa\tilde\kappa  \alpha^{1-1/k}\zeta^{-1/k})$ with  probability $1-\zeta$. 
     %The notation $\widetilde O$ in hides poly-logarithmic factors in $ d$.
     Further, the same guarantee holds even if $ \alpha$-fraction of the samples are arbitrarily corrupted as in Assumption~\ref{asmp:lr}. 
\end{coro}

%to achieve the same error bound, \cite{cai2019cost} together with the lower bound in Section~\ref{sec:mean-kmoment} imply a lower bound $n=\tilde\Omega(d/\alpha^{2-2/k}+d/(\alpha\varepsilon))$. Our approach achieves optimal error bound and nearly optimal sample complexity at the same time.

The error bound is optimal under $\alpha$-corruption: namely the error bound 
$(1/\gamma) \|\Sigma^{1/2}(\hat\beta-\beta)\| = O(  \alpha^{1-1/k})$ matches the lower bound $(1/\gamma) \|\Sigma^{1/2} (\hat\beta-\beta) \|=\Omega( \alpha^{1-1/k })$ by \cite{bakshi2021robust} where %they assume $\eta_i\in [-\gamma_3, \gamma_3]$ and 
the noise $\eta_i$ is $(1,k)$-hypercontractive and independent of $x_i$, which is also $(1,k)$-hypercontractive. For completeness, we provide the lower bound in Appendix~\ref{sec:lb}.  
HPTR is the first algorithm that guarantees both differential privacy and optimal robust error bound of $O(\alpha^{1-1/k})$ for hypercontractive distributions. 
If only robust error bound under $\alpha$-corruption is concerned, 
 \cite{zhu2019generalized} also achieves the same optimal error bound, but does not provide differential privacy. Further, in this robust but not private case with $\varepsilon=\infty$, our sample complexity improves by a factor of $\alpha^{2/k}$ upon the state-of-the-art sample complexity of \cite[Theorem 3.3]{zhu2019generalized} which shows that $n=O(d/\alpha^2)$ is sufficient to achieve $(1/\gamma) \|\Sigma^{1/2}(\hat\beta-\beta)\|=O(\alpha^{1-1/k})$.

\medskip
\noindent{\bf Remark.} Suppose $k,\kappa,\tilde\kappa$, and $\zeta$ are $\Theta(1)$. 
HPTR  achieves $(1/\gamma)\|\Sigma^{1/2}(\hat\beta-\beta)\| =  O(\alpha^{1-1/k})$ with $n = \tilde O ( d/(\alpha^{2-2/k}) + (d+\log(1/\delta))/(\alpha \varepsilon)) $ samples, where $\tilde O$ hides logarithmic factors in $d$. 
The first term cannot be improved as it matches the first term of a lower bound of 
 $n=\tilde\Omega(d/\alpha^{2-2/k}+d/(\alpha^{1-1/k}\varepsilon))$  from  \cite[Theorem 4.1]{cai2019cost}, 
 which holds even for standard  non-robust sub-Gaussian (which is $(c_k,k)$-hypercontractive for any $k\in{\mathbb Z}_+$ and a constant $c_k$ that depends only on $k$) linear regression with independent noise (see Appendix~\ref{sec:lb} for a precise statement). 
%{\bf The first term cannot be improved as it matches a lower bound of \cite{shamir2015sample}, which holds even for standard non-private and non-robust Gaussian (which is $(c_k,k)$-hypercontractive for any $k\in{\mathbb Z}_+$ and a constant $c_k$ that depends only on $k$) linear regression with independent noise (see Appendix~\ref{sec:lb} for a precise statement). } 
However, we do not have a matching lower bound for the second term.
% Although we do not have a matching lower bound for the second term under a general ambient dimension $d$, we show  in the following theorem a nearly matching lower bound of $n=\Omega(1/(\alpha\varepsilon))$  when $d=O(1)$, thus we believe our upper bound is tight.
To the best of our knowledge, HPTR is the first algorithm for linear regression that guarantees $(\varepsilon,\delta)$-DP under hypercontractive distributions with independent noise. 

\subsubsection{Hypercontractive distributions with dependent noise}\label{sec:lr:dependent}
	
We assume $x_i$ and $\eta_i$ may be  dependent and marginally 
$(\kappa,k)$-hypercontractive and $(\tilde\kappa,k)$-hypercontractive, respectively, as defined in Definition~\ref{def:hyper}. 
In this case, the first resilience $\rho_1$ that determines the error rate increases from $O(\alpha^{1-1/k}) $ to $O(\alpha^{1-2/k}) $ as a result of the input and the noise being potentially correlated. 
The next lemma shows that  the the resulting dataset has a subset of size at least $(1-\alpha)n$ that is $(O(\alpha) ,O(\alpha^{1-2/k}),O(\alpha^{1-2/k}),O(\alpha^{1-1/k}), O(\alpha^{1-2/k}))$-resilient. 

\begin{lemma}[Resilience for hypercontractive samples with dependent noise]\label{lem:hyper_contractive_res_dependent} For some integer $k \geq 4  $ and positive scalar parameters $\kappa$ and $\tilde\kappa$, let ${\cal D}_1$ be a $(\kappa, k)$-hypercontractive distribution on $x_i\in\reals^d$ with zero mean and covariance $\Sigma \succ 0$. Let ${\cal D}_2$ be a $(\tilde\kappa,k)$-hypercontractive distribution on $\eta_i\in\reals$ with variance $\gamma^2$. 
A multiset of labeled samples $S=\{(x_i, y_i)\}_{i=1}^n$ is generated from a linear model:
$y_i = x_i^\top\beta+\eta_i$,
where  $\{(x_i,\eta_i)\}_{i\in[n]}$ are i.i.d. samples from some distribution ${\cal D}$ whose marginal distribution for $x_i$ is  ${\cal D}_1$, the marginal distribution for $\eta_i$ is ${\cal D}_2$, and $\E[x_i\eta_i]=0$. 
For any  $\alpha\in(0,1/2)$ and $c_3>0$, there exist constants $c_1$ and $c_2>0$ that only depend on $c_3$ such that if 
\begin{eqnarray}
	n \;\geq\; c_1\Big(
	\frac{d}{\zeta^{2(1-1/k)}\alpha^{2(1-1/k)}} + 
	\frac{k^2\alpha^{2-4/k}(1+1/\tilde{\kappa}^2) d\log d}{\zeta^{2-4/k}\kappa^2\tilde{\kappa}^2} + \frac{\kappa^2(\tilde{\kappa}^2+1) d\log d}{\alpha^{4/k}}
	\Big)\;,
	\label{eq:lr_kmoment_sample_dependent}
\end{eqnarray} 
then $S$ 
    is $(c_3\alpha, \alpha,c_2k\kappa\tilde{\kappa} \alpha^{1-2/k}\zeta^{-2/k},c_2 k^2\kappa^2\alpha^{1-2/k}\zeta^{-2/k}, c_2k\kappa\alpha^{1-1/k}\zeta^{-1/k}, c_2k^2\tilde{\kappa}^2\alpha^{1-2/k}\zeta^{-2/k} )$-corrupt good  with respect to $(\beta,\Sigma,\gamma)$ with probability $1-\zeta$. % The notation $\tilde\Omega$ hides poly-logarithmic factors in $d$.
  \end{lemma}

\begin{proof}
	Since $\eta_i$ and $x_i$ are dependent, we can only bound $k/2$-th moment of $\gamma^{-1}\Sigma^{-1/2}x\eta$. By Holder inequality, we have
\begin{eqnarray*}
	\E\left[\left|\ip{v}{\Sigma^{-1/2}\gamma^{-1}x\eta}\right|^{k/2}\right]\leq \sqrt{ \E\left[\left|\ip{v}{\Sigma^{-1/2}x}\right|^{k}\right]\E\left[|\gamma^{-1}\eta|^k\right]}\leq \kappa^{k/2}\tilde{\kappa}^{k/2}\;.
\end{eqnarray*}
The rest of the proof follows similarly as the proof of Lemma~\ref{lem:hyper_contractive_res}. 

\end{proof}

The above resilience lemma and  Theorem~\ref{thm:robust_linear_regression} imply the following optimal utility guarantee achieving an error rate of $O(\alpha^{1-2/k})$. %Choosing $c_3=1/11$, we get the following result. 

\begin{coro}  
    \label{coro:linear_regression_hypercontractive_dependent_noise} 
    Under the hypothesis of Lemma~\ref{lem:hyper_contractive_res_dependent},
 there exists a constant $c>0$ such that for any $\alpha\leq c$ and $k^2 \kappa^2 \alpha^{1-2/k} \leq c$, it is sufficient to have a dataset of size 
 \begin{eqnarray*}
	n = O \Big(
	\frac{d+\log(1/\delta)}{\alpha\varepsilon}+ \frac{d}{\zeta^{2(1-1/k)}\alpha^{2(1-1/k)}} + 
	\frac{k^2\alpha^{2-4/k}(1+1/\tilde{\kappa}^2) d\log d}{\zeta^{2-4/k}\kappa^2\tilde{\kappa}^2} + \frac{\kappa^2(\tilde{\kappa}^2+1) d\log d}{\alpha^{4/k}} \Big)\;,
\end{eqnarray*} 
a sensitivity $\sens=O(1/(n\alpha^{2/k}))$, and a threshold $\thresh=O(\alpha^{1-2/k})$, 
with  large enough constants  
     for \HPTR$(S)$ with the distance function in Eq.~\eqref{def:lr_dist} to achieve $(1/\gamma) \|\Sigma^{1/2} (\hat\beta-\beta) \|=O(k\kappa\tilde\kappa  \alpha^{1-2/k}\zeta^{-2/k})$ with  probability $1-\zeta$. 
     %The notation $\widetilde O$ in hides poly-logarithmic factors in $ d$.
     Further, the same guarantee holds even if $\alpha$-fraction of the samples are arbitrarily corrupted as in Assumption~\ref{asmp:lr}. 
\end{coro}

This error rate is optimal in its dependence in $\alpha$ under $\alpha$-corruption. 
When $\eta_i$ and $x_i$ are dependent, \cite{bakshi2021robust} gives a lower bound of error rate $(1/\gamma) \|\Sigma^{1/2} (\hat\beta-\beta) \|=\Omega(\tilde{\kappa}\alpha^{1-2/k})$ that   holds regardless of how many samples we have and without the privacy constraints. 
For completeness, we provide the lower bound in Appendix~\ref{sec:lb}. 
If only robust error bound under $\alpha$-corruption is concerned, 
 \cite{zhu2019generalized} also achieves the same optimal error bound, but does not provide differential privacy. Further, in this robust but not private case with $\varepsilon=\infty$, our sample complexity improves by a factor of $\alpha^{2/k}$ upon the state-of-the-art sample complexity of \cite[Theorem 3.3]{zhu2019generalized} which shows that $n=O(d/\alpha^2)$ is sufficient to achieve $(1/\gamma) \|\Sigma^{1/2}(\hat\beta-\beta)\|=O(\alpha^{1-2/k})$. 

\medskip\noindent{\bf Remark.} 
Suppose $\zeta,\kappa,\tilde\kappa$, and $k$ are $\Theta(1)$.  The sample complexity of HPTR is 
$n=\tilde O((d+\log(1/\delta))/\alpha^{2(1-1/k)}+d/(\alpha\varepsilon))$. 
The first term has a gap of $\alpha^{-2/k}$ factor compared to the first term of a lower bound  of
 $n=\tilde\Omega(d/\alpha^{2(1-2/k)}+d/(\alpha^{1-2/k}\varepsilon))$  from  \cite[Theorem 4.1]{cai2019cost}, 
 which holds even for standard  non-robust sub-Gaussian DP linear regression. It remains an open question whether this gap can be closed, either by a tighter analysis of the resilience for HPTR or a tighter analysis for a lower bound.  
 
 On the upper bound, the gap comes from the fact that we are ensuring stronger resilience than we need. From Theorem~\ref{thm:linear_regression}, we know that we require $\rho_1\leq c$ and $\rho_3^2\leq c\alpha$, and from the optimal error rate, we want  $\rho_1\leq c\alpha^{1-2/k}$. The resilience we ensure in Lemma~\ref{lem:hyper_contractive_res_dependent} is $(\alpha,\rho_1=\alpha^{1-2/k},\rho_2=\alpha^{1-2/k},\rho_3=\alpha^{1-1/k})$ which is guaranteeing unnecessarily small $\rho_2$ and $\rho_3$. A similar slack was also there in mean estimation, which did not affect the final sample complexity. In this case with linear regression and hypercontractive distributions, it causes  sample complexity to be larger. Tighter analysis of the resilience which guarantees larger  $\rho_2$ and $\rho_3$ can improve the  the first term in the sample complexity in its dependence on $\alpha$, but cannot close the $\alpha^{-2/k}$ gap. On the lower bound, we are using a construction of \cite[Theorem 4.1]{cai2019cost}, which uses Gaussian distributions and an independent noise. One could potentially tighten the lower bound with a construction that uses hypercontractive distributions and a dependent noise. 
 
For the second term,  we provide a nearly matching lower bound of $n=\Omega(\min\{d,\log(1/\delta)\}/\alpha \varepsilon)$ to achieve $(1/\gamma) \|\Sigma^{1/2}(\hat\beta-\beta)\|^2\leq O(\alpha^{2-4/k})$ in  Proposition~\ref{thm:lowerbound_regression_dependent} proving that it is tight when $\delta=\exp(-\Theta(d))$.  To the best of our knowledge, HPTR is the first algorithm for linear regression that guarantees $(\varepsilon,\delta)$-DP under hypercontractive distributions with dependent noise.

% ----------------------------------

\begin{propo}[Lower bound of hypercontractive linear regression with dependent noise]
\label{thm:lowerbound_regression_dependent}
For any $k\geq 4$, 
	let $\cP_{\kappa,k,\Sigma,\gamma^2}$ be a distribution over $(x_i,\eta_i)\in\reals^{d}\times \reals$ where $x_i$ is $(\kappa, k)$-hypercontractive with zero mean and covariance $\Sigma$, and  $\eta_i$ is $(\kappa,k)$-hypercontractive with zero mean and variance $\gamma^2$. 
	We observe labelled examples a linear model 
	$y_i= x_i^\top \beta + \eta_i$ with  $\E[x_i\eta_i]=0$ such that 
	%For any $P\in\cP_{\Sigma, k}$, define optimal parameter
	$\beta = \Sigma^{-1}\E[y_ix_i] %\argmin_{\beta}\E_{(x,y)\sim P}[(y-\beta^\top x)^2]
	$\;. 
	Let $\cM_{\varepsilon,\delta}$ denote a class of $(\varepsilon, \delta)$-DP estimators that are  measurable functions over $n$ i.i.d.~samples $S= \{(x_i,y_i) \}_{i=1}^n$ from a distribution.   There exist  positive constants $c,\gamma,\kappa=O(1)$ such that, for $\varepsilon\in(0,10)$,
	 \begin{eqnarray*}
 	\inf_{\hat\beta\in \cM_{\varepsilon, \delta}}
 	\sup_{\Sigma\succ 0, P\in \cP_{\kappa,k,\Sigma,\gamma^2}} \frac{1}{\gamma} \E_{  P^{n}}[\|\Sigma^{1/2}(\hat\beta(S)-\beta )\|^2]
 	\;\geq\;
 	c \,\min \left\{\left(\frac{d\wedge \log((1-e^{-\varepsilon})/\delta)}{n\varepsilon}\right)^{2-4/k}, 1\right\}\;.
 \end{eqnarray*}
\end{propo}

\begin{proof}
We adopt the same framework as the proof of Proposition~\ref{thm:lowerbound_mean_hypercontractive}. We choose $\cP$ to be $\cP=\cP_{\Sigma, k}$. It suffices to construct index set $\cV$ and indexed family of distributions $\cP_{\cV}$ such that $d_{\rm TV}(P_v, P_{v'})= \alpha$ and $\rho(\beta_v, \beta_{v'})\geq t$ where $\beta_v$ is the least square solution of $P_v$. 
By {\cite[Lemma 6]{acharya2021differentially}}, there exists a finite set $\cV\subset \reals^d$ with cardinality $|\cV|=2^{\Omega(d)}$, $\|v\|=1$ for all $v\in \cV$, and $\|v-v'\|\geq 1/2$ for all $v\neq v'\in \cV$. Let $f_{\mu, \Sigma}(x)$ be density function of $\cN(\mu, \Sigma)$. 
We construct a marginal distribution over $\reals^d$ as follows,
\begin{eqnarray}
	D_x^v(x)=\left\{\begin{array}{ll}
	\alpha / 2, & \text { if } x =- \alpha^{-1/k}v \;,\\ 
	\alpha /2, & \text { if } x =\alpha^{-1/k}v \;,\\ 
	 (1-\alpha)f_{0, \mathbf{I}_{d \times d}}(x) & \text { otherwise\;, }\end{array}\right. \;.\label{eq:construct_hypercontractive_x}
\end{eqnarray}

It is easy to verify that $\E_{P_x^v}[x]=0$,  $\E_{P_x^v}[xx^\top]=(1-\alpha)\mathbf{I}_{d \times d}+\alpha^{1-2/k}vv^\top$ and thus $\frac12\mathbf{I}_{d \times d}\preceq\E_{P_x^v}[xx^\top]\preceq 2\mathbf{I}_{d \times d}$ for $\alpha\leq1/2$.  
Furthermore, we have
\begin{eqnarray*}
	\E_{x\sim P_x^v}[\,|\ip{u}{x}|^k\, ] &\leq & \ip{u}{v}^k +(1-\alpha) c_k^k =O(1)\;,
\end{eqnarray*}
where we use the fact that there exists a constant $c_k>0$ such that the $k$-th moment of Gaussian distribution is bounded by $c_k^k$. Since $\frac12\mathbf{I}_{d \times d}\preceq\E_{P_x^v}[xx^\top]\preceq 2\mathbf{I}_{d \times d}$, we know $x$ is $(O(1),k)$-hypercontractive. 
We construct conditional distribution $D^v(y|x)$ as follows
\begin{eqnarray*}
	y|x=\left\{\begin{array}{ll}
	-\alpha^{-1/k} & \text { if } x =- \alpha^{-1/k}v \\ 
	\alpha^{-1/k} & \text { if } x =  \alpha^{-1/k}v \\
\cN(0, 1) & \text { otherwise }\end{array}\right. \;.
\end{eqnarray*}

Then we have
\begin{eqnarray*}
	\beta_v &=& \E_{x\sim P_x^v}[xx^\top]^{-1}\E_{x, y\sim P_{x, y}^v}[xy]\\
	%&=&\E_{x\sim P_x^v}[xx^\top]^{-1} \int_{\reals^d\times \reals}xy D_{x,y}^v(x, y)dxdy\\
	&=&\E_{x\sim P_x^v}[xx^\top]^{-1} \alpha^{1-2/k}v\;.
\end{eqnarray*}

This implies $t = \min_{v\neq v'\in \cV}\|\beta_v-\beta_{v'}\|\geq 1/2\alpha^{1-2/k}\min_{v\neq v'\in \cV}\|v-v'\|=\Omega(\alpha^{1-2/k})$.
We are left to verify that $\eta = y-\ip{\beta_v}{x}$ is also hypercontractive: 
\begin{eqnarray*}
\E[|\eta|^k] &=& \alpha\,\big|\alpha^{-1/k}-v^\top \E_{x\sim P_x^v}[xx^\top]^{-1} v  \alpha^{1-3/k} \,\big|^k + (1-\alpha)\, \E_{x\sim \cN(0, 2\mathbf{I}_{d \times d})}[|x|^k]=O(1)\;,
\end{eqnarray*}
where we used the fact that $k$-th moment of standard Gaussian is bounded by some constants $C_k>0$ and $k=O(1)$. 
It is easy to see that total variation distance $d_{\rm TV}(P_{x,y}^v, P_{x,y}^{v'})
	=\alpha$.

Next, we apply the similar reduction of estimation to testing with this packing $\cV$ as in the proof of Proposition~\ref{thm:lowerbound_mean_hypercontractive}. For $(\varepsilon, \delta)$-DP estimator $\hat\beta$, using Theorem~\ref{thm:packing}, we have
\begin{eqnarray*}
	&&\sup_{P\in \cP}\E_{ P^n}[\|\Sigma(P)^{1/2}(\hat{\beta}(S)-\beta(P))\|^2]\\
		&\geq & \frac{1}{|\cV|}\sum_{v\in \cV}\E_{ P_v^n}[\|\Sigma(P_v)^{1/2}(\hat{\beta}(S)-\beta(P_v))\|^2]\\
		&= &t^2\frac{1}{|\cV|}\sum_{v\in \cV}P_{v}\left(\|\Sigma(P_v)^{1/2}(\hat{\beta}(S)-\beta(P_v))\|\geq t\right)\\
		& \asymp &t^2\frac{1}{|\cV|}\sum_{v\in \cV}P_{v}\left(\|\hat{\beta}(S)-\beta(P_v)\|\geq t\right)\\
		&\gtrsim &t^2 \frac{e^{d/2} \cdot\left(\frac{1}{2} e^{-\varepsilon\lceil n \alpha\rceil}-\frac{\delta}{1-e^{-\varepsilon}}\right)}{1+e^{d/2} e^{-\varepsilon\lceil n \alpha\rceil}}\;,
\end{eqnarray*}
where $\beta(P)$ is the least squares solution of the distribution $P$, $\Sigma(P)$ is the covariance of $x$ from $P$, and  the last inequality follows from the fact that $d\geq 2$. 
The rest of the proof follows from  {\cite[Proposition 4]{barber2014privacy}}. We choose 
\begin{eqnarray*}
	\alpha =\frac{1}{n \varepsilon} \min \left\{\frac{d}{2}-\varepsilon, \log \left(\frac{1-e^{-\varepsilon}}{4 \delta e^{\varepsilon}}\right)\right\}
\end{eqnarray*}
and $t=\Omega(\alpha^{1-2/k})$ for  $\varepsilon\in(0,10)$, so that 
 \begin{eqnarray*}
 	\sup_{P\in \cP}\E_{ P^n}[\|\Sigma(P)(\hat{\beta}(S)-\beta(P))\|^2]\gtrsim  \alpha^{2-4/k} \;.
 \end{eqnarray*}
 This means that for all $k\geq 4$ there exist some  $\kappa,\gamma=O(1)$ such that 
 \begin{eqnarray*}
 	\inf_{\hat{\beta}\in \cM_{\varepsilon,\delta}}\sup_{\Sigma\succ 0, P\in \cP_{\kappa,k,\Sigma,\gamma^2}}\E_{ P^n}[\|\Sigma^{1/2}(\hat{\beta}(S)-\beta(P))\|^2]\gtrsim\min \left\{\left(\frac{d\wedge \log(1-e^{-\varepsilon}/\delta)}{n\varepsilon}\right)^{2-4/k}, 1\right\}\;,
 \end{eqnarray*}
 which completes the proof by noting that $\gamma=\Theta(1)$.

\end{proof}

\section{Covariance estimation}
\label{sec:cov}

In a standard covariance estimation, we are given i.i.d.~samples $S=\{x_i\in\reals^d\}_{i\in[n]}$ drawn from a distribution $P_{\Sigma,\Psi}$ with zero mean, an unknown covariance matrix $0\prec \Sigma\in\reals^{d\times d}$, and an unknown 
  positive semidefinite  matrix 
$\Psi :=  \E[(x_i\otimes x_i- \Sigma^\flat)(x_i\otimes x_i-\Sigma^\flat )^\top]\in\reals^{d^2\times d^2}$, where $\otimes$ denotes the Kronecker product. 
The fourth moment matrix $\Psi$ will be treated as a linear operator on a subspace  ${\cal S}_{\rm sym} \subset \reals^{d^2}$ defines as ${\cal S}_{\rm sym}:=\{M^\flat \in \reals^{d^2}: M \text{ is symmetric}\}$ following the definitions and notations from \cite{diakonikolas2018robustly}.

\begin{definition}
    \label{def:flat}
    For any matrix $M\in\reals^{d\times d}$, let $M^\flat\in\reals^{d^2}$ denote its canonical flattening into a vector in $\reals^{d^2}$, and for any vector $v\in\reals^{d^2}$, let $v^\sharp$ denote the unique matrix $M\in\reals^{d\times d}$ such that $M^\flat=v$. 
\end{definition}

This definition of $\Psi$ as an operator on ${\cal S}_{\rm sym}$ is without loss of generality, as in this section we only apply $\Psi$ to flattened symmetric matrices, and also significantly lightens the notations, for example for Gaussian distributions. All $d^2\times d^2$ matrices in this section will be considered as linear operators on ${\cal S}_{\rm sym}$, and we restrict our support of the exponential mechanism in {\sc Release} to be the set of positive definite matrices: $\{\hat\Sigma\in\reals^{d\times d}:\hat\Sigma\succ 0\}$.

\begin{lemma}[{\cite[Theorem~4.12]{diakonikolas2018robustly}}]
\label{lem:Isserlis}
    If $P_{\Sigma,\Psi} ={\cal N}(0,\Sigma)$ then 
    $\E[x_i\otimes x_i] = \Sigma^\flat$, and 
    as a matrix in $\reals^{d^2\times d^2}$, we have  $\Psi_{n(i-1)+j,n(k-1)+\ell}=\Sigma_{i,k}\Sigma_{j,\ell} + \Sigma_{i,\ell}\Sigma_{j,k}$ for all  $(i,j,k,\ell)\in[d]^4$, and as an operator on ${\cal S}_{\rm sym}$, we can equivalently write it as  $\Psi=2(\Sigma\otimes \Sigma)$. 
\end{lemma}

Further, we can assume an invertible operator $\Psi$ and define the 
 Mahalanobis distance for $x_i\otimes x_i$, which is  
 $D_\Psi(\hat\Sigma,\Sigma) = \| \Psi^{-1/2}(\hat\Sigma^\flat- \Sigma^\flat )\|$.  
 For Gaussian distributions, for example, we have 
 $D_\Psi(\hat\Sigma,\Sigma)= (1/\sqrt{2}) \|\Sigma^{-1/2} \hat\Sigma \Sigma^{-1/2}-{\bf I}_{d\times d}\|_F $, where $\|\cdot\|_F$ denotes the Frobenius norm of a matrix. 
 This is a natural choice of a distance because the total variation distance between two Gaussian distributions is $d_{\rm TV}({\cal N}(0,\Sigma),{\cal N}(0,\Sigma')) = O(\|\Sigma^{-1/2}\hat\Sigma\Sigma^{-1/2}-{\bf I}_{d\times d}\|_F)$ (see for example \cite[Lemma 2.9]{KLSU19}). 
We want a DP estimate of the covariance $\Sigma$ with a small Mahalanobis distance $D_\Psi(\hat\Sigma,\Sigma)$. 
If the sample generating distribution is not zero-mean, we can either apply a robust mean estimation with a subset of samples to estimate the mean or estimate the covariance using zero mean samples of the form $\{x_i-x_{i+ \lceil n/2\rceil }\}_{i\in[n/2]}$.

% -----------------------------------
\subsection{Step 1: Designing the surrogate \texorpdfstring{$\robdist_S(\hat\Sigma)$}{} for the Mahalanobis distance} 

To sample only positive definite matrices, we restrict the domain of out score function to be  $D_\Sigma:\{\hat\Sigma\in\reals^{d\times d}: \hat\Sigma \succ 0\} \to \reals_+$, and assume $D_\Sigma(\hat\Sigma)=\infty$ for non positive definite $\hat\Sigma$: 
\begin{eqnarray}
\robdist_S(\hat{\Sigma}) = \max_{V\in\reals^{d \times d}: V^\top=V, \|V\|_F=1}\frac{\langle{V},{\hat{\Sigma}\rangle}-\Sigma_V({\cal M}_{V,\alpha})}{\psi_V({\cal M}_{V,\alpha}) }\;, \label{eq:cov_score}
\end{eqnarray}
where we define the set ${\cal M}_{V, \alpha}$ similarly as in  Section~\ref{sec:mean1}.  We consider a projected dataset  $\{\langle{V},{x_ix_i^\top}\rangle\}_{i\in S}$ and partition $S$ into three sets ${\cal B}_{V, \alpha}$, ${\cal M}_{V, \alpha}$ and ${\cal T}_{V, \alpha}$, where 
${\cal B}_{V,\alpha}$ corresponds to the subset of  $(2/5.5)\alpha n$ data points with smallest values in $\{\langle{V},{x_ix_i^\top}\rangle\}_{i\in S}$, ${\cal T}_{V,\alpha}$ is the subset of top  $(2/5.5)\alpha n$ data points with largest  values, and ${\cal M}_{V,\alpha}$ is the subset of remaining  $1-(4/5.5)\alpha n$ data points.
For a fixed symmetric  matrix $V\in\reals^{d\times d}$ with $\|V\|_F=1$, we define 
$\Sigma_V({\cal M}_{V, \alpha})=\frac{1}{|{\cal M}_{V, \alpha}|}\sum_{x_i\in {\cal M}_{V, \alpha}} \ip{V}{x_ix_i^\top}$,  and $\psi_V({\cal M}_{V, \alpha})^2 = \frac{1}{|{\cal M}_{V, \alpha}|}\sum_{x_i\in {\cal M}_{V, \alpha}}\left(\ip{V}{x_ix_i^\top}-\Sigma_V({\cal M}_{V, \alpha})\right)^2$, which are robust estimates of the population projected covariance $\Sigma_V=\ip{V}{\Sigma}$ and projected fourth moment $\psi_V^2=(V^\flat)^\top\Psi V^\flat $. 
Next, we show that this score function $\robdist_S(\hat{\Sigma})$ recovers our target error metric  $D_\Psi(\hat\Sigma,\Sigma) = \| \Psi^{-1/2}( \hat\Sigma^\flat - \Sigma^\flat )\|$ when we substitute $\Sigma_V({\cal M}_{V,\alpha})$ and  $\psi_V({\cal M}_{V,\alpha})$ with population statistics $\Sigma_V$ and $\psi_V$, respectively. This justifies the choice of $\robdist_S(\hat\Sigma)$ as discussed in Section~\ref{sec:mean1}.

\begin{lemma}
\label{lem:true_covariance}
For any $0\prec \Sigma \in\reals^{d\times d}$, $0\prec \hat\Sigma$ and any invertible linear operator $  \Psi \in\reals^{d^2\times d^2}$ on ${\cal S}_{\rm sym}$,  we have
	\begin{eqnarray}
		\max_{V\in\reals^{d\times d}:V^\top=V,\|V\|_F=1}\frac{\langle{V},{\hat{\Sigma}\rangle}-\Sigma_V}{\psi_V} = \left\| \Psi^{-1/2}( \hat\Sigma^\flat - \Sigma^\flat )\right\|\;,
	\end{eqnarray}
	where $\Sigma_V = \langle V,\Sigma\rangle$ and $\psi_V^2 = (V^\flat)^\top \Psi V^\flat  $. 
\end{lemma} 
This  follows immediately from Lemma~\ref{lem:equidist}.

%\begin{lemma}
%	\begin{eqnarray}
%		\max_{V:\|V\|_F=1}\frac{\ip{V}{\widehat{\Sigma}-\Sigma}}{\sqrt{\Tr(V^\top\Sigma V \Sigma)}} = \left\|\Sigma^{-1/2}\widehat{\Sigma}\Sigma^{-1/2}-\mathbf{I}_{d \times d}\right\|_F
%	\end{eqnarray}
%\end{lemma}
%\begin{proof}
%\begin{eqnarray}
%	\max_{V:\|V\|_F=1}\frac{\ip{V}{\widehat{\Sigma}-\Sigma}}{\sqrt{\Tr(V^\top\Sigma V \Sigma)}} &=& \max_{V:\|{\rm vec}(V)\|_2=1} \frac{{\rm vec}(V)^\top {\rm vec}(\widehat{\Sigma}-\Sigma)}{\sqrt{{\rm vec}(V)^\top (\Sigma\otimes \Sigma) {\rm vec}(V)}}\\
%	&=& \left\|(\Sigma\otimes \Sigma)^{-1/2}{\rm vec}(\widehat{\Sigma}-\Sigma)\right\|\\
%	&=& \left\|{\rm vec}(\Sigma^{-1/2}(\widehat{\Sigma}-\Sigma)\Sigma^{-1/2})\right\|\\
%	&=& \left\|{\rm vec}(\Sigma^{-1/2}\widehat{\Sigma}\Sigma^{-1/2}-\mathbf{I}_{d \times d})\right\|\\
%	&=& \left\|\Sigma^{-1/2}\widehat{\Sigma}\Sigma^{-1/2}-\mathbf{I}_{d \times d}\right\|_F\;.
%\end{eqnarray}
%	
%\end{proof}

% -----------------------------------
\subsection{Step 2: Utility analysis under resilience}
The following resilience property of the dataset is critical in selecting $\sens$ and $\thresh$, and analyzing utility.  

\begin{definition}[Resilience]
    \label{def:resilience_covariance}
    For some $\alpha\in(0,1)$, $\rho_1\in\reals_+$, and $\rho_2\in \reals_+$, we say a set of $n$ data points $S_{\rm good}$ is  $(\alpha,\rho_1,\rho_2)$-resilient with respect to $(\Sigma,\Psi)$ if for any $T\subset S_{\rm good}$ of size $|T|\geq (1-\alpha)n$, the following holds 
    %for some universal constant $c>0$ and 
    for all symmetric matrix $V\in\reals^{d\times d}$ with $\|V\|_F=1$:     
    \begin{eqnarray}
         \Big| \frac{1}{|T|}\sum_{x_i\in T} \ip{V}{x_ix_i^\top} -\ip{V}{\Sigma}\Big|  & \leq &  \rho_1 \, \psi_V    \;,\text{ and } \label{def:res_cov_1}\\
      \left| \frac{1}{|T|}\sum_{x_i\in T}  \big(\ip{V}{x_ix_i^\top}-\ip{V}{\Sigma} \big)^2  -   \psi_V^2\right| & \leq &  \rho_2 \, \psi_V\;. \label{def:res_cov_2}
    \end{eqnarray}
\end{definition}
Note that covariance estimation for $\{x_i\}$ is equivalent to  mean estimation for $\{x_i\otimes x_i\}$. We can immediately apply the mean estimation utility guarantee in  Theorem~\ref{thm:mean} to show that  $\|\Psi^{-1/2}  (\hat\Sigma^\flat-\Sigma^\flat)\|=O(\rho_1)$ can be achieved with $n=O(d^2/\varepsilon\alpha)$ samples.

\begin{coro}[Corollary of Theorem~\ref{thm:mean}]
    \label{coro:cov}  There exist positive constants $c$ and $C>0$ such that for any  $(\alpha,\rho_1,\rho_2)$-resilient dataset $S$ with respect to $(\Sigma,\Psi)$ satisfying $\alpha < c$, $\rho_1<c$ and $\rho_2<c$, and $\rho_1^2\leq c \alpha$, $\HPTR$ with the distance function in Eq.~\eqref{eq:cov_score},  $\sens=110 \rho_1 /(\alpha n)$, and $\thresh=42\rho_1 $  achieves 
    $\|\Psi^{-1/2} (\hat\Sigma^\flat - \Sigma^\flat ) \|\leq 32 \rho_1 $ with probability $1-\zeta$, if \begin{eqnarray}
         n \;\geq \; C \frac{d^2 + \log(1/(\delta\zeta))}{\varepsilon \alpha} \; .
    \end{eqnarray}
\end{coro}
Under Assumption~\ref{asmp:mean} on $\alpha_{\rm corrupt}$-corruption and Definition~\ref{def:corruptgoodset} on corrupt good sets extended to $\{x_i\otimes x_i\}_{i=1}^n$, it follows from Theorem~\ref{thm:robust_mean} that the same guarantee holds under an adversarial corruption.   

\begin{coro}[Corollary of Theorem~\ref{thm:robust_mean}]
    \label{coro:robust_cov}  There exist positive constants $c$ and $C>0$ such that for any  $((1/11)\alpha,\alpha,\rho_1,\rho_2)$-corrupt good set $S$ with respect to $(\Sigma,\Psi)$ satisfying $\alpha < c$, $\rho_1<c$ and $\rho_2<c$, and $\rho_1^2\leq c \alpha$, $\HPTR$ with the distance function in Eq.~\eqref{eq:cov_score},  $\sens=110 \rho_1 /(\alpha n)$, and $\thresh=42\rho_1 $  achieves 
    $\|\Psi^{-1/2} (\hat\Sigma^\flat - \Sigma^\flat ) \|\leq 32 \rho_1 $ with probability $1-\zeta$, if \begin{eqnarray}
         n \;\geq \; C \frac{d^2 + \log(1/(\delta\zeta))}{\varepsilon \alpha} \; .
    \end{eqnarray}
\end{coro}

% -----------------------------------
\subsection{Step 3: Near-optimal guarantees}

Covariance estimation  has been studied for  Gaussian distributions under differential privacy \cite{KV17,KLSU19,aden2020sample} and robust estimation under $\alpha$-corruption \cite{li2020robust, diakonikolas2019robust,chen2018robust, rousseeuw1985multivariate,zhu2019generalized}. 
Note that from Lemma~\ref{lem:Isserlis}, we know that $\Psi=2(\Sigma\otimes \Sigma)$ and the Mahalanobis distance simplifies to  $\robdist_\Psi(\hat\Sigma,\Sigma) = \|\Sigma^{1/2}\hat\Sigma\Sigma^{-1/2}-{\bf I}_{d\times d}\|_F $ for Gaussian distributions.

\subsubsection{Gaussian distributions} 

For Gaussian distributions, the second moment resilience in Eq.~\eqref{def:res_cov_1} is satisfied with $\rho_1=O(\alpha\log(1/\alpha))$ and the 4th moment resilience in Eq.~\eqref{def:res_cov_2} is satisfied with $\rho_2=O(\alpha\log^2(1/\alpha))$.

\begin{lemma}[Resilience for Gaussian]\label{lem:gaussian_cov_res}  Consider a dataset $S=\{x_i\in \reals^d\}_{i=1}^n$ of $n$ i.i.d. samples from $\cN(0,\Sigma)$. If $n=\Omega\left((d^2+\log(1/\zeta))/(\alpha^2\log(1/\alpha)) \right)$ with a large enough constant, then there exists a constant $C>0$ such that $S$ 
    is $(\alpha,C\alpha\log(1/\alpha),C\alpha\log^2(1/\alpha) )$-corrupt good  with respect to $(\Sigma, \Psi=2\Sigma\otimes \Sigma)$ with probability $1-\zeta$.\end{lemma}
    
\begin{proof}
Since $x$ is Gaussian, by Lemma~\ref{lem:Isserlis}, we have $\Psi=\E[(x\otimes x - \Sigma^\flat)(x\otimes x-\Sigma^\flat )^\top]=2\Sigma\otimes \Sigma$. We can write $\psi_V^2=2\Tr(V^\top \Sigma V \Sigma)=2\ip{V}{\Sigma V\Sigma}$.

\begin{lemma}[{\cite[Lemma~B.1]{li2020robust}} and {\cite[Fact~4.2]{dong2019quantum}}]
\label{lem:resilience_cov_gaussian}
Let $\delta>0$ and $\alpha\in (0,0.5)$. A dataset $S=\{x_1, x_2, \cdots, x_n\}$ consists of $n$ i.i.d. samples from $\cN(0,\mathbf{I}_{d\times d})$. If $n=\Omega\left((d^2+\log(1/\zeta )) / (\alpha^2\log(1/\alpha)) \right)$ with a large enough constant, then there exists a universal constant $C_1>0$ and $C_2>0$ such that with probability $1-\zeta $, for any subset $T\subset S $ and $|T|\geq (1-\alpha)n$, we have
\begin{eqnarray*}
    \Big\| \frac{1}{|T|}\sum_{ x_i\in T} x_i\otimes x_i - \mathbf{I}_{d\times d}^\flat \Big\|  & \leq &  C_1\alpha\log(1/\alpha)\,    \;,\text{ and } \\
      \left\| \frac{1}{|T|}\sum_{ x_i\in T}  \big(x_i\otimes x_i- \mathbf{I}_{d\times d}^\flat \big)\big(x_i\otimes x_i-  \mathbf{I}_{d\times d}^\flat \big)^\top  -   2\mathbf{I}_{d\times d} \otimes\mathbf{I}_{d\times  d}\right\| & \leq &  C_2\alpha\log(1/\alpha)^2\;. 
\end{eqnarray*}
\end{lemma}
    By Lemma~\ref{lem:resilience_cov_gaussian}, we know with probability $1-\zeta $, for any subset $T\subset S $ and $|T|\geq (1-\alpha)n$, we have
    \begin{eqnarray*}
      \Big\| \frac{1}{|T|}\sum_{x_i\in T} (\Sigma^{-1/2}x_i)\otimes (\Sigma^{-1/2}x_i) - \mathbf{I}_{d\times d}^\flat \Big\|  & \leq &  C_1\alpha\log(1/\alpha)\;.
    \end{eqnarray*}
    This is equivalent to 
    \begin{eqnarray*}
    	 \Big|  (V^\flat)^\top\frac{1}{|T|}\sum_{ x_i\in T}(\Sigma^{-1/2}\otimes \Sigma^{-1/2}) (x_i\otimes x_i) - (V^\flat)^\top \mathbf{I}_{d\times d}^\flat \Big|  & \leq & C_1\alpha\log(1/\alpha)\;,
    \end{eqnarray*}
    for any $\|V\|_F=1$. 
    This implies
    \begin{eqnarray*}
    	\Big|  (V^\flat)^\top\frac{1}{|T|}\sum_{ x_i\in T} (x_i\otimes x_i)  - (V^\flat)^\top(\Sigma\otimes \Sigma)^{1/2} \mathbf{I}_{d\times d}^\flat \Big|  & \leq &  C_1\alpha\log(1/\alpha) \sqrt{ (V^\flat )^\top(\Sigma\otimes \Sigma) V^\flat} \;,
    \end{eqnarray*}
    which is also equivalent to, for some constant $C$
    \begin{eqnarray*}
    	\left|\ip{V}{\frac{1}{|T|}\sum_{ x_i\in T}x_ix_i^\top}-\ip{V}{\Sigma}\right|\leq C\alpha\log(1/\alpha)\sqrt{2\ip{V}{\Sigma V\Sigma}}\;,
    \end{eqnarray*}
    which proves the first resilience Eq.~\eqref{def:res_cov_1} in Definition~\ref{def:resilience_covariance}.
    
    Similarly, by Lemma~\ref{lem:resilience_cov_gaussian}, we have
    \begin{eqnarray*}
    	\left\| \frac{1}{|T|}\sum_{x_i\in T}  \big(\Sigma^{-1/2}x_i\otimes \Sigma^{-1/2}x_i- \mathbf{I}_{d\times d}^\flat \big)\big(\Sigma^{-1/2}x_i\otimes \Sigma^{-1/2}x_i- \mathbf{I}_{d\times d}^\flat \big)^\top  -   2\mathbf{I}_{d\times d}\otimes\mathbf{I}_{d\times d}\right\| & \leq &  C_2\alpha\log(1/\alpha)^2\;.
    \end{eqnarray*}

    This is equivalent to for any $\|V\|_F=1$,
        \begin{eqnarray*}
    	 \Big| \frac{1}{|T|}\sum_{ x_i\in T}\ip{ V^\flat}{\Sigma^{-1/2}x_i\otimes \Sigma^{-1/2}x_i - \mathbf{I}_{d \times d}^\flat }^2 -2\Big|  & \leq &C_2\alpha\log(1/\alpha)^2\;.
    \end{eqnarray*}
    This implies
    \begin{eqnarray*}
    	 \Big| \frac{1}{|T|}\sum_{ x_i\in T}\ip{V^\flat}{x_i\otimes x_i- \Sigma^\flat }^2 -2(V^\flat)^\top(\Sigma \otimes \Sigma ) V^\flat \Big|  & \leq &C_2\alpha\log(1/\alpha)^2\ip{V}{\Sigma V\Sigma}\;,
    \end{eqnarray*}
    which is also equivalent to, for some constant $C$
    \begin{eqnarray*}
    	\Big| \frac{1}{|T|}\sum_{ x_i\in T}  \big(\ip{V}{x_ix_i^\top}-\ip{V}{\Sigma} \big)^2  -   2\Tr(V^\top\Sigma V\Sigma)\Big| & \leq &  2C\alpha\log(1/\alpha)^2\ip{V}{\Sigma V\Sigma}\;,
    \end{eqnarray*}
   which proves the second resilience Eq.~\eqref{def:res_cov_2} in Definition~\ref{def:resilience_covariance}.
    
    \end{proof}

    The second and fourth moment resilience properties of Gaussian distributions in Lemma~\ref{lem:gaussian_cov_res}, together with the utility analysis of HPTR in Corollary.~\ref{coro:robust_cov}, implies the following utility guarantee.   
    
 \begin{coro} 
    \label{coro:cov_gaussian} 
    Under the hypotheses of Lemma~\ref{lem:gaussian_cov_res} there exists a constant $c>0$ such that  for any $\alpha\in(0,c)$, a dataset of size 
     $$ n = O\Big(\,\frac{d^2+\log(1/\zeta)}{\alpha^2\log(1/\alpha)} + \frac{d^2 + \log(1/(\delta\zeta))}{\alpha\varepsilon} \,\Big)\;, $$ 
     a sensitivity of $\sens=O(\log(1/\alpha)/n) $, and a threshold $\thresh=O(\alpha\log(1/\alpha)) $ with large enough constants 
     are sufficient for  \HPTR$(S)$ with a choice of distance function in Eq.~\eqref{eq:cov_score} to achieve 
     \begin{eqnarray} 
     \|\Sigma^{-1/2} \hat{\Sigma}\Sigma^{-1/2} -\mathbf{I}_{d \times d}\|_F \;=\; O(\alpha\log(1/\alpha))\;, \label{eq:cov_error_bound}
     \end{eqnarray}
     with  probability $1-\zeta$. 
     Further, the same guarantee holds even if $\alpha$-fraction of the samples are arbitrarily corrupted as in Assumption~\ref{asmp:mean}. 
\end{coro}
This Mahalanobis distance  guarantee (for the  Kronecker product, $\{x_i\otimes x_i\}$, of the samples) implies that the predicted Gaussian distribution is  close to the sample generating one in total variation distance
 (see for example \cite[Lemma 2.9]{KLSU19}): $d_{\rm TV}({\cal N}(0,\hat\Sigma),{\cal N}(0,\Sigma)) = O(\| \Sigma^{-1/2}\hat\Sigma\Sigma^{-1/2} - {\bf I}_{d\times d}\|_F) = O(\alpha\log(1/\alpha))$. 
 This relation also implies that the error bound is near-optimal under $\alpha$-corruption, matching a lower bound up to a factor of $O(\log(1/\alpha))$. Even if DP is not required and we are given infinite samples, an adversary can move $\alpha$ fraction of the probability mass to switch a Gaussian distribution into another one at Mahalanobis distance $\|\Sigma_1^{-1/2}\Sigma_2\Sigma_1^{-1/2}-{\bf I}_{d\times d}\|_F=\Omega(\alpha)$. Hence, we cannot tell which of the two distributions the (potentially infinite) samples came from.

The sample complexity is near-optimal, matching a lower bound up to a factor of $O(\log(1/\alpha))$ when   $\delta=e^{-\Theta(d^2)}$.  For a constant $\zeta$, HPTR requires $n=O(d^2/(\alpha^2\log(1/\alpha))+d^2/(\alpha\varepsilon) + \log(1/\delta)/(\alpha\varepsilon))$. This nearly matches a lower bound (that holds even if there is no corruption) on $n$ to achieve the guarantee of 
Eq.~\eqref{eq:cov_error_bound}: 
$n=\Omega(d^2/(\alpha\log(1/\alpha))^2 + \min\{d^2,\log(1/\delta)\}/(\varepsilon\alpha\log(1/\alpha)) + \log(1/\delta)/\varepsilon )$. The first term follows from  the classical estimation of the covariance without DP, and matches the first term in our upper bound up to a $O(\log(1/\alpha))$ factor. The second term follows from extending the lower bound in \cite{KLSU19} constructed for pure differential privacy with $\delta=0$  and matches the second term in our upper bound up to a $O(\log(1/\alpha))$ factor when $\delta=e^{-\Theta(d^2)}$. The last term is from \cite{KV17} and has a gap of $O(1/\alpha)$ factor compared to the third term in our upper bound, but this term is typically not dominating when  $\delta$ is large enough: $\delta=e^{-O(d^2)}$.
We note that a slightly tighter upper bound is achieved  by the state-of-the-art algorithm in \cite{aden2020sample} that only requires $O(d^2/(\alpha\log(1/\alpha))^2+d^2/(\varepsilon\alpha\log(1/\alpha))+\log(1/\delta)/\varepsilon)$.  The state-of-the-art polynomial time algorithm in  \cite{kamath2021private}  requires no assumptions on $\Sigma$ but the sample complexity is larger:  $n=\tilde O(d^2/(\alpha\log(1/\alpha))^2 + d^2{\rm polylog}(1/\delta)/(\varepsilon\alpha\log(1/\alpha)) + d^{5/2}{\rm polylog}(1/\delta)/\varepsilon)$. 

If privacy is not concerned (i.e., $\varepsilon = \infty$), HPTR achieves the error in Eq.~\eqref{eq:cov_error_bound} with $n=O(d^2 / \alpha^2\log(1/\alpha))$ samples. There are polynomial time estimators achieving the same guarantee \cite{li2020robust, diakonikolas2019robust}. The gap of $\log(1/\alpha)$ to the lower bound in the error can be tightened using  algorithms that are not computationally efficient as shown in \cite{chen2018robust, rousseeuw1985multivariate}.

\medskip\noindent
{\bf Remark.} When we only have a sample size of $n=O(d/\alpha^2)$, our analysis does not provide any guarantees. 
However, for robust covariance estimation under $\alpha$-corruption,
one can still guarantee a bound on  a weaker error metric in spectral norm:  $\|\Sigma^{-1/2}\hat\Sigma\Sigma^{-1/2}-{\bf I}_{d\times d}\| = O(\alpha\log(1/\alpha))$ \cite[Theorem 3.4]{zhu2019generalized}. There is no corresponding differentially private covariance estimator in that small sample regime. 
A promising direction is to  apply HPTR framework, but designing a score function for this spectral norm distance that only depends on one-dimensional robust statistics remains challenging.

\section{Principal component analysis}
\label{sec:pca}

In Principal Component Analysis (PCA), 
we are given i.i.d. samples $S=\{x_i\in\reals^d \}_{i=1}^n$ drawn from a zero mean distribution $P_\Sigma$ with an unknown covariance matrix  $\Sigma$. We want to find a top eigenvector of $\Sigma$, $u \in \arg\max_{\|v\|=1} v^\top \Sigma v$, privately. The performance of our estimate $\hat{u}$ is measured by  
how much of the covariance is captured in the direction $\hat u$ relative to that of $u$:  
$D_{\Sigma}(\hat{u})=1-(\hat{u}^\top\Sigma\hat{u}/u^\top \Sigma u)$, where $u$ is one of the top eigenvector of $\Sigma$. 
When the mean is not zero, this can be handled similarly as in covariance estimation in Section~\ref{sec:cov}. 
%In this section, we assume we know the top eigenvalue $\lambda_1= u^\top \Sigma u$ and discuss how to estimate it privately from samples in Appendix~\ref{sec:pca_lambda}. 

\subsection{Step 1: Designing the surrogate score function \texorpdfstring{$D_S(\hat{u})$}{}}
\label{sec:pca1}

It is straightforward to design a score function of $D_S : {\mathbb S}^{(d-1)}\to \reals_+$ where ${\mathbb S}^{(d-1)}$ is the unit sphere in $\reals^d$, 
\begin{eqnarray}
	D_S(\hat{u})\;=\;1-\frac{\hat{u}^\top \Sigma(\cM_{\hat{u}, \alpha}) \hat{u}}{\max_{v\in\reals^d:\|v\|=1} v^\top \Sigma(\cM_{v, \alpha}) v }\;,\label{score:pca}
\end{eqnarray} 
where $\cM_{\hat{u}, \alpha} \subset S$ is  the subset of data points corresponding to the smallest $(1-(2/3.5)\alpha) n$ values  in the projected set $S_{\hat{u}}=\{\ip{\hat{u}}{x_i}^2\}_{x_i\in S}$ and $\Sigma(\cM_{\hat{u}, \alpha})= (1/|\cM_{\hat{u}, \alpha}|)\sum_{x_i\in \cM_{\hat{u}, \alpha}} x_ix_i^\top $.
Note that when we replace $\Sigma(\cM_{\hat{u}, \alpha})$ with the population covariance matrix $\Sigma$, we recover the target error metric of $D_\Sigma(\hat u) = 1- (\hat u ^\top \Sigma \hat u / \max_{\|v\|=1}v^\top \Sigma v)$. For this choice of $\robdist_S(\hat u)$, the support  of the exponential mechanism is already compact, and we do not restrict it any further, say, to be in $B_{\tau,S}$. This simplifies the HPTR algorithm and also the analysis as follows. 
We define
\begin{eqnarray*}
    \unsafe_{\varepsilon}&=&
    \Big\{S'\subset\reals^{d\times n}\,|\, \text{$\exists S''\sim S'$ and $\exists E$ such that } \prob_{\hat{u}\sim r_{(\varepsilon,\sens,S'')}}(\hat{u}\in E) > e^\varepsilon \prob_{\hat{u}\sim r_{(\varepsilon,\sens,S')}} (\hat{u}\in E)  \nonumber\\
    && \;\; \text{ or } \prob_{\hat{u}\sim r_{(\varepsilon,\sens,S')}}(\hat{u}\in E) > e^\varepsilon \prob_{\hat{u}\sim r_{(\varepsilon,\sens,S'')}} (\hat{u}\in E)  \Big\}\;.
\end{eqnarray*}
 Note that sine the support is the same for all $S$, we can achieve a stronger pure DP with $\delta=0$ in the exponential mechanism. However, we still need $\delta>0$ in the {\sc Test} step. 
$\HPTR$ for PCA proceeds as follows: 
\begin{itemize} 
    \item[1.] {\sc Propose}: 
        Propose a target sensitivity bound $\sens=80\rho_2/(\alpha n)$.
    \item[2.] {\sc Test}: 
    \begin{itemize}
    \item[2.1.]     Compute the safety margin $m = \min_{S'} d_H(S,S')$ such that $S'\in \unsafe_{\varepsilon/2}$.
    \item[2.2.]      If $\hat{m} = m  + {\rm Lap}(2/\varepsilon)<(2/\varepsilon)\log(2/\delta)$ then output $\perp$, and otherwise continue.
    \end{itemize} 
    \item[3.] {\sc Release}: Output $\hat{u}$ sampled from a distribution with a pdf:
    \begin{eqnarray*}
    r_{(\varepsilon,\sens,S)}(\hat{u})\;\;=\;\; \frac{1}{Z}\,\exp\left(-\frac{\varepsilon}{4 \sens}\robdist_S(\hat{u}) \right) \;,
     \end{eqnarray*} 
     from ${\mathbb S}^{(d-1)} =\{\hat u\in\reals^d: \|\hat u\|=1\}$ 
where $Z=\int_{{\mathbb S}^{(d-1)}} \exp\{-(\varepsilon \robdist_S(\hat{u}))/(4\sens)\} \,d\hat{u}$.
\end{itemize} 

The choice of $\rho_2$ depends on your hypothesis on the tail of the sample generating distribution, and $\alpha$ depends on the target accuracy as guided by Theorem~\ref{thm:pca_utility} (or the fraction of adversarial corruption in the case of outlier robust PCA setting in Theorem~\ref{thm:robust_pca_utility}). The target privacy guarantee determines $(\varepsilon,\delta)$. 

\subsection{Step 2: Utility analysis under resilience}

The following resilience properties are critical in selecting the sensitivity $\sens$ and also in analyzing the utility. 
\begin{definition}[Resilience for PCA]
    \label{def:resilience_pca}
    For some $\rho_1\in \reals_+, \rho_2\in\reals_+ $ we say a set of $n$  data points $S_{\rm good}=\{x_i\in\reals^d\}_{i=1}^n$ is $(\alpha, \rho_1, \rho_2)$-resilient with respect to $\Sigma $ for some positive semidefinite $\Sigma\in\reals^{d\times d}$ if for any $T\subset S_{\rm good}$ of size $|T|\geq (1-\alpha)n$, the following holds for all $v\in \reals^d$ with $\|v\|=1$:     
    \begin{eqnarray}
    \Big| \frac{1}{|T|}\sum_{x_i\in T} \langle v, x_i\rangle \Big| & \leq &  \rho_1 \,\sigma_v \label{def:res_pca1} \text{ and }\\
	\Big| \frac{1}{|T|}\sum_{x_i\in T} \langle v, x_i\rangle^2  - \sigma_v^2 \Big| & \leq &  \rho_2\, \sigma_v^2 \label{def:res_pca2} \;.
	\end{eqnarray}
    where  $\sigma_v^2 = v^\top \Sigma v$.
\end{definition}
We refer to Section~\ref{sec:mean2} for the explanation of how resilience is fundamentally connected to sensitivity. 
For an example of a Gaussian distribution, the samples are $(\alpha,O(\alpha\sqrt{\log(1/\alpha)} ),O(\alpha\log(1/\alpha)))$-resilient (with a large enough $n$). We show next how resilience implies an error bound for HPTR, which is $O(\alpha\log(1/\alpha))$ for Gaussian distributions.

\begin{thm}
\label{thm:pca_utility}
	There exist positive constants $c$ and $C$ such that for any  $(\alpha,\rho_1,\rho_2)$-resilient set $S$ with respect to some $\Sigma$ and  satisfying $\alpha<\rho_2<c$, $\HPTR$ Section~\ref{sec:pca1} for PCA with the choices of the distance function in Eq.~\eqref{score:pca} and  $\sens= 80 \rho_2/(\alpha n)$  achieves 
    $1-(\hat{u}^\top\Sigma\hat{u}/\|\Sigma\| ) \leq 20 \rho_2$ with probability $1-\zeta$, if \begin{eqnarray}
         n \;\geq \; C \left(\frac{\log(1/(\delta\zeta))+d\log(1/\rho_2)}{\varepsilon\alpha }        \right)   \;.
    \end{eqnarray} 
\end{thm}
We discuss the implications of this result in Section~\ref{sec:pca3} for specific instances of the problem. 
Under Assumption~\ref{asmp:mean} on $\alpha_{\rm corrupt}$-corruption of the data and Definition~\ref{def:corruptgoodset} on the corrupt good sets, we show that HPTR is also robust against corruption.
\begin{thm}
\label{thm:robust_pca_utility}
    %Under the hypotheses of Theorem~\ref{thm:pca_utility}, 
	There exist positive constants $c$ and $C$ such that for any  $((2/7)\alpha,\alpha,\rho_1,\rho_2)$-corrupt good set $S$ with respect to some $\Sigma$ satisfying $\alpha<rho_2<c$, $\HPTR$ in Section~\ref{sec:pca1} for PCA with the choices of the distance function in Eq.~\eqref{score:pca} and  $\sens=80 \rho_2/(\alpha n)$  achieves 
    $1-(\hat{u}^\top\Sigma\hat{u}/\|\Sigma\| ) \leq 20\rho_2$ with probability $1-\zeta$, if \begin{eqnarray}
         n \;\geq \; C \left(\frac{\log(1/(\delta\zeta))+d\log(1/\rho_2)}{\varepsilon\alpha }        \right)   \;.
    \end{eqnarray} 
\end{thm}
We provide a proof of the robust and DP PCA in Section~\ref{sec:pca_proof3}, where Theorem~\ref{thm:pca_utility} follows immediately by selecting $\alpha$ as a free parameter. 
As the HPTR Section~\ref{sec:pca1} for  PCA is significantly simpler, we do not apply the general analysis in Theorem~\ref{thm:utility} and instead we prove The above theorem directly. 
To this end, we first show a bound on sensitivity  and next show that safety test succeeds with high probability in 
Section~\ref{sec:pca_proof1}.

\subsubsection{Resilience implies bounded local sensitivity}
\label{sec:pca_proof1}

Given the resilience properties of a corrupt good set $S$, we show that the sensitivity of $\robdist_S(\hat u)$ is bounded by $\sens$.  

\begin{lemma}
    \label{lem:local_asmp_pca}
    Suppose $\alpha\leq c$  for some small enough constant $c$. For $\sens = 80\rho_2/(\alpha n)$,
    and a $((2/7)\alpha, \alpha,\rho_1,\rho_2)$-corrupt good  $S$, if 
    \begin{eqnarray*}
        n\;=\;\Omega\Big(
        \frac{\log(1/(\delta\zeta))}{\alpha\varepsilon}
        \Big)\;,
    \end{eqnarray*} 
    with a large enough constant 
    then the for all $S'$ within Hamming distance $k^* = (2/\varepsilon) \log(4/(\zeta\delta))$ from $S$, we have
    \begin{eqnarray}
    	\max_{S''\sim S'}|D_{S''}(\hat{u})-D_{S'}(\hat{u})|\leq \Delta\;,
    \end{eqnarray}
    for all unit vector $\hat{u}$ and all neighboring dataset $S''$.
\end{lemma}

\begin{proof}
The proof is similar to the proof of Lemma~\ref{lem:local_asmp}. We first assume $(k^*+1)/n\leq \alpha/7$, which requires $n=\Omega(\log(1/\delta\zeta))/(\alpha\varepsilon)$ with a large enough constant. This implies that $S'$ is a $((3/7)\alpha,\alpha,\rho_1,\rho_2)$-corrupt good set. The rest of this proof is under this assumption. 
Let ${\cal T}_{\hat{u}, \alpha}(S') \subset S$ be the subset of data points corresponding to the largest $(2/3.5)\alpha n$ values  in the projected set $S'_{\hat{u}}=\{\ip{\hat{u}}{x_i}^2\}_{x_i\in S'}$. 
Recall that $S_{\rm good}$ is the original resilient dataset before corruption by an adversary. 
From Lemma~\ref{lem:deviate} and  the fact that $|S_{\rm good}\cap {\cal T}_{\hat u, \alpha}(S')|\geq (1/7)\alpha n$, it follows that $ (1/|S_{\rm good}\cap {\cal T}_{\hat u, \alpha}(S')|)\sum_{x_i\in S_{\rm good}\cap {\cal T}_{\hat u, \alpha}} \langle\hat u, x_i \rangle^2 \leq (1 + (2\rho_2 )/((1/7)\alpha))\sigma_{\hat u}^2$, where $\sigma_{\hat u}=\sqrt{\hat u^\top\Sigma\hat u}$. This implies 
\begin{eqnarray}
	\min_{x_i\in S_{\rm good}\cap {\cal T}_{\hat{u}, \alpha}}\ip{\hat{u}}{x_i}^2
	\;\leq \;\Big( 1 + \frac{2\rho_2}{(1/7)\alpha  } \Big) \sigma_{\hat{u}}^2\;.
\end{eqnarray}

Let ${\cal M}_{\hat{u}, \alpha}(S')$ be the remaining  subset of $S'$ with $(1-(2/3.5)\alpha)n$ smallest values in $\{ (\langle \hat u, x_i\rangle)^2\}_{i\in [n]} $. 
${\cal M}_{\hat{u}, \alpha}(S')$ and ${\cal M}_{\hat{u}, \alpha}(S'')$ can differ at most by one data point. 
Let $x'$ and $x''$ be the unique pair of data points that are in ${\cal M}_{\hat{u}, \alpha}(S')$ and ${\cal M}_{\hat{u}, \alpha}(S'')$, respectively. 
If there is no such pair, then the two filtered subsets are the same and the following claims are trivially true. 

%Suppose $x' \in {\cal M}_{\hat{u}, \alpha}(S')$. 
If $\ip{\hat{u}}{x''}^2\leq  \max_{x_i\in  {\cal M}_{\hat{u}, \alpha}(S')}\ip{\hat{u}}{x_i}^2 \leq \min_{x_i\in S_{\rm good}\cap {\cal T}_{\hat{u}, \alpha}(S')}\ip{\hat{u}}{x_i}^2$, we have 
$|\ip{\hat{u}}{x'}^2-\ip{\hat{u}}{ x''}^2|\leq (1+14\rho_2/\alpha) \sigma_{\hat u}^2$, where $\sigma_{\hat u}^2=\hat u^\top \Sigma \hat u$.  If $\ip{\hat{u}}{x''}^2>  \max_{x_i\in  {\cal M}_{\hat{u}, \alpha}(S')}\ip{\hat{u}}{x_i}^2$, then $x''$ is at most $\langle \hat u , x'' \rangle ^2 \leq \min_{x_i\in S_{\rm good}\cap {\cal T}_{\hat{u}, \alpha}(S')}\langle\hat u, x_i\rangle^2$, where equality holds if the smallest point in the top subset enters  ${\cal M}_{\hat u, \alpha}(S'')$. 
%will have $x_j$ replaced by $\argmin_{x_i\in S_{\rm good}\cap {\cap T}_{\hat{u}, \alpha}}\ip{\hat{u}}{x_i}^2$, which 
This also implies $|\ip{\hat{u}}{x'}^2-\ip{\hat{u}}{x''}^2|\leq (1+14\rho_2/\alpha) \sigma_{\hat u}^2$. %The case where $x_j \notin {\cal M}_{\hat{u}, \alpha}(S')$ follows similarly. 
Let $\sigma'^2_v = v^\top \Sigma({\cal M}_{v,\alpha}(S')) v$ and 
$\sigma''^2_v = v^\top \Sigma({\cal M}_{v,\alpha}(S'')) v$, then  for any $\|v\|=1$, 
\begin{eqnarray*}
	\left|\sigma'^2_v -\sigma''^2_v \right|
	&= & \left| v^\top \left(\frac{1}{(1-(2/3.5)\alpha)n}\sum_{x_i\in \cM_{v, 2\alpha} (S') }x_ix_i^\top-\frac{1}{(1-(2/3.5) \alpha)n}\sum_{x_i\in \cM_{v, 2\alpha}(S'')}x_ix_i^\top\right) v \right|\\
	&\leq & \frac{2}{n}| \langle v,x'\rangle^2 - \langle v,x''\rangle^2|  
	\;\leq\; \frac{2}{n}\Big( 1 + \frac{14\rho_2}{\alpha }\Big) v^\top \Sigma v \;,
\end{eqnarray*}
for $\alpha\leq c$ small enough. 
	Then for the local sensitivity, we have
\begin{eqnarray*}
	\left|D_{S'}(\hat{u})-D_{S''}(\hat{u})\right|
	&\leq & \Big| \frac{\sigma'^2_{\hat u}-\sigma''^2_{\hat u}}{\max_{\|v\|=1} \sigma'^2_v} \Big| + \Big| \frac{\sigma''^2_{\hat u}}{\max_{\|v\|=1}\sigma'^2_v} - \frac{\sigma''^2_{\hat u}}{\max_{\|v\|=1}\sigma''^2_v} \Big| \\ 
	&\leq& \frac{2}{n} \Big(1+\frac{14\rho_2}{\alpha}\Big)\frac{\hat u^\top \Sigma \hat u}{0.9\|\Sigma\|} + \frac{1.1 \hat u^\top \Sigma \hat u}{0.9^2\|\Sigma\|^2} \frac{2}{n} \Big(1+\frac{14\rho_2}{\alpha}\Big)\|\Sigma\|\;, 
\end{eqnarray*}
where we used the resilience in  Eq.~\eqref{def:res_pca2} with small enough $\rho_2\leq c$ such that $0.9 v^\top \Sigma v \leq \sigma'^2_v\leq 1.1 v^\top \Sigma v$ and $0.9 v^\top \Sigma v \leq \sigma''^2_v\leq 1.1 v^\top \Sigma v$ (which follow from Lemma~\ref{lem:pca_robustness}).
When $\rho_2\leq \alpha$, this is bounded by $|D_{S'}\hat u)-D_{S''}(\hat u)|\leq 80\rho_2/(\alpha n)=\sens$.

\end{proof}

%\subsubsection{Large safety margin}
%\label{sec:pca_proof2}
Since the support is the same for all exponential mechanisms regardless of the dataset, sensitivity bound immediately implies safety. The following lemma shows that we have sufficient safety margin to succeed with probability at least $1-\zeta$, since $k^*=(2/\varepsilon)\log(4/(\delta\zeta))$ and the threshold is $(2/\varepsilon)\log(2/\delta)$.

\begin{lemma}
\label{lem:safety_margin_pca}
	Under the hypothesis of Lemma~\ref{lem:local_asmp_pca}, for any $S'$ at Hamming distance at most $k^*$ from $S$, we have $S'\in \safe_{\varepsilon/2}$.
\end{lemma}

%
%\begin{proof}
%	Under the hypothesis of Lemma~\ref{lem:local_asmp_pca}, we have $\left|D_{S'}(\hat{u})-D_{S''}(\hat{u})\right| \leq 110\rho_1\lambda_1/(\alpha n)$, this implies, for any $E\subset \{v\in \reals^d:\|v\|=1\}$, we have
%	\begin{eqnarray}
%		\prob_{\hat{u}\sim r_{(\varepsilon, \Delta, S')}}\left(\hat{u}\in E\right)\leq e^{\varepsilon/2}\prob_{\hat{u}\sim r_{(\varepsilon, \Delta, S'')}}\left(\hat{u}\in E\right)
%	\end{eqnarray}
%	
%\end{proof}
\subsubsection{Proof of Theorem~\ref{thm:robust_pca_utility}}
\label{sec:pca_proof3}

This proof is similar as the proof of a universal utility analysis in  Theorem~\ref{thm:utility}. First, we show we pass the safety test with high probability. By Lemma~\ref{lem:safety_margin_pca}, we know $m>k^*=2/\varepsilon\log(4/(\zeta\delta))$. Then we have
	\begin{eqnarray*}
		\prob\left(\text{output}\perp\right) = \prob\left(m+{\rm Lap}(2/\varepsilon)<(2/\varepsilon) \log(2/\delta)\right) \;\; \leq \;\; \frac{\zeta}{2}\;.
	\end{eqnarray*}
Next, we assume the dataset passed the safety test and show that $\prob_{\hat{u}\sim r_{(\varepsilon, \Delta, S)}}(\hat{u}^\top\Sigma\hat{u}\geq (1-4\rho_2)\|\Sigma\|  )\geq 1-\zeta/2$. 
%Let the good event be $G=\{v: v^\top \Sigma v\geq (1-\eta)\lambda_1\}$ and bad event be $B = \{v: v^\top \Sigma v\leq (1-2\eta)\lambda_1\}$. 
\begin{lemma}
\label{lem:pca_robustness}
	For an $((2/7)\alpha,\alpha,\rho_1, \rho_2)$-corrupt good set $S$ with respect to $\Sigma$,  then $|\hat{u}^\top \Sigma \hat{u}-\hat{u}^\top \Sigma(\cM_{\hat{u}, \alpha}) \hat{u}| \; \leq \; 4 \rho_2 \hat u^\top \Sigma \hat u $.
\end{lemma}
\begin{proof} 

We have
\begin{eqnarray}
    |\hat{u}^\top \Sigma \hat{u}-\hat{u}^\top \Sigma(\cM_{\hat{u}, \alpha}) \hat{u}| &=& \frac{|\sum_{i\in \cM_{\hat{u}, \alpha}}(\ip{\hat{u}}{x_i}^2-\sigma_{\hat{u}}^2)|}{(1-(2/3.5)\alpha)n}\nonumber\\
    &&\hspace*{-3cm}\leq   \frac{|\sum_{i\in \cM_{\hat{u}, \alpha}\cap S_{\rm good}}(\ip{\hat{u}}{x_i}^2-\sigma_{\hat{u}}^2)|}{(1-(2/3.5)\alpha)n} +  \frac{|\sum_{i\in \cM_{\hat{u}, \alpha}\cap S_{\rm good}}(\ip{\hat{u}}{x_i}^2-\sigma_{\hat{u}}^2)|}{(1-(2/3.5)\alpha)n}\label{eq:pcares1}
\end{eqnarray}
For $i\in \cM_{\hat{u}, \alpha}\cap S_{\rm bad}$, by Lemma~\ref{lem:deviate}, we have
\begin{eqnarray}
    |\ip{\hat{u}}{x_i}^2-\sigma_{\hat{u}}^2| &\leq & \max\left\{\frac{\sum_{i\in \cT_{\hat{u}, \alpha}\cap S_{\rm good}}(\ip{\hat{u}}{x_i}^2-\sigma_{\hat{u}}^2)}{|\cT_{\hat{u}, \alpha}\cap S_{\rm good}|}, \sigma_{\hat{u}}^2\right\}\nonumber\\
   &\leq & \frac{2\rho_2\sigma_{\hat{u}}^2}{(1/3.5)\alpha}\;,\label{eq:pcares2}
\end{eqnarray}
where in the last inequality, we applied our assumption that $\rho_2\geq \alpha$.

By the resilience property Eq.~\eqref{def:res_pca2} on $\cM_{\hat{u}, \alpha}\cap S_{\rm good}$, we also have
\begin{eqnarray}
    \frac{|\sum_{i\in \cM_{\hat{u}, \alpha}\cap S_{\rm good}}(\ip{\hat{u}}{x_i}^2-\sigma_{\hat{u}}^2)|}{|\cM_{\hat{u}, \alpha}\cap S_{\rm good}|}\leq \rho_2\sigma_{\hat{u}}^2\;.\label{eq:pcares3}
\end{eqnarray}

Plugging Eq.~\eqref{eq:pcares2} and \eqref{eq:pcares3} into \eqref{eq:pcares1}, we have
\begin{eqnarray*}
    |\hat{u}^\top \Sigma \hat{u}-\hat{u}^\top \Sigma(\cM_{\hat{u}, \alpha}) \hat{u}| \leq \frac{2\rho_2\sigma_{\hat{u}}^2+(1-(2/3.5)\alpha)\rho_2\sigma_{\hat{u}}^2}{1-(2/3.5)\alpha} \leq 4\rho_2\sigma_{\hat{u}}^2\;,
\end{eqnarray*}
for $\alpha\leq c$ small enough.
\end{proof}

This implies $|D_\Sigma(\hat{u})-D_S(\hat{u})|\leq 4\rho_2$ for an $((2/7)\alpha, \alpha, \rho_1, \rho_2)$-corrupt good set $S$.

Let $\mu(\cdot)$ denote the uniform measure on the unit sphere. By the fact that 
for any $0<r<2$, a cap of radius $r$ on the $(d-1)$-dimensional unit sphere ${\mathbb S}^{(d-1)}$ has measure at least $(1/2)(r/2)^{d-1}$
from, for example 
\cite[Fact~3.1]{kapralov2013differentially},  we have for some constant $c_2>0$ and $\rho_2\leq 1/8$, 
\begin{eqnarray}
    \mu(\{v\in \reals^d: v^\top \Sigma v\geq (1-4\rho_2)\|\Sigma\|, \|v\|=1\}) \geq \left(\cos^{-1}(1-4\rho_2) / 2\right)^{d-1} \geq e^{-c_2 d \log (1 / \rho_2)}\;.
\end{eqnarray}

By Lemma~\ref{lem:pca_robustness}, the choice of $\sens=80\rho_2/(\alpha n)$, we have 
\begin{eqnarray*}
&&\prob_{\hat{u}\sim r_{(\varepsilon, \Delta, S)}}\left(\|\Sigma\| -\hat{u}^\top \Sigma \hat{u}\leq 4\rho_2 \|\Sigma\| \right) \\
=&& \int_{\{v\in \reals^d: v^\top \Sigma v\geq (1-4\rho_2)\|\Sigma\|, \|v\|=1\}}  \,r_{(\varepsilon, \Delta, S)}(\hat{u}) \, d \hat \mu\\
\geq && {\rm Vol}(\{v\in \reals^d: v^\top \Sigma v\geq (1-4\rho_2)\|\Sigma\|, \|v\|=1\}) \min_{\hat \mu \in \{v\in \reals^d: v^\top \Sigma v\geq (1-4\rho_2)\|\Sigma\|, \|v\|=1\}} \,r_{(\varepsilon, \Delta, S)}(\hat{u}) \\
\geq && 
{\rm Vol} 
({\mathbb S}^{(d-1)})  \,\mu(\{v\in \reals^d: v^\top \Sigma v\geq (1-4\rho_2) \|\Sigma\|, \|v\|=1\}) 
\min_{\hat{u}\in \{v\in \reals^d: v^\top \Sigma v\geq (1-4\rho_2)\|\Sigma\|, \|v\|=1\}} r_{(\varepsilon, \Delta, S)}(\hat{u})\\
\geq && {\rm Vol}({\mathbb S}^{(d-1)})\,e^{-c_2d\log(1/\rho_2)} \frac{1}{Z} \exp \Big\{-\frac{\varepsilon}{4\sens} \max_{\|\hat u\|=1,4\rho_2\geq 1-\frac{\hat u^\top \Sigma \hat u}{\|\Sigma\|}} 
1-\frac{\hat u^\top \Sigma({\cal M}_{\hat u,\alpha}) \hat u} {\|\Sigma\| } \Big\} 
\\
%\geq && e^{-cd\log(1/\rho_2)} \frac{1}{Z}\exp\{-\varepsilon \alpha n(4\rho_2+\rho_2)\|\Sigma\| /(56\rho_2\lambda_1)\}\;,
\geq && {\rm Vol}({\mathbb S}^{(d-1)})\, e^{-c_2d\log(1/\rho_2)} \frac{1}{Z}\exp\Big\{-\frac{\alpha \varepsilon n}{40}\Big\}\;,
	\end{eqnarray*}
	and similarly, 
	\begin{eqnarray*}
		\prob_{\hat{u}\sim r_{(\varepsilon, \Delta, S)}}\left(\|\Sigma\|-\hat{u}^\top \Sigma \hat{u}\geq 20\rho_2 \|\Sigma\| \right) 
		&\leq& {\rm Vol}({\mathbb S}^{(d-1)})\, \max_{\hat{u}\in \{v\in \reals^d: v^\top \Sigma v\leq (1- 20\rho_2)\|\Sigma\|, \|v\|=1\}} r_{(\varepsilon, \Delta, S)}(\hat{u})\\
		&\leq & {\rm Vol}({\mathbb S}^{(d-1)})\, \frac{1}{Z}e^{-\varepsilon \alpha n( 20\rho_2-4\rho_2) \|\Sigma\| /(320\rho_2\|\Sigma\|)}	\\
		&\leq &{\rm Vol}({\mathbb S}^{(d-1)})\, \frac{1}{Z}\exp\Big\{-\frac{\alpha \varepsilon n}{20}\Big\}
	\end{eqnarray*}
	
	This implies
	\begin{eqnarray*}
		\log\left(\frac{\prob_{\hat{u}\sim r_{(\varepsilon, \Delta, S)}}\left(\lambda_1-\hat{u}^\top \Sigma \hat{u}\leq 4\rho_2\|\Sigma\|\right)}{\prob_{\hat{u}\sim r_{(\varepsilon, \Delta, S)}}\left(\lambda_1-\hat{u}^\top \Sigma \hat{u}  \geq  20 \rho_2\|\Sigma\|\right)}\right) \; \geq \; \frac{\varepsilon\alpha n}{40}-c_2d\log(1/\rho_2)\;.
	\end{eqnarray*}
	
	If we set $n=\Omega\left(\frac{\log(1/\zeta)+d\log(1/\rho_2)}{\varepsilon\alpha }\right)$, we get
	\begin{eqnarray*}
		\frac{\prob_{\hat{u}\sim r_{(\varepsilon, \Delta, S)}}\left(\lambda_1-\hat{u}^\top \Sigma \hat{u}\leq 4\rho_2\lambda_1\right)}{\prob_{\hat{u}\sim r_{(\varepsilon, \Delta, S)}}\left(\lambda_1-\hat{u}^\top \Sigma \hat{u}\geq 20 \rho_2\lambda_1\right)}\geq \frac{2}{\zeta}\;,
	\end{eqnarray*}
	which completes the proof.

\subsection{Step 3: Achievability guarantees}
\label{sec:pca3}

We provide utility guarantees for private PCA for sub-Gaussian and hypercontractive distributions.

\subsubsection{Sub-Gaussian distributions}

Using the resilience of sub-Gaussian distributions with respect to $(\mu=0,\Sigma)$ in 
 Lemma~\ref{lem:mean_subgaussian}, which is the same as the resilience properties we need for PCA in Definition \ref{def:resilience_pca},  Theorem~\ref{thm:robust_pca_utility} implies  the following corollary.

\begin{coro} 
    \label{coro:pca} 
    Under the hypothesis of Lemma~\ref{lem:mean_subgaussian} with $\mu=0$ and any PSD matrix $\Sigma\in\reals^{d\times d}$, there exist universal constants $c$ and $C>0$ such that for any $\alpha\in(0,c)$, a dataset of size 
     $$n \;=\; O \left(\frac{d+\log(1/\zeta)}{(\alpha\log(1/\alpha))^2} + \frac{\log(1/(\delta\zeta))+d\log(1/(\alpha\log(1/\alpha)))}{\varepsilon \alpha}       \right)   \;,$$ 
     and sensitivity of $\sens= O( \log(1/\alpha)/n)$ with large enough constants are sufficient for  $\HPTR (S)$ in Section~\ref{sec:pca1} for PCA
      with the choices of the distance function in Eq.~\eqref{score:pca}    
     to achieve 
     \begin{eqnarray} 1-\frac{\hat{u}^\top \Sigma \hat{u}}{\|\Sigma\|} \;\leq\; C\alpha\log(1/\alpha)  \;,
     \label{eq:pca_subgauss}
     \end{eqnarray}
     with  probability $1-\zeta$. 
     Further, the same guarantee holds even if $ \alpha$-fraction of the samples are arbitrarily corrupted as in Assumption~\ref{asmp:mean}. 
\end{coro}
 The error bound is near-optimal under $\alpha$-corruption, matching a lower bound up to a factor of $O(\log(1/\alpha))$. HPTR is the first estimator that guarantees $(\varepsilon,\delta)$-DP and also achieves the  robust error rate of $ 1- \hat u^\top \Sigma \hat u /\|\Sigma\| = O(\alpha\log(1/\alpha))$, nearly matching the information theoretic lower bound of $ 1- \hat u^\top \Sigma \hat u /\|\Sigma\|=\Omega(\alpha)$. This lower  bound, which can be easily constructed using ${\cal N}(0,{\bf I}+\alpha e_1e_1^\top)$ and ${\cal N}(0,{\bf I}+\alpha e_2e_2^\top)$, holds for any estimator that is not necessarily private and regardless of how many samples are available.  If privacy is not required,  near-optimal robust error rate can be achieved by outlier-robust PCA approaches in \cite{kong2020robust,jambulapati2020robust}. 

 The sample complexity is near-optimal, matching a lower bound up to a factor of $O(\log(1/\alpha))$ when  $\delta=e^{-\Theta(d)}$. Even for DP PCA without corrupted samples, HPTR is the first estimator for sub-Gaussian distributions to nearly match the information-theoretic lower bound of $n = \Omega(d/(\alpha\log(1/\alpha))^2 + \min\{d,\log((1-e^{-\varepsilon})/\delta)\}/(\varepsilon\alpha\log(1/\alpha)))$ to achieve the error in Eq.~\eqref{eq:pca_subgauss}. 
 The first term is unavoidable as even without DP and robustness, when the data comes from a Gaussian distribution, estimating the principal component up to error $\alpha\log(1/\alpha)$ requires $\Omega(d/(\alpha \log(1/\alpha))^2)$ samples (Proposition~\ref{prop:lowerbound_pca}). 
 The second term in the lower bound follows from Proposition~\ref{thm:lowerbound_pca_subgaussian}, which matches the second term  in the upper bound up to a factor of $O(\log(1/\alpha))$ when $\delta=e^{-\Theta(d)}$ and $\varepsilon>0$.  
 Existing DP PCA approaches from \cite{PPCA, kapralov2013differentially,dwork2014analyze} are designed for arbitrary samples not  necessarily drawn i.i.d., and hence require a larger samples size of $n=\tilde O(d/\alpha^2 + d^{1.5}\sqrt{\log(1/\delta)}/(\alpha\varepsilon))$ i.i.d.~samples from a Gaussian distribution to achieve the guarantee in Eq.~\eqref{eq:pca_subgauss}, where $\tilde O$ hides polylogarithmic terms in $1/\alpha$ and $1/\zeta$.
 
 \medskip 
 \noindent
 {\bf Remark.} 
Rank-$k$ PCA under $\alpha$-corruption from a Gaussian dataset is of great practical interest. 
An outlier-robust PCA algorithm in \cite[Appendix D]{kong2020robust} outputs an  orthonormal matrix $\hat U\in\reals^{d\times k}$ achieving 
\begin{eqnarray*}
    {\rm Tr}(U_k^\top\Sigma U_k) - {\rm Tr}(\hat U^\top \Sigma \hat U) \;=\; O\big(\,\alpha {\rm Tr}(U_k^\top \Sigma U_k) + \nu k^{1/2} \alpha\log(1/\alpha)\, \big) \;,
\end{eqnarray*}
where $U_k\in\arg\max_{U^\top U={\bf I}_{k\times k} } U^\top \Sigma U$ and $\nu^2= \max_{V\in\reals^{d\times d}, \|V\|_F=1, V=V^\top, {\rm rank}(V)\leq k} \langle V,\Sigma V\Sigma \rangle $.
It is a promising direction to design a DP rank-$k$ PCA algorithm by applying the HPTR framework that can achieve a similar error rate.  It is not immediate how to 
design an appropriate score function for general rank $k$, and 
a simple technique of peeling off rank-one components one-by-one (using the rank-one PCA with HPTR) will not achieve the target error bound.

%\Xiyang{captured variance $v^\top \Sigma v\geq \lambda_1-\alpha$ bounded distributions: $\alpha=O(d^{1.5}/(n\varepsilon\alpha))$ \cite{PPCA, kapralov2013differentially}, $\alpha=O(d/(n\varepsilon\alpha))$ gaussian mechanism \cite{dwork2014analyze}, power methods}

\begin{propo}[Lower bound for private sub-Gaussian PCA]
\label{thm:lowerbound_pca_subgaussian}
	Let $\cP_{\Sigma}$ be the set of zero-mean sub-Gaussian distributions with covariance $\Sigma\in \reals^{d\times d}$.  Let  $\cM_{\varepsilon,\delta}$ be a class of $(\varepsilon, \delta)$-DP, $d$-dimensional estimators of the top principal component of $\Sigma$  using $n$ i.i.d.~samples from $P\in{\cal P}_{\Sigma}$.
	Then, for $\varepsilon\in(0,10)$,  there exists a universal constant $c>0$ such that 
	\begin{eqnarray*}
	\inf_{\hat{u}\in \cM_{\varepsilon,\delta}}
	\sup_{\Sigma\succ 0 ,P\in \cP_{ \Sigma}}\E_{S\sim  P^n}\left[1-\frac{\hat{u}(S)^\top\Sigma\hat{u}(S)}{\|\Sigma\|}\right]
	\; \geq\; c \cdot 
	\min\left\{\frac{d\wedge \log((1-e^{-\varepsilon})/\delta)}{n\varepsilon}, 1\right\}\;.
	\end{eqnarray*}
	
\end{propo}

\begin{proof}
    We adopt the same proof strategy as the proof of Proposition~\ref{thm:lowerbound_mean_hypercontractive} for mean estimation. By {\cite[Lemma~6]{acharya2021differentially}}, there exists a finite index set $\cV\subset \reals^d$ with cardinality $|\cV|=2^{\Omega(d)}$, $\|v\|=1$ for all $v\in \cV$ and $\|v-v'\|\geq 1/2$ for all $v\neq v'\in \cV$. 
    For each $v\in \cV$, we define $\Sigma_v :=\mathbf{I}_{d \times d}+\alpha vv^\top$ and $P_v:=\cN(0, \Sigma_v)$ for some $\alpha\in (0,1/2)$. It is easy to see that $\mathbf{I}_{d \times d} \preceq\Sigma_v\preceq 3\mathbf{I}_{d \times d}/2$ and the top eigenvector of $\Sigma_v$ is $v$. For $v\neq v'\in \cV$, we know $\|\Sigma_{v'}^{-1/2}\Sigma_v\Sigma_{v'}^{-1/2}-\mathbf{I}_{d \times d}\|_F=O(\alpha)$. By {\cite[Lemma~2.9]{KLSU19}}, this implies $d_{\rm TV}(\cN(0, \Sigma_v), \cN(0, \Sigma_v'))=O(\alpha)$.
    
    % Define distance function $\rho(P_{\Sigma_v}, P_{\Sigma_{v'}}):=\frac{1}{2}(D_{\Sigma_{v'}}(u_v)+D_{\Sigma_v}(u_{v'}))$, where $u_v=v$ and $u_{v'}=v'$ are the top eigenvector of $\Sigma_v$ and $\Sigma_{v'}$ respectively.
    
    Since $\|v-v'\|\geq 1/2$, we have
    \begin{eqnarray*}
        D_{\Sigma_{v'}}(v) = 1-\frac{v^\top \Sigma_v'v}{\|\Sigma_{v'}\|} = 1-\frac{1+\alpha\ip{v}{v'}^2}{1+\alpha}\geq \frac{\alpha}{8(1+\alpha)}> \frac{\alpha}{12}\;.
    \end{eqnarray*}
    
    The principal component estimation problem can be reduced to a testing problem  with this packing $\cV$. For $(\varepsilon, \delta)$-DP estimator $\hat{u}$, using Lemma~\ref{thm:packing}, let $t=\frac{\alpha^{1-2/k}}{12}$, we have
\begin{eqnarray*}
	\sup_{P\in \cP_\Sigma}\E_{S \sim P^n}[D_{\Sigma}(\hat{u})]
		&\geq & \frac{1}{|\cV|}\sum_{v\in \cV}\E_{S \sim P_v^n}[D_{\Sigma_v}(\hat{u})]\\
		&= &\frac{1}{|\cV|}\sum_{v\in \cV}P_{v}\left(D_{\Sigma_v}(\hat{u})\geq t\right)\\
		&\gtrsim &t \frac{e^{d/2} \cdot\left(\frac{1}{2} e^{-\varepsilon\lceil n \alpha\rceil}-\frac{\delta}{1-e^{-\varepsilon}}\right)}{1+e^{d/2} e^{-\varepsilon\lceil n \alpha\rceil}}\;,
\end{eqnarray*}
where the last inequality follows from the fact that $d\geq 2$. 
The rest of the proof follows from  {\cite[Proposition 4]{barber2014privacy}}. We choose 
\begin{eqnarray*}
	\alpha =\frac{1}{n \varepsilon} \min \left\{\frac{d}{2}-\varepsilon, \log \left(\frac{1-e^{-\varepsilon}}{4 \delta e^{\varepsilon}}\right)\right\}
\end{eqnarray*}
 so that 
 \begin{eqnarray*}
 	\sup_{P\in \cP_\Sigma}\E_{S\sim P^n}[D_{\Sigma_v}(\hat{u})] \gtrsim  \alpha \;.
 \end{eqnarray*}
 This implies, for $t=\alpha/12$ and  $\varepsilon\in(0,10)$, 
 \begin{eqnarray*}
 	\inf_{\hat{u}\in \cM_{\varepsilon, \delta}}\sup_{\Sigma\succ 0 ,P\in \cP_\Sigma }\E_{S \sim P^n}[D_{\Sigma}(\hat{u})]
 	\;\gtrsim \; \min\left\{\frac{d\wedge \log((1-e^{-\varepsilon})/\delta)}{n\varepsilon}, 1\right\}\;,
 \end{eqnarray*}
which completes the proof.
\end{proof}
It is well known that even for Gaussian distribution, learning the principal component up to error $\alpha$ requires $\Omega(d/\alpha^2)$. We provides a lower bound proof here for completeness.
\begin{propo}[Sample Complexity Lower bound for PCA]
\label{prop:lowerbound_pca}
	Let $\cP_{\Sigma}$ be the set of zero-mean Gaussian distributions with covariance $\Sigma\in \reals^{d\times d}$.  Let $\cM_{d}$ be the class of estimators of the $d$-dimensional top principal component of $\Sigma$ using $n$ i.i.d.~samples from $P\in{\cal P}_{\Sigma}$. There exists a universal constant $c>0$ such that 
	\begin{eqnarray*}
	\inf_{\hat{u}\in \cM_{d}}
	\sup_{\Sigma\succ 0 ,P\in \cP_{ \Sigma}}\E_{S\sim  P^n}\left[1-\frac{\hat{u}(S)^\top\Sigma\hat{u}(S)}{\|\Sigma\|}\right]
	\; \geq\; c \cdot 
	\min\left\{\sqrt{\frac{d}{n}}, 1\right\}\;.
	\end{eqnarray*}
\end{propo}
\begin{proof}

The following proposition will help us prove a minimax lower bound on estimating $\|\Sigma\|$. Let us first define some notations.
\begin{definition}[Definition 3.1 in~\cite{diakonikolas2017statistical}]
For a distribution A on the real line with probability density function $A(x)$ and a unit vector $v \in\reals^d$, consider the distribution over $\reals^n$ with probability density function
$P_v(x) = A(v^\top x) \exp(-\|x - (v^\top x) v\|_2^2/2)\cdot (2\pi)^{-(d-1)/2}$
\end{definition}

\begin{propo}[Proposition 7.1 in~\cite{diakonikolas2017statistical}]\label{prop:op-lowerbound} Let $A$ be a distribution on $\reals$ such that $A$ has mean $0$ and $\chi^2
(A, N(0, 1))$ is finite. Then, there is no algorithm that, for any $d$, given $n < d/(8\chi^2
(A, N(0, 1)))$ samples from a distribution $D$ over
$\reals^d$ which is either $N(0, I)$ or $P_v$, for some unit vector $v \in \reals^d$, correctly distinguishes between the two cases with probability at least $2/3$.
\end{propo}
To apply Proposition~\ref{prop:op-lowerbound}, let $A$ be Gaussian distribution $\cN(0, 1+\alpha)$. Through simple calculation, it can be shown that 
$\chi^2(\cN(0,1), \cN(0,1+\alpha))=\frac{1}{\sqrt{1-\alpha^2}}-1\le \alpha^2$ whenever $\alpha^2\le 1/2$. Then for the first case in Proposition~\ref{prop:op-lowerbound}, $\|\Sigma\|=\|I\|=1$, the second case has $\|\Sigma\| = 1+\alpha$, and Proposition~\ref{prop:op-lowerbound} implies there exists absolute constant $c$ such that
\begin{eqnarray*}
	\inf_{\hat{\lambda}}
	\sup_{\Sigma\succ 0 ,P\in \cP_{ \Sigma}}\E_{S\sim  P^n}\left[1-\frac{\hat\lambda(S)}{\|\Sigma\|}\right]
	\; \geq\; c \cdot 
	\min\left\{\sqrt{\frac{d}{n}}, 1\right\}\;.
\end{eqnarray*}
Since we can turn a principal component estimator $u(S)$ into an estimator of $\|\Sigma\|$ through $n$ additional fresh samples to estimate ${u}(S)^\top\Sigma{u}(S)$ up to a minor multiplicative error $O(1/\sqrt{n})$. This implies there exists a universal constant $c>0$ such that 
	\begin{eqnarray*}
	\inf_{\hat{u}\in \cM_{d}}
	\sup_{\Sigma\succ 0 ,P\in \cP_{ \Sigma}}\E_{S\sim  P^n}\left[1-\frac{\hat{u}(S)^\top\Sigma\hat{u}(S)}{\|\Sigma\|}\right]
	\; \geq\; c \cdot 
	\min\left\{\sqrt{\frac{d}{n}}, 1\right\}\;.
	\end{eqnarray*}
	\end{proof}
\subsubsection{Hypercontractive distributions}
In this section, we apply our results on hypercontractive distributions in Definition~\ref{def:hyper}. Using the resilience of hypercontractive distributions with respect to $(\mu=0,\Sigma)$ in 
 Lemma~\ref{lem:mean_kmoment}, which is the same as the resilience properties we need for PCA in Definition \ref{def:resilience_pca},  Theorem~\ref{thm:robust_pca_utility} implies  the following corollary.

\begin{coro} 
    \label{coro:pca_hyper} 
    Under the hypothesis of Lemma~\ref{lem:mean_kmoment} with $k\geq 3$, $\mu=0$ and any PSD matrix $\Sigma\in\reals^{d\times d}$, there exist universal constants $c$ and $C>0$ such that for any $\alpha\in(0,c)$, a dataset of size 
     $$n \;=\; O \left(\frac{d}{\zeta^{2(1-1/k)}\alpha^{2(1-1/k)}} + \frac{k^2\alpha^{2-2/k} d\log d}{\zeta^{2-4/k}\kappa^2} + \frac{\kappa^2 d\log d}{\alpha^{2/k}} + \frac{\log(1/(\delta\zeta))+d\log(1/\alpha^{1-2/k})}{\varepsilon \alpha}       \right)   \;,$$ 
     and sensitivity of $\sens= O( \alpha^{1-2/k}/n)$ with large enough constants are sufficient for  $\HPTR (S)$ in Section~\ref{sec:pca1} for PCA
      with the choices of the distance function in Eq.~\eqref{score:pca}    
     to achieve 
     \begin{eqnarray} 1-\frac{\hat{u}^\top \Sigma \hat{u}}{\|\Sigma\|} \;\leq\; C\alpha^{1-2/k} \;,
     \label{eq:pca_hyper}
     \end{eqnarray}
     with  probability $1-\zeta$. 
     Further, the same guarantee holds even if $ \alpha$-fraction of the samples are arbitrarily corrupted as in Assumption~\ref{asmp:mean}. 
\end{coro}

The error bound is optimal under $\alpha$-corruption up to a constant factor. HPTR is the first estimator that guarantees $(\varepsilon,\delta)$-DP and also achieves the robust error rate of $1-\hat u^\top \Sigma \hat u/\|\Sigma\| = O(\alpha^{1-2/k})$, matching the information theoretic lower bound of $1-\hat u^\top \Sigma \hat u/\|\Sigma\|=\Omega(\alpha^{1-2/k})$. 
This lower bound can be easily constructed using the construction in  Eq.~\eqref{eq:construct_hypercontractive_x}, where two hypercontractive distributions are at total variation distance $O(\alpha)$ and the top principal component of one distribution achieves an error lower bounded by $1-\hat u^\top \Sigma \hat u/\|\Sigma\|=\Omega(\alpha^{1-2/k})$.
Even if privacy is not required, there is no outlier-robust PCA estimator matching this optimal error rate for general $k$.

The sample complexity is $n=\tilde O(d/\alpha^{2(1-1/k)} + (d+\log(1/\delta))/(\varepsilon\alpha) )$ for constant $\zeta,k$, and $\kappa$, where $\tilde O$ hides logarithmic factors in $1/\alpha$ and $d$. 
Even for DP PCA without corrupted samples, HPTR is the first estimator for hypercontractive distributions to 
guarantee differential privacy.  
 The information-theoretic lower bound is  $n=\Omega(d/\alpha^{2(1-2/k)} +  \min\{d,\log((1-e^{-\varepsilon})/\delta)\}/(\alpha \varepsilon) )$ to achieve the error in Eq.~\eqref{eq:pca_hyper}.  The first term is unavoidable  even without DP and robustness, when the data comes from a Gaussian distribution, because estimating the principal component up to error $\alpha^{1-2/k}$ requires $\Omega(d/\alpha^{2(1-2/k)})$ samples (Proposition~\ref{prop:lowerbound_pca}). 
 There is a gap of factor $O(\alpha^{-2/k})$ compared to the first term in our upper bound.  Since the sample complexity lower bound in Proposition~\ref{prop:lowerbound_pca} is constructed using  Gaussian distributions, 
 it might be possible to tighten it further using hypercontractive distributions. 
The second term in the lower bound follows from  Proposition~\ref{thm:lowerbound_pca_hyper}, which matches the last term  in the upper bound up to a factor of $O(\log(1/\alpha))$ when $\delta= e^{-\Theta(d)}$ and $\varepsilon>0$. 
 To the best of our knowledge, HPTR is the first algorithm for PCA that guarantees $(\varepsilon,\delta)$-DP under hypercontractive distributions. 
 
\begin{propo}[Lower bound for hypercontractive private PCA]
\label{thm:lowerbound_pca_hyper}
	Let $\cP_{\Sigma}$ be the set of zero-mean hypercontractive distributions with covariance $\Sigma\in \reals^{d\times d}$.  Let  $\cM_{\varepsilon,\delta}$ be a class of $(\varepsilon, \delta)$-DP estimators using $n$ i.i.d.~samples from $P\in{\cal P}_{\Sigma}$.
	Then, for $\varepsilon\in(0,10)$, there exists a constant $c$ such that 
	\begin{eqnarray}
	\inf_{\hat{u}\in  \cM_{\varepsilon,\delta}}
	\sup_{\Sigma\succ 0 ,P\in \cP_{ \Sigma}}\E_{S\sim  P^n}\left[1-\frac{\hat{u}^\top\Sigma\hat{u}}{\|\Sigma\|}\right]
	\; \geq\; c 
	\min\left\{\left(\frac{d\wedge \log((1-e^{-\varepsilon})/\delta)}{n\varepsilon}\right)^{1-2/k}, 1\right\}\;.
	\end{eqnarray}
	
\end{propo}

\begin{proof}
    We use the same construction as the distribution of $x$ in the proof of Proposition~\ref{thm:lowerbound_regression_dependent}.
    By {\cite[Lemma~6]{acharya2021differentially}}, there exists a finite index set $\cV\subset \reals^d$ with cardinality $|\cV|=2^{\Omega(d)}$, $\|v\|=1$ for all $v\in \cV$ and $\|v-v'\|\geq 1/2$ for all $v\neq v'\in \cV$.  For each $v\in \cV$ and $\alpha\in (0,1/2)$, we construct the density function of distribution $P_v$ as defined in Eq.~\eqref{eq:construct_hypercontractive_x}. Let $\Sigma_v$ denote the covariance matrix of $P_v$. The proof of Proposition~\ref{thm:lowerbound_regression_dependent} shows that $\Sigma_v= (1-\alpha)\mathbf{I}_{d \times d}+\alpha^{1-2/k} vv^\top$, $d_{\rm TV}(P_v, P_v')=\alpha$ and that $P_v$ is $(O(1), k)$-hypercontractive.

    Since $\|v-v'\|\geq 1/2$, we know $\ip{v}{v'}\leq 7/8$ and we have
    \begin{eqnarray*}
        D_{\Sigma_{v'}}(v) = 1-\frac{v^\top \Sigma_v'v}{\|\Sigma_{v'}\|} = 1-\frac{1-\alpha+\alpha^{1-2/k}\ip{v}{v'}^2}{1-\alpha+\alpha^{1-2/k}}\geq \frac{\alpha^{1-2/k}}{8(1-\alpha+\alpha^{1-2/k})}> \frac{\alpha^{1-2/k}}{12}\;,
    \end{eqnarray*}
    for $\alpha<c$ small enough.
    
    Next, we apply the reduction of estimation to testing with this packing $\cV$. For $(\varepsilon, \delta)$-DP estimator $\hat{u}$, using Lemma~\ref{thm:packing}, let $t=\frac{\alpha^{1-2/k}}{12}$, we have
\begin{eqnarray*}
	\sup_{P\in \cP_\Sigma}\E_{S \sim P^n}[D_{\Sigma}(\hat{u})]
		&\geq & \frac{1}{|\cV|}\sum_{v\in \cV}\E_{S \sim P_v^n}[D_{\Sigma_v}(\hat{u})]\\
		&= &\frac{1}{|\cV|}\sum_{v\in \cV}P_{v}\left(D_{\Sigma_v}(\hat{u})\geq t\right)\\
		&\gtrsim &t \frac{e^{d/2} \cdot\left(\frac{1}{2} e^{-\varepsilon\lceil n \alpha\rceil}-\frac{\delta}{1-e^{-\varepsilon}}\right)}{1+e^{d/2} e^{-\varepsilon\lceil n \alpha\rceil}}\;,
\end{eqnarray*}
where the last inequality follows from the fact that $d\geq 2$.

The rest of the proof follows from  {\cite[Proposition 4]{barber2014privacy}}. We choose 
\begin{eqnarray*}
	\alpha =\frac{1}{n \varepsilon} \min \left\{\frac{d}{2}-\varepsilon, \log \left(\frac{1-e^{-\varepsilon}}{4 \delta e^{\varepsilon}}\right)\right\}
\end{eqnarray*}
 so that 
 \begin{eqnarray*}
 	\sup_{P\in \cP}\E_{S\sim P^n}[D_{\Sigma_v}(\hat{u})] \gtrsim  \alpha^{1-2/k} \;.
 \end{eqnarray*}
 This means, for $t=(1/12)\alpha^{1-2/k}$ and $\varepsilon\in(0,10)$,
 \begin{eqnarray*}
 	\inf_{\hat{u}\in \cM_{\varepsilon, \delta}}\sup_{P\in \cP}\E_{S \sim P^n}[D_{\Sigma}(\hat{u})]\gtrsim \min\left\{\left(\frac{d\wedge \log((1-e^{-\varepsilon})/\delta)}{n\varepsilon}\right)^{1-2/k}, 1\right\}\;,
 \end{eqnarray*}
which completes the proof.
\end{proof}

% ---------------------------------

\section{Conclusion}

We provide a universal framework for characterizing the statistical efficiency of statistical estimation problems with differential privacy guarantees. Our framework, which we call High-dimensional Propose-Test-Release (HPTR), is computationally inefficient and builds upon three key components: the exponential mechanism, robust statistics, and the Propose-Test-Release mechanism. 
The key insight is that if we design an exponential mechanism that accesses the data only via one-dimensional robust statistics, then the resulting local sensitivity can be dramatically reduced. Using resilience, which is a central concept in robust statistics, we can  provide tight local sensitivity bounds.
These tight bounds readily translate into near-optimal utility guarantees in several statistical estimation problems of interest: mean estimation, linear regression, covariance estimation, and principal component analysis. Although our framework is written as a conceptual algorithm without a specific implementation, it is possible to implement it with exponential computational complexity following the guidelines of \cite{brown2021covariance} where a similar exponential mechanism with PTR was proposed and an implementation was explicitly provided. 

To protect against membership inference attacks, significant progress was made in training differentially private models that are practical \cite{abadi2016deep,yu2021differentially,anil2021large}. To protect against data poisoning attacks, a recent work utilizes robust statistics with a great success \cite{hayase2021defense}. In practice, however, we need to protect against both types of attacks, to facilitate learning and analysis from shared data. Currently,  there is an algorithmic deficiency in this space. Efficient algorithms achieving both  differential privacy and robustness against adversarial corruption are known only for mean estimation \cite{liu2021robust}. It is an important direction to  design such algorithms for a broad class of problems, including covariance estimation, principal component analysis, and linear regression. 

Further, these computationally efficient algorithms typically require more samples. For sub-Gaussian mean estimation with known covariance $\Sigma$, an efficient approach of \cite{liu2021robust} requires $\tilde O(d/\alpha^2 + d^{3/2}/(\varepsilon\alpha))$ samples under $\alpha$-corruption and $(\varepsilon,\delta)$-DP to achieve an error of $\|\Sigma^{-1/2}(\hat\mu-\mu)\|=\tilde O(\alpha)$. HPTR only requires $O(d/\alpha^2 +  d/(\varepsilon\alpha))$ samples. It remains an important open question if this $d^{1/2}$ gap is fundamental and cannot be improved.

% -----------------------------
\section*{Acknowledgement}

This work is supported by 
Google faculty research award, NSF grants CNS-2002664, IIS-1929955, DMS-2134012 and CCF-2019844 as a part of Institute for Foundations of Machine Learning (IFML) and CNS-2112471 as a part of AI Institute for Future Edge Networks and Distributed Intelligence (AI-EDGE).

% -------------------------------------
\bibliographystyle{alpha}
\bibliography{references} 

\appendix
%\section*{Appendix}
\newpage
\section*{Appendix: Complete Proofs}

% ------------------------------------

\section{General case: utility analysis of HPTR}
\label{sec:utility}

We prove the following theorem that provides a utility guarantee for HPTR output $\hat\theta$  measured in $\popdist(\hat\theta,\theta)$.

\begin{thm} 
\label{thm:utility} 
For a given dataset $S$, a target error function $\popdist:\reals^p\times \reals^p\to\reals_+$, probability $\zeta\in(0,1)$, and privacy $(\varepsilon,\delta)$,  
    {\HPTR} achieves $\popdist(\hat\theta,\theta) = c_0 \rho $ for some $\rho>0$ and any constant $c_0>3c_1$ with probability $1-\zeta$ 
if there exist constants $c_1,c_2 >0$ 
and $(\sens\in\reals^+,\rho\in\reals^+)$ 
such that with the choice of 
$k^*=(2/\varepsilon)\log(4/(\delta\zeta))$, 
$\thresh=(c_0+c_1)\rho$, 
the following assumptions are satisfied: 
    \begin{enumerate}[label=(\alph*)] 
        \item (Bounded volume) $(7/8)\thresh-(k^*+1)\sens>0$,  
        \begin{eqnarray*}
        \frac{{\rm Vol}(B_{ \thresh+(k^*+1)\sens+c_1\rho,S}) }
        {{\rm Vol}(B_{ (7/8)\thresh-(k^*+1)\sens-c_1\rho,S})}
        &\leq& e^{c_2 p}\;, \text{ and }\\
        \frac{{\rm Vol}(\{\hat\theta:\popdist (\hat \theta ,\theta)\leq (c_0+2c_1)\rho\})}
        {{\rm Vol}(\{\hat\theta: \popdist (\hat \theta ,\theta)\leq c_1\rho\})}
        &\leq &e^{c_2 p}\;,
        \end{eqnarray*}
        \label{asmp_vol}
        \item (Local sensitivity) For all $S'$ within Hamming distance $k^*$ from S,
    $\max_{S''\sim S'} \|\robdist_{S''}(\hat\mu) - \robdist_{S'}(\hat\mu)\|\leq   \sens $ for all 
    $\hat\mu\in B_{\tau+\backoff,S} $, \label{asmp_local}
        \item (Bounded sensitivity) $\sens\leq \frac{(c_0-3c_1)\rho   \varepsilon}{32(c_2 p + (\varepsilon/2) + \log(16/\delta\zeta))}$, and  \label{asmp_sens} 
        \item (Robustness) $|\popdist(\hat\theta,\theta) - \robdist_S(\hat\theta)| \leq  c_1 \rho$ for all $\hat\theta\in B_{\thresh,S}$.  \label{asmp_resilience}
    \end{enumerate} 
\end{thm}
% ---------------------------------------------
The parameter $\rho\in\reals_+$ represents the target error up to a constant factor and depends on the resilience of the underlying distribution $P_{\theta,\phi}$ that the samples are drawn from. We explicitly prescribe how to choose the parameter $\rho$ for each problem instance in Sections~\ref{sec:mean}, \ref{sec:lin}, \ref{sec:cov}, and \ref{sec:pca}. 
Following the standard analysis techniques for exponential mechanisms, we show that the output concentrates around an inner set $\{\hat\theta:\popdist(\hat\theta,\theta)\leq c_0\rho\}$, by comparing its probability mass with an outer set $\{\hat\theta:\popdist(\hat\theta,\theta)\geq c_1\rho\}$. This uses the ratio of the volumes in the assumption~\ref{asmp_vol} and the closeness of the error metric and $\robdist(\hat\theta)$ in the assumption~\ref{asmp_resilience}. 
When there is a strict gap between the two, which happens if $ \varepsilon\rho/\sens \gg  p + \log(1/\zeta)$ as in the  assumption~\ref{asmp_sens}, this implies $\popdist(\hat\theta,\theta)\leq c_0 \rho$ with probability $1-\zeta$. We provide a proof in Section~\ref{sec:proof_utility}.

A major challenge in analyzing HPTR is in showing that the safety test threshold $k^*=(2/\varepsilon)\log(4/(\delta\zeta))$ is not only large enough to ensure that datasets with safety violation is screened with probability $1-\delta/2$ but also small enough such that good datasets satisfying the assumptions~\ref{asmp_vol}, \ref{asmp_local}, and \ref{asmp_sens} pass the test with probability $1-\zeta/2$. 
We establish this first in Section~\ref{sec:proof_safetymargin}.

% ---------------------------------------------

\subsection{Large safety margin}
\label{sec:proof_safetymargin}

In this section, we show in Lemma~\ref{lem:safetymargin} that under the assumptions of  Theorem~\ref{thm:utility}, we get a large enough margin for safety such that we pass the safety test with high probability.
We follow the proof strategy introduced in \cite{brown2021covariance}  adapted to our more general framework. 
A major challenge is the lack of a uniform bound on the sensitivity, which the analysis of \cite{brown2021covariance}  relies on. 
We generalize the analysis by showing that while the data does not satisfy uniform sensitivity bound, we can still exploit its 
{\em local} sensitivity bound in the assumption~\ref{asmp_local}. 

The following main technical lemma is a counter part of \cite[Lemma 3.7]{brown2021covariance}, where we have an extra challenge that the sensitivity bound is only local; there exists $\hat\theta$ far from $\theta$ where the sensitivity bound fails.  We rely on the assumption~\ref{asmp_local} to resolve it. 
Let $w_S(B) \triangleq \int_{B} \exp\{-(\varepsilon/4\sens) \robdist_S(\hat\mu)\}d\hat\mu$ be the weight of a subset $B\subset \reals^p$.
The following lemma will be used to show that the denominator of the exponential distribution in {\sc Release} step does not change too fast between two neighboring datasets.

\begin{lemma} 
\label{lem:safe_ratio} Under  the assumption~\ref{asmp_local} 
and $\delta\in (0,1/2)$, 
for a dataset $S'$ at Hamming distance at most $k^*$ from $S$, 
if  $w_{S'}(B_{\thresh -\sens,S'}) \geq (1-\delta) w_{S'}(B_{ \thresh +\sens,S'})$ then $S' \in\safe_{\varepsilon,4e^{2\varepsilon}\delta,\thresh}$. 
\end{lemma}

\begin{proof}
We follow the proof strategy of  \cite[Lemma 3.7]{brown2021covariance} but there are key differences due to the fact that we do not have a universal sensitivity bound, but only local bound. 
In particular, we first establish that under the local sensitivity assumption,  
$  B_{\thresh,S''}\subseteq B_{\thresh + \sens,S'}$ for all $S''\sim S'$, which will be used heavily throughout the proof.  
Since $\robdist_{S''}(\hat\theta) \leq \robdist_{S'}(\hat\theta)+\sens $ for all $\hat\theta\in B_{\thresh+\backoff,S}$, we have 
$B_{\thresh,S''}\cap B_{\thresh+\backoff,S} \subseteq   B_{\thresh+\sens,S'}$. 
We are left to show that 
$B_{\thresh,S''}\setminus B_{\thresh+\backoff,S}=\emptyset$, which follows from the fact that 
$(B_{\thresh,S''}\setminus B_{\thresh+(k^*+1.5)\sens,S})\cap B_{\thresh+\backoff,S})=\emptyset$ and $D_{S''}(\hat\theta)$ is a Lipschitz continuous function. Similarly, it follows that $  B_{\thresh - \sens,S'}\subseteq B_{\thresh,S''}$.
In particular, this implies that $B_{\thresh,S'} \subseteq B_{\thresh+\backoff,S}$ for any $S'$ with $d_H(S',S)\leq k^*$.

We first show that for any $E \subset B_{\thresh,S'}$ one side of the $(\varepsilon/2,4e^{\varepsilon/2}\delta)$-DP condition is met: 
$\prob_{\hat\theta\sim r_{(\varepsilon,\sens,\thresh,S')}} 
    ( \hat\theta \in E)\leq e^{\varepsilon /2}
    \prob_{\hat\theta\sim r_{(\varepsilon,\sens,\thresh,S'')}} (\hat\theta\in E) + 4e^{\varepsilon/2}\delta$ for all $S''\sim S'$ 
where $r_{(\varepsilon,\sens,\thresh,S')}$ and $r_{(\varepsilon,\sens,\thresh,S'')}$ are the distributions used in the exponential mechanism as defined in \eqref{def:exp} respectively.
For $B = B_{\thresh,S'} \cap B_{\thresh,S''} $, we have 
\begin{eqnarray*}
    \prob_{\hat\theta\sim r_{(\varepsilon,\sens,\thresh,S')}} 
    ( \hat\theta \in E) &=& \prob_{\hat\theta\sim r_{(\varepsilon,\sens,\thresh,S')}} 
    ( \hat\theta \in E\cap B) + \prob_{\hat\theta\sim r_{(\varepsilon,\sens,\thresh,S')}} 
    ( \hat\theta \in E \setminus B) \\
    &=& \frac{\prob_{\hat\theta\sim r_{(\varepsilon,\sens,\thresh,S')}} 
    ( \hat\theta \in E\cap B)}{\prob_{\hat\theta\sim r_{(\varepsilon,\sens,\thresh,S'')}} 
    ( \hat\theta \in E\cap B)}\prob_{\hat\theta\sim r_{(\varepsilon,\sens,\thresh,S'')}} 
    ( \hat\theta \in E\cap B) + \prob_{\hat\theta\sim r_{(\varepsilon,\sens,\thresh,S')}} 
    ( \hat\theta \in E \setminus B) \\
    &\leq& \frac{\prob_{\hat\theta\sim r_{(\varepsilon,\sens,\thresh,S')}} 
    ( \hat\theta \in E\cap B)}{\prob_{\hat\theta\sim r_{(\varepsilon,\sens,\thresh,S'')}} 
    ( \hat\theta \in E\cap B)}\prob_{\hat\theta\sim r_{(\varepsilon,\sens,\thresh,S'')}} 
    ( \hat\theta \in E) + \prob_{\hat\theta\sim r_{(\varepsilon,\sens,\thresh,S')}} 
    ( \hat\theta \not\in B_{\thresh,S''})\;.
\end{eqnarray*}
The ratio is bounded due to the local sensitivity bound at $S'$ as   
\begin{eqnarray*}
\frac{\prob_{\hat\theta\sim r_{(\varepsilon,\sens,\thresh,S')}} 
    ( \hat\theta \in E\cap B)}{\prob_{\hat\theta\sim r_{(\varepsilon,\sens,\thresh,S'')}} 
    ( \hat\theta \in E\cap B)} &\leq& e^{\varepsilon/4}\frac{w_{S''}(B_{\thresh,S''})}{w_{S'}(B_{\thresh,S'})} \\
    &\leq& e^{\varepsilon/2}\frac{w_{S'}(B_{\thresh,S''})}{w_{S'}(B_{\thresh,S'})}\\
    &\leq& 
    e^{\varepsilon/2}\frac{w_{S'}(B_{\thresh+\sens,S})}{w_{S'}(B_{\thresh,S'})}\;\leq\; e^{\varepsilon/2}(1+2\delta)
    \; ,
\end{eqnarray*}
where the second inequality follows from the fact that 
$w_{S''}(A)\leq e^{\varepsilon/6}w_{S'}(A)$ for any set $A\subset B_{\thresh,S'}\cup B_{\thresh,S''}\subseteq B_{\thresh+\backoff,S}$  and the third inequality follows from the fact that 
$B_{\thresh,S''} \subseteq B_{\thresh+\sens,S'} $.
From the  assumption on the weights,  it follows that ${w_{S'}(B_{\thresh+\sens,S'})}/{w_{S'}(B_{\thresh,S'})}\leq{w_{S'}(B_{\thresh+\sens,S'})}/{w_{S'}(B_{\thresh-\sens,S'})} \leq 1/(1-\delta)\leq 1+2\delta$ for $\delta<1/2$.
Similarly, 
\begin{eqnarray*}
    \prob_{\hat\theta\sim r_{(\varepsilon,\sens,\thresh,S')}} 
    ( \hat\theta \not\in B_{\thresh,S''}) &\leq & \prob_{\hat\theta\sim r_{(\varepsilon,\sens,\thresh,S')}} 
    ( \hat\theta \not\in B_{\thresh-\sens,S'})\\
    &\leq&1- \frac{w_{S'}(B_{\thresh-\sens,S'})}{w_{S'}(B_{\thresh,S'})}\leq 1- \frac{w_{S'}(B_{\thresh-\sens,S'})}{w_{S'}(B_{\thresh+\sens,S'})} \leq \delta\;.
\end{eqnarray*}
Putting these together, we get $\prob_{\hat\theta\sim r_{(\varepsilon,\sens,\thresh,S')}} 
    ( \hat\theta \in E)\leq e^{\varepsilon /2}
    \prob_{\hat\theta\sim r_{(\varepsilon,\sens,\thresh,S'')}} (\hat\theta\in E) + 4e^{\varepsilon/2}\delta$.

Next, we show the other side of the $(\varepsilon/2,4e^{\varepsilon/2}\delta)$-DP condition: 
$\prob_{\hat\theta\sim r_{(\varepsilon,\sens,\thresh,S')}} 
    ( \hat\theta \in E)\leq e^{\varepsilon /2}
    \prob_{\hat\theta\sim r_{(\varepsilon,\sens,\thresh,S)}} (\hat\theta\in E) + 4e^{2\varepsilon}\delta
$ for all $S'\sim 
S$. 
We need to show an upper bound on the ratio:
\begin{eqnarray*}
    \frac{\prob_{\hat\theta\sim r_{(\varepsilon,\sens,\thresh,S')}} 
    ( \hat\theta \in E\cap B)}{\prob_{\hat\theta\sim r_{(\varepsilon,\sens,\thresh,S)}} 
    ( \hat\theta \in E\cap B)} &\leq& e^{\varepsilon/4}\frac{w_{S}(B_{\thresh,S})}{w_{S'}(B_{\thresh,S'})} \\
    &\leq& e^{\varepsilon/2}\frac{w_{S}(B_{\thresh,S})}{w_{S}(B_{\thresh,S'})}\\
    &\leq& 
    e^{\varepsilon/2}\frac{w_{S}(B_{\thresh,S})}{w_{S}(B_{\thresh-\sens,S})}
    \;\leq\; (1+2\delta)e^{\varepsilon/2}\; ,
\end{eqnarray*}
For the probability outside $B_{\thresh,S'}$,
\begin{eqnarray*}
    \prob_{\hat\theta\sim r_{(\varepsilon,\sens,\thresh,S'')}} 
    ( \hat\theta \not\in B_{\thresh,S'}) &\leq & \prob_{\hat\theta\sim r_{(\varepsilon,\sens,\thresh,S'')}} 
    ( \hat\theta \in B_{\thresh+\sens,S'}\setminus B_{\thresh,S'} )\\
    &\leq& \frac{w_{S''}(B_{\thresh+\sens,S'}\setminus B
    _{\thresh,S'})}{w_{S''}(B_{\thresh,S''})}\\
    &\leq& e^{\varepsilon/2} \frac{w_{S'}(B_{\thresh+\sens,S'}\setminus B
    _{\thresh,S'})}{w_{S'}(B_{\thresh,S''})}\\
    & \leq & e^{\varepsilon/2} \frac{w_{S'}(B_{\thresh+\sens,S'})-w_{S'}(B
    _{\thresh,S'})}{w_{S'}(B_{\thresh-\sens,S'})}\\
    &\leq & e^{\varepsilon/2} (1+2\delta -1) \;=\; 2e^{\varepsilon/2}\delta
    \;.
\end{eqnarray*}
where the first inequality follows from $B_{\thresh,S''}\subseteq B_{\thresh+\sens,S'}$, the second inequality follows from $ (B_{\thresh+\sens,S'}\setminus B_{\thresh,S'})\cap B_{\thresh,S''}\subseteq B_{\thresh+\sens,S'}\setminus B_{\thresh,S'}$, the third inequality follows from the fact that $B_{\thresh,S''} \subseteq B_{\thresh+\sens,S'}$ and the local sensitivity assumption, 
and the last inequality follows from the weight assumption and $B_{\thresh-\sens,S'}\subseteq B_{\thresh,S'}$.

\end{proof}

The next lemma identifies the range of the threshold $k^*=O(\thresh/\sens)$ that ensures safety. 

\begin{lemma}
Under the assumption~\ref{asmp_local}, % and $\delta\in(0,(1/8)e^{-\varepsilon/2})$, 
if there exists a $g>0$ such that 
$\thresh-\sens(k^*+g+1)> 0$ and  
\begin{eqnarray}
\frac{{\rm Vol}(B_{\thresh + \sens (k^*+1),S})}{{\rm Vol}(B_{\thresh - \sens(k^* +g+1),S})} e^{\frac{-\varepsilon g}{4}  }\;\;\leq \;\;\frac{1}{8}e^{-\varepsilon/2}\delta \;, 
\end{eqnarray} 
then $S'\in\safe_{(\varepsilon/2,\delta/2,\thresh)}$ for all $S'$ within Hamming distance $k^*$ from $S$.
\label{lem:safe_vol}
\end{lemma}

\begin{proof}
    Consider $S'$ at Hamming distance $k$ away from $S$. From Lemma~\ref{lem:safe_ratio} 
    it suffices to show that 
    $w_{S'}(B_{\thresh-\sens,S'})/w_{S'}(B_{\thresh+\sens,S'})\geq 1-\delta'$ for $\delta'=(1/8)e^{-\varepsilon/2}\delta$, which is equivalent to 
    $$ w_{S'}(B_{\thresh+\sens,S'}\setminus B_{\thresh - \sens,S'})/w_{S'}(B_{\thresh+\sens,S'}) \leq \delta'\;.$$ 
    The denominator is lower bounded by 
    \begin{eqnarray*}
    w_{S'}(B_{\tau+\sens,S'})\;\geq\; w_{S'}(B_{\tau-\sens(1+g),S'})&\geq& {\rm Vol}(B_{\tau-\sens(1+g),S'})e^{-\varepsilon(\tau-\sens(1+g))/(4\sens)} \\
    &\geq& {\rm Vol}(B_{\tau-\sens(1+g+k),S})e^{-\varepsilon(\tau-\sens(1+g))/(4\sens)}\;,
    \end{eqnarray*}
    where the last inequality uses the local sensitivity  (the assumption~\ref{asmp_local}). 
    The numerator is upper bounded by 
    \begin{eqnarray*}
    w_{S'}(B_{\thresh+\sens,S'}\setminus B_{\thresh - \sens,S'}) \;\leq\; 
    w_{S'}(B_{\thresh+(k+1)\sens,S}\setminus B_{\thresh - \sens,S'}) \;\leq\; 
    {\rm Vol}(B_{\thresh+(k+1)\sens,S})e^{-\varepsilon(\thresh-\sens)/(4\sens)}\;,
    \end{eqnarray*} 
    where the first inequality uses the local sensitivity. 
    Together, it follows that 
    $$\frac{w_{S'}(B_{\thresh+\sens,S'}\setminus B_{\thresh - \sens,S'})}{w_{S'}(B_{\thresh+\sens,S'})} \;\leq\; \frac{{\rm Vol}(B_{\thresh+(k+1)\sens,S})e^{-\varepsilon(\thresh-\sens)/(4\sens)}}{{\rm Vol}(B_{\tau-\sens(1+g+k),S})e^{-\varepsilon(\tau-\sens(1+g))/(4\sens)}}\;\leq\;\delta' \;=\; \frac{1}{8}e^{\varepsilon/2}\delta
    \;,$$ 
    as $e^{-\varepsilon(\thresh-\sens)/(4\sens)}/e^{-\varepsilon(\tau-\sens(1+g))/(4 \sens)}=e^{-\varepsilon g/4}$, which implies safety. 
    
\end{proof}

We next show  that $k^*=O((1/\varepsilon)\log(1/(\delta\zeta)))$ is sufficient to ensure a large enough safety margin of $m_\thresh-k^*=\Omega((1/\varepsilon)\log(1/\zeta))$.

\begin{lemma}
    \label{lem:safetymargin} 
Under the assumptions \ref{asmp_vol}, \ref{asmp_local}, and \ref{asmp_sens} of Theorem~\ref{thm:utility}, for $k^*=(2/\varepsilon)\log(4/(\delta\zeta))$, %$ \sens = O((\thresh\varepsilon)/(d+ \varepsilon + \log(1/\delta)))$,  and $\thresh-\sens(k^*+g+1)\geq 7\rho_1(5\alpha)$, 
%$0<g<(\thresh/\sens)-k^*-1 $, 
%$\zeta\leq e^{-\tau\varepsilon/(24\sens)}$, and $\delta \leq 1/2$, %\min\{1/2, e^{-\tau\varepsilon/(24\sens)}\}$, 
    if $d_H(S',S)\leq (2/\varepsilon)\log(4/(\zeta\delta))$ then $S'\in\safe_{(\varepsilon/2,\delta/2,\thresh )}$. 
\end{lemma}
\begin{proof}
    Applying Lemma~\ref{lem:safe_vol} with $k^*=(2/\varepsilon)\log(4/(\delta\zeta))$ and $g=(1/(8\sens))\thresh$, we require 
\begin{eqnarray*}
\frac{{\rm Vol}(B_{\thresh + \sens (k^*+1),S})}{{\rm Vol}(B_{(7/8)\thresh-\sens(k^*+1),S})} e^{\frac{-\varepsilon \thresh}{32\sens}  }\;\;\leq \;\;\frac{1}{8}e^{-\varepsilon/2}\delta \;. 
\end{eqnarray*} 
From the assumption \ref{asmp_vol}, 
it is sufficient to have 
\begin{eqnarray*}
 \exp\Big\{c_2 p - \frac{\thresh \varepsilon}{32\sens}\Big\}\;\;\leq \;\; \frac{1}{8}e^{-\varepsilon/2}\delta \;. 
\end{eqnarray*}
    For $\sens\leq (\thresh  \varepsilon)/(32(c_2 p + (\varepsilon/2) + \log(8/\delta)))$, which follows from the assumption~\ref{asmp_sens}, this is satisfied. 
\end{proof}

% ----------------------------------------------
 \subsection{Proof of Theorem~\ref{thm:utility}} \label{sec:proof_utility}
% ----------------------------------------------

We first show that we pass the safety test with high probability. 
Define the  error event $E$ as the event that we output $\perp$ in the {\sc Test} step.
%    From Lemma~\ref{lem:sensitivitymargin}, we have $m_\sens>(1/2)\alpha n$. This implies that 
%   \begin{eqnarray*}
 %   \prob(E_1) &=& \prob\big(\,m_\sens + {\rm Lap}(3/\varepsilon)<(3/\varepsilon)\log(3/\delta)\,\big) \nonumber\\
  %  &\leq&  \frac12 \exp\{ -\frac{\varepsilon \alpha n}{6} + \log\frac3\delta \}     \;\leq\; \frac\zeta3\;,
  %  \end{eqnarray*}
  %  for $n=\Omega((1/(\varepsilon \alpha))\log(1/(\delta\zeta)))$. 
    From Lemma~\ref{lem:safetymargin}, we have  $m_\thresh>(2/\varepsilon)\log(4/(\delta\zeta))$ under the assumptions \ref{asmp_vol}, \ref{asmp_local}, and \ref{asmp_sens}. This  implies that 
    \begin{eqnarray*}
        \prob(E ) &=& \prob\big(\, m_\thresh + {\rm Lap}(2/\varepsilon)< (2/\varepsilon)\log(2/\delta)\,\big)
        \;\leq\; \frac{\zeta}{2} \;.
    \end{eqnarray*}

    We next show that resilience implies good utility (once safety test has passed). 
    We want the exponential mechanism to output an accurate $\hat\theta$ near $\theta$ with high probability, i.e., $\prob_{\hat\theta\sim r_{(\varepsilon,\sens,\thresh,S)}}( \popdist(\hat\theta,\theta) \geq c_0 \rho )\leq \zeta/2$. 
    We omit the subscript in the probability for brevity, and it is assumed that randomness is in the sampling of the exponential mechanism. 
    We want to bound by ${\zeta}/{2}$ the failure probability: 
    \begin{eqnarray*}
         \prob\big(\, \popdist(\hat\theta,\theta) \geq c_0\rho \,\big) &\leq& \frac{\prob\big(\,\popdist(\hat\theta,\theta)\geq c_0\rho  \,\big)}{\prob\big(\,\popdist(\hat\theta,\theta)\leq  c_1\rho_1\,\big)} \\
        &\leq& \frac{{\rm Vol}(B_{\thresh,S})}{{\rm Vol}
        (\{\hat\theta:\popdist(\hat\theta,\theta)\leq  c_1\rho\}  ) } 
        \frac{\max_{\hat\theta:\popdist(\hat\theta,\theta)\geq c_0\rho } \prob(\hat\theta)}{\min_{\hat\theta:\popdist(\hat\theta,\theta)\leq  c_1\rho_1  }\prob(\hat\theta)}\;,
    \end{eqnarray*}
as long as $ \{\hat\theta: \popdist(\hat\theta,\theta) \leq c_0\rho \} \subseteq B_{\thresh,S}$ (otherwise we are under-estimating the volume), which follows from the assumption \ref{asmp_resilience}; $ \robdist_S(\hat\theta) \leq (\popdist(\hat\theta,\theta)+c_1\rho) \leq (c_0+c_1)\rho = \thresh$. 

Similarly, since $\hat\theta \in B_{\thresh,S}$ implies $\popdist(\hat\theta,\theta) \leq \thresh+c_1\rho = (c_0+2c_1)\rho$, the volume ratio is bounded by 
\begin{eqnarray*}
    \frac{{\rm Vol}(B_{\thresh,S})}
    {{\rm Vol}(\{\hat\theta:\popdist(\hat\theta,\theta) \leq c_1\rho )} &\leq& 
    \frac{{\rm Vol}(\{\hat\theta:\popdist(\hat\theta,\theta)\leq (c_0+2 c_1) \rho \})}{{\rm Vol}(\{\hat\theta: \popdist(\hat\theta,\theta) \leq c_1\rho\})} 
\;\leq \; e^{c_2p}\;,
\end{eqnarray*}
under the assumption \ref{asmp_vol}. 
The probability ratio can be bounded similarly. 
From the assumption~\ref{asmp_resilience}, 
we have 
\begin{eqnarray*}   
    \frac{\max_{\hat\theta: \popdist(\hat\theta,\theta) \geq c_0\rho} \prob(\hat\theta)}{\min_{\hat\theta:\popdist(\hat\theta,\theta)\leq c_1\rho }\prob(\hat\theta)} \;\leq \;
    \exp\Big\{ -\frac{\varepsilon}{4\sens}\big( (c_0-c_1)-(2c_1)\big)\rho  \Big\}\;\leq\; \exp\Big\{-\frac{\varepsilon(c_0-3c_1)\rho }{4\sens}\Big\}\;.
\end{eqnarray*}
When $ e^{c_2p-(\varepsilon(c_0-3c_1)\rho/(4\sens))})\leq \zeta/2$, we have the desired bound. 
This is guaranteed with our assumption~\ref{asmp_sens}.

% ----------------------------

% ------------------------------------
\section{Auxiliary lemmas}
\label{sec:aux}

\begin{lemma} For any symmetric $\Sigma \succ 0$ and vector $u\in \reals^d$,
\begin{eqnarray}
		\max_{v:\|v\|=1} \frac{\ip{v}{u}}{v^\top \Sigma v} =\left\|\Sigma^{-1/2}u\right\|\;.
\end{eqnarray}
\label{lem:vector_norm}
\end{lemma}
\begin{proof}
 This follows analogously from the proof of Lemma~\ref{lem:equidist}. 
\end{proof}

\begin{lemma}\label{lem:bounded_norm}
	Let $\Sigma, A\in \reals^{d\times d}$ be a symmetric matrix. If $ -c \mathbf{I}_{d \times d} \preceq \Sigma^{-1/2}A\Sigma^{-1/2}-\mathbf{I}_{d \times d}\preceq c \mathbf{I}_{d \times d}$ for some $c>0$, then we have for any $u\in \reals^d$, 
	\begin{eqnarray}
		\|\Sigma^{-1/2}(A-\Sigma)u\|\leq c\|\Sigma^{1/2}u\|\;.
	\end{eqnarray}
\end{lemma}
\begin{proof}
	Using the fact that  $-\mathbf{I}_{d \times d}\preceq M\preceq \mathbf{I}_{d \times d} $ implies $-\mathbf{I}_{d \times d}\preceq M^2\preceq \mathbf{I}_{d \times d}$, for any symmetric matrix $M$, we know \begin{eqnarray}
		 -c^2 \mathbf{I}_{d \times d} \preceq \Sigma^{-1/2}(A-\Sigma)\Sigma^{-1}(A-\Sigma)\Sigma^{-1/2}\preceq c^2 \mathbf{I}_{d \times d}\;,
	\end{eqnarray}
	which implies that 
	\begin{eqnarray}
		 -c^2 \Sigma \preceq (A-\Sigma)\Sigma^{-1}(A-\Sigma)\preceq c^2 \Sigma\;.
	\end{eqnarray}
	
	Thus, we know 
	\begin{eqnarray}
	\|\Sigma^{-1/2}(A-\Sigma)u\|^2 = u^\top (A-\Sigma)\Sigma^{-1}(A-\Sigma) u\leq c^2 u^\top \Sigma u= c^2\|\Sigma^{1/2}u\|^2\;.
	\end{eqnarray}
\end{proof}

\section{Existing lower bounds}
\label{sec:lb}

%\begin{theorem}[Lower bound for sub-Gaussian DP mean estimation \cite{KV17,KLSU19}]
%	Let $\cP_{\mu, \Sigma}$ be the set of sub-Gaussian distributions with mean $\mu\in \reals^d$, covariance $\Sigma\in \reals^{d\times d}$ and sub-Gaussian proxy $\Gamma\leq c_1\Sigma$ for some constant $c_1$.  Let $\cM_{\varepsilon,\delta}$ be a class of $(\varepsilon, \delta)$-DP estimators using $n$ i.i.d.~samples from $P\in{\cal P}_{\mu,\Sigma}$.
%	Then there exists a constant $c$ such that 
%	\begin{eqnarray}
%	\inf_{\hat{\mu}\in \cM_{\varepsilon,\delta}}\sup_{\mu\in \reals^d, \Sigma\succ 0, P\in \cP_{\mu, \Sigma}}\E_{S \sim P^n}[\|\Sigma^{-1/2}(\hat\mu(S)-\mu)\|^2]
%	\; \geq\; c\left(\frac{d}{n}+\frac{d^2}{n^2\varepsilon^2}\right)\;.
%	\end{eqnarray}
%\end{theorem}

%\begin{theorem}[Lower bound for DP Gaussian mean estimation with known covariance {\cite[Theorem~6.5]{KLSU19}}]
%For any $\alpha \in(0, c)$, where $c\in (0,1)$ is some absolute constant, any $(\varepsilon, \delta)$ mechanism (for $\delta\leq \widetilde{O}\left(\sqrt{d}/(Rn)\right)$) which estimate a Gaussian distributions (with mean $\mu\in[-R, R]^d$ and known covariance $\sigma^2\mathbf{I}_{d \times d}$) to accuracy $\leq \alpha$ in total variation distance with probability $\geq 2/3$ requires $n=\Omega\left(d/(\alpha\varepsilon\log(dR))\right)$.
%\end{theorem}

\begin{theorem}[Lower bound for DP Gaussian mean estimation with known covariance {\cite[Lemma~6.7]{KLSU19}}]
Let $\hat{\mu}:\reals^{n\times d}\rightarrow [-R\sigma, R\sigma]^d$ be an $(\varepsilon, \delta)$-differentially private estimator (with $\delta\leq \sqrt{d}/(48\sqrt{2}Rn\sqrt{\log(48\sqrt{2}Rn/\sqrt{d})})$) such that for every Gaussian distribution $P=\cN(\mu, \sigma^2\mathbf{I}_{d \times d})$ (for $-R\sigma\leq \mu_j \leq R\sigma$ where $j\in [d]$) and  
\begin{eqnarray}
	\E_{S\sim P^n}\left[\|\hat{\mu}(S)-\mu\|^2\right]\leq \alpha^2\leq \frac{d\sigma^2R^2}{6}\;,
\end{eqnarray}
then $n\geq \frac{d\sigma}{24\alpha\varepsilon}$.
\end{theorem}

\begin{theorem}[Lower bound for DP covariance bounded mean estimation {\cite[Theorem~6.1]{kamath2020private}}]
Suppose $\hat{\mu}$ is an $(\varepsilon, 0)$-DP estimator such that, for every product distribution $P\in \reals^d$ such that $\E[P]=\mu$, $\sup_{v:\|v\|=1}\E_{x\sim P}[\ip{v}{x-\mu}^2]\leq 1$ and 
\begin{eqnarray}
    \E_{S\sim P^n}\left[\|\hat{\mu}(S)-\mu\|^2\right]\leq \alpha^2\;.
\end{eqnarray}
Then $n=\Omega\left(d/(\varepsilon\alpha^2)\right)$
\end{theorem}

\begin{theorem}[Lower bound on the error rate for hypercontractive linear regression with independent noise{\cite[Theorem~6.1]{bakshi2021robust}}]
Consider linear model $y=\ip{\beta}{x}+\eta$, where optimal hyperplane $\beta$ is used to generate data, and the noise $\eta$ is independent of the samples $x$. Then there exists two distribution $D_1$ and $D_2$ over $\reals^2\times \reals$ such that the marginal distribution over $\reals^2$ has covariance $\Sigma$ and is $(\kappa_k, k)$-hypercontractive yet $\|\Sigma^{1/2}(\beta_1-\beta_2)\|=\Omega(\sqrt{\kappa_k}\gamma\alpha^{1-1/k})$, where $\beta_1$ and $\beta_2$ are the optimal hyperplanes for $D_1$ and $D_2$ respectively, $\gamma<1/\alpha^{1/k}$ and the noise $\eta$ is uniform over $[-\gamma, \gamma]$.
\end{theorem}

\begin{theorem}[Lower bound on the error rate for hypercontractive linear regression with dependent noise{\cite[Theorem~6.2]{bakshi2021robust}}]
There exists two distributions $D_1$, $D_2$ over $\reals^2\times \reals$ such that the marginal distribution over $\reals^2$ has covariance $\Sigma$ and is $\kappa_k, k$-hypercontractive yet $\|\Sigma^{1/2}(\beta_1-\beta_2)\|=\Omega(\sqrt{\kappa_k}\gamma\alpha^{1-2/k})$, where $\beta_1$ and $\beta_2$ are least square solutions for $D_1$ and $D_2$, respectively, $\gamma<1/\alpha^{1/k}$ and the noise is a function of the marginal distribution of $\reals^2$,
\end{theorem}
\begin{theorem}[Lower bound for DP sub-Gaussian linear regression {\cite[Theorem~4.1]{cai2019cost}}]
	Given i.i.d. samples $S=\{(x_i, y_i)\}_{i=1}^n$ drawn from model $y_i=\ip{\beta}{x_i}+\eta_i$, where $\eta_i\sim \cN(0,\gamma^2)$, $\beta\in \Theta=\{\beta\in \reals^d:\|\beta\|\leq 1\}$, $\prob(\|x\|\leq 1)=1$, $\Sigma=\E[xx^\top]$ is diagonal and satisfies $0<1/L<d\lambda_{\min}(\Sigma)\leq d\lambda_{\max}(\Sigma)<L$ for some constant $L=O(1)$. Denote this class of distribution as $\cP_{\gamma,  \Theta, \Sigma}$. Denote $\cM_{\varepsilon, \delta}$ as a class of $(\varepsilon, \delta)$-DP algorithms. Then suppose $\varepsilon\in (0,1)$, $\delta\in (0, n^{-(1+w)})$ for some fixed $w>0$, then there exists a constant such that
	\begin{eqnarray}
		\inf_{\hat\beta\in \cM_{\varepsilon, \delta}}\sup_{\Sigma\succ0, P\in \cP_{\gamma,  \Theta, \Sigma}}\E_{ P^n}\left[\|\Sigma^{1/2}(\hat{\beta}(S)-\beta)\|^2\right]\geq c\gamma^2\left(\frac{d}{n}+\frac{d^2}{n^2\varepsilon^2}\right)\;.
	\end{eqnarray}
\end{theorem}

%For sub-Gaussian DP linear regression 
%
%%\cite[Theorem 4.1]{cai2019cost}
%
%\begin{theorem}[Lower bound for sub-Gaussian  DP linear regression {\cite[Theorem~4.1]{cai2019cost}}]
%	Let $\cP_{\Sigma}$ be the set of zero-mean sub-Gaussian distributions with covariance $\Sigma\in \reals^{d\times d}$ and sub-Gaussian proxy $\Gamma\leq c_1\Sigma$ for some constant $c_1>0$. Let $\cP_{\gamma}$ be the set of zero mean one dimensional sub-Gaussian with covariance $\gamma^2$ and sub-Gaussian proxy $\gamma_0^2\leq c_1\gamma^2$ for some constant $c_1>0$. A multiset of i.i.d.~labeled samples $S=\{(x_i, y_i )\}_{i=1}^n$ is generated from a linear model with noise  $\eta_i$ independent of $x_i$:
%\begin{eqnarray}
%	y_i\;=\;x_i^\top\beta+\eta_i\;,
%\end{eqnarray} 
%where the input $x_i$ and the independent noise $\eta_i$ are i.i.d. samples from $\cP_{\Sigma}$ and $\cP_{\gamma}$. 
%	
%	Let $\cM_{\varepsilon,\delta}$ be a class of $(\varepsilon, \delta)$-DP estimators using $n$ i.i.d.~samples from $\cP_{\Sigma}\times \cP_{\gamma}$. Suppose $0<\varepsilon<1$ and $\delta<n^{}$
%	Then there exists a constant $c$ such that 
%	\begin{eqnarray}
%	\inf_{\hat{\beta}\in \cM_{\varepsilon,\delta}}\sup_{\|\beta\|\leq 1, \Sigma\succ 0, \gamma\geq 0, P\in \cP_{\Sigma}\times \cP_{\gamma}}\E_{S \sim P^n}[\|\Sigma^{1/2}(\hat\beta(S)-\beta)\|^2]
%	\; \geq\; c\gamma^2\left(\frac{d}{n}+\frac{d^2}{n^2\varepsilon^2}\right)\;.
%	\end{eqnarray}
%\end{theorem}
%For hypercontrative linear regression with independent noise 
%
%\cite{bakshi2021robust}
%
%

%Linear regression lower bound of \cite{shamir2015sample} 

\begin{theorem}[Lower bound of linear regression {\cite[Theorem~1]{shamir2015sample}}]
A multiset of i.i.d. samples $S=\{(x_i, y_i)\}_{i=1}^n$ is drawn from distribution $P\in \reals^{d}\times \reals$ in a class $\cP_{B, Y}$, where $|y|\leq Y$, $\|x\|\leq 1$ and $\beta\in \Theta_B=\{\beta\in \reals^d:\|\beta\|\leq B\}$. Then there exists a constant $c$ such that
\begin{eqnarray}
	\inf_{\hat{\beta}\in\Theta_B}\sup_{P\in \cP_{B,Y}}\E_{ P^n}\left[\left(y-\ip{\hat{\beta}(S)}{x}\right)^2-\min_{\beta\in \Theta_B}\left(y-\ip{\beta}{x}\right)^2\right]\geq c\min\left\{Y^2,B^2, \frac{dY^2}{n}, \frac{BY}{\sqrt{n}}\right\}\;.
\end{eqnarray}
	
\end{theorem}

\begin{theorem}[Lower bound of Gaussian DP covariance estimation {\cite[Lemma~6.11]{KLSU19}}]
Let $\widehat{\Sigma}:\reals^{n\times d}\rightarrow \Theta$ be an $(\varepsilon, 0)$-DP estimator (where $\Theta$ is the space of all $d\times d$ PSD matriaces), and for every $\cN(0, \Sigma)$ over $\reals^d$ such that $1/2\mathbf{I}_{d \times d}\preceq \Sigma\preceq 3/2\mathbf{I}_{d \times d}$, 
\begin{eqnarray}
	\E_{S\sim \cN(0, \Sigma)^n}\left[\|\widehat{\Sigma}(S)-\Sigma\|_F^2\right]\leq \frac{\alpha^2}{64}\;,
\end{eqnarray}
then $n\geq \Omega(d^2/(\varepsilon\alpha))$.
	
\end{theorem}

\end{document}